\documentclass[a4paper,11pt]{amsart}
\usepackage{amsmath}
\usepackage{amssymb}
\usepackage{longtable}
\usepackage{enumerate}

\newtheorem{thm}{Theorem}[section]
\newtheorem{lem}[thm]{Lemma}%
\newtheorem{prop}[thm]{Proposition}%
\newtheorem{cor}[thm]{Corollary}%
\theoremstyle{remark}
\newtheorem{definition}{Definition}[section] %
\newtheorem{remark}{Remark}[section] %

\theoremstyle{plain}

\numberwithin{equation}{section}

\def\PP{{\mathbb P}}

\def\RR{{\mathbb R}}
\def\SS{{\mathbb S}}
\def\TT{{\mathbb T}}

\def\XX{{\mathbb X}}
\def\YY{{\mathbb Y}}

\def\mm{{\mathsf m}}

\def\I{\operatorname{I}}

\def\Leb{\operatorname{Leb}}

\def\C{\operatorname{C}}

\def\undef{\operatorname{undef}}
\def\tdef{\operatorname{def}}

\def\Ad{\operatorname{Ad}}

\def\new{{\operatorname{new}}}

\def\UB{{\scrB_1^{d-1}}}

\def\US{{\S_1^{d-1}}}

\def\Pac{P_{\mathrm{ac}}}

\def\veca{{\text{\boldmath$a$}}}
\def\vecb{{\text{\boldmath$b$}}}

\def\vecc{{\text{\boldmath$c$}}}

\def\vece{{\text{\boldmath$e$}}}

\def\vech{{\text{\boldmath$h$}}}

\def\vecell{{\text{\boldmath$\ell$}}}
\def\vecm{{\text{\boldmath$m$}}}
\def\vecn{{\text{\boldmath$n$}}}
\def\vecq{{\text{\boldmath$q$}}}
\def\vecQ{{\text{\boldmath$Q$}}}
\def\vecp{{\text{\boldmath$p$}}}

\def\vecu{{\text{\boldmath$u$}}}
\def\vecv{{\text{\boldmath$v$}}}
\def\vecV{{\text{\boldmath$V$}}}
\def\vecVp{{\text{\boldmath$V$}}_{\!\!+}}
\def\vecVone{{\text{\boldmath$V$}}_{\!\!1}}

\def\vecw{{\text{\boldmath$w$}}}
\def\vecW{{\text{\boldmath$W$}}}

\def\vecx{{\text{\boldmath$x$}}}

\def\vecy{{\text{\boldmath$y$}}}
\def\vecz{{\text{\boldmath$z$}}}
\def\vecalf{{\text{\boldmath$\alpha$}}}
\def\vecbeta{{\text{\boldmath$\beta$}}}

\def\scrA{{\mathcal A}}
\def\scrB{{\mathcal B}}

\def\scrD{{\mathcal D}}
\def\scrE{{\mathcal E}}

\def\scrH{{\mathcal H}}

\def\scrK{{\mathcal K}}
\def\scrL{{\mathcal L}}

\def\scrO{{\mathcal O}}
\def\scrQ{{\mathcal Q}}
\def\scrP{{\mathcal P}}

\def\scrS{{\mathcal S}}
\def\scrT{{\mathcal T}}
\def\scrU{{\mathcal U}}
\def\scrV{{\mathcal V}}

\def\scrX{{\mathcal X}}

\def\scrZ{{\mathcal Z}}

\def\fD{{\mathfrak D}}

\def\fJ{{\mathfrak J}}

\def\fL{{\mathfrak L}}

\def\fQ{{\mathfrak Q}}

\def\fW{{\mathfrak W}}

\def\fZ{{\mathfrak Z}}

\def\diag{\operatorname{diag}}

\def\dim{\operatorname{dim}}

\def\id{\operatorname{id}}

\def\L{\operatorname{L{}}}

\def\M{\operatorname{M}}
\def\I{\operatorname{I}}
\def\GL{\operatorname{GL}}

\def\S{\operatorname{S}}

\def\SL{\operatorname{SL}}
\def\ASL{\operatorname{ASL}}

\def\SO{\operatorname{SO}}

\def\T{\operatorname{T{}}}
\def\Tr{\operatorname{Tr}}

\def\supp{\operatorname{supp}}

\def\vol{\operatorname{vol}}

\def\trans{^{\!\mathsf{T}}}

\def\Onder#1#2#3#4#5{#1 \setbox0=\hbox{$#1$}\setbox1=\hbox{$#2$}
       \dimen0=.5\wd0 \dimen1=\dimen0 \dimen2=\dp0 \dimen3=\dimen2
       \advance\dimen0 by .5\wd1 \advance\dimen0 by -#4
       \advance\dimen1 by -.5\wd1 \advance\dimen1 by -#4
       \advance\dimen2 by -#3 \advance\dimen2 by \ht1
       \advance\dimen2 by 0.3ex \advance\dimen3 by #5
        \kern-\dimen0\raisebox{-\dimen2}[0ex][\dimen3]{\box1}
       \kern\dimen1}

\def\nbar{\overline{n}}

\def\tc{\widetilde{c}}

\def\tvecc{\widetilde{\vecc}}
\newcommand{\oPhi}{\overline{\Phi}}
\newcommand{\tXX}{\widetilde{\XX}}
\newcommand{\ck}{\:^{\circ\hspace{-1pt}}k}
\newcommand{\cp}{\:^{\circ\hspace{-1pt}}p}
\newcommand{\ckappa}{\:^{\circ\hspace{-1pt}}\kappa}

\newcommand{\tscrZ}{\widetilde{\scrZ}}

\newcommand{\tSS}{\widetilde{\SS}}

\newcommand{\tPsi}{\widetilde{\Psi}}
\newcommand{\tpsi}{\widetilde{\psi}}
\newcommand{\tvecV}{\widetilde{\vecV}}
\newcommand{\tG}{\widetilde{G}}

\newcommand{\tGamma}{\widetilde{\Gamma}}
\newcommand{\tgamma}{\widetilde{\gamma}}
\newcommand{\tTT}{\widetilde{\TT}}

\newcommand{\tp}{\widetilde{p}}
\newcommand{\tiota}{\widetilde{\iota}}
\newcommand{\hf}{\widehat{f}}
\newcommand{\g}{\mathsf{g}}
\newcommand{\p}{\mathsf{r}}
\renewcommand{\r}{\mathsf{r}}
\newcommand{\tr}{\widetilde{\mathsf{r}}}
\renewcommand{\a}{\mathsf{a}}
\newcommand{\oomega}{\overline{\omega}}
\newcommand{\tomega}{\widetilde{\omega}}
\newcommand{\toomega}{\widetilde{\overline{\omega}}}

\newcommand{\vs}{\varsigma}

\newcommand{\oA}{\overline{A}}

\newcommand{\tV}{\widetilde{V}}

\newcommand{\oB}{\overline{B}}
\newcommand{\tB}{\widetilde B}

\newcommand{\tP}{\widetilde\scrP}

\newcommand{\bs}{\backslash}

\newcommand{\tL}{\widetilde{L}}

\newcommand{\tfZ}{\widetilde\fZ}

\newcommand{\tlambda}{\widetilde\lambda}
\newcommand{\tmu}{\widetilde\mu}

\newcommand{\hmu}{\widehat\mu}

\newcommand{\tX}{\widetilde X}
\newcommand{\tnu}{\widetilde\nu}

\newcommand{\GaG}{\Gamma\backslash G}

\newcommand{\lsl}{\mathfrak{sl}}

\newcommand{\ig}{\mathfrak{g}}
\newcommand{\is}{\mathfrak{s}}

\newcommand{\ih}{\mathfrak{h}}

\newcommand{\il}{\mathfrak{l}}

\newcommand{\Q}{\mathbb{Q}}
\newcommand{\R}{\mathbb{R}}
\newcommand{\Z}{\mathbb{Z}}

\renewcommand{\mod}{\text{ mod }}

\newcommand{\col}{\: : \:}

\newcommand{\bn}{\mathbf{0}}

\newcommand{\tvecu}{\widetilde{\vecu}}
\newcommand{\tM}{\widetilde{M}}

\newcommand{\txi}{\widetilde{\xi}}

\newcommand{\trho}{{\widetilde{\rho}}}

\newcommand{\tg}{\tilde{g}}

\newcommand{\tf}{\widetilde{f}}

\newcommand{\ve}{\varepsilon}

\newcommand{\matr}[4]{\left( \begin{matrix} #1 & #2 \\ #3 & #4 \end{matrix} \right) }

\newcommand{\smatr}[4]{\bigl( \begin{smallmatrix} #1 & #2 \\ #3 & #4 \end{smallmatrix} \bigr) }
\newcommand{\scmatr}[2]{\left( \begin{smallmatrix} #1 \\ #2 \end{smallmatrix} \right) }

\title{The Boltzmann-Grad limit of the Lorentz gas in a union of lattices}
\author{Matthew Palmer and Andreas Str\"ombergsson}

\begin{document}

\begin{abstract}
The Lorentz gas describes an ensemble of noninteracting point particles in an infinite array of spherical scatterers.
In the present paper we consider the case when the scatterer configuration $\scrP$
is a fixed union of (translated) lattices in $\R^d$,
and %
prove that in the limit of low scatterer density,
the particle dynamics converges to a
random flight process.  %
In the special case when the lattices in $\scrP$ are pairwise incommensurable, this settles a conjecture from \cite{jMaS2013a}.
The proof is carried out 
by applying a %
framework developed in recent work %
by Marklof and Str\"ombergsson
\cite{jMaS2019}, %
and central parts of our proof
are the construction of an admissible marking of the point set $\scrP$, and 
the verification of the uniform spherical equidistribution %
condition required in \cite{jMaS2019}.
Regarding the random flight process %
obtained in the low density limit of the Lorentz gas,
we prove that it can be reconstructed from the
corresponding limiting 
flight processes arising from the %
individual commensurability classes of lattices in $\scrP$. %
We furthermore prove that 
the free path lengths of the limit flight process have a 
distribution with a power law tail,
whose exponent depends on the 
number of commensurability classes in $\scrP$. %
\end{abstract}

\maketitle

\tableofcontents

\section{Introduction}
\label{INTROsec}

The Lorentz gas 
\cite{hL1905} 
describes the dynamics of a cloud of non-interacting 
point particles in an array of fixed spherical scatterers of radius $\rho>0$, centered at the elements of a 
given locally finite point set $\scrP\subset\RR^d$.
Each particle travels with constant velocity along straight lines, 
and each time it hits a scatterer 
it is deflected
by elastic %
reflection
or by a more general (fixed) scattering process.
We denote the position and velocity of a point particle at time $t$
by $\vecq(t)$ and $\vecv(t)$.
Since the particle speed outside the scatterers is a constant of motion
we may without loss of generality consider only point particles
having unit speed, viz., $\|\vecv(t)\|=1$.
This means that the particle dynamics 
takes place in the unit tangent bundle
$\T^1(\scrK_\rho):=\scrK_\rho\times\US$ of the domain
\begin{align*}
\scrK_\rho:=\R^d\setminus(\scrP+\scrB_\rho^d),
\end{align*}
where $\scrB_\rho^d$ denotes the open ball of radius $\rho$, centered at the origin.\label{scrBrhodDEF}
The Liouville measure on $\T^1(\scrK_\rho)$ is
$\vol\times\sigma$,
where $\vol$ denotes %
the Lebesgue measure on $\R^d$
and $\sigma:=\vol_{\US}$ is the Lebesgue measure on $\US$. \label{sigmadef}

Since the gas particles are assumed to be non-interacting, 
to study the evolution of a particle cloud, 
we may just as well %
consider the orbit $t \mapsto (\vecq(t),\vecv(t))$   %
of a single %
point particle 
starting from a random %
point %
$(\vecq_0,\vecv_0)$, 
chosen according to a 
given %
probability measure on the phase space $\T^1(\scrK_\rho)$. %
Then %
$t \mapsto (\vecq(t),\vecv(t))$
becomes a random flight process, which we call 
the \textit{Lorentz process.}
A central challenge is to determine whether,
in the limit of small scatterer density
(that is %
as $\rho\to0$),
the %
Lorentz process
converges to a limiting stochastic process.
In order to give a precise formulation of this question,
we assume from now on that 
$\scrP$ has an asymptotic density,
meaning that there exists a constant
$\nbar_{\scrP}>0$  \label{nbarPDEF1} 
such that for any bounded set $\scrD\subset\RR^d$ 
with boundary of Lebesgue measure zero, we have
\begin{equation}\label{density000}
\lim_{R\to\infty} \frac{ \#(\scrP\cap R \scrD)}{R^d} = \nbar_{\scrP}\vol(\scrD) .
\end{equation}
Then a simple heuristic argument shows that the mean free path length,
i.e.\ the mean time between consecutive collisions,
should be expected to scale as $\rho^{1-d}$ as $\rho\to0$.
It is therefore natural
to consider the so-called \textit{Boltzmann-Grad} scaling,
in which length and time units are rescaled by a factor of $\rho^{1-d}$.
That is, we consider the macroscopic coordinates
\begin{align*}
\bigl(\vecQ(t),\vecV(t)\bigr)
=\bigl(\rho^{d-1}\vecq(\rho^{-(d-1)}t),\vecv(\rho^{-(d-1)}t)\bigr).
\end{align*}
The challenge now is to prove that %
the rescaled random flight process
$t\mapsto(\vecQ(t),\vecV(t))$
converges to a limiting random flight process as $\rho\to0$.

\vspace{5pt}

There are two %
instances where this problem has been fully understood
for some years.
The first is the
case when $\scrP$ is a fixed realisation of a Poisson point process. Here Boldrighini, Bunimovich
and Sinai \cite{cBlByS83} 
proved that the Lorentz process converges to a limit that is consistent with the
linear Boltzmann equation (cf.\ also %
\cite{gG69}
and \cite{hS78}).
In fact the paper 
\cite{cBlByS83} 
is restricted to dimension $d=2$ and hard sphere scatterers,
but it was proved in \cite{jMaS2019}
that the results generalise to general dimensions %
and soft scattering potentials.

The second instance is when the scatterer configuration $\scrP$ equals %
a Euclidean lattice
$\scrL$ of full rank in $\R^d$. %
For this case, Marklof and Str\"ombergsson 
\cite{jMaS2010a,jMaS2011a,jMaS2008,jMaS2010b}
proved convergence of the Lorentz process to a limiting random flight process which in fact 
only depends on the asymptotic density of $\scrL$. %
The limit process is Markovian only on an extended phase
space which, in addition to position and momentum, also includes the impact parameter
and distance to the next collision. The corresponding transport equation is in particular not
consistent with the linear Boltzmann equation. This new transport equation was obtained
independently in dimension $d = 2$ for $\scrP=\Z^2$ by 
Caglioti and Golse 
\cite{eCfG2008,eCfG2010},
subject to a
heuristic assumption that was proved (in any dimension) in \cite{jMaS2011a}. 
In the lattice %
setting, the limit transport process in fact satisfies a superdiffusive
central limit theorem \cite{jMbT2014}, with a mean-square displacement proportional to $t\log t$ 
(where $t$ is time measured in units of the mean collision time), rather than the standard linear scaling
which appears in the case of random scatterer configurations.

\vspace{5pt}

In a recent paper by Marklof and Str\"ombergsson
\cite{jMaS2019}, a general framework was developed
which, under a certain set of hypotheses on
the scatterer configuration $\scrP$, %
allows the proof of convergence of 
the rescaled Lorentz process
$t\mapsto(\vecQ(t),\vecV(t))$
to a limiting
random flight process.
This framework was proved %
to apply when $\scrP$ 
belongs to a certain class of quasicrystals
(this includes in particular the case when $\scrP$ is a general \textit{periodic} point set),
and also in the case when $\scrP$ is a fixed realisation of a Poisson point process
of constant intensity.

Our main goal in the present paper is to prove that the 
framework from \cite{jMaS2019}
also applies %
in the case when the scatterer configuration $\scrP$
is an arbitrary finite union of \textit{grids}. %
(By definition, a `grid' is a translate of a full rank lattice in $\R^d$.)
That is, we will assume that 
\begin{align}\label{Pform0}
\scrP=\bigcup_{i=1}^N\scrL_i,
\end{align}
where each $\scrL_i$ is a grid.
As we will see, this case serves as a nice 
testing ground for the framework developed in 
\cite{jMaS2019}, and exhibits new features
compared with the previous cases for which %
that framework has been proved to apply. %

In the special case when the %
$\scrL_i$ are pairwise \textit{incommensurable}\footnote{Two grids $\scrL$ and $\scrL'$ in $\R^d$ are said to be 
commensurable if there exist ${\delta}>0$ and $\vecv\in\R^d$
such that $\scrL\cap({\delta}\scrL'+\vecv)$ is a grid.\label{commensurableDEF}},
the scatterer configurations  in \eqref{Pform0}
have previously been considered in the paper \cite{jMaS2013a},
where among other things a conjectural description of the
Boltzmann-Grad limit of the Lorentz process was given.
The main result of the present paper settles that conjecture as a special case.
However, %
the general case of $\scrP$ %
as in \eqref{Pform0}, without the %
incommensurability assumption,
is considerably more difficult %
and involves
interesting new phenomena
(see Section \ref{SigmaremarksexSEC}).

\begin{remark}\label{overlapREMARK}
In the case when some scatterers in the family
$\{\vecp+\scrB_\rho^d\col\vecp\in\scrP\}$
\textit{overlap,}
certain technical annoyances appear in the definition of the 
Lorentz process $t\mapsto(\vecq(t),\vecv(t))$.
Let us note that %
overlapping scatterers in general \textit{do} exist %
in the situation studied in the
present paper. 
Indeed, if $\scrP$ is any finite union of grids which is not periodic,
then for every $\rho>0$ there exist points $\vecp\neq\vecp'$ in $\scrP$ with 
$\|\vecp-\vecp'\|<2\rho$.
In the classical case when the particles interact with the scatterers
through specular reflection,
this issue is handled in a standard manner \cite[Ch.\ 1.2]{jMaS2019}: in this case the
Lorentz flow equals the standard billiard flow in the region $\scrK_\rho$,
and this flow is technically defined 
only on a subset of $\T^1(\scrK_\rho)$
of full measure with respect to the Liouville measure $\vol\times\sigma$;
the exceptional points include all points $(\vecq,\vecv)\in\partial\scrK_\rho\times\US$
for which $\vecv$ points into a scatterer,
and also, in the case of scatterer overlaps,
any initial condition for which the particle at some time point either in the past or in the future collides with
an intersection point of two or more scatterer boundaries.
In the case of scatterer interaction through %
a more general scattering map,
for example a scattering process generated by a spherically symmetric potential,
a simple way to avoid intricacies in the definition of the Lorentz flow is to
\textit{remove} an arbitrary %
maximal subset of 
scatterer centers in $\scrP$ causing overlap \cite[Ch.\ 1.3]{jMaS2019}.
Since the probability of the particle hitting a scatterer which 
overlaps with %
another scatterer %
tends to zero in the Boltzmann-Grad limit $\rho\to0$,\footnote{This follows from the conclusion
$\Lambda(\fW(n;\rho))\to1$ in
\cite[Remark 4.3]{jMaS2019},
since $\fW(n;\rho)$ is by definition 
(see
\cite[(4.44) and end of Section 3.4]{jMaS2019})
the set of initial conditions, in macroscopic coordinates, which are such that the first $n$ collisions of the point particle
only involve \textit{separated} scatterers, i.e.\ scatterers which do not overlap any other scatterer.}
the limiting random flight process becomes the same independently of 
the choice of convention
and the choice of which overlapping scatterers to remove.
\end{remark}

\begin{remark}
The methods developed in the present paper
can be expected to be useful also for
the treatment of a Lorentz gas in
a \textit{polycrystal},
in the case when the typical grain size is comparable to the mean free path length,
and with a periodic scatterer configuration inside each grain;
cf.\ \cite{jMaS2015}.
Note however that a rigorous treatment of this case
(from the point of view 
of mimimal randomness of the scatterer configuration,
as adopted in the present paper as well as in e.g.\ \cite{cBlByS83},
and \cite{jMaS2019})
would require first extending the 
framework developed in \cite{jMaS2019}.
Another case where the methods developed in this paper
may prove useful\footnote{In combination with
the unfolding method for polyhedral billiards.}
is that of a %
periodic Lorentz gas contained inside
a \textit{polyhedral container}, 
assuming that the point particles are specularly reflected
whenever hitting the boundary of the container.
\end{remark}

\subsection{The limiting flight process}
\label{LIMITFLIGHTPROCESS}
We will assume throughout the %
paper that the fixed
scattering process of the Lorentz gas %
satisfies the conditions in
\cite[Sec.\ 3.4]{jMaS2019}.
Special cases %
include the case of hard sphere scatterers
as well as a general class of soft scattering potentials.

The main result of the present paper is
Theorem \ref{MAINTHM} in Section \ref{KINTHEORYrecapsec}, %
which states that all the assumptions required in the
framework in \cite{jMaS2019}
are satisfied in the case when $\scrP$ is an arbitrary finite union of grids.  
By the results of \cite{jMaS2019},
this implies the existence of the Boltzmann-Grad limit of the particle dynamics 
in the Lorentz gas.
To make a precise statement,   %
let us write $F_{t,\rho}$ for the rescaled Lorentz flow;
\begin{align*}
\bigl(\vecQ(t),\vecV(t)\bigr)=F_{t,\rho}\bigl(\vecQ(0),\vecV(0)\bigr).
\end{align*}
For notational reasons we extend the dynamics to the inside of each scatterer trivially,
that is, set $F_{t,\rho}=\id$ whenever $\vecQ$ is inside the scatterer.
Thus $F_{t,\rho}$ is now a flow defined on all of $\T^1(\R^d)$, the unit tangent bundle of $\R^d$.
Let $\Pac(\T^1(\R^d))$ be the set of Borel probability measures on $\T^1(\R^d)$
which are absolutely continuous 
with respect to the Liouville measure $\vol\times\sigma$.

The following theorem 
is a consequence of 
Theorem \ref{MAINTHM} below,
in combination with %
the main results of \cite{jMaS2019}
and in particular 
\cite[Sec.\ 4.5]{jMaS2019}.
\begin{thm}\label{limitflightprocessTHM}
Let the scatterer configuration $\scrP$ be a finite union of grids.
Then, for any $\Lambda\in\Pac(\T^1(\R^d))$,
there exists a random flight process $\Xi=\{\Xi(t)\col t\geq 0\}$
such that the random process
\begin{align*}
\Xi^{(\rho)}:t\mapsto\Xi^{(\rho)}(t)=F_{t,\rho}(\vecQ_0,\vecV_0)
\end{align*}
obtained by taking the initial data $(\vecQ_0,\vecV_0)$ random with respect to $\Lambda$,
converges to $\Xi$ in distribution, as $\rho\to0$.
\end{thm}

The next result, Theorem \ref{limitflightprocessexplTHM},
gives a %
description
of the limiting process $\Xi$ in Theorem \ref{limitflightprocessTHM}.
The key point is that, after introducing a certain \textit{marking} of the points in $\scrP$,
the process $\Xi$ %
can be described as the flow with unit speed along a random piecewise linear curve, whose
path segments, %
when considered in combination with the marks of the scatterers involved in the collisions,
are generated by a Markov process with memory two.
This is a special case of a corresponding result in
\cite{jMaS2019};
in the setting of the present paper with $\scrP$ being a finite union of grids,
the set of marks can be taken to be a certain concrete \textit{finite} set,
which we now turn to describe.

From now on, we will express $\scrP$ as
\begin{align}\label{Ppsidec1}
\scrP=\bigcup_{\psi\in\Psi}\scrL_\psi,
\end{align}
where $\Psi$ is a finite set of indices, and each $\scrL_\psi$ is a grid in $\R^d$.
In the statement of Theorem~\ref{limitflightprocessexplTHM}
we must require that the above presentation of $\scrP$ is \textit{admissible.}
This is a somewhat technical notion which we will define in 
Section \ref{SSjpsi0Ojpsi0sec};
we prove in Section \ref{ATTAINADMsec} that every finite union of grids possesses an admissible presentation.
We give a couple of simple examples:
If the grids $\scrL_\psi$ are pairwise incommensurable,
then the presentation in \eqref{Ppsidec1} is always admissible.
On the other hand,
for any $q\in\Z_{\geq2}$ and $\vecv\in\R^d\setminus \Z^d$,
the presentation $\scrP=\Z^d\cup (q\Z^d+\vecv)$
is admissible only if $\vecv\notin\Q^d$;
but in any case, regardless of whether $\vecv$ belongs to $\Q^d$ or not,
an admissible presentation of $\scrP=\Z^d\cup (q\Z^d+\vecv)$ 
can be given as the union of the grids $q\Z^d+\vecalf$ with $\vecalf$ running through $\{1,\ldots,q\}^d\sqcup\{\vecv\}$.
However, the discussion of admissibility in the general case is considerably more complicated
(see in particular the proof of Lemma \ref{ADMproplem2new}
and the comments at the end of Section \ref{ATTAINADMsec}).

Assuming that \eqref{Ppsidec1} 
is an admissible presentation of $\scrP$,
let us now fix a choice of a function
\begin{align}\label{psimarkdef}
\psi:\scrP\to\Psi,
\qquad\text{subject to $\vecp\in\scrL_{\psi(\vecp)}\:$ ($\forall\vecp\in\scrP$).}
\end{align}
We will call this function %
a \textit{crude marking}\footnote{See Section \ref{outlineSEC} regarding why we call it ``crude''.}
of $\scrP$.
Note that the condition
in \eqref{psimarkdef}
forces the choice of $\psi(\vecp)$ for 
every $\vecp\in\scrP$ 
which does not lie in more than one of the sets $\scrL_\psi$.
That is, %
the flexibility involved in the choice of the crude marking %
$\psi:\scrP\to\Psi$
concerns only the points in the union
$\bigcup_{\psi\neq\psi'\in\Psi}(\scrL_\psi\cap\scrL_{\psi'})$.
It will turn out that %
the exact choice of the mark $\psi(\vecp)$ at these points has no %
influence on the description of the limiting flight process.

\vspace{5pt}

For each $\psi\in\Psi$,
let $\nbar_\psi$
be the asymptotic density of $\scrL_\psi$. \label{nbarDEF0}
It follows that 
the asymptotic density of $\scrP$
is given by
$\nbar_{\scrP}:=\sum_{\psi\in\Psi}\nbar_{\psi}.$ \label{nbarPDEF0} 
We equip $\Psi$ with the probability measure $\mm$ defined by
\begin{align}\label{mmpsiDEF}
\mm(\psi)=\frac{\nbar_\psi}{\nbar_\scrP}
\qquad(\psi\in\Psi).
\end{align}
Thus $\mm(\psi)$ gives the relative density of scatterers belonging to $\scrL_{\psi}$.
(We are abusing notation for
increased readability; naturally, ``$\mm(\psi)$'' 
denotes $\mm(\{\psi\})$.)

\vspace{5pt}

The following result
follows from %
\cite[Theorem 4.6]{jMaS2019},
in view of Theorem \ref{MAINTHM} below.
\begin{thm}\label{limitflightprocessexplTHM}
Let $\scrP=\bigcup_{\psi\in\Psi}\scrL_\psi$ be an admissible presentation of $\scrP$,
and let $\psi:\scrP\to\Psi$
be a corresponding crude marking.
Let $\Lambda\in\Pac(\T^1(\R^d))$.
For $(\vecQ_0,\vecV_0)$ random with respect to $\Lambda$,
let the corresponding random trajectory
$t\mapsto F_{t,\rho}(\vecQ_0,\vecV_0)$
be described by
the random variables
$\xi_j^{(\rho)}\in\R_{>0}$,
$\psi_j^{(\rho)}\in\Psi$
and $\vecV_j^{(\rho)}\in\US$,
where $\xi_j^{(\rho)}$ is the length of the $j$th path segment,
$\psi_j^{(\rho)}$ is the mark of the scatterer involved in the $j$th collision
and $\vecV_j^{(\rho)}\in\US$ is the velocity after the $j$th collision.
Then as $\rho\to0$, the random process
\begin{align*}
\bigl(\big\langle\xi_j^{(\rho)},\psi_j^{(\rho)},\vecV_j^{(\rho)}\big\rangle\bigr)_{j=1,2,\ldots}
\end{align*}
converges in distribution to the second-order Markov process
\begin{align}\label{limitingMarkovprocess}
\bigl(\big\langle\xi_j,\psi_j,\vecV_j\big\rangle\bigr)_{j=1,2,\ldots},
\end{align}
where for any Borel set $A\subset\R_{>0}\times\Psi\times\US$,
\begin{equation}\label{Markovprocexpl1}
\PP\Big( \langle\xi_1,\psi_1,\vecV_1\rangle \in A \,\Big|\, (\vecQ_0,\vecV_0) \Big) 
=\int_{A} p^{(\psi)}(\vecV_0;\xi,\vecV)\,  d\xi\,d\mm(\psi)\,d\sigma(\vecV),
\end{equation}
and for $j\geq 2$,
\begin{align}\label{Markovprocexpl2}
\PP\Big( \langle\xi_j,\psi_j,\vecV_j\rangle \in A \,\Big|\, (\vecQ_0,\vecV_0),\; \big\langle (\xi_i,{\psi}_i,\vecV_i)\big\rangle_{i=1}^{j-1} \Big) 
\hspace{60pt}
\\\notag
= \int_{A} p^{(\psi_{j-1}\to\psi)}(\vecV_{j-2},\vecV_{j-1};\xi,\vecV)\,  d\xi\,d\mm(\psi)\,d\sigma(\vecV).
\end{align}
The functions $p^{(\psi)}$ and $p^{(\psi'\to\psi)}$
depend on $\scrP$ but are independent of $\Lambda$,
and for any fixed $\vecV',\vecV'',\psi'$,
both $p^{(\psi)}(\vecV';\xi,\vecV)$
and $p^{(\psi'\to\psi)}(\vecV'',\vecV';\xi,\vecV)$
are probability densities on $\R_{>0}\times\Psi\times\US$
with respect to the measure 
$d\xi\times d\mm(\psi)\times d\sigma(\vecV)$.
\end{thm}

We call the functions $p^{(\psi)}$ and $p^{(\psi'\to\psi)}$ 
\textit{collision kernels.}
Explicit formulas for these will be given in Section~\ref{transkerdefSEC},
as pushforwards of invariant measures 
on certain homogeneous spaces %
associated to the given point set $\scrP$. %
Both $p^{(\psi)}$ and $p^{(\psi'\to\psi)}$
are rotationally invariant %
in the sense that
$p^{(\psi)}(\vecV'K;\xi,\vecV K)=p^{(\psi)}(\vecV';\xi,\vecV)$
and %
$p^{(\psi'\to\psi)}(\vecV''K,\vecV'K;\xi,\vecV K)
=p^{(\psi'\to\psi)}(\vecV'',\vecV';\xi,\vecV)$
for any $K\in\SO(d)$.
We also mention that by definition we have
$p^{(\psi)}(\vecV';\xi,\vecV)=0$ 
unless $\vecV\in\scrV_{\vecV'}$,
and $p^{(\psi'\to\psi)}(\vecV'',\vecV';\xi,\vecV)=0$
unless $\vecV'\in\scrV_{\vecV''}$ and $\vecV\in\scrV_{\vecV'}$.
Here, for any $\vecV\in\US$,
$\scrV_{\vecV}$ is the set of possible exit velocities
in a scatterer collision with entry velocity $\vecV$
\cite[Secs.\ 3.4--5]{jMaS2019}.

\vspace{5pt}

The limiting random flight process $\Xi$ %
discussed above
is closely related to the dynamics of a particle cloud 
in the Boltzmann-Grad limit of the Lorentz gas.
Indeed, the time evolution of an initial particle density $f_0\in\L^1(\T^1(\R^d))$
in the Lorentz gas with fixed scatterer radius $\rho$ is given by
$f_t=L_\rho^t f_0$ where $L_\rho^t$ is the Liouville operator defined by
\begin{align*}
[L_\rho^tf_0](\vecQ,\vecV):=f_0(F_{-t,\rho}(\vecQ,\vecV)).
\end{align*}
The existence of the limiting stochastic process $\Xi(t)$ implies that for every $t>0$
there exists a %
linear operator $L_t:\L^1(\T^1(\R^d))\to\L^1(\T^1(\R^d))$
such that for every $f_0\in\L^1(\T^1(\R^d))$ and every set $A\subset\T^1(\R^d)$
with boundary of Liouville measure zero,
\begin{equation*}
\lim_{\rho\to 0} \int_{A}  [L_t^{(\rho)}f_0](\vecQ,\vecV) \, d\vecQ\,d\sigma(\vecV) =
 \int_{A}  L_t f_0(\vecQ,\vecV) \, d\vecQ\,d\sigma(\vecV)
\end{equation*}
\cite[Cor.\ 1.4]{jMaS2019}.
Under suitable continuity assumptions 
we can in fact express $L_tf_0$ as
\begin{align*}
[L_tf_0](\vecQ,\vecV)=
\int_{\R_{>0}\times\Psi\times\US}f^{(\psi)}(t,\vecQ,\vecV,\xi,\vecVp)\,d\xi\,d\mm(\psi)\,d\sigma(\vecVp),
\end{align*}
where the functions $f^{(\psi)}(t,\vecQ,\vecV,\xi,\vecVp)$ %
form the unique solution of the system of differential equations
\begin{align}\notag
\bigl(\partial_t+\vecV &\cdot\nabla_\vecQ-\partial_{\xi}\bigr)
f^{(\psi_+)}(t,\vecQ,\vecV,\xi,\vecVp) 
\\\label{genlinBOLTZMANN}
&=\int_{\Psi\times\US} f^{(\psi)}\bigl(t,\vecQ,\vecV_0,0,\vecV\bigr)
\,p^{(\psi\to\psi_+)}(\vecV_0,\vecV;\xi,\vecVp)\,d\mm(\psi)\,d\sigma(\vecV_0),
\end{align}
subject to the initial condition
$f^{(\psi_+)}(0,\vecQ,\vecV,\xi,\vecVp)=f_0(\vecQ,\vecV)\, p^{(\psi_+)}(\vecV;\xi,\vecVp).$
See \cite[Sections 1.4 and 4.6]{jMaS2019}.

Equation \eqref{genlinBOLTZMANN}
may be viewed as a generalization of the linear Boltzmann equation;
it is the forward Kolmogorov equation
(or Fokker-Planck-Kolmogorov equation)
of a natural extension of $\Xi$ to a 
Markov flight process on the space $\T^1(\R^d)\times\R_{>0}\times\Psi\times\US$
\cite[(1.24)]{jMaS2019}.

In the special case when the grids $\scrL_i$ in
\eqref{Pform0} are pairwise incommensurable,
the limiting random flight process
described in Theorem \ref{limitflightprocessexplTHM},
as well as the generalized linear Boltzmann equation in 
\eqref{genlinBOLTZMANN},
agree with the corresponding
limits conjectured in
\cite{jMaS2013a};
see Remark \ref{oldprodformulacompareREM} below.

\subsection{Expressing $\Xi$ %
via the commensurability classes} %
\label{prodformulasSEC}
An important result of the present paper is that 
the collision kernels can be expressed as 
products over collision kernels associated to 
the commensurability classes of 
the grids appearing in $\scrP$.
In the special case when the grids in $\scrP$ are pairwise incommensurable,
such product formulas were given in
\cite[Sec.\ 5]{jMaS2013a}.

Recall that two grids $\scrL$ and $\scrL'$ in $\R^d$ are said to be 
\textit{commensurable} if there exist ${\delta}>0$ and $\vecv\in\R^d$
such that $\scrL\cap({\delta}\scrL'+\vecv)$ is a grid.
This is an equivalence relation on the family of grids in $\R^d$.
Let us say that two marks $\psi,\psi'\in\Psi$ 
are \textit{equivalent}
(and write $\psi\sim\psi'$) when
$\scrL_{\psi}$ and $\scrL_{\psi'}$ are commensurable.
For any $\psi\in\Psi$ we denote by
$[\psi]:=\{\psi'\in\Psi\col\psi'\sim\psi\}$  %
the corresponding 
equivalence class.

Let $C_{\Psi}$ be the family of equivalence classes in $\Psi$, \label{CPsidef}
and set, for any $c\in C_{\Psi}$,
\begin{align}\label{Pcdef}
\scrP_c:=\bigcup_{\psi\in\, c}\scrL_{\psi}.
\end{align}
Note that $\scrP_c$ is a subset of $\scrP$,
and $\scrP_c$ is itself
a finite union of grids, 
so that all the results discussed so far
also hold for the Lorentz gas \textit{with $\scrP_c$}
as scatterer configuration.
The asymptotic density of $\scrP_c$ is
\begin{align}
\nbar_c:=\sum_{\psi\in\,c}\nbar_\psi,
\end{align}
and it will be immediate from our definition of admissibility that
for any $c\in C_\Psi$, 
the presentation $\scrP_c=\cup_{\psi\in c}\scrP_\psi$ is admissible;
hence Theorem \ref{limitflightprocessexplTHM} applies %
to the scatterer configuration $\scrP_c$
with this presentation. %
For any two $\psi,\psi'\in c$,    %
let us denote by
$\cp^{(\psi)}$
and $\cp^{(\psi'\to\psi)}$
the collision kernels for the corresponding limiting flight process.\label{cpcollkernelDEF}
Thus $\cp^{(\psi)}$ is defined for any $\psi\in\Psi$, 
and $\cp^{(\psi'\to\psi)}$
is defined for any $\psi',\psi\in\Psi$ satisfying $\psi'\sim\psi$.
Also, for any $c\in C_\Psi$, letting $\mm_c$ be the probability measure on $c$ defined by
\begin{align}\label{mcDEF}
\mm_c(\psi)=\frac{\nbar_\psi}{\nbar_c},
\end{align}
we have that for any fixed $\vecV',\vecV''$ and $\psi'\in c$,
both $\cp^{(\psi)}(\vecV';\xi,\vecV)$ and 
$\cp^{(\psi'\to\psi)}(\vecV'',\vecV';\xi,\vecV)$
are probability densities on $\R_{>0}\times c\times\US$ 
with respect to the measure $d\xi\times d\mm_c(\psi)\times d\sigma(\vecV)$.

The product formula for collision kernels is stated in
Theorem \ref{Productformulasthm} below
(see also Remark~\ref{Productformulasthm2}).
Here we give an equivalent formulation
saying that the Markov process which generates $\Xi$
can be obtained by \textit{merging} %
in a certain way mutually independent versions of the Markov
processes which arise %
in the Boltzmann-Grad limit for Lorentz gases which scatterer configurations $\scrP_c$.
To prepare for the statement, let $\vecV',\vecV''$ be arbitrary, fixed  %
vectors in $\US$ subject to $\vecV'\in\scrV_{\vecV''}$,
that is, $\vecV'$ is a possible exit velocity in a scatterer collision with entry velocity $\vecV''$.
For every $c\in C_\Psi$,
let $\langle \txi_{c},\tpsi_{c},\tvecV_{c}\rangle$ 
be a random variable 
in $\R_{>0}\times c\times\US$,  %
with distribution %
\begin{align}\label{Markovprocexpl1forPc}
\cp^{(\psi)}(\vecV';\xi,\vecV)\,  d\xi\,d\mm_c(\psi)\,d\sigma(\vecV).
\end{align}
Also for every 
$\psi'\in\Psi$, let
$\langle \txi_{[\psi'],\psi'},\tpsi_{[\psi'],\psi'},\tvecV_{[\psi'],\psi'}\rangle$
be a random variable
in $\R_{>0}\times [\psi']\times\US$,  %
with distribution
\begin{align}\label{Markovprocexpl2forPc}
\cp^{(\psi'\to\psi)}(\vecV'',\vecV';\xi,\vecV)\,  d\xi\,d\mm_{[\psi']}(\psi)\,d\sigma(\vecV).
\end{align}
We take all these random triples to be 
mutually independent. %
Furthermore, for any $\psi'\in\Psi$ and $c\in C_{\Psi}$ with $\psi'\notin c$, we set
$\langle \txi_{c,\psi'},\tpsi_{c,\psi'},\tvecV_{c,\psi'}\rangle
:=\langle \txi_{c},\tpsi_{c},\tvecV_{c}\rangle$.
Thus the random variable 
$\langle \txi_{c,\psi'},\tpsi_{c,\psi'},\tvecV_{c,\psi'}\rangle$ is now defined for all $\psi'\in\Psi$ and all $c\in C_\Psi$.
Note that 
\eqref{Markovprocexpl1forPc} and \eqref{Markovprocexpl2forPc}
are exactly the distributions which appear
in \eqref{Markovprocexpl1} and \eqref{Markovprocexpl2}
when Theorem \ref{limitflightprocessexplTHM} is applied 
\textit{to $\scrP_c$.}
Next let the random variable $\tc$ be equal to
that $c\in C_\Psi$ which \textit{minimizes}
$\txi_{c}$, and let 
the random variable $\tc_{\psi'}$ 
be equal to that $c$ which minimizes $\txi_{c,\psi'}$.
(These $\tc$ and $\tc_{\psi'}$ are almost surely uniquely defined.)
\begin{thm}\label{indepmergeTHM}
For given $\vecV''\in\US$ and $\vecV'\in\scrV_{\vecV''}$
and $\psi'\in\Psi$,
introduce the random variables
$\langle \txi_{c},\tpsi_{c},\tvecV_{c}\rangle$
and $\langle \txi_{c,\psi'},\tpsi_{c,\psi'},\tvecV_{c,\psi'}\rangle$
(for all $c\in C_\Psi$)
and $\tc$ and $\tc_{\psi'}$ as above.
Then the conditional distributions appearing in 
\eqref{Markovprocexpl1} and \eqref{Markovprocexpl2} of
Theorem \ref{limitflightprocessexplTHM}
are given by the following formulas:
For any Borel set $A\subset\R_{>0}\times\Psi\times\US$,
\begin{align}\label{indepmergeCORres1}
\PP\Big( \langle\xi_1,\psi_1,\vecV_1\rangle \in A \,\Big|\, (\vecQ_0,\vecV_0),\:\vecV_0=\vecV' \Big) 
=\PP\Big(\langle \txi_{\tc},\tpsi_{\tc},\tvecV_{\tc}\rangle\in A \Big),
\end{align}
and for $j\geq 2$,
\begin{align}\notag
\PP\Big( \langle\xi_j,\psi_j,\vecV_j\rangle \in A \,\Big|\, (\vecQ_0,\vecV_0),\; \big\langle (\xi_i,{\psi}_i,\vecV_i)\big\rangle_{i=1}^{j-1},
\: \vecV_{j-2}=\vecV'',\: \vecV_{j-1}=\vecV',\:\psi_{j-1}=\psi'\Big) \hspace{20pt}
\\\label{indepmergeCORres2}
= \PP\Big( \langle \txi_{\tc_{\psi'},\psi'},\tpsi_{\tc_{\psi'},\psi'},\tvecV_{\tc_{\psi'},\psi'}\rangle \in A  \Big).
\end{align}
\end{thm}
Note that Theorem \ref{indepmergeTHM} is a tautology
when $\#C_{\Psi}=1$, i.e.\ when all the grids $\scrL_\psi$ ($\psi\in\Psi$) are commensurable;
but when $\#C_{\Psi}>1$,
Theorem \ref{indepmergeTHM} 
in effect expresses the  %
collision kernel
of the limiting flight process for the scatterer configuration $\scrP$
in terms of the collision kernels
for the subconfigurations $\scrP_c$, $c\in C_{\Psi}$
(again, see Remark \ref{Productformulasthm2} below).

\vspace{5pt}

We next point out a consequence of 
Theorem \ref{indepmergeTHM}
concerning the limiting free path length distribution in the 
Boltzmann-Grad limit. %
Set, for any $\vecV'\in\US$,
\begin{align}\label{PhiPdef}
\Phi_{\scrP}(\xi)=\int_{\Psi\times\US}p^{(\psi)}(\vecV';\xi,\vecV)\,d\mm(\psi)\,d\sigma(\vecV).
\end{align}
The right hand side in \eqref{PhiPdef}
is independent of $\vecV'$ in view of the 
rotational invariance of $p^{(\psi)}$,
and $\Phi_{\scrP}(\xi)$ is a probability density on $\R_{>0}$ with respect to
Lebesgue measure.
In fact $\Phi_{\scrP}(\xi)$ is the density function for the free path length in the 
limiting flight process $\Xi$
when starting from a generic point inside the billiard domain.
We will show in Section~\ref{transkerSEC}
that the function $\Phi_{\scrP}(\xi)$ is %
continuous, decreasing, and $0\leq\Phi_{\scrP}(\xi)\leq\nbar_{\scrP}\,v_{d-1}$
for all $\xi\in\R_{>0}$, where $v_{d-1}:=\vol(\UB)$. \label{vdm1DEF}
Note that for any $c\in C_\Psi$, the formula \eqref{PhiPdef} applied to $\scrP_c$ reads
\begin{align}\label{PhiPcformula}
\Phi_{\scrP_c}(\xi)=\int_{c\times\US}\cp^{(\psi)}(\vecV';\xi,\vecV)\,d\mm_c(\psi)\,d\sigma(\vecV).
\end{align}
We now have:
\begin{cor}\label{ProductformulasthmCOR}
\begin{align*}
\Phi_{\scrP}(\xi)=
-\frac d{d\xi}\prod_{c\in C_{\Psi}}\int_\xi^\infty\Phi_{\scrP_{c}}(\xi')\,d\xi'.
\end{align*}
\end{cor}
In the special case when the grids in $\scrP$ are pairwise incommensurable,
the formula in Corollary \ref{ProductformulasthmCOR} was given in
\cite[(2.8)]{jMaS2013a}.

\subsection{Asymptotic estimates for the free path length distribution}

We will %
prove the following upper and lower bounds
showing that the free path length distribution
in the Boltzmann-Grad limit 
has a power-law tail.
\begin{thm}\label{AsymptTHM}
Let $\scrP$ be a finite union of grids,
and let $N=\# C_\Psi$ be the number of commensurability classes 
of grids appearing in $\scrP$.
Then there exist constants $0<c_1<c_2$
such that
\begin{align}\label{genoPhiPhirel7disc4}
c_1\xi^{-(N+1)}<\Phi_{\scrP}(\xi)< c_2\xi^{-(N+1)},
\qquad\forall \xi\geq1.
\end{align}
\end{thm}
In particular it follows that the expected value of the free path length
from starting at a generic point, $\int_0^\infty\xi\Phi_{\scrP}(\xi)\,d\xi$,
is finite if and only if $N\geq2$.

In the special case when the 
grids %
in $\scrP$ are pairwise incommensurable,
a more precise asymptotic formula for
$\Phi_{\scrP}(\xi)$ as $\xi\to\infty$
was given in
\cite[Theorem 2]{jMaS2013a};
the proof uses the product formula and the
asymptotics for $\Phi_{\scrP}(\xi)$ in the special case when $\scrP$ is a lattice,
which was obtained in 
\cite[Theorem 1.13]{jMaS2010b}.
It would be interesting to try to generalize this asymptotic formula,
and also the precise asymptotic estimates for the collision kernels
obtained in \cite{jMaS2010b}, 
to the case when $\scrP$ is a general union of grids as considered in the present paper.

We remark that %
in the proof of Theorem \ref{AsymptTHM},
the critical case to handle is $N=1$,
as the case of general $N\geq2$ 
then follows immediately %
via the product formula in Corollary \ref{ProductformulasthmCOR}.

Next we point out a consequence of
Theorem \ref{AsymptTHM}
concerning the density function 
$\oPhi_{\scrP}(\xi)$ 
of the free path length between consecutive collisions
in the random flight process $\Xi$
describing the Boltzmann-Grad limit of the Lorentz gas.
We give an explicit formula 
for $\oPhi_{\scrP}(\xi)$
in \eqref{genoPhiPhirel5} in Section \ref{transkerSEC};
here we merely point out the relations
\begin{align}\label{PhioPhirel}
\Phi_{\scrP}(\xi)=\nbar_{\scrP}\,v_{d-1}\int_\xi^\infty\oPhi_{\scrP}(\xi')\,d\xi'
\end{align}
and 
\begin{align}\label{genoPhiPhirel6}
\int_0^\infty\xi\,\oPhi_{\scrP}(\xi)\,d\xi=\frac1{\nbar_{\scrP}\,v_{d-1}}.
\end{align}
Note that the value in \eqref{genoPhiPhirel6} is
the expected value of the free path length between consecutive collisions.
The following result is an easy consequence of 
Theorem \ref{AsymptTHM} and the relation \eqref{PhioPhirel}.
\begin{cor}\label{Asympt2momentCOR}
The second moment of the free path length between consecutive collisions,
$\int_0^\infty\xi^2\,\oPhi_{\scrP}(\xi)\,d\xi$,
is finite if and only if $N\geq2$.
\end{cor}

It is interesting to ask about the long-time asymptotics of the limiting random flight
process $t\mapsto\Xi(t)$;
we expect that %
this process converges to Brownian motion
under an appropriate rescaling.
In the special case when $\scrP$ is a lattice,
such a convergence result has been proved in
\cite{jMbT2014};
the correct scaling factor in this case is $t\log t$ (superdiffusion).
In view of Theorem~\ref{AsymptTHM} and 
Corollary \ref{Asympt2momentCOR}, %
it is natural to guess
that a similar superdiffusive central limit result holds
whenever $\scrP$ is a finite union of \textit{commensurable} grids ($N=1$);
however if the grids in $\scrP$ are not all commensurable ($N\geq2$),
we expect that a standard
diffusive central limit result holds, i.e.\ 
with the rescaling factor being linear in $t$.

\subsection*{Acknowledgement}
We are grateful to Jens Marklof for valuable comments and suggestions.
We are also grateful to the referees for several valuable suggestions 
helping to improve %
exposition of the paper.

\section{The  hypotheses from [\ref{jMaS2019label}] on the scatterer configuration $\scrP$}  %
\label{KINTHEORYrecapsec}

We now %
recall the precise formulation of the hypotheses on 
the scatterer configuration $\scrP$ 
which in the paper \cite{jMaS2019}
were proved to imply %
the convergence of the Lorentz gas particle dynamics in the Boltzmann-Grad limit.
We first need to introduce some technical notation.

For any topological space $S$, we write $P(S)$ for the set of Borel probability measures on $S$,  
\label{PSdef} 
equipped with the weak topology. We will only consider $P(S)$
when $S$ is separable and metrizable;
recall that then also $P(S)$ is metrizable
\cite[pp.\ 72-73]{pB99}.
Also, given any locally compact second countable Hausdorff space $\scrX$,
we let $N(\scrX)$ be the family of locally finite counting measures on $\scrX$,
equipped with the vague topology (then $N(\scrX)$ is a Polish space),\label{NscrXDEF}
and let $N_s(\scrX)$ be the subset of \textit{simple} counting measures;
this is a Borel subset of $N(\scrX)$.
The elements of $N_s(\scrX)$ may be identified with the family
of locally finite subsets of $\scrX$ through $\nu\mapsto\supp(\nu)$. The inverse map is
$\{x_i\}\mapsto\sum_i\delta_{x_i}$.
We will use this identification between point sets and simple counting measures
throughout this work, often using the same notation for point set and counting measure.
A \textit{point process} in $\scrX$ is,
by definition,
a random element $\xi$ in $N(\scrX)$.
It is called \textit{simple} if $\xi\in N_s(\scrX)$ almost surely.
We identify $P(N_s(\scrX))$ with the set of probability measures $\nu\in P(N(\scrX))$ with 
$\nu(N_s(\scrX))=1$.
Then a point process $\xi$ is simple if and only if its law is in $P(N_s(\scrX))$.

In the present paper %
we will always have 
\begin{align}\label{scrXDEF}
\scrX:=\R^d\times\Sigma
\end{align}
for some compact metric space $\Sigma$.
We will view vectors in $\R^d$ as \textit{row} vectors;
thus $\GL_d(\R)$ acts on $\R^d$ from the right. %
We extend the natural actions on $\R^d$ of $\R^d$ (by translation),
of $\R^\times$ (by dilation)
and of $\GL_d(\R)$ (by multiplication from the right)
by %
the trivial action on $\Sigma$;
thus for any $(\vecw,\vs)\in\scrX$ and any
$\vecx\in\RR^d$, $T\in\RR^\times$ and $A\in\GL_d(\RR)$ 
we set
$(\vecw,\vs)+\vecx:=(\vecw+\vecx,\vs)$, $T(\vecw,\vs):=(T\vecw,\vs)$ and $(\vecw,\vs)A:=(\vecw A,\vs)$. 
These actions also give rise to natural, continuous, actions on %
$N_s(\scrX)$.
For example, for any $A\in\GL_d(\RR)$
and $\scrQ\in N_s(\scrX)$ (viewed as a locally finite subset of $\scrX$),
$\scrQ A:=\{x A\col x\in\scrQ\}$.

For $\rho>0$ we denote by $D_\rho$ the diagonal matrix \label{Drhodef}
\begin{align}
D_\rho= %
\diag(\rho^{d-1},\rho^{-1},\cdots,\rho^{-1})\in\SL_d(\R).
\end{align}
We also fix, once and for all, a map $R:\US\to\SO(d)$ 
with the property %
$\vecv R(\vecv)=\vece_1$ for all $\vecv\in\US$, \label{Rdef} \label{vece1DEF}
where $\vece_1:=(1,0,\ldots,0)$. 
We assume that $R$ is continuous when restricted to $\US$ minus one point; the choice of $R$ is otherwise arbitrary.

\vspace{5pt}

Now let %
$\scrP$ be an arbitrary
fixed locally finite subset of $\R^d$ with constant asymptotic density $\nbar_{\scrP}$. 
We also assume given a compact metric space $\Sigma$, a map $\vs:\scrP\to\Sigma$, a Borel probability measure $\mm$ on $\Sigma$,
and a continuous map $\vs\mapsto\mu_\vs$ from $\Sigma$ to $P(N(\scrX))=P(N(\R^d\times\Sigma))$.\label{muvsmapDEF}
Set\label{mmDEF2}
\begin{align}\label{tPDEF}
\tP=\{(\vecp,{\vs}(\vecp))\col\vecp\in\scrP\}\subset\scrX.
\end{align}
We refer to $\vs$ as a {\em marking} of $\scrP$, and $\Sigma$ as the corresponding {\em space of marks}.\footnote{For the
point sets $\scrP$ which we consider in the present paper, %
the marking $\vs:\scrP\to\Sigma$ will be constructed to be a certain 
refinement of the \textit{crude} marking
$\psi:\scrP\to\Psi$ which we introduced in Section \ref{LIMITFLIGHTPROCESS}.}

For any $\vecq\in\R^d$, $\vecv\in\US$ and $\rho>0$, we set
\begin{equation}\label{repPqdef}
\tP_\vecq=
\begin{cases}
\tP\setminus\{(\vecq,\vs(\vecq))\} & (\vecq\in\scrP)\\
\tP & (\vecq\notin\scrP) 
\end{cases}
\end{equation}
and
\begin{align}\label{repXIRHOqv}
\scrQ_\rho(\vecq,\vecv)=(\tP_\vecq-\vecq)\,R(\vecv)\,D_\rho .
\end{align}
This point set $\scrQ_\rho(\vecq,\vecv)$ 
represents the (marked) scatterer configuration
in a coordinate system which is 
normalized in a natural way from the point of view of  %
a point particle
at position $\vecq$ and direction of travel $\vecv$.
Namely, in this coordinate system the direction of travel has been
rotated so that it equals $\vece_1=(1,0,\ldots,0)$;
lengths in the direction of travel are measured in units of $\rho^{1-d}$
(which is proportional to the mean free path length),
whereas lengths perpendicular to the direction of travel are measured in units of $\rho$.

Given any $\vecq\in\R^d$ and $\lambda\in P(\US)$,
if we take $\vecv$ random in $(\US,\lambda)$ then $\scrQ_\rho(\vecq,\vecv)$ becomes a 
point process;
we write $\mu_{\vecq,\rho}^{(\lambda)}\in P(N_s(\scrX))$ for its distribution.\label{repmuqrholambdaDEF}
Finally set $\mu_{\scrX}=\vol\times\mm$. \label{muscrXDEF}

\vspace{5pt}

The assumption %
on the scatterer configuration $\scrP$,
under which the convergence of the rescaled Lorentz process
to a limiting flight process was proved in \cite{jMaS2019},
is that $\scrP$ can be equipped with data 
$\vs,\Sigma,\mm$ and $\vs\mapsto\mu_\vs$ as above,
in such a way that the following six 
conditions [P1]--[P3] and [Q1]--[Q3] are fulfilled
(see \cite[Sec.\ 2.3]{jMaS2019}):
\begin{enumerate}[{\bf [P1]}]
\item {\em Uniform density:} For any bounded $B \subset\scrX$ with $\mu_\scrX(\partial B)=0$, we have
\begin{align}\label{repASYMPTDENSITY1b}
\lim_{T\to\infty}\frac{\#(\tP\cap TB)}{T^d}=\nbar_{\scrP}\,\mu_\scrX(B).
\end{align}
\item {\em Spherical equidistribution:}
Let $\Pac(\US)$ be the set of $\lambda\in P(\US)$ which are absolutely continuous with respect to $\sigma$.\label{PacUSdef}
There exists a subset $\scrE\subset\scrP$ of \label{scrEDEF}
density zero\footnote{We say that 
a subset $\scrE$ of $\R^d$ has \textit{density zero}
if it has asymptotic density zero in the sense of \eqref{density000};
note that this holds if and only if 
it holds for $\scrD=\scrB_1^d$,
viz., if and only if
$\lim_{R\to\infty} R^{-d}\#(\scrE\cap\scrB_R^d)=0$.}
such that
for any fixed $T\geq1$ and $\lambda\in \Pac(\US)$, we have
\begin{align}\label{repASS:KEY}
\mu_{\vecq,\rho}^{(\lambda)}\xrightarrow[]{\textup{ w }}\mu_{{\vs}(\vecq)}
\quad\text{as }\:\rho\to0,\:\text{ uniformly for $\vecq\in\scrP_T(\rho):=\scrP\cap\scrB^d_{T\rho^{1-d}}\setminus\scrE$. }\footnotemark
\end{align}
\footnotetext{The statement in \eqref{repASS:KEY} means by definition:
$\bigl[\forall\ve>0$: $\exists\rho_0>0$:
$\forall\rho\in(0,\rho_0)$: $\forall\vecq\in\scrP_T(\rho)$:
$d(\mu_{\vecq,\rho}^{(\lambda)},\mu_{{\vs}(\vecq)})<\ve\bigr]$,
where $d$ is some metric on $P(N(\scrX))$ realizing the weak topology.
This definition is independent of the choice of $d$;
indeed see \cite[Lemma 2.1]{jMaS2019} and note that
$\{\mu_{\vs}\col\vs\in\Sigma\}$ is a compact subset of $P(N(\scrX))$, since it is the continuous image of the compact set $\Sigma$.}

\item {\em No escape of mass:} For every bounded Borel set $B\subset\RR^d$,
\begin{align*}%
\lim_{\xi\to\infty}\limsup_{\rho\to0}\hspace{7pt}
[\vol\times\sigma]\bigl(\bigl\{(\vecq,\vecv)\in B\times\US\col
\scrQ_{\rho}(\rho^{1-d}\vecq,\vecv)\cap(\fZ_\xi\times\Sigma)=\emptyset\bigr\}\bigr)=0,
\end{align*}
with the open cylinder\label{fZxidef} $\fZ_\xi=(0,\xi)\times\scrB_1^{d-1}\subset\RR^d$.
\end{enumerate}

\vspace{3pt}

\begin{enumerate}[{\bf [Q1]}]
\item {\em $\SO(d-1)$-invariance:}\hspace{10pt} For every $\vs\in\Sigma$, 
\begin{align*}
\text{$\mu_{\vs}$ is invariant under the action of
$\SO(d-1):=\{k\in\SO(d)\col\vece_1k=\vece_1\}$.} \label{repASS:sodm1inv}
\end{align*}
\item {\em Coincidence-free first coordinates:} For every $\vs\in\Sigma$, 
\begin{equation*}%
\mu_{\vs}(\{\nu\in N(\scrX)\col \exists x_1\in\R\text{ s.t.\ }\nu(\{x_1\}\times\R^{d-1}\times\Sigma)>1\})=0 . 
\end{equation*}
\item {\em Small probability of large voids:}  For every $\ve>0$ there exists $R>0$ such that for all $\vs\in\Sigma$ and $\vecx\in\RR^d$ we have 
\begin{equation} \label{Q3cond}
\mu_{\vs}\bigl(\bigl\{\nu\in N(\scrX)\col \nu\bigl(\scrB^d_R(\vecx)\times\Sigma\bigr)=0\bigr\}\bigr)<\ve .
\end{equation}
Here $\scrB_R^d(\vecx):=\vecx+\scrB_R^d$, the open ball of radius $R$ centered at $\vecx$.\label{scrBRxdef}
\end{enumerate}

\begin{remark}
Among  [P1--3] and [Q1--3], the \textit{key} assumption is [P2],
the (uniform) spherical equidistribution.
It postulates the convergence in distribution
of the random point set $\scrQ_\rho(\vecq,\vecv)$,
which represents the scatterer configuration
in the natural coordinate system %
adapted to a point particle
starting from a scatterer center $\vecq$ in a random direction $\vecv$.
This condition [P2] is also the one which, by far, is the most difficult to verify, for every example of %
point sets $\scrP$ which are aware of that provably satisfy [P1--3].
\end{remark}

\vspace{5pt}

For later usage, %
we here point out %
an immediate %
consequence of %
[P2],
which partly motivates
the construction of $\Sigma$
in the present paper
(see the discussion in Section \ref{SigmaremarksexSEC}).
\begin{lem}\label{MAMScondsimpleconsLEM}
Assume given a locally finite subset $\scrP\subset\R^d$,
a map $\vs$ from $\scrP$ to a compact metric space $\Sigma$,
and a continuous map $\vs\mapsto\mu_\vs$ from $\Sigma$ to $P(N(\R^d\times\Sigma))$,
such that the condition [P2] holds. %
Let $p$ be the natural projection from $N(\R^d\times\Sigma)$ to $N(\R^d)$.
Then the map $\vs\mapsto\hmu_{\vs}:=p_*(\mu_{\vs})$
from $\Sigma$ to $P(N(\R^d))$ is continuous,
and for each fixed $\vecq\in\scrP\setminus\scrE$
(with $\scrE$ as in [P2])
and each $\lambda\in\Pac(\US)$,
we have $p_*(\mu_{\vecq,\rho}^{(\lambda)})
\xrightarrow[]{\textup{ w }} \hmu_{\vs(\vecq)}$ as $\rho\to0$.
\end{lem}
Note here that %
$p_*(\mu_{\vecq,\rho}^{(\lambda)})$ is the distribution of the random point set
$((\scrP\setminus\{\vecq\})-\vecq)R(\vecv)D_\rho$ in $\R^d$,
for $\vecv$ random in $(\US,\lambda)$.
\begin{proof}
As a map of counting measures, we have
$p=(p_1)_*$ where $p_1$ is the projection map $\R^d\times\Sigma\to\R^d$;
and one immediately verifies that $p$ is continuous 
(using the fact that $\Sigma$ is compact).
Now the lemma follows by the continuous mapping theorem.
\end{proof}

\subsection{Main technical result, and outline of the remainder of the paper}
\label{outlineSEC}

From a technical point of view, %
the following is  %
the main result of the present paper:
\begin{thm}\label{MAINTHM}
Let $\scrP$ be a union of grids in $\R^d$.
Then there exists a compact metric space $\Sigma$,
a marking $\vs:\scrP\to\Sigma$,
a continuous map $\vs\mapsto\mu_\vs$
from $\Sigma$ to $P(N(\scrX))$,
and a Borel probability measure $\mm$ on $\Sigma$,
such that all the assumptions [P1]--[P3] and [Q1]--[Q3]
hold.
\end{thm}
The explicit description of the 
data 
$\bigl[\Sigma,\vs,\vs\mapsto\mu_\vs$, $\mm\bigr]$
which we use in the proof of %
Theorem~\ref{MAINTHM}
is given in Section \ref{KINTHEORYsec}.

\begin{remark}
In the special case when $\scrP$ is \textit{periodic},
i.e.\ when $\scrP$ may be represented as in 
\eqref{Pform0} with all the $\scrL_i$s being translates of one and the same fixed lattice,
Theorem~\ref{MAINTHM}
was proved in 
\cite[Prop.\ 5.6]{jMaS2019}.
See Remark \ref{MAMSSec5p2comparisonREM} below for a more detailed comparison.
\end{remark}

\vspace{3pt}

We next describe the content of the rest %
of the paper.
In short, Sections \ref{SETUPsec}--\ref{NEWISOsec} are devoted to the proof of
Theorem~\ref{MAINTHM},
from which also Theorems \ref{limitflightprocessTHM} and \ref{limitflightprocessexplTHM} follow,
and Section~\ref{transkerSEC} contains the proofs of
Theorems \ref{indepmergeTHM} and \ref{AsymptTHM},
and their corollaries.

\vspace{2pt}

Specifically, %
in Section \ref{SETUPsec} we introduce 
some further basic set-up and notation.
In particular we fix the index set $\Psi$
in the presentation $\scrP=\bigcup_{\psi\in\Psi}\scrL_\psi$
to be
\begin{align*}
\Psi=\{\psi=(j,i)\col j\in\{1,\ldots,N\},\: i\in\{1,\ldots,r_j\}\},
\end{align*}
for some positive integers $N$ and $r_1,\ldots,r_N$,
with the indices chosen in such a way that the grids
$\scrL_{(j,i)}$ and $\scrL_{(j',i')}$ are commensurable if and only if $j=j'$.
We also introduce 
a certain   %
product of homogeneous spaces $\XX=\XX_1\times\cdots\times\XX_N$
which naturally parametrizes the families of unions of grids which appear 
when studying the spherical equidistribution condition [P2].
Here each space $\XX_j$ is a torus fiber bundle over the space 
$\SL_d(\Z)\bs\SL_d(\R)$ of $d$-dimensional lattices of covolume one in $\R^d$;
the fiber space is an $r_jd$-dimensional torus 
which we will denote by $\TT_j^d$;
it is equipped with a natural action by $\SL_d(\Z)$ from the right.

In Section \ref{nonunifsphericalequidistrSEC} we state an 
equidistribution result in the space $\XX$,
Theorem \ref{HOMDYNintrononunifTHM},
which can be viewed as a precise description of
the limit considered in the spherical equidistribution condition [P2].
In fact Theorem \ref{HOMDYNintrononunifTHM} applies for an arbitrary 
translation vector $\vecq$ in $\R^d$; %
however the theorem 
includes no statement about uniformity
with respect to $\vecq$,
as is required in [P2].
The corresponding limit probability measure on $\XX$ is called $\mu^{(\vecq)}$.
In the remainder of 
Section \ref{nonunifsphericalequidistrSEC} we give an alternative way
to represent these measures $\mu^{(\vecq)}$
via certain measures $\omega_1^{(\vecq)},\ldots,\omega_N^{(\vecq)}$,
where for each $j$, $\omega_j^{(\vecq)}$ 
is an $\SL_d(\Z)$-invariant probability measure on $\TT_j^d$.
We will denote by $P(\TT_j^d)'$
the set of all $\SL_d(\Z)$-invariant probability measures on $\TT_j^d$;
this is a compact subset of $P(\TT_j^d)$.

In Section \ref{omegajqlimitSEC} we study the behaviour of
each measure $\omega_j^{(\vecq)}$ as $\vecq$ tends to infinity
within a given subgrid $\scrL_\psi$ of $\scrP$, and we explicitly identify 
the limit measure(s) obtained.   %
In order for this limit to be \textit{unique} in a certain sense,
we need to introduce the
admissibility property of the presentation
$\scrP=\bigcup_{\psi\in\Psi}\scrL_\psi$ mentioned in Section \ref{LIMITFLIGHTPROCESS}.
The notion of admissibility is defined in Section \ref{SSjpsi0Ojpsi0sec},
and we prove in Section \ref{ATTAINADMsec} that an admissible presentation 
of $\scrP$ always exists.

In Section \ref{KINTHEORYsec}, building on the results from the previous sections,
we are finally able to give the precise definition of the space of marks
$\Sigma$, the marking $\vs:\scrP\to\Sigma$,
and the map $\vs\mapsto\mu_\vs$, %
which we will use in the proof of %
Theorem \ref{MAINTHM}.
The construction of the space of marks $\Sigma$
is one of the key steps in the present paper; 
this construction is more complicated and non-intuitive than 
for any of the previous types of scatterer configurations for which %
the framework from \cite{jMaS2019} is known to apply 
(see \cite[Ch.\ 5]{jMaS2019}).
The crude marking $\psi:\scrP\to\Psi$
which we introduced in
Section \ref{LIMITFLIGHTPROCESS} %
will appear as a \textit{factor} of $\vs:\scrP\to\Sigma$;
in fact we will define $\Sigma$ to be a subset of 
the Cartesian product of $\Psi$ and 
the compact set $\Omega:=\prod_{j=1}^N P(\TT_j^d)'$.
In Section \ref{SigmaremarksexSEC} we give a
discussion motivating our construction of $\Sigma$.

In Section \ref{verQ1Q2etcSEC} we prove that the conditions
[Q1]--[Q3] and [P1] and [P3] hold, %
and we reduce %
the remaining condition, [P2],
to a certain uniform equidistribution result taking place in the homogeneous space $\XX$,
Theorem \ref{HOMDYNMAINTHM}.
In Section \ref{UNIPOTAPPLsec} we prove a key result,
Theorem \ref{MAINcleanUNIFCONVthm},
on equidistribution of certain expanding unipotent orbits in a slightly
generalized version of the homogeneous space
$\XX=\GaG$; the proof
builds on Ratner's classification of ergodic measures invariant under unipotent flows
\cite{mR91a}.
Theorem \ref{MAINcleanUNIFCONVthm}
is then used in 
Section \ref{NEWISOsec}
in order to complete the proof of 
Theorem \ref{HOMDYNMAINTHM},
and thus also of Theorem~\ref{MAINTHM}.

Finally, in Section \ref{transkerSEC}
we start by recalling the formulas for the collision kernels
$p^{(\psi)}$ and $p^{(\psi'\to\psi)}$ %
(see Theorem \ref{limitflightprocessexplTHM})
in the general setting of the paper \cite{jMaS2019},
and we then give the proofs of
Theorem \ref{Productformulasthm} (which we discussed above in Section \ref{prodformulasSEC}),
Theorem \ref{indepmergeTHM}, Theorem \ref{AsymptTHM}
and Corollary \ref{Asympt2momentCOR}.

\section{Further setup and notation}
\label{SETUPsec}

\subsection{Representing the point set $\scrP$}
\label{Preprsec}
As mentioned in the introduction, we will always view vectors in $\R^d$ as \textit{row} vectors.\footnote{But note that
for other spaces $\R^r$  %
appearing below,
it will be natural to instead consider the vectors to be 
\textit{column} vectors.
The first instance of this is %
on p.\ \pageref{TTjdef},
for the space $\R^{r_j}$ in the definition of the torus $\TT_j=\R^{r_j}/\Z^{r_j}$.}

Recall that we are assuming that the scatterer configuration $\scrP$ is of the form
$\scrP=\bigcup_{i=1}^M\scrL_i$,
where each $\scrL_i$ is a grid
(viz., a translate of a lattice)
in $\R^d$.
Recall also that two grids $\scrL$ and $\scrL'$ in $\R^d$ are said to be 
\textit{commensurable} if there exist $c>0$ and $\vecv\in\R^d$
such that $\scrL\cap(c\scrL'+\vecv)$ is a grid.
This is an equivalence relation on the family of grids in $\R^d$.
By collecting the given grids $\scrL_1,\ldots,\scrL_M$ %
into commensurability classes,
and then applying Lemma \ref{PformredLEM1} below to each class,
one sees that the set $\scrP$ can be represented as follows:
\begin{align}\label{GENPOINTSET1pre}
\scrP=\bigcup_{j=1}^N\bigcup_{i=1}^{r_j}c_{j,i}(\scrL_j+\vecv_{j,i}),
\end{align}
where $\scrL_1,\ldots,\scrL_N$ are \textit{pairwise incommensurable lattices} 
in $\R^d$ of \textit{covolume one},
$r_1,\ldots,r_N$ are positive integers,
$c_{j,i}$ are positive real numbers
and $\vecv_{j,i}$ are vectors in $\R^d$,
and furthermore:
\begin{align}\label{THINDISJcond1}
\forall j\in\{1,\ldots,N\}:\:\:\forall i\neq i'\in\{1,\ldots,r_j\}:\:\: c_{j,i}/c_{j,i'}\in\Q\Rightarrow 
\bigl[c_{j,i}(\scrL_j+\vecv_{j,i})\cap c_{j,i'}(\scrL_j+\vecv_{j,i'})=\emptyset\bigr].
\end{align}
Our reason for requiring \eqref{THINDISJcond1} is that if
$c_{j,i}/c_{j,i'}\in\Q$ and the two grids 
$c_{j,i}(\scrL_j+\vecv_{j,i})$ and $c_{j,i'}(\scrL_j+\vecv_{j,i'})$ 
were \textit{not} disjoint,
then their intersection would be a grid and in particular would have positive aymptotic density;
this is a situation that we wish to avoid 
(see Lemma~\ref{THINDISJcondLEM} and Remark \ref{psiaedef} below).

In order to arrive at the statement above, %
we made use of the following lemma.
\begin{lem}\label{PformredLEM1}
Let $\scrL_1,\ldots,\scrL_m$ be commensurable grids in $\R^d$.  %
Then there exists a lattice $\scrL$ in $\R^d$ of covolume one,
a positive integer $r$, 
positive real numbers $c_1,\ldots,c_r$
and vectors $\vecv_1,\ldots,\vecv_r\in\R^d$ 
such that
$\scrL_1\cup\cdots\cup\scrL_m=\bigcup_{i=1}^r c_i(\scrL+\vecv_i).$
In this representation, we may furthermore assume that for any $i\neq i'$ in $\{1,\ldots,r\}$,
if $c_i/c_{i'}\in\Q$ 
then $c_i(\scrL+\vecv_i)\cap c_{i'}(\scrL+\vecv_{i'})=\emptyset$.
\end{lem}
\begin{proof}
Choose $\vecw_1\in\R^d$ such that $\scrL':=\scrL_1+\vecw_1$ is a lattice.
Now for each $j\in\{2,\ldots,m\}$,
the fact that $\scrL_1$ and $\scrL_j$ are commensurable implies that we can choose
$\delta_j>0$ and $\vecw_j\in\R^d$
such that $\scrL'\cap(\delta_j\scrL_j+\vecw_j)$ is a lattice.
(In particular then $\bn\in \delta_j\scrL_j+\vecw_j$, meaning that $\delta_j\scrL_j+\vecw_j$ is a lattice.)
It follows that also $\cap_{j=2}^m(\delta_j\scrL_j+\vecw_j)$
intersects $\scrL'$ in a lattice,
and we can express this intersection as $\delta\scrL$, where $\scrL$ is a lattice of covolume one
and $\delta$ is a positive real number.
The fact that $\delta\scrL$ is a finite index subgroup of $\scrL'$ implies that
$\scrL'$ is a union of a finite number of translates of $\delta\scrL$;
hence so is $\scrL_1$. %
Similarly for each $j\in\{2,\ldots,m\}$,
since $\delta\scrL$ is a finite index subgroup of $\delta_j\scrL_j+\vecw_j$,
the grid $\scrL_j$ 
is a union of a finite number of translates of 
$(\delta/\delta_j)\scrL$.
The first statement of the lemma follows by collecting all these unions into a single union.

To prove the second statement, 
we may for example proceed as follows.
We start from a union
$\bigcup_{i=1}^r c_i(\scrL+\vecv_i)$ with arbitrary numbers $c_1,\ldots,c_r\in\R_{>0}$
and vectors $\vecv_1,\ldots,\vecv_r\in\R^d$.
Partition the index set $\{1,\ldots,r\}$ by the equivalence relation 
$i\sim i'\stackrel{\tdef}{\Leftrightarrow} c_i/c_{i'}\in\Q$,
and for each equivalence class $J\subset\{1,\ldots,r\}$,
let $c_J$ be the least common multiple of the numbers $\{c_i\col i\in J\}$
(this exists since $c_i/c_{i'}\in\Q$ for all $i,i'\in J$).
Then for each $i\in J$ we have $n_i:=c_J/c_i\in\Z^+$,
and now
$\bigcup_{i\in J}c_i(\scrL+\vecv_i)=\bigcup_{i\in J}\bigcup_{\vecw\in R_i}c_J(\scrL+\vecw+n_i^{-1}\vecv_i)$,
where $R_i\subset n_i^{-1}\scrL$ is any fixed set of 
representatives for the quotient $n_i^{-1}\scrL/\scrL$.
Removing any duplicate translate of $c_J\scrL$ from the last union,
we arrive at an expression of the form
$\bigcup_{i\in J}c_i(\scrL+\vecv_i)=\bigcup_{i=1}^{r_J}c_J(\scrL+\vecv_i')$
for some vectors $\vecv_1',\ldots,\vecv_{r_J}'\in\R^d$ which are pairwise incongruent modulo $\scrL$.
Applying this rewriting procedure to each equivalence class $J\subset\{1,\ldots,r\}$,
we arrive at a representation of $\scrL_1\cup\cdots\cup\scrL_m$
which has the required property.
\end{proof}

\begin{remark}
Note that for a given $\scrP$,
the representation \eqref{GENPOINTSET1pre} is far %
from unique,
even when we require \eqref{THINDISJcond1}.
\end{remark}

\vspace{5pt}

Let us rewrite \eqref{GENPOINTSET1pre} slightly:
Choose $M_1,\ldots,M_N\in\SL_d(\R)$ so that $\scrL_j=\Z^dM_j$
(this is possible since $\scrL_j$ has covolume one).
Set
\begin{align}\label{PSIdef}
\Psi=\{(j,i)\col j\in\{1,\ldots,N\},\: i\in\{1,\ldots,r_j\}\}.
\end{align}
Also for each $\psi=(j,i)\in\Psi$,
we set $c_\psi=c_{j,i}$,
$\vecw_\psi=\vecw_{j,i}=\vecv_{j,i}M_j^{-1}\in\R^d$,
and
\begin{align}\label{LpsiDEF}
\scrL_\psi:=c_{j,i}(\scrL_j+\vecv_{j,i})=c_{\psi}(\Z^d+\vecw_{\psi})M_{j}.
\end{align}
With this notation,
the asymptotic density of $\scrL_\psi$ is
$\nbar_\psi=c_\psi^{-d}$,
and we have: %
\begin{align}\label{GENPOINTSET1}
\scrP=\bigcup_{\psi\in\Psi}\scrL_\psi, %
\end{align}
as in \eqref{Ppsidec1}.

Given $\psi\in\Psi$, we denote by $j_\psi$ and $i_\psi$
the indices such that $\psi=(j_\psi,i_\psi)$. %
The condition\label{ipsijpsidef}
\eqref{THINDISJcond1} now takes the following form:
\begin{align}\label{THINDISJcond2}
\forall \psi\neq\psi'\in\Psi:\:\:
\bigl[j_\psi=j_{\psi'}\text{ and }c_\psi/c_{\psi'}\in\Q\bigr]
\Rightarrow 
\scrL_\psi\cap\scrL_{\psi'}=\emptyset.
\end{align}
Also, the fact that $\scrL_1,\ldots,\scrL_N$ are pairwise incommensurable
implies that 
\begin{align}\label{GENPOINTSET1req}
M_jM_{j'}^{-1}\notin\scrS,\qquad\forall j\neq j'\in\{1,\ldots,N\},
\end{align}
where $\scrS$ is the commensurator of $\SL_d(\Z)$ in $\SL_d(\R)$,
i.e.
\begin{align}\label{COMMENSURATOR}
\scrS=
\{(\det T)^{-1/d}T\col T\in\GL(d,\Q),\:\det T>0\}.
\end{align}
The presentation of $\scrP$ in \eqref{GENPOINTSET1} is the one which we will work with
throughout the paper,
and the conditions \eqref{THINDISJcond2} and \eqref{GENPOINTSET1req}
will always be assumed to hold.
\begin{lem}\label{THINDISJcondLEM}
In the above representation of $\scrP$,
for any two $\psi\neq\psi'\in\Psi$,
the intersection $\scrL_\psi\cap\scrL_{\psi'}$ is contained in an affine hyperplane of $\R^d$.
\end{lem}
\begin{proof}
Assume that $\scrL_\psi\cap\scrL_{\psi'}$ is not contained in an affine hyperplane.
Then there exist points $\vecq_0,\ldots,\vecq_d\in\scrL_\psi\cap\scrL_{\psi'}$
such that the vectors $\vecq_j-\vecq_0$ for $j=1,\ldots,d$ %
form a linear basis of $\R^d$.
Then $\vecq_0+\Z(\vecq_1-\vecq_0)+\cdots+\Z(\vecq_d-\vecq_0)$ is a grid
contained in $\scrL_\psi\cap\scrL_{\psi'}$,
and so $\scrL_\psi$ and $\scrL_{\psi'}$ are commensurable,
i.e.\ $j_\psi=j_{\psi'}$.
Hence by \eqref{LpsiDEF},
both $c_{\psi}(\Z^d+\vecw_\psi)$ and 
$c_{\psi'}(\Z^d+\vecw_{\psi'})$ contain the points 
$\vecq_iM_{j_\psi}^{-1}$ ($i=0,1,\ldots,d$).
Therefore $(c_\psi/c_{\psi'})(\Z^d+\vecw_\psi-(c_{\psi'}/c_\psi)\vecw_{\psi'})$ 
intersects $\Z^d$
in more than one point, and hence
$c_\psi/c_{\psi'}\in\Q$.
When combined with \eqref{THINDISJcond2},
this leads to a contradiction.
\end{proof}

Given $\scrP$ as in \eqref{GENPOINTSET1}
we now also fix a choice of a crude marking
$\psi:\scrP\to\Psi$, as in
\eqref{psimarkdef}.
\begin{remark}\label{psiaedef}
As we discussed below \eqref{psimarkdef},
$\psi(\vecp)$ is uniquely defined for 
every point $\vecp\in\scrP$ outside the union
$\bigcup_{\psi\neq\psi'\in\Psi}(\scrL_\psi\cap\scrL_{\psi'})$.
By Lemma \ref{THINDISJcondLEM}, this union is contained in a finite union of affine hyperplanes.
In particular $\psi(\vecp)$ is uniquely defined for 
every $\vecp\in\scrP$ away from a subset of density zero.
\end{remark}

\subsection{Lie groups and homogeneous spaces}
\label{LIEGPHOMSPsec}
We now introduce 
appropriate homogeneous spaces 
to keep track of the unions of grids which will appear in our discussion. %
(This is very standard;
cf.\ \cite[Sections 5.2 \& 5.3.3]{jMaS2019},
\cite[Sec.\ 7]{cDjMaS2016}
as well as many other references.)

For any positive integer $r$, we let $\M_{r\times d}(\R)$ be the space of real $r\times d$ matrices,
and let $\S_r(\R)$ be the semidirect product
\begin{align}\label{Srdefrep}
\S_r(\R):=\SL_d(\R)\ltimes\M_{r\times d}(\R)
\end{align}
with multiplication law
\begin{align}\label{Srgrouplaw}
(M,U)(M',U')=(MM',UM'+U').
\end{align}
In the special case $r=1$, we have $\S_1(\R)=\ASL_d(\R)$, %
the affine special linear group.
This group acts on $\R^d$ %
from the right through
\begin{align*}
\vecw (M,\vecu):=\vecw M+\vecu
\qquad(\vecw\in\R^d,\: (M,\vecu)\in\ASL_d(\R)),
\end{align*}
that is, by affine linear transformations preserving volume and orientation.
Similarly, for general $r$, 
$\S_r(\R)$ can be thought of as the group of such affine linear transformations
acting on $r$ \textit{copies} of $\R^d$, 
with the constraint that the linear part (``$M$'') of the action is the same on each copy.
To make this precise,
let $\r_i:\M_{r\times d}(\R)\to\R^d$ ($i=1,\ldots,r$) %
be the projection map which takes any matrix to its $i$th row,\label{riDEF}
and let $\a_i$ be the following Lie group homomorphism:
\begin{align*}
\a_i:\S_r(\R)\to\ASL_d(\R);\qquad
\a_i(M,U):=(M,\r_i(U)).
\end{align*}
Taken together, these maps $\a_1,\ldots,\a_r$ give %
an isomorphism of
$\S_r(\R)$ with the subgroup of $\ASL_d(\R)^r$
consisting of all $\langle (M_i,\vecu_i)\rangle_{i=1,\ldots,r}$ with $M_1=\cdots=M_r$.

It is also natural to identify %
$\S_r(\R)$ with a subgroup of $\SL_{d+r}(\R)$
through $(M,U)\leftrightarrow\matr M0UI$,
since the group law \eqref{Srgrouplaw} then corresponds to multiplication of matrices;
however we will not make use of this identification in the present paper.

We will always consider $\SL_d(\R)$ to be embedded in $\S_r(\R)$ through 
$M\mapsto (M,0)$;
in other words any $M\in\SL_d(\R)$ is understood to also denote the element $(M,0)$ in $\S_r(\R)$.
In the opposite direction, we denote by $\iota$ the projection homomorphism\label{IOTAdef}
\begin{align*}
\iota:\S_r(\R)\to\SL_d(\R);\qquad\iota(M,U):=M.
\end{align*}
For any $U\in \M_{r\times d}(\R)$
we let $\I_U$ be the element 
\begin{align}\label{IUdef}
\I_U:=(\I,U).  %
\end{align}
in $\S_r(\R)$.
Note that the map $U\mapsto\I_U$ is a homomorphism from 
$\M_{r\times d}(\R)$ to $\S_r(\R)$.

\vspace{5pt}

From now on we assume a point set $\scrP$ with presentation as in \eqref{PSIdef}--\eqref{GENPOINTSET1}
to be \textit{fixed.}
In particular this means that the numbers $N,r_1,\ldots,r_N$ are fixed,
and so are the numbers $c_\psi$,
vectors $\vecw_\psi$ ($\psi\in\Psi$)
and matrices $M_1,\ldots,M_N\in\SL_d(\R)$.
Associated to this fixed presentation of $\scrP$, we now introduce 
the following homogeneous space 
$\XX$ which will be of fundamental importance throughout the rest of the paper:
\begin{align*}
\XX:=\GaG,
\end{align*}
where \label{GXDEF}
\begin{align*}
& G:=\S_{r_1}(\R)\times\cdots\times\S_{r_N}(\R)
\qquad\text{and}\qquad %
\Gamma:=\S_{r_1}(\Z)\times\cdots\times\S_{r_N}(\Z),
\end{align*}
with  %
\begin{align*}
\S_r(\Z):=\SL_d(\Z)\ltimes\M_{r\times d}(\Z).
\end{align*}
The key point of this space $\XX$ is that it %
parametrizes   %
the set of all point sets that can be obtained from
$\scrP=\cup_{\psi\in\Psi}\scrL_\psi$
by applying arbitrary translations to each individual grid $\scrL_{\psi}$,
and arbitrary linear maps of determinant one to %
each individual commensurability class,
$\cup_{i=1}^{r_j}\scrL_{(j,i)}$.
To describe this parametrization, %
let us write $G_j=\S_{r_j}(\R)$, \label{GjDEF}
$\Gamma_j=\S_{r_j}(\Z)$ and $\XX_j=\Gamma_j\bs G_j$,
so that $G=G_1\times\cdots\times G_N$ 
and $\XX=\XX_1\times\cdots\times\XX_N$;
also write $p_j$ and $\tp_j$ for the projection maps onto the $j$th factor:
\begin{align}\label{pjdef}
p_j:G\to\S_{r_j}(\R)
\qquad\text{and}\qquad
\tp_j:\XX\to\XX_j
\qquad (j\in\{1,\ldots,N\}).
\end{align}
Then for each $\psi=(j,i)\in\Psi$, let $p_\psi$ be the Lie group homomorphism
\begin{align}\label{ppsiDEF}
p_\psi:=\a_i\circ p_j:G\to\ASL_d(\R).
\end{align}
Now for any $g\in G$, we take the point $\Gamma g$ in $\XX$ to parametrize 
the point set $J_0(\Gamma g)$ in $\R^d$,
where
\begin{align}\label{Gammagpointset}
J_0:\XX\to N_s(\R^d);\qquad
J_0(\Gamma g):=\bigcup_{\psi\in\Psi} c_\psi\bigl(\Z^d\,p_\psi(g)\bigr)\qquad(g\in G).
\end{align}

To see that the above map $J_0$ %
is well-defined,
note that for each $\psi$ we have $p_\psi(\Gamma)=\ASL_d(\Z)$;
hence if $\Gamma g=\Gamma g'$ then $p_\psi(g')=\gamma\,p_\psi(g)$ for some $\gamma\in\ASL_d(\Z)$,
implying that $\Z^d p_\psi(g') %
=\Z^d p_\psi(g)$.
However it should be noted that the map $J_0$ is discontinuous
at any point $\Gamma g\in\XX$
for which the grids $c_\psi\bigl(\Z^d\,p_\psi(g)\bigr)$ ($\psi\in\Psi$) are not pairwise disjoint.

We next introduce notation
for expressing an arbitrary translate 
of $\scrP$
in terms of the parametrization by $\XX$.
Namely, for any $\vecq\in\R^d$ 
we will define an element $g_0^{(\vecq)}\in G$ such that
\begin{align}\label{J0UqtMeqPmq}
J_0(\Gamma\,g_0^{(\vecq)})=\scrP-\vecq
\qquad (\forall\vecq\in\R^d).
\end{align}
This is done as follows:
For any $\vecq\in\R^d$ and $j\in\{1,\ldots,N\}$
we define the matrix $U_j^{(\vecq)}$ by specifying its row vectors:
\begin{align}\label{Ujqdef}
U_j^{(\vecq)}\in\M_{r_j\times d}(\R);
\qquad
\p_i(U_j^{(\vecq)})=\vecw_{j,i}-c_{j,i}^{-1}\vecq M_j^{-1}
\qquad(i=1,\ldots,r_j),
\end{align}
and we also write
\begin{align}\label{tUqdef}
U^{(\vecq)}:=(U_1^{(\vecq)},\cdots,U_N^{(\vecq)})\in 
\prod_{j=1}^N\M_{r_j\times d}(\R).   %
\end{align}
We also set %
(cf.\ \eqref{IUdef})
\begin{align}\label{IVdef}
\I_V:=\bigl(\I_{V_1},\ldots,\I_{V_N}\bigr)\in G
\qquad\text{for any }\: V=(V_1,\ldots,V_N)\in\prod_{j=1}^N\M_{r_j\times d}(\R),
\end{align}
and
\begin{align}\label{tMdef}
\tM:=(M_1,\ldots,M_N)\in G.
\end{align}
Finally, we define:
\begin{align}\label{g0qdef}
g_0^{(\vecq)}:=\I_{U^{(\vecq)}}\hspace{0pt}\tM\in G\qquad(\vecq\in\R^d).
\end{align}
We then compute that, %
for any $\psi=(j,i)\in\Psi$:
\begin{align}\label{KEYUjqdefmotivation}
c_\psi\bigl(\Z^d\,p_\psi(g_0^{(\vecq)})\bigr)
=
c_\psi\bigl(\Z^d\,\a_i(\I_{U_j^{(\vecq)}} M_j)\bigr)
=
c_\psi\,(\Z^d+\vecw_\psi-c_\psi^{-1}\vecq M_j^{-1})\,M_j
=
\scrL_\psi-\vecq
\end{align}
(cf.\ \eqref{LpsiDEF}).
Using \eqref{KEYUjqdefmotivation} and \eqref{Gammagpointset},
it now follows that the key relation \eqref{J0UqtMeqPmq} holds.

\vspace{5pt}

Finally we introduce some further notation
relating to the structure of the homogeneous spaces $\XX_j$ and $\XX$.
Recall that for each $j\in\{1,\ldots,N\}$
we have a projection map $\iota:\S_{r_j}(\R)\to\SL_d(\R)$; 
this map takes $\Gamma_j$ to $\SL_d(\Z)$, and hence induces a projection map
from $\XX_j$ to $\SL_d(\Z)\bs\SL_d(\R)$, which we will denote by $\tiota$:\label{tIOTAdef}
\begin{align}\label{tiotaDEF}
\tiota:\XX_j\to\SL_d(\Z)\bs\SL_d(\R),
\qquad\tiota(\Gamma_jg)=\SL_d(\Z)\,\iota(g)\quad (g\in G_j).
\end{align}
We also set\label{TTjdef}
\begin{align*}
\TT_j:=\R^{r_j}/\Z^{r_j}
\qquad\text{and}\qquad
\TT_j^d:=\underbrace{\TT_j\times\cdots\times\TT_j}_{d\: \text{ copies}}.
\end{align*}
Note that $\TT_j^d$ is an $r_jd$-dimensional torus.
We will use ``$\pi$'' to denote any of the standard projection maps\label{piXprojDEF}
\begin{align}\label{pitorusprojDEF}
\pi:\R^{r_j}\to\TT_j;
\qquad %
\pi:(\R^{r_j})^d\to\TT_j^d;
\qquad
\pi:G_j\to\XX_j,
\qquad\text{and}\qquad
\pi:G\to\XX.
\end{align}
We will identify $\M_{r_j\times d}(\R)$ with $(\R^{r_j})^d$,
by identifying any matrix in $\M_{r_j\times d}(\R)$ with the sequence of its $d$ column vectors
(in order).
This also induces an identification
\begin{align*}
\TT_j^d=%
\M_{r_j\times d}(\R/\Z).
\end{align*}
We write $\tr_i:\TT_j^d\to(\R/\Z)^d$ for the projection induced by\label{triDEF}
the map $\r_i$. %
For each $j$ we also introduce the embedding
\begin{align}\label{xmapDEF}
x:\TT_j^d\to\XX_j;
\qquad
x(\pi(U))=\Gamma_j\hspace{1pt} \I_U,
\quad\forall U\in\M_{r_j\times d}(\R).
\end{align}
This map is well-defined since $\I_U\in\Gamma_j$ for all $U\in\M_{r_j\times d}(\Z)$.

\section{Spherical equidistribution (without uniformity)}
\label{nonunifsphericalequidistrSEC}

\subsection{The limit measure $\mu^{(\vecq)}$}
\label{HOMDYNintrononunifTHMSEC}

In this section we state a result in homogeneous dynamics,
Theorem \ref{HOMDYNintrononunifTHM} below,
which gives a precise description of
the limit considered in the spherical equidistribution condition [P2]
in Section \ref{KINTHEORYrecapsec}.
Theorem \ref{HOMDYNintrononunifTHM} 
expresses the answer
in terms of the parametrizing space $\XX$,
and in fact applies for an arbitrary point $\vecq$ in $\R^d$; %
however the theorem 
includes no statement about uniformity
with respect to 
the translation vector $\vecq$,
as is required in [P2].

This Theorem \ref{HOMDYNintrononunifTHM} is a special case of 
more general equidistribution results which we will prove in Sections
\ref{UNIPOTAPPLsec} and \ref{NEWISOsec}
(see in particular Theorem \ref{nonunifTHM1}),
and the main reason for stating 
this special case %
already here is 
to help motivating
the introduction of the measure $\mu^{(\vecq)}$
(see \eqref{muqDEF} below),
which will be a crucial object of study in the remainder of
Section \ref{nonunifsphericalequidistrSEC} as well as in Section \ref{omegajqlimitSEC}.
Note that we will not make use of the statement %
of Theorem \ref{HOMDYNintrononunifTHM}
before Section \ref{NEWISOsec}. %

For a given $\vecq\in\R^d$
and each $j\in\{1,\ldots,N\}$,
we let $U_{j,1}^{(\vecq)},\ldots,U_{j,d}^{(\vecq)}\in\R^{r_j}$
be the column vectors of
the matrix $U_j^{(\vecq)}$,
and then define 
$L_j^{(\vecq)}$ to be the identity component of the 
smallest closed subgroup of $\R^{r_j}$
containing both $\Z^{r_j}$ and $U_{j,1}^{(\vecq)},\ldots,U_{j,d}^{(\vecq)}$.\label{Ljqfirstdef}
This $L_j^{(\vecq)}$ 
is a \textit{rational} (linear) subspace of $\R^{r_j}$,
i.e., $L_j^{(\vecq)}\cap\Z^{r_j}$ is a lattice in $L_j^{(\vecq)}$.
Next, given any linear subspace $L$ of $\R^{r_j}$, we let $\S_L(\R)$ be the closed connected subgroup of $\S_{r_j}(\R)$ given by
\begin{align}\label{SVdefrepnew}
\S_L(\R):=\SL_d(\R)\ltimes L^d=\bigl\{(M,U)\in\S_{r_j}(\R)\col %
U\in L^d\bigr\}.
\end{align}
Here %
``$L^d$'' should be understood via our identification
$\M_{r_j\times d}(\R)=(\R^{r_j})^d$,
viz., $L^d$ is the set of matrices in 
$\M_{r_j\times d}(\R)$ all of whose column vectors lie in $L$.

\begin{lem}\label{mujqjustifyLEM}
For any $\vecq\in\R^d$, the $\S_{L_j^{(\vecq)}}(\R)$-orbit of the point
$\Gamma_j\I_{U_j^{(\vecq)}}$ in $\XX_j=\Gamma_j\bs G_j$
is a closed embedded submanifold of $\XX_j$
which carries a unique $\S_{L_j^{(\vecq)}}(\R)$-invariant
probability measure.
\end{lem}
We remark that this $\S_{L_j^{(\vecq)}}(\R)$-orbit,
i.e.\ $\Gamma_j\bs\Gamma_j\I_{U_j^{(\vecq)}}\S_{L_j^{(\vecq)}}(\R)$,
equals the closure in $\XX_j$ of the $\SL_d(\R)$-orbit of 
$\Gamma_j\I_{U_j^{(\vecq)}}$ (this is an easy consequence of Theorem \ref{HOMDYNintrononunifTHM} below).
\begin{proof}[Proof of Lemma \ref{mujqjustifyLEM}]
Set $U:=U_j^{(\vecq)}$ and $L:=L_j^{(\vecq)}$.
It follows from the construction of $L$ %
that there exists 
a matrix 
$X\in\M_{r_j\times d}(\Q)$
such that $U-X\in L^d$ %
(indeed, see Lemma \ref{IDCLcharLEM} below).
This implies 
$\I_{U}\I_X^{-1}\in\S_{L}(\R)$,
and hence the $\S_{L}(\R)$-orbit in the statement of the lemma can be expressed as
$\Gamma_j\bs\Gamma_j\I_{U}\S_{L}(\R)
=\Gamma_j\bs\Gamma_j\I_{X}\S_{L}(\R)$.
Hence by 
\cite[Theorem 1.13]{mR72},
it suffices to verify that
$\Gamma_j$
intersects
$\I_X \S_{L}(\R)\I_X^{-1}$ 
in a lattice.
To show this, let $n$ be a denominator of $X$,
i.e.\ a positive integer such that $X\in n^{-1}\M_{r_j\times d}(\Z)$,
and then set
\begin{align*}
\Lambda:=\{(M,V)\col M\in\Gamma(n),\: V\in XM-X+(L\cap\Z^{r_j})^d\},
\end{align*}
where $\Gamma(n)$ is the 
principal congruence subgroup of $\SL_d(\Z)$ of level $n$.
One verifies that $\Lambda$ is a subgroup of $\Gamma_j\cap\I_X \S_{L}(\R)\I_X^{-1}$,
and a lattice in 
$\I_X \S_{L}(\R)\I_X^{-1}$.
Hence the lemma is proved.
\end{proof}

In view of Lemma \ref{mujqjustifyLEM},
it makes sense to define 
$\mu_j^{(\vecq)}\in P(\XX_j)$
to be the unique 
$\S_{L_j^{(\vecq)}}(\R)$-invariant
probability measure
on the orbit
$\Gamma_j\bs\Gamma_j\I_{U_j^{(\vecq)}}\S_{L_j^{(\vecq)}}(\R)$
in $\XX_j$.
Finally we set 
\begin{align}\label{muqDEF}
\mu^{(\vecq)}:=\mu_1^{(\vecq)}\otimes\cdots\otimes\mu_N^{(\vecq)}\in P(\XX).
\end{align}

Let $\varphi:\SL_d(\R)\to G$ be the diagonal embedding.\label{varphiDEF}
For any topological space $S$, we denote by $\C_b(S)$ the space
of bounded continuous real-valued functions on $S$.
\begin{thm}\label{HOMDYNintrononunifTHM}
Given any $\vecq\in\R^d$, $f\in\C_b(\XX)$ and $\lambda\in\Pac(\US)$,
we have
\begin{align}\label{HOMDYNintrononunifTHMRES}
\int_{\US}f(\Gamma g_0^{(\vecq)}\varphi(R(\vecv)D_\rho))\,d\lambda(\vecv)
\to \int_{\XX} f \,d\mu^{(\vecq)}
\end{align}
as $\rho\to0$.
\end{thm}
In the left hand side of 
\eqref{HOMDYNintrononunifTHMRES},
note that the point 
$\Gamma g_0^{(\vecq)}\varphi(R(\vecv)D_\rho)$ in $\XX$
corresponds to the point set
$(\scrP-\vecq)R(\vecv)D_\rho$
in our parametrization, because of \eqref{J0UqtMeqPmq}.
Later we will introduce a refined parametrization which also keeps track of the marking,
and which removes the point $\vecq$ itself if $\vecq\in\scrP$
(see \eqref{Jpsidef} below); through this parametrization
the point $\Gamma g_0^{(\vecq)}\varphi(R(\vecv)D_\rho)$ in $\XX$
will correspond to the point set $\scrQ_\rho(\vecq,\vecv)$
which we defined in \eqref{repXIRHOqv}.
Therefore, as we mentioned previously, the limit result of
Theorem \ref{HOMDYNintrononunifTHM} can be viewed as giving a 
description of the limit considered in the spherical equidistribution condition [P2]
in Section \ref{KINTHEORYrecapsec};
in particular we will see that the limit measure $\mu_{\vs(\vecq)}$ in \eqref{repASS:KEY}
equals the appropriate push-forward of the measure $\mu^{(\vecq)}$ defined in \eqref{muqDEF}.

\vspace{3pt} 

We will prove Theorem \ref{HOMDYNintrononunifTHM}
at the end of Section \ref{nonunifequidistrSEC},
as a corollary of
more general %
results which we prove in Section \ref{UNIPOTAPPLsec}.
We remark that it is possible to give a considerably simpler
deduction of Theorem \ref{HOMDYNintrononunifTHM} 
by combining %
the ideas in the
proof of 
\cite[Theorem 10]{jMaS2013a}
and the proof of \cite[Lemma 5.22]{jMaS2019}.
In any case, the crucial ingredient in either of these two proofs
is the deep measure classification theorem of Ratner \cite{mR91a},
and strong use is also made of the work of Shah \cite{nS96}.

\subsection{Re-expressing %
the limit measures $\mu_j^{(\vecq)}$ as $\overline{\omega_j^{(\vecq)}}$}
\label{reexprmuqSEC}

The task of finding the appropriate 
space of marks $\Sigma$ and marking $\vs:\scrP\to\Sigma$
to satisfy the conditions in 
Section \ref{KINTHEORYrecapsec}
is closely related to the task of
understanding the limit measures 
$\mu^{(\vecq)}=\mu_1^{(\vecq)}\otimes\cdots\otimes\mu_N^{(\vecq)}$
appearing in Theorem~\ref{HOMDYNintrononunifTHM},
and in particular how these vary as $\vecq$ varies within the set $\scrP$.
As a first step towards this goal,
in this section we will introduce an
alternative way to 
define the measures $\mu_j^{(\vecq)}$.
Recall that each space $\XX_j$ is a torus bundle
over the space $\SL_d(\Z)\bs\SL_d(\R)$
(cf.\ \eqref{tiotaDEF}),
and the point here is
that each measure $\mu_j^{(\vecq)}$
can be expressed as a \textit{product measure}, 
of the normalized Haar measure on $\SL_d(\Z)\bs\SL_d(\R)$
and some fixed measure on the fiber $\TT_j^d$.
We make this precise in
Lemma \ref{oomegapartDEFlem}
and
Proposition \ref{SLRinvprobmeasLEM} below.

We keep $j\in\{1,\ldots,N\}$ throughout this section.
Our first step will be to introduce the relevant measures which can appear on the fiber $\TT_j^d$.

Recall that $\pi$ denotes (among other things) the projection map
$\R^{r_j}\to\TT^{r_j}:=\R^{r_j}/\Z^{r_j}$.
For any %
subset $S$ of $\R^{r_j}$ (resp.\ $\TT^{r_j}$),
we write $\langle S\rangle$ for the subgroup of $\R^{r_j}$ (resp.\ $\TT^{r_j}$)
generated by $S$,
and $\overline{\langle S\rangle}$ for its closure. %
For any topological group $H$ we write $H^\circ$ for
the connected component of the identity. %
Now for any non-empty subset $S$ of $\TT^{r_j}$, we introduce the notation:
\begin{align}\label{fLdef}
\fL(S)
:=\overline{\big\langle \pi^{-1}(S)\big\rangle}^{\hspace{3pt}\circ}.
\end{align}
This is a \textit{rational} subspace of $\R^{r_j}$,
i.e.\ a linear subspace of $\R^{r_j}$ which is spanned by its vectors in $\Z^{r_j}$.
If $S$ is a non-empty subset \textit{of} $\R^{r_j}$ then we also write
$\fL(S):=\fL(\pi(S))=\overline{\langle S+\Z^{r_j}\rangle}^{\hspace{3pt}\circ}$,
and if $V_1,\ldots,V_m$ %
is any finite sequence of points in $\TT^{r_j}$ or $\R^{r_j}$,
then we also write $\fL(V_1,\ldots,V_m):=\fL(\{V_1,\ldots,V_m\})$.
One reason why this notation ``$\fL\,$'' is useful is that
the rational subspace 
$L_j^{(\vecq)}\subset\R^{r_j}$ introduced in
Section \ref{HOMDYNintrononunifTHMSEC}
can now be expressed as:
\begin{align}\label{Ljqexplformula}
L_j^{(\vecq)}=\fL\bigl(U_{j,1}^{(\vecq)},\ldots,U_{j,d}^{(\vecq)}\bigr)
\qquad (\forall\vecq\in\R^d).
\end{align}

\begin{lem}\label{IDCLcharLEM}
For any %
non-empty subset $S\subset\R^{r_j}$,
$\fL(S)$
is %
the unique
smallest rational subspace $L\subset\R^{r_j}$ 
with the property that %
there is some $n\in\Z^+$ such that $\langle S\rangle\subset n^{-1}\Z^{r_j}+L$.
\end{lem}
\begin{proof}
Set $G=\overline{\langle \pi(S)\rangle}$;
this is a closed subgroup of $\TT^{r_j}$;
hence its group of components, $G/G^\circ$,
is finite,
i.e.\ there exists a positive integer $n$ such that 
$n\vecy\in G^\circ$ for all $\vecy\in G$.
We also have $G^\circ=\pi(\fL(S))$.
It follows that $n\cdot\pi(S)\subset\pi(\fL(S))$,
i.e.\ $n\cdot S\subset\Z^{r_j}+\fL(S)$.
This is equivalent with
$S\subset n^{-1}\Z^{r_j}+\fL(S)$,
and also with
$\langle S\rangle\subset n^{-1}\Z^{r_j}+\fL(S)$.

Next assume that $L$ is \textit{any} rational subspace of $\R^{r_j}$ satisfying 
$\langle S\rangle\subset n^{-1}\Z^{r_j}+L$ for some $n\in\Z^+$.
Now $n^{-1}\Z^{r_j}+L$ is a closed subgroup of $\R^{r_j}$
which contains $S+\Z^{r_j}$;
hence $\overline{\langle S+\Z^{r_j}\rangle}\subset n^{-1}\Z^{r_j}+L$,
and so 
$\fL(S):=\overline{\langle S+\Z^{r_j}\rangle}^{\hspace{3pt}\circ}\subset(n^{-1}\Z^{r_j}+L)^\circ=L$.
\end{proof}

Next, for any $V=(V_1,\ldots,V_d)\in\TT_j^d$,
we introduce the notation
\begin{align}\label{SSjVdef}
\SS_j^{(V)}:=\overline{\langle V_1,\ldots,V_d\rangle}
\end{align}
and
\begin{align}\label{LjDEF}
L_j^{(V)}:=\fL(V_1,\ldots,V_d)=\bigl(\pi^{-1}(\SS_j^{(V)})\bigr)^\circ.
\end{align}
Thus $\SS_j^{(V)}$ is a closed subgroup of $\TT_j$
and $L_j^{(V)}$ is a rational subspace of $\R^{r_j}$.
An important special choice of $V$ is $V=\pi(U_j^{(\vecq)})\in\TT_j^d$,
and by \eqref{Ljqexplformula}
and \eqref{LjDEF} we now have:
\begin{align}\label{Ljqexplformula2}
L_j^{(\vecq)}=L_j^{(\pi(U_j^{(\vecq)}))}
\qquad (\forall\vecq\in\R^d).
\end{align}
Let us also define $\SS_j^{(\vecq)}:=\SS_j^{(\pi(U_j^{(\vecq)}))}$,
so that $L_j^{(\vecq)}=\bigl(\pi^{-1}(\SS_j^{(\vecq)})\bigr)^\circ$.

Recall that we have identified $\M_{r_j\times d}(\R)$ with $(\R^{r_j})^d$;
this means that for any linear subspace $L\subset\R^{r_j}$,
$L^d=L\times\cdots\times L$ is a linear subspace of $\M_{r_j\times d}(\R)$.
Similarly for any subgroup $\SS\subset\TT_j$, $\SS^d$ is a subgroup 
of $\TT_j^d=\M_{r_j\times d}(\R)/\M_{r_j\times d}(\Z)$.
Note that $\SL_d(\R)$ acts from the right on $\M_{r_j\times d}(\R)$ 
by matrix multiplication,
and this action preserves the subspace $L^d$ for any $L\subset\R^{r_j}$.
Furthermore, the subgroup $\SL_d(\Z)$ preserves
$\M_{r_j\times d}(\Z)$,
and hence we obtain an induced right action of $\SL_d(\Z)$ on
$\TT_j^d=\M_{r_j\times d}(\R/\Z)$. %
This action preserves the subgroup $\SS^d$, for any subgroup $\SS\subset\TT_j$.

Now for any $V\in\TT_j^d$,
we define:
\begin{align}\label{OjVDEF}
\scrO_j^{(V)}=\bigcup_{\gamma\in\SL_d(\Z)}\bigl(V\gamma+\bigl(\SS_j^{(V)}\bigr)^{\circ\,d}\bigr).
\end{align}
The point of this definition is that, as we will see below,
if $V=\pi(U_j^{(\vecq)})$ then $\scrO_j^{(V)}$
equals \textit{the support of $\mu_j^{(\vecq)}$ in each torus fiber in $\XX_j$.}
It should be noted that for any $V\in\TT_j^d$ we have
$V\in\bigl(\SS_j^{(V)}\bigr)^d$;
hence $\scrO_j^{(V)}$ is a union of some of the connected components of
the Lie group $\bigl(\SS_j^{(V)}\bigr)^d$.
In particular $\scrO_j^{(V)}$ is open and closed in $\bigl(\SS_j^{(V)}\bigr)^d$.
Note also that since $\bigl(\SS_j^{(V)}\bigr)^d$ is compact, 
its total number of components is finite.
We also define
$\omega_j^{(V)}\in P(\TT_j^d)$
\label{mujVDEF}
to be the restriction to $\scrO_j^{(V)}$ of the Haar measure on $\bigl(\SS_j^{(V)}\bigr)^d$.
normalized so that $\omega_j^{(V)}(\scrO_j^{(V)})=1$.
Note that for each $V\in\TT_j^d$, the measure $\omega_j^{(V)}$ is 
$\SL_d(\Z)$-invariant by construction.
We denote by 
$P(\TT_j^d)'$ the subset of all $\SL_d(\Z)$-invariant measures $\omega\in P(\TT_j^d)$.\label{PTjdprimDEF}
This is a closed, and hence compact, subset of $P(\TT_j^d)$.

Next, for any $\vecq\in\R^d$, we write 
$\scrO_j^{(\vecq)}$
and $\omega_j^{(\vecq)}$
to denote 
$\scrO_j^{(\pi(U_j^{(\vecq)}))}$
and $\omega_j^{(\pi(U_j^{(\vecq)}))}$,
respectively.
Thus in particular,
$\scrO_j^{(\vecq)}$ is a union of some of the components of
the closed subgroup $\bigl(\SS_j^{(\vecq)}\bigr)^d\subset\TT_j^d$,
and $\omega_j^{(\vecq)}$ is a measure in $P(\TT_j^d)'$
supported on $\scrO_j^{(\vecq)}$.
We will see that $\omega_j^{(\vecq)}$
is the measure on the fiber $\TT_j^d$
which gives back the measure $\mu_j^{(\vecq)}$
on $\XX_j$ via the product construction which we will presently describe.
\label{SSjdef}
\label{Ljqdef}
\label{OjqDEF}
\label{mujDEF}

\vspace{5pt}

We now introduce a map
\begin{align}\label{mapPTtoPXjrep}
P(\TT_j^d)'\to P(\XX_j),\qquad\omega\mapsto\oomega,
\end{align}
as follows. 
For any $\omega\in P(\TT_j^d)$,
let $\tomega$ be corresponding 
$\M_{r_j\times d}(\Z)$-invariant
Borel measure on $\M_{r_j\times d}(\R)$,
and define the Borel measure $\toomega$ on $G_j$ through
\begin{align}\label{toomegaDEF}
d\toomega(g)=d\tomega(U)\,d\nu(A)
\qquad\text{when }\: g=\I_UA\in G_j
\quad (U\in\M_{r_j\times d}(\R),\: A\in\SL_d(\R)),
\end{align}
where $\nu$ denotes the Haar measure on $\SL_d(\R)$, 
normalized by $\nu(\SL_d(\Z)\bs\SL_d(\R))=1$.  \label{nuDEF}
By the following lemma,
if $\omega\in P(\TT_j^d)'$
then the measure $\toomega$ is left $\Gamma_j$-invariant,
and we finally define $\oomega$ to be the corresponding Borel measure on $\XX_j$.
\label{oomegapartDEF}

\begin{lem}\label{oomegapartDEFlem}
For any $\omega\in P(\TT_j^d)'$, the measure 
$\toomega$ is left $\Gamma_j$-invariant,
and the corresponding measure $\oomega$ on $\XX_j$ satisfies,
for any Borel set $E\subset\XX_j$,
\begin{align}\label{oomegapartDEFlemres}
\oomega(E)=\int_{F_d}\int_{\TT_j^d}I\bigl(x(U)A\in E\bigr)\,d\omega(U)\,d\nu(A),
\end{align}
where $F_d$ is any fixed Borel set in $\SL_d(\R)$ which is a fundamental domain for $\SL_d(\Z)\bs\SL_d(\R)$,
and where $I(\cdot)$ is the indicator function.\label{indicatorfunction}
\end{lem}
Note that \eqref{oomegapartDEFlemres}
in particular shows that 
$\oomega$ is a probability measure,
i.e.\ $\oomega\in P(\XX_j)$,
and so we indeed have a map as in \eqref{mapPTtoPXjrep}.
One may think of $\oomega$ as the product measure 
of the normalized Haar measure on $\SL_d(\Z)\bs\SL_d(\R)$
and the measure $\omega$ on the fiber $\TT_j^d$.
\begin{proof}
Let $\omega\in P(\TT_j^d)'$.
In order to verify that $\toomega$ is left $\Gamma_j$-invariant,
it suffices to verify that for any Borel set $E\subset G_j$ we have
$\toomega(\I_M %
E)=\toomega(E)$ for all $M\in\M_{r_j\times d}(\Z)$
and $\toomega(\gamma\,E)=\toomega(E)$ for all $\gamma\in\SL_d(\Z)$.
The first of these two relations is immediate from \eqref{toomegaDEF} and
the fact that $\tomega$ is $\M_{r_j\times d}(\Z)$-invariant.
The second relation follows by noticing that, for any
$U\in\M_{r_j\times d}(\R)$ and $A\in\SL_d(\R)$,
$\I_UA\in\gamma E$ holds if and only if
$\gamma^{-1}\I_UA\in E$,
viz., $\I_{U\gamma}\gamma^{-1}A\in E$,
and then using the 
invariance of the Haar measure $\nu$, and
the fact that
$\tomega$ is $\SL_d(\Z)$-invariant,
since 
$\omega\in P(\TT_j^d)'$.

Next, in order to verify \eqref{oomegapartDEFlemres},
note that if $C_d$ is the set %
of matrices in $\M_{r_j\times d}(\R)$ all of whose entries lie in $[0,1)$,
then the set
\begin{align*}
F_d':=\{\I_UA\col U\in C_d,\: A\in F_d\}\quad\subset G_j
\end{align*}
is a fundamental domain for $\Gamma_j\bs G_j$.
Hence for any Borel set $E\subset\XX_j$
we have 
\begin{align*}
\oomega(E)=\toomega(\pi^{-1}(E)\cap F_d')
=\int_{F_d}\int_{C_d}I\bigl(\I_UA\in \pi^{-1}(E)\bigr)\,d\tomega(U)\,d\nu(A),
\end{align*}
which is the same as \eqref{oomegapartDEFlemres}.
\end{proof}

For later use we record a few simple properties of the map $\omega\mapsto\oomega$ just defined.
\begin{lem}\label{rhosoomegaeqnuLEM}
For any $\omega\in P(\TT_j^d)'$,
the measure $\oomega$ is $\SL_d(\R)$ invariant,
and $\tiota_*\,\oomega=\nu$.
\end{lem}
(Here, by abuse of notation, 
we write $\nu$ also for the measure on $\SL_d(\Z)\bs\SL_d(\R)$
induced by the Haar measure $\nu$.)
\begin{proof}
It is obvious from 
\eqref{toomegaDEF}
that $\toomega$ is right $\SL_d(\R)$ invariant;
hence also $\oomega$ is $\SL_d(\R)$ invariant.
The fact that $\tiota_*\,\oomega=\nu$
is immediate from \eqref{oomegapartDEFlemres}.
\end{proof}

\begin{lem}\label{mapPTtoPXjcontLEM}
The map in \eqref{mapPTtoPXjrep} is continuous.
(Here $P(\TT_j^d)'$ is equipped with the subspace topology from $P(\TT_j^d)$.)
\end{lem}
\begin{proof}
Recall that $P(S)$ is metrizable 
for any separable and metrizable topological space $S$
\cite[pp.\ 72--73]{pB99};
in particular $P(\TT_j^d)$ and $P(\XX_j)$ are metrizable.
Hence it suffices to prove that if $(\omega_k)$ is any sequence in $P(\TT_j^d)$ converging
to $\omega\in P(\TT_j^d)$,
we have $\oomega_k\to\oomega$ in $P(\XX_j)$.
To prove this, we have to prove that $\oomega_k(\varphi)\to\oomega(\varphi)$ for any $\varphi\in \C_b(\XX_j)$;
but it is a well-known fact that it suffices to prove this for $\varphi\in \C_c(\XX_j)$.
Thus for a fixed
$\varphi\in \C_c(\XX_j)$,
our task is to prove that
\begin{align*}
\int_{\TT_j^d}\int_{F_d}\varphi(x(U) A)\,d\nu(A)\,d\omega_k(U)
\end{align*}
converges to the corresponding integral with $\omega$, as $k\to\infty$.
The fact that $\varphi$ has compact support in $\XX_j$
implies that $\alpha(U):=\int_{F_d}\varphi(x(U)A)\,d\nu(A)$
is a continuous function on $\TT_j^d$.

Now the task is to prove that
$\int_{\TT_j^d}\alpha \,d\omega_k\to
\int_{\TT_j^d}\alpha \,d\omega$,
and this holds by definition
since $\omega_k\to\omega$ in $P(\TT_j^d)$.
\end{proof}

Finally, for any $V\in\TT_j^d$
we will now identify the measure 
$\overline{\omega^{(V)}}\in P(\XX_j)$
as an invariant measure on a certain homogeneous submanifold of $\XX_j$.
Given $V\in\TT_j^d$,
it follows from Lemma~\ref{IDCLcharLEM}
that there exists some $X\in\M_{r_j\times d}(\Q)$ 
such that 
$V-\pi(X)\in\bigl(\SS_j^{(V)}\bigr)^{\circ\, d}$,
and this means that we can choose some
$\tV\in X+\bigl(L_j^{(V)}\bigr)^d\subset\M_{r_j\times d}(\R)$ with $\pi(\tV)=V$.
Let $L=L_j^{(V)}$ and recall the definition of $\S_L(\R)$ in \eqref{SVdefrepnew}.
It follows that
the $\S_L(\R)$-orbit of the point $x(V)$ in $\XX_j$
can be expressed as:
\begin{align}\label{xVSLR}
x(V)\cdot \S_L(\R)
=\Gamma_j\bs\Gamma_j\I_{\tV}\S_L(\R)
=\Gamma_j\bs\Gamma_j\I_{X}\S_L(\R).
\end{align}
By the proof of Lemma \ref{mujqjustifyLEM}
(applied to $\tV$ in place of $U_j^{(\vecq)}$),
the orbit in \eqref{xVSLR}
is a closed embedded submanifold of $\XX_j$
which carries a unique $\S_{L}(\R)$-invariant
probability measure.

\begin{prop}\label{SLRinvprobmeasLEM}
For any $V\in\TT_j^d$, %
the unique
$\S_{L_j^{(V)}}(\R)$-invariant probability measure on 
the orbit $x(V)\cdot \S_{L_j^{(V)}}(\R)$
equals $\overline{\omega_j^{(V)}}$.
In particular, for any $\vecq\in\R^d$ we have
$\mu_j^{(\vecq)}=\overline{\omega_j^{(\vecq)}}$.
\end{prop}
\begin{proof}
Note that the second part of the proposition is an immediate consequence of the first part,
since, by the definition in Section \ref{HOMDYNintrononunifTHMSEC},
$\mu_j^{(\vecq)}$ is the 
$\S_{L_j^{(\vecq)}}(\R)$-invariant 
probability measure on
the orbit $x(\pi(U_j^{(\vecq)}))\cdot \S_{L_j^{(\vecq)}}(\R)$ in $\XX_j$,
and furthermore we have $L_j^{(\vecq)}=L_j^{(\pi(U_j^{(\vecq)}))}$
and $\omega_j^{(\vecq)}=\omega_j^{(\pi(U_j^{(\vecq)}))}$.

To prove the first part of the proposition, 
set $\omega:=\omega_j^{(V)}$,
and let $X$, $L$ and $\tV$ be as in \eqref{xVSLR}.
Since 
$\oomega=\overline{\omega_j^{(V)}}$
by definition
is the probability measure on $\XX_j$ 
which corresponds to the
$\Gamma_j$-invariant measure $\toomega$ on $G_j$
given by \eqref{toomegaDEF},
it suffices to verify that $\toomega$ is supported on
$\Gamma_j\I_{\tV}\S_L(\R)$ 
and that $\toomega$ is right $\S_L(\R)$ invariant.
However, it follows from \eqref{toomegaDEF} that
\begin{align}\label{SLRinvprobmeasLEMpf1}
\supp(\toomega)=\{\I_UA\col U\in\pi^{-1}(\scrO_j(V)),\: %
A\in\SL_d(\R)\},
\end{align}
and here we have
\begin{align}\label{SLRinvprobmeasLEMpf2}
\pi^{-1}(\scrO_j(V))
=\{(\tV+W)\gamma\col \gamma\in\SL_d(\Z),\:W\in L^d+\M_{r_j\times d}(\Z)\}
\end{align}
(cf.\ \eqref{OjVDEF}).
Note here that both $L^d$ and $\M_{r_j\times d}(\Z)$
are $\SL_d(\Z)$-invariant subsets of $\M_{r_j\times d}(\R)$.
Using \eqref{SLRinvprobmeasLEMpf1}, \eqref{SLRinvprobmeasLEMpf2}
and $\I_{(\tV+W)\gamma}A=\gamma^{-1}\I_{\tV+W}\gamma A$
it follows that $\supp(\toomega)=\Gamma_j\I_{\tV}\S_L(\R)$, as desired.

Finally we verify that $\toomega$ is right
$\S_L(\R)$-invariant.
We have noted that 
$\toomega$ is right $\SL_d(\R)$ invariant
(cf.\ Lemma \ref{rhosoomegaeqnuLEM});
hence it suffices to verify that $\toomega$ is 
right $\I_W$-invariant for every $W\in L^d$.
However this also follows from %
\eqref{toomegaDEF},
by noticing that
$\I_U A \I_W=\I_{U+WA^{-1}}A$   %
for all $U\in\M_{r_j\times d}(\R)$ and $A\in\SL_d(\R)$,
and using $WA^{-1}\in L^d$
and the fact that $\tomega$ is invariant under $L^d$-translations %
(since $\omega=\omega_j^{(V)}$ is invariant under $(\SS_j(V)^\circ)^d$-translations). 
\end{proof}

\subsection{A Siegel integration formula for the measure $\oomega$}
The classical Siegel integration formula 
\cite{cS45a}
states that for any measurable function $f:\R^d\to\R_{\geq0}$
(or for any $f\in\C_c(\R^d)$),
defining the \textit{Siegel transform} $\hf$ on $\SL_d(\Z)\bs\SL_d(\R)$ by
\begin{align*}
\hf(\SL_d(\Z)g)=\sum_{\vecm\in\Z^d}f(\vecm g)\qquad(g\in\SL_d(\R)),
\end{align*}
then
\begin{align*}
\int_{\SL_d(\Z)\bs\SL_d(\R)}\hf\,d\nu=\int_{\R^d}f(\vecx)\,d\vecx+f(\bn).
\end{align*}
The following proposition gives an analogous formula involving the measure $\oomega$,
for any given $\omega\in P(\TT_j^d)'$.
This formula will be used later in our proof of [P2] (uniform spherical equidistribution);
see the proof of Lemma \ref{P2condthmLEM6} below.
\begin{prop}\label{SIEGELFORMULALEM1}
Let $\psi=(j,i)\in\Psi$.
For any Borel measurable function $f\in\R^d\to\R_{\geq0}$, we
define its ``$\psi$-Siegel transform'' $\hf_\psi:\XX_j\to\R_{\geq0}\cup\{+\infty\}$ by
\begin{align}\label{SIEGELFORMULALEM1res1}
\hf_\psi(\Gamma_jg)=\sum_{\vecm\in\Z^d} f\bigl(c_\psi\cdot(\vecm\,\a_i(g))\bigr)
\qquad(g\in G_j).
\end{align}
Then for any $\omega\in P(\TT_j^d)'$ %
we have
\begin{align}\label{SIEGELFORMULALEM1res2}
\int_{\XX_j}\hf_\psi\,d\oomega=
\nbar_\psi %
\int_{\R^d} f(\vecx)\,d\vecx
+\omega(\tr_i^{-1}(\{\bn\})) f(\bn),
\end{align}
as an equality of extended real numbers in $\R_{\geq0}\cup\{+\infty\}$.
\end{prop}
\begin{remark}
The reason for the formula \eqref{SIEGELFORMULALEM1res1}
is that this corresponds to the ``$\psi$-part'' of the union in \eqref{Gammagpointset}.
Note that $\widehat{f_\psi}$ is well-defined, in the sense that the right hand side of 
\eqref{SIEGELFORMULALEM1res1} remains the same if $g$ is replaced by $\gamma g$ for any $\gamma\in\Gamma_j$.
\end{remark}
\begin{proof}
Let $\omega\in P(\TT_j^d)'$ be given.
Define the Borel measure $\mu$ on $\R^d$ by
$\mu(B):=\int_{\XX_j}\widehat{(\chi_B)_\psi}\,d\oomega$
for any Borel set $B\subset\R^d$.
Let us prove that $\mu$ is finite on compact sets, so that $\mu$ is a Radon measure.
For this, it suffices to prove that $\mu(B)<\infty$ for any ball $B=\scrB_R^d$. %
In this case we have, for all $g\in G_j$:
\begin{align}\label{SIEGELFORMULALEM1pf1}
\widehat{(\chi_B)_\psi}(\Gamma_jg)=\#\bigl(B\cap c_\psi(\Z^d\a_i(g))\bigr)
\ll\#\bigl(B\cap c_\psi\Z^d\iota(g)\bigr),
\end{align}
by \cite[Proposition 5]{hGcS91}
(the implied constant in the last bound depends only on $d$).
Furthermore,
\begin{align*}
\int_{\XX_j}\#\bigl(B\cap c_\psi\Z^d\iota(g)\bigr)\,d\oomega(g)
=\int_{\SL_d(\Z)\bs\SL_d(\R)}\#\bigl(B\cap c_\psi\Z^dg\bigr)\,d\nu(g)
=1+\vol(c_\psi^{-1}B)<\infty,
\end{align*}
where we applied Siegel's original integration formula
\cite{cS45a}.

Next, one verifies that for any Borel set $B\subset\R^d$ and any $T\in\SL_d(\R)$ and $g\in G_j$,
we have
$\widehat{(\chi_{BT})_\psi}(\Gamma_jg)=\widehat{(\chi_{B})_\psi}(\Gamma_jgT^{-1})$.
Using this identity, and the fact that $\oomega$ is $\SL_d(\R)$ invariant
(cf.\ Lemma \ref{rhosoomegaeqnuLEM}),
it follows that 
$\mu(BT)=\mu(B)$.
Hence $\mu$ is $\SL_d(\R)$ invariant. %
By the well-known characterization of $\SL_d(\R)$ invariant Radon measures on $\R^d$,
it follows that
\begin{align}\label{SIEGELFORMULALEM1pf2}
\mu=c_1\delta_{\bn}+c_2\vol
\end{align}
for some constants $c_1,c_2\in\R_{\geq0}$,
where $\vol$ is the Lebesgue measure on $\R^d$.
Here 
\begin{align*}
c_1=\mu(\{\bn\})=\int_{F_d}\int_{\TT_j^d}\widehat{(\chi_{\{\bn\}})_{\psi}}\bigl(x(U)A\bigr)\,d\omega(U)\,d\nu(A)
=\omega(\tr_i^{-1}(\{\bn\})),
\end{align*}
where the second equality holds by \eqref{oomegapartDEFlemres},
and the last equality holds since 
$\widehat{(\chi_{\{\bn\}})_{\psi}}\bigl(x(U)A\bigr)=I(\tr_i(U)=\bn)$
for all $U\in\TT_j^d$ and all $A\in\SL_d(\R)$.
Furthermore, it follows from \eqref{SIEGELFORMULALEM1pf2} that
\begin{align*}
c_2=\lim_{R\to\infty}\frac{\mu(\scrB_R^d)}{\vol(\scrB_R^d)}
=\lim_{R\to\infty}\int_{\XX_j}\frac{\widehat{(\chi_{\scrB_R^d})_\psi}}{\vol(\scrB_R^d)}\,d\oomega,
\end{align*}
and here for each fixed point in $\Gamma_jg\in\XX_j$, the value of the 
integrand tends to $c_\psi^{-d}=\nbar_\psi$ as $R\to\infty$,
by \eqref{SIEGELFORMULALEM1res1} and
since $c_\psi\cdot(\Z^d\a_i(g))$ is an affine lattice of covolume $c_\psi^d$ in $\R^d$.
Hence by Lebesgue's dominated convergence theorem,
the application of which is justified by the bound 
\begin{align*}
\frac{\widehat{(\chi_{\scrB_R^d})_\psi}(\Gamma_jg)}{\vol(\scrB_R^d)}
\ll\frac{\#(\scrB_R^d\cap c_\psi\Z^d\iota(g))}{\vol(\scrB_R^d)}
\ll\frac{\#(\scrB_1^d\cap c_\psi\Z^d\iota(g))}{\vol(\scrB_1^d)}
\end{align*}
(where we first used \eqref{SIEGELFORMULALEM1pf1}
and then \cite[Proposition 4]{hGcS91})
and %
Siegel's integration formula,
\cite{cS45a},
we conclude that $c_2=\nbar_\psi$.

We have thus proved that \eqref{SIEGELFORMULALEM1res2} holds
whenever $f$ is the characteristic function of a Borel set $B\subset\R^d$.
By taking finite linear combinations, it follows that 
\eqref{SIEGELFORMULALEM1res2} holds whenever $f$ is a simple function,
and finally by a standard approximation argument using the 
monotone convergence theorem one proves
\eqref{SIEGELFORMULALEM1res2} in the general case.
\end{proof}

\section{Limit behaviour of $\omega_j^{(\vecq)}$}
\label{omegajqlimitSEC}

Recall that a central task for us will be to construct an appropriate marking
$\vs:\scrP\to\Sigma$ and a map $\vs\mapsto\mu_\vs$
from $\Sigma$ to $P(N(\scrX))$. %
The composed map %
$\vecq\mapsto\mu_{\vs(\vecq)}$ gives,
for any $\vecq\in\scrP$, 
the limit measure appearing in the spherical equidistribution statement [P2],
and as explained in Section \ref{HOMDYNintrononunifTHMSEC},
we will obtain this limit measure as a push-forward of the measure
$\mu^{(\vecq)}=\mu_1^{(\vecq)}\otimes\cdots\otimes\mu_N^{(\vecq)}$ on $\XX$.
It follows from Proposition \ref{SLRinvprobmeasLEM} that 
$\mu^{(\vecq)}$ can be obtained from the measures $\omega_j^{(\vecq)}\in P(\TT_j^d)'$
($j=1,\ldots,N$).
Therefore, as a crucial preparation for the construction of the maps
$\vs:\scrP\to\Sigma$ and $\vs\mapsto\mu_\vs$ (to be given in Section \ref{spaceofmarksandmapsSEC}),
we will in this section %
study how
each measure $\omega_j^{(\vecq)}$ varies as $\vecq$ varies in $\scrP$;
in particular we are interested in the behaviour of
$\omega_j^{(\vecq)}$ as $\vecq$ tends to infinity
within $\scrP$.
The main result of the present section states that,
under a certain admissibility assumption 
on the presentation of the given point set $\scrP$ as a union of grids,
it holds for each $\psi\in\Psi$
that the measure $\omega_j^{(\vecq)}$ 
tends to a unique limit measure $\omega_j^\psi\in P(\TT_j^d)'$
as $\vecq$ tends to infinity within a full density subset of $\scrL_\psi$.

\subsection{The limit spaces $L_j^{\psi}$ and $L_j$}
\label{Ljpsi0SEC}
As a first step, we will study the %
behaviour of the linear spaces $L_j^{(\vecq)}$
as $\vecq$ varies through a fixed grid $\scrL_\psi$.
(Recall that the space $L_j^{(\vecq)}$ plays a central role in the definition of
the measure $\omega_j^{(\vecq)}$.)
For each $\psi$ and $j$, %
we will define a certain rational space $L_j^\psi\subset\R^{r_j}$,
and will prove that this space is,
in a certain sense,
the limit of the spaces $L_j^{(\vecq)}$ as $\vecq$
tends to infinity within a full density subset of $\scrL_\psi$.

Let $\psi\in\Psi$ and $j\in\{1,\ldots,N\}$.
We start by introducing some auxiliary notation 
for keeping track of the matrix $U_j^{(\vecq)}$ as $\vecq$ varies within $\R^d$ and $\scrL_\psi$,
respectively.
Recall from \eqref{LpsiDEF}
that we have introduced $M_j\in\SL_d(\R)$, $c_\psi\in\R_{>0}$ and $\vecw_\psi\in\R^d$
so that 
$\scrL_\psi=c_{\psi}(\Z^d+\vecw_{\psi})M_{j_\psi}$.
Now set, for any $j\in\{1,\ldots,N\}$ and $\psi\in\Psi$,
\begin{align}\label{cjdef}
\tvecc_j=\begin{pmatrix} c_{j,1}^{-1} \\ \vdots \\ c_{j,r_j}^{-1}\end{pmatrix}\in\R^{r_j}
\qquad\text{and}\qquad
\vecc_j^{\psi}:=c_\psi\,\tvecc_{j}\in\R^{r_j},
\end{align}
and also
\begin{align}\label{Wjdef}
W_j:=\begin{pmatrix}\vecw_{j,1}\\\vdots\\\vecw_{j,r_j}\end{pmatrix}\in\M_{r_j\times d}(\R)
\qquad\text{and}\qquad
W_j^\psi:=W_j-\vecc_j^\psi\vecw_\psi T_j^\psi\in\M_{r_j\times d}(\R),
\end{align}
where
\begin{align*}
T_j^{\psi}:=M_{j_\psi}M_j^{-1}\in\SL_d(\R).
\end{align*}
The point of this notation is that we now have   
(cf.\ \eqref{Ujqdef})
\begin{align}\label{Ujqgenformula}
U_j^{(\vecq)}=W_j-\tvecc_j\vecq M_j^{-1},
\qquad\forall\vecq\in\R^d.
\end{align}
It follows that for any point 
$\vecq$ in $\scrL_\psi=c_{\psi}(\Z^d+\vecw_{\psi})M_{j_\psi}$,
\begin{align}\label{Ujqformula}
U_j^{(\vecq)}=W_j^{\psi}-\vecc_j^{\psi}\vecm T_j^{\psi}
\qquad\text{when }\: \vecq=c_\psi(\vecm+\vecw_\psi)M_{j_\psi}\:\: (\forall\vecm\in\Z^d).
\end{align}

Now let $L_j^{\psi}$ be the rational subspace of $\R^{r_j}$ given by
\begin{align}\label{Ljpsi0DEF}
L_j^{\psi}:=
\begin{cases}
\fL\bigl(\{\vecc_j^{\psi}\}\cup\{W_{j,\ell}^{\psi}\col\ell\in\{1,\ldots,d\}\}\bigr)&\text{if }\: j=j_\psi
\\[3pt]
\fL\bigl(\R\vecc_j^{\psi}\cup\{W_{j,\ell}^{\psi}\col\ell\in\{1,\ldots,d\}\}\bigr)&\text{if }\: j\neq j_\psi,
\end{cases}
\end{align}
where $W_{j,1}^{\psi},\ldots,W_{j,d}^{\psi}\in\R^{r_j}$ are the column vectors of $W_j^{\psi}$.
\begin{lem}\label{Ljpsi0keypropLEM2}
Let $\psi\in\Psi$ and $j\in\{1,\ldots,N\}$.
Let $\scrL'$ be any subgrid of $\scrL_\psi$.
Then
\begin{align*}
L_j^{\psi}=\fL\bigl(\{U_{j,\ell}^{(\vecq)}\col \vecq\in\scrL',\: \ell\in\{1,\ldots,d\}\}\bigr),
\end{align*}
where $U_{j,1}^{(\vecq)},\ldots,U_{j,d}^{(\vecq)}\in\R^{r_j}$ are the column vectors of $U_j^{(\vecq)}$.
\end{lem}

\begin{proof}
Set $S_0=\{W_{j,\ell}^{\psi}\col\ell\in\{1,\ldots,d\}\}$
and $S=\{\vecc_j^{\psi}\}\cup S_0$ if $j=j_\psi$;
otherwise $S=\R\vecc_j^{\psi}\cup S_0$.
Also set
$S':=\{U_{j,\ell}^{(\vecq)}\col \vecq\in\scrL',\: \ell\in\{1,\ldots,d\}\}$.
Then the task is to prove that $\fL(S)=\fL(S')$.
Let $T_{j,1}^{\psi},\ldots,T_{j,d}^{\psi}\in\R^{r_j}$ be the column vectors of $T_j^{\psi}$.
Then by \eqref{Ujqformula} we have
\begin{align}\label{Ujlqformula}
U_{j,\ell}^{(\vecq)}=W_{j,\ell}^{\psi}-\vecc_j^{\psi}\vecm T_{j,\ell}^{\psi}
\qquad\text{when }\: \vecq=c_\psi(\vecm+\vecw_\psi)M_{j_\psi}\:\: (\forall\vecm\in\Z^d, \: \ell\in\{1,\ldots,d\}).
\end{align}
Using \eqref{Ujlqformula} and the fact that $T_j^\psi=I$ if $j=j_\psi$, it follows that 
$S'\subset\langle S\rangle$,
and so
$\fL(S')\subset\fL(S)$.
It remains to prove that
$\fL(S)\subset\fL(S')$.

Let $\scrL''$ be the inverse image of $\scrL'$ under the 
bijection
$\vecm\mapsto c_{\psi}(\vecm+\vecw_{\psi})M_{j_\psi}$
from $\Z^d$ onto $\scrL_{\psi}$.
Then $\scrL''$ is a coset of a full rank sublattice of $\Z^d$;
hence there is some ${n}\in\Z^+$ and some $\vecm_0\in\Z^d$ such that
$\scrL''$ contains $\vecm_0+{n}\Z^d$,
and so, by \eqref{Ujlqformula}:
\begin{align}\label{Ljpsi0keypropLEM2pf10}
W_{j,\ell}^\psi-\vecc_j^\psi\vecm T_{j,\ell}^\psi\in S',
\qquad\forall\vecm\in\vecm_0+{n}\Z^d,\ \ell\in\{1,\ldots,d\}.
\end{align}
By Lemma \ref{IDCLcharLEM} there is some $m\in\Z^+$ such that
$S'\subset m^{-1}\Z^{r_j}+\fL(S')$.
Hence, since every point in $\Z^d$ can be written as an affine linear
combination of points in $\vecm_0+{n}\Z^d$,
with all weights in $n^{-1}\Z$,
it follows that
\begin{align}\label{Ljpsi0keypropLEM2pf10a}
W_{j,\ell}^\psi-\vecc_j^\psi\vecm T_{j,\ell}^\psi\in
\scrA,
\qquad\forall\vecm\in\Z^d,\ \ell\in\{1,\ldots,d\},
\end{align}
where
$\scrA:= ({n}m)^{-1}\Z^{r_j}+\fL(S')$
(this is a closed subgroup of $\R^{r_j}$).
We will prove that $S\subset\scrA$;
by Lemma \ref{IDCLcharLEM} this implies
$\fL(S)\subset\fL(S')$, and so the proof of Lemma \ref{Ljpsi0keypropLEM2}
will be complete.
By taking $\vecm=\bn$ in 
\eqref{Ljpsi0keypropLEM2pf10a}
it follows that $S_0\subset\scrA$;
hence it suffices to prove that 
$\vecc_j^\psi\in\scrA$ if $j=j_\psi$,
and $\R\vecc_j^\psi\subset\scrA$ if $j\neq j_\psi$.

We denote by $\vece_k$ the $k$th standard unit vector in $\R^d$.\label{vecekrowDEF}
If $j=j_\psi$ then
$T_j^\psi$ is the identity matrix,
and applying 
\eqref{Ljpsi0keypropLEM2pf10a} with $\ell=1$ and $\vecm\in\{\vece_1,\bn\}$
it follows that $\scrA$ contains both
$W_{j,\ell}^\psi-\vecc_j^\psi$
and $W_{j,\ell}^\psi$;
hence also
$\vecc_j^\psi\in\scrA$.

Next assume $j\neq j_\psi$.
Then $T_j^{\psi}\notin\scrS$
(cf.\ \eqref{GENPOINTSET1req}),
and so there exist two (non-zero) entries of the matrix 
$T_j^{\psi}$ which have an irrational ratio;
say entries $k,\ell$ and $k',\ell'$, respectively.
Writing $T_j^{\psi}=(t_{r,s})$ we thus have $t_{k,\ell}/t_{k',\ell'}\notin\Q$,
which implies that the set 
$\{at_{k,\ell}+bt_{k',\ell'}\col a,b\in\Z\}$
is dense in $\R$.
By considering the difference of two arbitrary vectors as in
\eqref{Ljpsi0keypropLEM2pf10a} we have
$\vecc_j^\psi\vecm T_{j,\ell}^\psi\in\scrA$
for all 
$\vecm\in\Z^d$.
Taking here $\vecm=\vece_k$ gives
$t_{k,\ell}\vecc_j^\psi\in\scrA$.
Similarly we also have
$t_{k',\ell'}\vecc_j^\psi\in\scrA$,
and hence 
$(at_{k,\ell}+bt_{k',\ell'})\vecc_j^\psi\in\scrA$
for all $a,b\in\Z$.
Hence since $\scrA$ is closed,
we have $\R\vecc_j^\psi\subset\scrA$.
\end{proof}

\begin{lem}\label{LjqsubsetLjpsi0LEM}
Let $\psi\in\Psi$ and $j\in\{1,\ldots,N\}$.
For every $\vecq\in\scrL_{\psi}$ 
we have $L_j^{(\vecq)}\subset L_j^{\psi}$.
\end{lem}
\begin{proof}
By definition,
$L_j^{(\vecq)}=\fL(\{U_{j,\ell}^{(\vecq)}\col\ell\in\{1,\ldots,d\}\})$;
hence the statement %
follows from 
Lemma \ref{Ljpsi0keypropLEM2},
and the fact that $\fL(S)$ is increasing in $S$.
\end{proof}

Our goal is now to prove Lemma \ref{MuconvauxLEM1} below,
which says that,
in an appropriate sense, 
the space $L_j^{(\vecq)}$ approaches $L_j^\psi$
as $\vecq$ tends to infinity within a full density subset of $\scrL_\psi$.
(Cf.\ also Remark \ref{MuconvauxLEM1rem} below.)
We will need the following auxiliary lemma.
We denote by ``$\:\cdot\:$'' the standard scalar product in the space $\R^{r_j}$.

\begin{lem}\label{IDCLcharLEM2}
Let $\psi\in\Psi$ and $j\in\{1,\ldots,N\}$.
For any non-empty subset $S\subset\R^{r_j}$ and any vector $\veca\in\Q^{r_j}$,
we have $\veca\perp\fL(S)$ if and only if
$\{\veca\cdot\vecv\col\vecv\in S\}\subset n^{-1}\Z$
for some $n\in\Z^+$.
\end{lem}
\begin{proof}
The case $\veca=\bn$ is trivial;
hence from now on we assume $\veca\neq\bn$.
The statement of the lemma is
invariant under rescaling of $\veca$ by a non-zero rational number;
hence we may assume that $\veca\in\Z^{r_j}$ and 
$\gcd(a_1,\ldots,a_{r_j})=1$.
It follows that 
$\{\veca\cdot\vecm\col\vecm\in\Z^{r_j}\}=\Z$,
and hence $\{\veca\cdot\vecm\col\vecm\in n^{-1}\Z^{r_j}\}=n^{-1}\Z$
for any $n\in\Z^+$.
Therefore, if
$\{\veca\cdot\vecv\col\vecv\in S\}\subset n^{-1}\Z$
then $S\subset n^{-1}\Z^{r_j}+\veca^\perp$,
and so by Lemma \ref{IDCLcharLEM},
$\fL(S)\subset\veca^\perp$,
i.e.\ $\veca\perp\fL(S)$.
Conversely, 
assume $\veca\perp\fL(S)$.
By Lemma \ref{IDCLcharLEM}
there is some $n\in\Z^+$ such that 
$S\subset n^{-1}\Z^{r_j}+\fL(S)\subset n^{-1}\Z^{r_j}+\veca^\perp$,
and this implies
$\{\veca\cdot\vecv\col\vecv\in S\}\subset n^{-1}\Z$.
\end{proof}
\begin{remark}\label{IDCLcharLEM2rem2}
If $S$ is a \textit{finite} subset of $\R^{r_j}$,
then Lemma \ref{IDCLcharLEM2}
implies that a vector $\veca\in\Q^r$
is orthogonal to $\fL(S)$
if and only if $\{\veca\cdot\vecv\col\vecv\in S\}\subset\Q$.
Furthermore, 
if $S=\R\vecc\cup S'$ for some vector $\vecc\in\R^r$ and a finite subset $S'\subset\R^r$,
then Lemma~\ref{IDCLcharLEM2} implies that
a vector $\veca\in\Q^r$ is orthogonal to $\fL(S)$
if and only if $\veca\perp\vecc$ %
and $\{\veca\cdot\vecv\col\vecv\in S'\}\subset\Q$.
\end{remark}

\begin{lem}\label{MuconvauxLEM1}
Let $\psi\in\Psi$ and $j\in\{1,\ldots,N\}$.
For every $\veca\in\Z^{r_j}$ with $\veca\not\perp L_j^\psi$, we have
\begin{align}\label{MuconvauxLEM1res}
\#\{\vecq\in\scrL_{\psi}\cap \scrB_R^d\col \veca\perp L_j^{(\vecq)}\}\ll R^{d-1}
\qquad\text{as }\: R\to\infty.
\end{align}
\end{lem}
\begin{proof}
As before,
let us parametrize the points in $\scrL_{\psi}=c_{\psi}(\Z^d+\vecw_{\psi})M_{j_\psi}$
as
$\vecq=\vecq(\vecm)=c_{\psi}(\vecm+\vecw_{\psi})M_{j_\psi}$,
with $\vecm$ running through $\Z^d$.
We will prove the following bound, which clearly implies 
\eqref{MuconvauxLEM1res}:
\begin{align}\label{MuconvauxLEM1pf1}
\#\{\vecm\in\Z^d\cap[-R,R]^d\col \veca\perp L_j^{(\vecq(\vecm))}\}\leq(2R+1)^{d-1}
\qquad(\forall R>0).
\end{align}
Using $L_j^{(\vecq)}=\fL(\{U_{j,\ell}^{(\vecq)}\col\ell\in\{1,\ldots,d\}\})$
together with \eqref{Ujlqformula} and Lemma \ref{IDCLcharLEM2},
the bound \eqref{MuconvauxLEM1pf1} can equivalently be stated as:
\begin{align}\label{MuconvauxLEM1pf2}
\#\bigl\{\vecm\in\Z^d\cap[-R,R]^d\col 
\veca\cdot\bigl(W_{j,\ell}^{\psi}-\vecc_j^{\psi}\vecm T_{j,\ell}^{\psi}\bigr)\in\Q
\:\:\forall \ell\in\{1,\ldots,d\}\bigr\}
\leq(2R+1)^{d-1}.
\end{align}
We are assuming that $\veca$ is not orthogonal to $L_j^{\psi}$.
By \eqref{Ljpsi0DEF} and 
Lemma \ref{IDCLcharLEM2} 
(cf.\ also Remark~\ref{IDCLcharLEM2rem2}),
this implies that
\begin{align*}
\begin{cases}
\text{If }j=j_\psi:&
\veca\cdot\vecc_j^{\psi}\notin\Q\:\text{ or }\: \veca\cdot W_{j,\ell}^{\psi}\notin\Q\text{ for some }\ell\in\{1,\ldots,d\},
\\
\text{If }j\neq j_\psi:&
\veca\cdot\vecc_j^{\psi}\neq0\:\text{ or }\: \veca\cdot W_{j,\ell}^{\psi}\notin\Q\text{ for some }\ell\in\{1,\ldots,d\}.
\end{cases}
\end{align*}

Let us first assume $j\neq j_\psi$.
If $\veca\cdot\vecc_j^{\psi}=0$ 
then 
$\veca\cdot W_{j,\ell}^{\psi}\notin\Q$ for some $\ell$,
and it follows that
the set in \eqref{MuconvauxLEM1pf2} is \textit{empty}, for all $R$.
Hence we may assume that
$\veca\cdot\vecc_j^{\psi}\neq0$.
Now $j\neq j_\psi$ implies that $T_j^{\psi}\notin\scrS$,  %
which means that there exist two (non-zero) entries of the matrix 
$T_j^{\psi}$ which have an irrational ratio.
Hence there exist $k,\ell\in\{1,\ldots,d\}$ 
such that $(\veca\cdot\vecc_j^{\psi})(\vece_k T_{j,\ell}^{\psi})\notin\Q$.
This implies that, for any $R$,
the set in \eqref{MuconvauxLEM1pf2}
contains at most one point $\vecm$
along any line %
parallel with $\vece_{k}$.
Hence the bound in \eqref{MuconvauxLEM1pf2} holds.

Next assume $j=j_\psi$.
Then $T_j^{\psi}=\I$.
If $\veca\cdot\vecc_j^{\psi}\in\Q$ then $\veca\cdot W_{j,\ell}^{\psi}\notin\Q$ for some $\ell\in\{1,\ldots,d\}$,
and it follows that
the set in \eqref{MuconvauxLEM1pf2} is \textit{empty}, for all $R$.
Hence we may assume that
$\veca\cdot\vecc_j^{\psi}\notin\Q$.
Then for each $\ell\in\{1,\ldots,d\}$
there is at most one integer $m$
such that $\veca\cdot (W_{j,\ell}^{\psi}-m\vecc_j^{\psi})\in\Q$,
and so the set in \eqref{MuconvauxLEM1pf2} contains at most one point,
and the bound in \eqref{MuconvauxLEM1pf2} holds.
\end{proof}

We end this section by introducing 
a certain rational space $L_j\subset\R^{r_j}$
which equals $L_j^{(\vecq)}$ for \textit{$\vecq$ generic within $\R^d$}
(cf.\ Lemma \ref{LjbasicLEM1} below).
These spaces are closely related %
to %
the spaces $L_j^\psi$ introduced above,
and they appear in %
the explicit description of the
macroscopic limit measure
$\mu^{\g}$
which we discuss in Section \ref{P3proofSEC} below.
For each $j\in\{1,\ldots,N\}$, we set
\begin{align}\label{trueLjDEF}
L_j:=\fL\bigl(\R\tvecc_j\cup\{W_{j,\ell}\col\ell\in\{1,\ldots,d\}\}\bigr),
\end{align}
where $W_{j,1},\ldots,W_{j,d}\in\R^{r_j}$ are the column vectors of $W_j$ (see \eqref{Wjdef}).
\begin{lem}\label{LjbasicLEM1}
Let $j\in\{1,\ldots,N\}$.
For every $\vecq\in\R^d$ we have $L_j^{(\vecq)}\subset L_j$,
and for all except countably many $\vecq\in\R^d$ we even have
$L_j^{(\vecq)}=L_j$.
\end{lem}
\begin{proof}
Let $S_0=\{W_{j,\ell}\col\ell\in\{1,\ldots,d\}$ and $S=\R\tvecc_j\cup S_0$,
so that $L_j=\fL(S)$.
Recall that
for every $\vecq\in\R^d$ we have 
$L_j^{(\vecq)}=\fL(\{U_{j,\ell}^{(\vecq)}\col\ell\in\{1,\ldots,d\}\})$;
and by 
\eqref{Ujqgenformula},
$U_{j,\ell}^{(\vecq)}=W_{j,\ell}-\tvecc_j\vecq M_{j,\ell}'$,
where $M_{j,\ell}'\in\R^{r_j}$ is the $\ell$th column vector of $M_j^{-1}$.
Here $\tvecc_j\vecq M_{j,\ell}'\in\R\tvecc_j$,
and thus $U_{j,\ell}^{(\vecq)}\in\langle S\rangle$ for every $\ell$.
Therefore $L_j^{(\vecq)}\subset L_j$.

Next, since both $L_j^{(\vecq)}$ and $L_j$ are rational subspaces of $\R^{r_j}$,
if $L_j^{(\vecq)}\subsetneq L_j$ then there 
exists some $\veca\in\Z^{r_j}$
satisfying $\veca\perp L_j^{(\vecq)}$ but $\veca\not\perp L_j$.
By Lemma \ref{IDCLcharLEM2} 
(cf.\ also Remark~\ref{IDCLcharLEM2rem2})
this means that
\begin{align}\label{LjbasicLEM1pf1}
\veca\cdot(W_{j,\ell}-\tvecc_j\vecq M_{j,\ell}')\in\Q,\qquad \forall\ell\in\{1,\ldots,d\},
\end{align}
but either $\veca\not\perp\tvecc_j$ or $\veca\cdot W_{j,\ell}\notin\Q$ for some $\ell$.
Clearly this is not possible if $\veca\perp\tvecc_j$;
hence we must have $\veca\not\perp\tvecc_j$.
Noticing also that the condition \eqref{LjbasicLEM1pf1} is equivalent to\footnote{Note:
We view $\veca$ as a column vector, or equivalently as an $r_j\times 1$ matrix;
hence $\veca\trans$ is a $1\times r_j$ matrix.}
$\veca\trans(W_j-\tvecc_j\vecq M_j^{-1})\in\Q^{d}$, 
we conclude that:
\begin{align}\label{LjbasicLEM1pf2}
\{\vecq\in\R^d\col L_j^{(\vecq)}\subsetneq L_j\}
=\bigcup_{\substack{\veca\in\Z^{r_j}\\ (\veca\not\perp\tvecc_j)}}
\bigcup_{\vecb\in\Q^{d}} \bigl\{\vecq\in\R^d\col \veca\trans(W_j-\tvecc_j\vecq M_j^{-1})=\vecb\bigr\}.
\end{align}
But for every $\veca\in\Z^{r_j}$ with $\veca\not\perp\tvecc_j$
and every $\vecb\in\Q^d$,
the set
$\bigl\{\vecq\in\R^d\col \veca\trans(W_j-\tvecc_j\vecq M_j^{-1})=\vecb\bigr\}$
consists of exactly one point,
namely $\vecq=(\veca\trans\tvecc_j)^{-1}(\veca\trans W_j-\vecb)M_j$.
Hence the set in \eqref{LjbasicLEM1pf2} is countable, and the lemma is proved.
\end{proof}

\begin{lem}\label{LjvsLjpsiLEM}
For any $\psi\in\Psi$ and $j\in\{1,\ldots,N\}$
we have $L_j^\psi\subset L_j$,
and if $j\neq j_\psi$
then even $L_j^\psi=L_j$.
\end{lem}
\begin{proof}
Recall that $\vecc_j^\psi=c_\psi \tvecc_j$;
hence $\R\vecc_j^\psi=\R\tvecc_j$.
Furthermore, by comparing the definitions 
\eqref{Wjdef} and \eqref{Wjdef}
we note that
$W_{j,\ell}^\psi=W_{j,\ell}-\vecc_j^\psi \vecw_\psi T_{j,\ell}^\psi
\in W_{j,\ell}+\R\vecc_j^\psi
=W_{j,\ell}+\R\tvecc_j$.
Hence the two sets $\R\vecc_j^{\psi}\cup\{W_{j,\ell}^{\psi}\col\ell\in\{1,\ldots,d\}\}$
and $\R\tvecc_j\cup\{W_{j,\ell}\col\ell\in\{1,\ldots,d\}\}$
generate the same subgroups of $\R^{r_j}$.
In view of this fact,
the lemma now follows by inspecting the definitions of
$L_j^\psi$ and $L_j$, 
\eqref{Ljpsi0DEF} and \eqref{trueLjDEF}.
\end{proof}

\begin{lem}\label{LjvsLjpsiLEM2}
For every $\psi=(j,i)\in\Psi$,
$L_j^\psi=L_j\cap\vece_i^\perp$
and $L_j=L_j^\psi+\R\tvecc_j$.
\end{lem}
\begin{proof}
Since $\psi=(j,i)$, 
the $i$th coordinate of $\vecc_j^\psi$ is $1$ by \eqref{cjdef}
and $\r_i(W_j^\psi)=\bn$ by \eqref{Wjdef};
hence it follows from \eqref{Ljpsi0DEF} that 
$L_j^\psi\perp\vece_i$,
and so by Lemma \ref{LjvsLjpsiLEM}
we have 
$L_j^\psi\subset L_j\cap\vece_i^\perp$.
On the other hand, $\R\tvecc_j\subset L_j$ by \eqref{trueLjDEF}
and $\tvecc_j\cdot\vece_i=c_\psi^{-1}>0$;
in particular $L_j\not\perp\vece_i$.
Now to complete the proof of the lemma it suffices to prove that
\begin{align}\label{LjvsLjpsiLEM2pf1}
L_j\subset \tL,
\qquad\text{with }\:\tL:=L_j^\psi+\R\tvecc_j.
\end{align}
But it follows from \eqref{Ljpsi0DEF} and Lemma \ref{IDCLcharLEM}
that $\vecc_j^\psi\in \Q^{r_j}+L_j^\psi$;
hence since $\tL$ equals the linear span of $L_j^\psi$ and the vector $\vecc_j^\psi$,
$\tL$ is a rational subspace of $\R^{r_j}$.
It also follows from \eqref{Ljpsi0DEF} and Lemma \ref{IDCLcharLEM}
that each column vector of $W_j^\psi$ lies in
$\Q^{r_j}+L_j^\psi$;
thus $W_j^\psi\in\M_{r_j\times d}(\Q)+(L_j^\psi)^d$,
and hence we have
\begin{align*}
W_j=W_j^\psi+\vecc_j^\psi\vecw_\psi
\in\M_{r_j\times d}(\Q)+(L_j^\psi)^d+\vecc_j^\psi\vecw_\psi
\subset\M_{r_j\times d}(\Q)+\tL^d.
\end{align*}
This implies that $W_{\ell}^\psi\in \Q^{r_j}+\tL$ for each $\ell\in\{1,\ldots,d\}$,
and since also $\R\tvecc_j\subset\tL$,
the inclusion in \eqref{LjvsLjpsiLEM2pf1} 
now follows from \eqref{trueLjDEF} and Lemma \ref{IDCLcharLEM}.
\end{proof}

\subsection{The limit measures $\omega_j^\psi$ and $\omega_j^{\g}$}
\label{SSjpsi0Ojpsi0sec}
The following condition will be of crucial importance for us.
It involves the vectors $\vecc_j^\psi\in\R^{r_j}$
and the subspaces $L_j^\psi$
which were defined in
\eqref{cjdef} and \eqref{Ljpsi0DEF}
in the previous section.
\begin{definition}\label{admissibleDEF}
We say that a presentation of $\scrP$ as in \eqref{GENPOINTSET1},
and satisfying \eqref{THINDISJcond2} and \eqref{GENPOINTSET1req},
is \textit{admissible} if $\vecc_{j}^{\psi}\in L_{j}^{\psi}+\Z^{r_j}$ for all $\psi=(j,i)\in\Psi$.
\end{definition}
It should be carefully noted that the condition in Definition \ref{admissibleDEF}
only involves 
vectors $\vecc_j^\psi$ and spaces $L_j^\psi$
for pairs of $\psi$ and $j$ with $j=j_\psi$.

\vspace{5pt}

As we will see below, %
the admissibility of the presentation of $\scrP$
is a necessary 
(as well as sufficient)
condition to ensure that
all the measures $\omega_j^{(\vecq)}$
have a unique generic limit as $\vecq$ tends to infinity within any of the grids $\scrL_\psi$.
Luckily, it turns out that 
any point set $\scrP$ which is a finite union of grids
possesses an admissible presentation;
we will prove this in 
Section \ref{ATTAINADMsec} below.

\vspace{5pt}

\textit{From now on we assume that the given presentation \eqref{GENPOINTSET1} 
of $\scrP$ is admissible.}

\vspace{5pt}

We will start by defining,
for any $\psi\in\Psi$ and $j\in\{1,\ldots,N\}$,
a measure $\omega_j^\psi\in P(\TT_j^d)'$.
We will then prove that
$\omega_j^\psi$ is in fact the generic limit of the measures $\omega_j^{(\vecq)}$
as $\vecq$ tends to infinity within the grid $\scrL_\psi$.
As we will see, 
the assumption about admissibility is needed already for the definition of $\omega_j^\psi$ to make sense.   %
(Cf.\ the proof of Lemma \ref{Ojpsi0welldefLEM} below.)

Given $\psi\in\Psi$ and $j\in\{1,\ldots,N\}$ we
set $\SS_j^{\psi}:=\pi(L_j^{\psi})$.
\label{Sjpsi0DEF}
This is a closed subtorus of $\TT_j$, since $L_j^\psi$ is a rational subspace of $\R^{r_j}$.
Furthermore, we 
pick an arbitrary point $\vecq\in\scrL_{\psi}$, and define:
\begin{align}\label{Ojpsi0DEF}
&\tSS_j^{\psi}:=\SS_j^{(\vecq)}+\SS_j^{\psi}\subset\TT_j;
\qquad\text{and}\qquad
\scrO_j^{\psi}:=\bigcup_{\gamma\in\SL_d(\Z)}(\pi(U_j^{(\vecq)})\gamma+(\SS_j^{\psi})^d)\subset\TT_j^d.
\end{align}
\begin{lem}\label{Ojpsi0welldefLEM}
Both $\tSS_j^{\psi}$ and 
$\scrO_j^{\psi}$ are well-defined, 
i.e.\ the expressions in \eqref{Ojpsi0DEF} are independent of the choice of $\vecq$.
We have $\scrO_j^{(\vecq)}\subset\scrO_j^{\psi}$ for all $\vecq\in\scrL_{\psi}$.
Furthermore, $\tSS_j^{\psi}$ is a closed subgroup of $\TT_j$ whose connected component subgroup
equals $\SS_j^{\psi}$,
and $\scrO_j^{\psi}$ is a union of some of 
the connected components of $(\tSS_j^\psi)^d$.
\end{lem}
(It should be noted that since $(\tSS_j^\psi)^d$ is compact,
its total number of connected components is finite.)

In view of Lemma \ref{Ojpsi0welldefLEM},
we may now define 
$\omega_j^{\psi}\in P(\TT_j^d)$
\label{mujpsi0def}
to be the restriction to $\scrO_j^{\psi}$ of the Haar measure on $(\tSS_j^{\psi})^d$,
normalized so that $\omega_j^{\psi}(\scrO_j^{\psi})=1$.
Note that $\omega_j^{\psi}$ is $\SL_d(\Z)$-invariant by construction,
i.e.\ we actually have $\omega_j^{\psi}\in P(\TT_j^d)'$.

\begin{proof}[Proof of Lemma \ref{Ojpsi0welldefLEM}]
Let us first prove that 
\begin{align}\label{Ojpsi0welldefLEMpf1}
U_{j,\ell}^{(\vecq)}-U_{j,\ell}^{(\vecq')}\in L_j^{\psi}+\Z^{r_j},
\qquad
\forall \vecq,\vecq'\in\scrL_{\psi},
\:\:\forall\ell\in\{1,\ldots,d\}.
\end{align}
Indeed, pick $\vecm,\vecm'\in\Z^d$ so that
$\vecq=c_{\psi}(\vecm+\vecw_{\psi})M_{j_\psi}$
and $\vecq'=c_{\psi}(\vecm'+\vecw_{\psi})M_{j_\psi}$;
then by  \eqref{Ujlqformula}
we have $U_{j,\ell}^{(\vecq)}-U_{j,\ell}^{(\vecq')}=\vecc_j^{\psi}(\vecm'-\vecm) T_{j,\ell}^{\psi}$.
If $j\neq j_\psi$ then $\R\vecc_j^{\psi}\subset L_j^{\psi}$
by \eqref{Ljpsi0DEF}
and hence $U_{j,\ell}^{(\vecq)}-U_{j,\ell}^{(\vecq')}\in L_j^{\psi}$ for each $\ell$.
On the other hand if $j=j_\psi$ then $T_j^{\psi}=I$
and so $U_{j,\ell}^{(\vecq)}-U_{j,\ell}^{(\vecq')}\in\Z\,\vecc_j^{\psi}$,
and using the assumption that \textit{$\scrP$ is admissible}
(cf.\ Def.\ \ref{admissibleDEF}),
this implies that \eqref{Ojpsi0welldefLEMpf1} holds.

In order to prove that $\tSS_j^{\psi}$ is well-defined it suffices to
verify that 
for any two $\vecq,\vecq'\in\scrL_{\psi}$
we have
$\SS_j^{(\vecq')}\subset\SS_j^{(\vecq)}+\SS_j^{\psi}$.
However this follows from the definition of
$\SS_j^{(\vecq')}$
and the fact that, by \eqref{Ojpsi0welldefLEMpf1},
$\pi(U_{j,\ell}^{(\vecq')})\in\pi(U_{j,\ell}^{(\vecq)})+\SS_j^{\psi}\subset\SS_j^{(\vecq)}+\SS_j^{\psi}$.
Note also that \eqref{Ojpsi0welldefLEMpf1}
implies that 
$\pi(U_j^{(\vecq)})-\pi(U_j^{(\vecq')})\in(\SS_j^{\psi})^d$;
and this implies that $\scrO_j^{\psi}$ is well-defined.

Now consider an arbitrary point $\vecq\in\scrL_{\psi}$.
Recall that
$L_j^{(\vecq)}\subset L_j^{\psi}$,
by Lemma \ref{LjqsubsetLjpsi0LEM};
hence $(\SS_j^{(\vecq)})^\circ\subset\SS_j^{\psi}$,
and by inspection in 
\eqref{OjVDEF} and \eqref{Ojpsi0DEF}
(and recalling
$\scrO_j^{(\vecq)}:=\scrO_j^{(\pi(U_j^{(\vecq)}))}$),
this implies that $\scrO_j^{(\vecq)}\subset\scrO_j^{\psi}$.
Recall also that $\SS_j^{(\vecq)}$ is a finite union of $(\SS_j^{(\vecq)})^\circ$-cosets in $\TT_j$;
therefore
$\tSS_j^{\psi}=\SS_j^{(\vecq)}+\SS_j^{\psi}$
is a finite union of $\SS_j^{\psi}$-cosets.
Hence $\tSS_j^{\psi}$
is indeed a closed subgroup of $\TT_j$
whose connected component subgroup equals $\SS_j^{\psi}$.
Finally, from the definition of $\scrO_j^{\psi}$
and the fact that $\pi(U_j^{(\vecq)})\in(\SS_j^{(\vecq)})^d$,
it follows that $\scrO_j^{\psi}\subset(\tSS_j^{\psi})^d$;
and $\scrO_j^{\psi}$ is by definition a union of $(\SS_j^{\psi})^d$-cosets,
i.e.\ a union of some of 
the connected components of $(\tSS_j^\psi)^d$.
\end{proof}

This is a convenient point to interject the definition of
a measure $\omega_j^{\g}\in P(\TT_j^d)'$,
which is closely analogous to $\omega_j^\psi$,
and which we will need later for the explicit description of 
the macroscopic limit measure $\mu^{\g}$ 
which we discuss in Section \ref{P3proofSEC} below.

Given $j\in\{1,\ldots,N\}$
we set $\SS_j:=\pi(L_j)$. \label{SjDEF} This is a closed subtorus of $\TT_j$.
Picking an arbitrary $\vecq\in\R^d$, we define:
\begin{align}\label{OjDEF}
\tSS_j:=\SS_j^{(\vecq)}+\SS_j
\qquad\text{and}\qquad
\scrO_j:=\bigcup_{\gamma\in\SL_d(\Z)}(\pi(U_j^{(\vecq)})\gamma+\SS_j^d)\subset\TT_j^d.
\end{align}
\begin{lem}\label{OjwelldefLEM}
Both $\tSS_j$ and 
$\scrO_j$ are well-defined, 
i.e.\ the expressions in \eqref{OjDEF} are independent of the choice of $\vecq$.
We have $\scrO_j^{(\vecq)}\subset\scrO_j$ for all $\vecq\in\R^d$.
Furthermore, $\tSS_j$ is a closed subgroup of $\TT_j$ whose connected component subgroup
equals $\SS_j$,
and $\scrO_j$ is a union of some of 
the connected components of $\,\tSS_j^d$.
\end{lem}
In view of Lemma \ref{OjwelldefLEM}
we may now define $\omega_j^{\g}\in P(\TT_j^d)'$\label{omegagjdef}
to be the restriction to $\scrO_j$ of Haar measure on $\tSS_j^d$, normalized so that
$\omega_j^{\g}(\scrO_j)=1$.
\begin{proof}
It follows from 
\eqref{Ujqgenformula}
that for all $\vecq,\vecq'\in\R^d$
we have $U_{j,\ell}^{(\vecq)}-U_{j,\ell}^{(\vecq')}\in\R\vecc_j\subset L_j$,
and thus 
$\pi(U_{j,\ell}^{(\vecq)})-\pi(U_{j,\ell}^{(\vecq')})\in\SS_j$, for every $\ell\in\{1,\ldots,d\}$.
Recall that 
$\SS_j^{(\vecq)}=\overline{\langle\pi(U_{j,1}^{(\vecq)}),\ldots,\pi(U_{j,d}^{(\vecq)})\rangle}$.
It follows that $\SS_j^{(\vecq')}\subset\SS_j^{(\vecq)}+\SS_j$ for any $\vecq,\vecq'\in\R^d$.
Hence we get that $\tSS_j$ is well-defined.
It also follows that $\pi(U_j^{(\vecq)})-\pi(U_j^{(\vecq')})\in\SS_j^d$ for any $\vecq,\vecq'\in\R^d$;
hence $\scrO_j$ is well-defined.
The proof of the remaining assertions is essentially the same as in Lemma \ref{Ojpsi0welldefLEM}.
\end{proof}

\begin{lem}\label{OjqeqOjLEM}
For all except at most countably many $\vecq\in\R^d$ we have
$(\SS_j^{(\vecq)})^\circ=\SS_j$,
$\SS_j^{(\vecq)}=\tSS_j$,
$\scrO_j^{(\vecq)}=\scrO_j$, and 
$\omega_j^{(\vecq)}=\omega_j^{\g}$.
\end{lem}
\begin{proof}
We will prove that the stated equalities hold whenever
$L_j^{(\vecq)}=L_j$; by Lemma \ref{LjbasicLEM1} this 
gives the statement of the present lemma.
Thus assume that $L_j^{(\vecq)}=L_j$. %
Then $\bigl(\SS_j^{(\vecq)}\bigr)^\circ=\pi(L_j^{(\vecq)})=\pi(L_j)=\SS_j$.
In particular $\SS_j\subset\SS_j^{(\vecq)}$, and so
and $\SS_j^{(\vecq)} %
=\SS_j^{(\vecq)}+\SS_j
=\tSS_j$.
By comparing \eqref{OjDEF} with
\eqref{OjVDEF} (applied with $V=\pi(U_j^{(\vecq)})$),
we have $\scrO_j^{(\vecq)}=\scrO_j$.
Finally now also $\omega_j^{(\vecq)}=\omega_j^{\g}$
is immediate from the definitions.
\end{proof}

\begin{lem}\label{SjpsieqSjLEM}
For any $\psi\in\Psi$ and $j\in\{1,\ldots,N\}$,
if $j\neq j_\psi$
then 
$\SS_j^\psi=\SS_j$, $\tSS_j^\psi=\tSS_j$,
$\scrO_j^\psi=\scrO_j$ and
$\omega_j^\psi=\omega_j^{\g}$.
\end{lem}
\begin{proof}
If $j\neq j_\psi$.
then $L_j^\psi=L_j$ by Lemma \ref{LjvsLjpsiLEM},
and the stated equalities are immediate from this fact,  %
by inspection in the relevant definitions.
\end{proof}

We now return to the discussion on the measures $\omega_j^\psi$.
In Proposition \ref{muconvLEM} and Remark \ref{CONVremark} 
below, we will prove that $\omega_j^\psi$ is the generic limit of the 
measures $\omega_j^{(\vecq)}$
as $\vecq$ tends to infinity within the grid $\scrL_\psi$.
We will make use of the following auxiliary lemma.
\begin{lem}\label{OjpsicosetcarcLEM}
Let $\psi\in\Psi$, $j\in\{1,\ldots,N\}$ and $\vecq\in\scrL_\psi$.
Let $m=m(\vecq)$ be the number of $\bigl(\SS_j^{(\vecq)}\bigr)^{\circ\, d}$-cosets contained in $\scrO_j^{(\vecq)}$,
and let 
$n$ be the number of $(\SS_j^\psi)^d$-cosets contained in $\scrO_j^\psi$.
Then $n$ divides $m$, and every $(\SS_j^\psi)^d$-coset contained in $\scrO_j^\psi$
contains exactly $m/n$ distinct $\bigl(\SS_j^{(\vecq)}\bigr)^{\circ\, d}$-cosets.
\end{lem}
\begin{proof}
Let $U:=\pi(U_j^{(\vecq)})\in\bigl(\SS_j^{(\vecq)}\bigr)^{d}$,
and let $\Lambda$ be the stabilizer %
of the coset 
$U+\bigl(\SS_j^{(\vecq)}\bigr)^{\circ\, d}$, for the action of $\SL_d(\Z)$ on the finite group 
$\bigl(\SS_j^{(\vecq)}\bigr)^{d}/\bigl(\SS_j^{(\vecq)}\bigr)^{\circ\, d}$.
It then follows from the definition in \eqref{OjqDEF}
that $\scrO_j^{(\vecq)}$ equals the disjoint union of the cosets $U\gamma+\bigl(\SS_j^{(\vecq)}\bigr)^{\circ\, d}$,
when $\gamma$ runs through any set of representatives for $\Lambda\bs\SL_d(\Z)$.
In particular we have $m=\#(\Lambda\bs\SL_d(\Z))$.
Similarly
let $\Lambda'$ be the stabilizer of the coset
$U+(\SS_j^\psi)^d$ in $(\tSS_j^\psi)^d/(\SS_j^\psi)^d$;
then by \eqref{Ojpsi0DEF},
$\scrO_j^\psi$ equals the disjoint union of the cosets $U\gamma+(\SS_j^\psi)^d$,
when $\gamma$ runs through any set of representatives for $\Lambda'\bs\SL_d(\Z)$;
and in particular we have $n=\#(\Lambda'\bs\SL_d(\Z))$.
Note that $\Lambda\subset\Lambda'$,
since $\bigl(\SS_j^{(\vecq)}\bigr)^\circ\subset\SS_j^\psi$.
Therefore $n$ divides $m$,
and the last statement of the lemma follows from the fact that
each right $\Lambda'$-coset in $\SL_d(\Z)$ contains exactly
$m/n$ right $\Lambda$-cosets.
\end{proof}

\begin{prop}\label{muconvLEM}
Let $\psi\in\Psi$ and $j\in\{1,\ldots,N\}$,
and let $U$ be any open neighbourhood
of $\omega_j^\psi$ in $P(\TT_j^d)$.
Then the set $\{\vecq\in\scrL_{\psi}\col\omega_j^{(\vecq)}\notin U\}$
has density zero.
\end{prop}

\begin{proof}
To each matrix $A\in\M_{r_j\times d}(\Z)$ corresponds a character $\chi_A$ on 
$\TT_j^d=\M_{r_j\times d}(\R)/\M_{r_j\times d}(\Z)$ given by
$\chi_A(X)=e^{2\pi i\Tr(AX\trans)}$,
and every character on $\TT_j^d$ can be so expressed.
Using the fact that the set of finite linear combinations of characters on $\TT_j^d$ is dense in $C(\TT_j^d)$,
it follows that there exists a finite subset $S\subset\M_{r_j\times d}(\Z)$ and some $\ve>0$ such that
$U$ contains the set
\begin{align*}
\{\mu\in P(\TT_j^d)\col |\mu(\chi_A)-\omega_j^{\psi}(\chi_A)|<\ve\hspace{8pt} \forall A\in S\}.
\end{align*}
Hence, using also the fact that a finite union of density zero sets again has density zero,
it suffices to prove that for any fixed $A\in\M_{r_j\times d}(\Z)$,
the set
\begin{align}\label{muconvLEMPF1}
\{\vecq\in\scrL_\psi\col |\omega_j^{(\vecq)}(\chi_A)-\omega_j^{\psi}(\chi_A)|\geq\ve\}
\end{align}
has density zero.

For each $\vecq\in\scrL_\psi$,
let $\nu_j^{(\vecq)}\in P(\TT_j^d)$ be the normalized Haar measure on 
$\bigl(\SS_j^{(\vecq)}\bigr)^{\circ\, d}$,
and let $R(\vecq)$ be a set of representatives containing one point in each
$\bigl(\SS_j^{(\vecq)}\bigr)^{\circ\, d}$-coset contained in $\scrO_j^{(\vecq)}$ 
(recall that $\scrO_j^{(\vecq)}$ is a finite union of 
$\bigl(\SS_j^{(\vecq)}\bigr)^{\circ\, d}$-cosets).
Then by the definition of $\omega_j^{(\vecq)}$ around \eqref{OjqDEF},
we have  
\begin{align*}
\omega_j^{(\vecq)}=\frac1{\# R(\vecq)}\sum_{X\in R(\vecq)}\nu_j^{(\vecq)}\circ\tau_{X}^{-1},
\end{align*}
where $\tau_{X}:\TT_j^d\to\TT_j^d$ denotes translation by $X$.
Also let $\nu_j^{\psi}\in P(\TT_j^d)$ be the normalized Haar measure on $(\SS_j^{\psi})^d$.
Comparing \eqref{OjqDEF} with \eqref{Ojpsi0DEF} 
and the definition of $\omega_j^\psi$ 
on p.\ \pageref{mujpsi0def},
and using Lemmas \ref{Ojpsi0welldefLEM} and \ref{OjpsicosetcarcLEM},
we then have:
\begin{align}\label{muconvLEMPF3}
\omega_j^{\psi}=\frac1{\# R(\vecq)}\sum_{X\in R(\vecq)}\nu_j^{\psi_0}\circ\tau_{X}^{-1}.
\end{align}
It follows that
\begin{align}\label{muconvLEMPF4}
\bigl|\omega_j^{(\vecq)}(\chi_A)-\omega_j^{\psi}(\chi_A)\bigr|
&\leq\frac1{\# R(\vecq)}\sum_{X\in R(\vecq)}\bigl|\nu_j^{(\vecq)}(\chi_A\circ\tau_X)-\nu_j^{\psi}(\chi_A\circ\tau_X)\bigr|
=\bigl|\nu_j^{(\vecq)}(\chi_A)-\nu_j^{\psi}(\chi_A)\bigr|,
\end{align}
where the last equality holds since
$\chi_A\circ\tau_X=\chi_A(X)\cdot\chi_A$ for any $X\in\TT_j^d$.

Recall that %
$\bigl(\SS_j^{(\vecq)}\bigr)^{\circ\, d}=\pi\bigl(\bigl(L_j^{(\vecq)}\bigr)^d\bigr)$;
hence $\nu_j^{(\vecq)}(\chi_A)=1$ if 
$A$ is orthogonal to $\bigl(L_j^{(\vecq)}\bigr)^d$,
otherwise $\nu_j^{(\vecq)}(\chi_A)=0$.
Hence, letting $\veca_1,\ldots,\veca_d\in\Z^{r_j}$ be the column vectors of $A$,
we have $\nu_j^{(\vecq)}(\chi_A)=1$ if each $\veca_\ell$ is orthogonal to $L_j^{(\vecq)}$,
otherwise $\nu_j^{(\vecq)}(\chi_A)=0$.
Similarly $\nu_j^\psi(\chi_A)=1$ if each $\veca_\ell$ is orthogonal to $L_j^\psi$,
otherwise $\nu_j^{(\vecq)}(\chi_A)=0$.
Recall also that $L_j^{(\vecq)}\subset L_j^\psi$ for all $\vecq\in\scrL_\psi$
(cf.\ Lemma \ref{LjqsubsetLjpsi0LEM}).
These facts, together with \eqref{muconvLEMPF4}, imply that the set in \eqref{muconvLEMPF1} 
is contained in the following finite union:
\begin{align*}
\bigcup_{\substack{\ell\in\{1,\ldots,d\}\\(\veca_\ell\not\perp L_j^\psi)}}
\bigl\{\vecq\in\scrL_\psi\col \veca_\ell\perp L_j^{(\vecq)}\bigr\}.
\end{align*}
This set has density zero by Lemma \ref{MuconvauxLEM1}.
\end{proof}

\begin{remark}\label{CONVremark}
We have stated Proposition \ref{muconvLEM} in a way which will be convenient for later applications.
However let us note that it can be reformulated as follows:
\textit{We have $\omega_j^{(\vecq)}\to\omega_j^{\psi}$ in $P(\TT_j^d)$ 
as $\vecq$ tends to infinity within a full density subset of $\scrL_{\psi}$.}
Indeed, applying Proposition~\ref{muconvLEM} to a sequence of shrinking open sets $U$
containing no other element than $\omega_j^\psi$ in their intersection,
and using Lemma \ref{densityzeroLEM1} below,
we conclude that there exists a subset $\scrZ\subset \scrL_{\psi}$
of density zero such that 
$\lim_{k\to\infty}\omega_j^{(\vecq_k)}=\omega_j^{\psi}$ %
for any sequence of points $\vecq_1,\vecq_2,\ldots$ in $\scrL_{\psi}\setminus\scrZ$
with $\|\vecq_k\|\to\infty$.
\end{remark}

\begin{lem}\label{densityzeroLEM1}
Let $\scrL$ be a grid in $\R^d$,
and let $\scrZ_1,\scrZ_2,\ldots$ be subsets of $\scrL$
which each have  density zero.
Then there exists a subset $\scrZ$ of $\scrL$ of density zero
which has the property that for every $k$ there exists some $R>0$
such that $\scrZ_k\setminus\scrB_R^d\subset\scrZ$.
\end{lem}
\begin{proof}
Set $z_j(R)=R^{-d}\#(\scrZ_j\cap\scrB_R^d)$.
Using the fact that $z_j(R)\to0$ for each $j$,
we can pick $0<R_1<R_2<\cdots$ such that
for each $k\in\Z^+$ and each $R\geq R_k$ we have $\sum_{j=1}^kz_j(R)<k^{-1}$.
Now set
\begin{align*}
\scrZ=\bigcup_{j=1}^\infty(\scrZ_j\setminus\scrB_{R_j}^d).
\end{align*}
This set $\scrZ$ has  density zero,
since for each $k\in\Z^+$ and each $R\in[R_k,R_{k+1})$
we have $R^{-d}\#(\scrZ\cap\scrB_R^d)=
R^{-d}\#\cup_{j=1}^k(\scrZ_j\cap\scrB_R^d\setminus\scrB_{R_j^d})
\leq\sum_{j=1}^k z_j(R)<k^{-1}$.
It is also clear that $\scrZ$ has the last property stated in the lemma.
\end{proof}

\begin{remark}\label{MuconvauxLEM1rem}
By analogy with Remark \ref{CONVremark} we also note that
Lemma \ref{MuconvauxLEM1} implies that there
exists a subset $\scrZ\subset \scrL_{\psi}$
of density zero such that 
for every 
$\veca\in\Z^{r_j}$ with $\veca\not\perp L_j^\psi$
there is some $R>0$ such that
$\veca\not\perp L_j^{(\vecq)}$ holds for all $\vecq\in\scrL_{\psi}\setminus\scrZ$
with $\|\vecq\|>R$.
As in the proof of Proposition \ref{muconvLEM}, %
this implies that the normalized Haar measure on $\bigl(\SS_j^{(\vecq)}\bigr)^\circ=\pi(L_j^{(\vecq)})\subset\TT_j$
tends to the normalized Haar measure $\nu_j^\psi$ on $\SS_j^\psi=\pi(L_j^\psi)$
as $\vecq$ tends to infinity within $\scrL_{\psi}\setminus\scrZ$.
In this sense, we may say that ``$L_j^{(\vecq)}$ approaches $L_j^\psi$''
as $\vecq$ tends to infinity within a full density subset of $\scrL_\psi$.
\end{remark}

\subsection{Admissible presentations of $\scrP$}
\label{ATTAINADMsec}

In this section we will prove that an admissible presentation of $\scrP$ can always be obtained:

\begin{prop}\label{ADMpropMAIN}
Let $\scrP$ be a finite union of grids.
Then $\scrP$ possesses an admissible presentation,
i.e.\ there exist $N\in\Z^+$, $r_1,\ldots,r_N\in\Z^+$,
$M_1,\ldots,M_N\in\SL_d(\R)$,
and numbers $c_\psi\in\R$ and vectors $\vecw_\psi\in\R^d$
for $\psi\in\Psi$ (with $\Psi$ as in \eqref{PSIdef}),
such that $\scrP$ is given by \eqref{GENPOINTSET1}, \eqref{LpsiDEF},
and this presentation of $\scrP$
satisfies the admissibility condition in Definition \ref{admissibleDEF}.
\end{prop}

We will prove Proposition \ref{ADMpropMAIN}
by showing that there exists a presentation of $\scrP$
as in \eqref{GENPOINTSET1}, \eqref{LpsiDEF},
such that \eqref{THINDISJcond2} and \eqref{GENPOINTSET1req} hold
and also
\begin{align}\label{strongadmissible}
\vecc_{j}^{\psi}\in\fL\bigl(\{\vecc_j^{\psi}\}\bigr)+\Z^{r_j},
\qquad\forall\psi=(j,i)\in\Psi.
\end{align}
This implies that the presentation is admissible, since
we always have $\fL\bigl(\{\vecc_j^{\psi}\}\bigr)\subset L_j^\psi$ for all $\psi=(j,i)\in\Psi$,
because of \eqref{Ljpsi0DEF} and the fact that
$\fL(S)$ is increasing in $S$.

For any $r\geq1$ and any vector $\vecu=(u_1,\ldots,u_r)\in\R_{>0}^r$,
let us write $\tvecu:=(u_1^{-1},\ldots,u_r^{-1})\in\R_{>0}^r$.
We call the vector $\vecu\in\R_{>0}^r$
\textit{admissible}
if $u_i\tvecu\in\fL(u_i\tvecu)+\Z^r$ 
for every $i\in\{1,\ldots,r\}$.\label{admissibleDEF2}
In view of \eqref{cjdef},
we then have that
\eqref{strongadmissible}
is equivalent with the condition
that the vector
$(c_{j,1},\ldots,c_{j,r_j})$ is admissible, for every $j\in\{1,\ldots,N\}$.

To prove Proposition \ref{ADMpropMAIN},
we will start from an arbitrary presentation of $\scrP$
as obtained in Section \ref{Preprsec},
i.e.\ we assume that $\scrP$ is expressed as in \eqref{GENPOINTSET1}, \eqref{LpsiDEF},
and that the conditions \eqref{THINDISJcond2} and \eqref{GENPOINTSET1req}
hold.
If this presentation is not already admissible,
then we will modify it by making use of the simple fact that
for any positive integer $q$,
we can express
$\Z^d$ as the (disjoint) union of
the grids $q\Z^d+\vecalf$,
with $\vecalf$ running through the set $\{1,\ldots,q\}^d$;
hence for any $\vecw\in\R^d$ and $M\in\SL_d(\R)$,
the grid $(\Z^d+\vecw)M$ equals the union of the grids
$(q\Z^d+\vecalf+\vecw)M$.
It follows that for any choice of positive integers $q_\psi$ ($\psi\in\Psi$),
we have
\begin{align}
\scrP=\bigcup_{\psi\in\Psi}\scrL_\psi
=\bigcup_{\psi\in\Psi}c_\psi(\Z^d+\vecw_\psi)M_{j_\psi}
=\bigcup_{\psi\in\Psi}\bigcup_{\vecalf\in\{1,\ldots,q_\psi\}^d}
q_\psi c_\psi\bigl(\Z^d+q_\psi^{-1}(\vecalf+\vecw_\psi)\bigr)M_{j_\psi}.
\end{align}
In other words, we have obtained a new presentation of $\scrP$,
analogous to the original one: %
\begin{align}\label{Pnewpresentation}
\scrP=\bigcup_{\vartheta\in\Theta} c'_\vartheta(\Z^d+\vecw_\vartheta')M_{j_\vartheta},
\end{align}
where
\begin{align*}
\Theta=\{(j,i)\col j\in\{1,\ldots,N\},\: i\in\{1,\ldots,r_j'\}\}
\end{align*}
with
\begin{align*}
r_j'=\sum_{i=1}^{r_j}q_{j,i}^d
\qquad (j=1,\ldots,N),
\end{align*}
and where 
\begin{align*}
\vecw'_{\vartheta}=q_{\psi(\vartheta)}^{-1}(\vecalf_\vartheta+\vecw_{\psi(\vartheta)});
\qquad c'_\vartheta=q_{\psi(\vartheta)}c_{\psi(\vartheta)},
\end{align*}
with $\psi(\vartheta)\in\Psi$ and $\vecalf_\vartheta\in\Z^d$
chosen in such a way that
$j_{\psi(\vartheta)}=j_\vartheta$ for all $\vartheta\in\Theta$,
and the map $\vartheta\mapsto\langle\psi(\vartheta),\vecalf_\vartheta\rangle$
is a bijection from $\Theta$ onto $\{\langle\psi,\vecalf\rangle\col \psi\in\Psi,\:\vecalf\in\{1,\ldots,q_\psi\}^d\}$.
It is obvious that the new presentation %
again satisfies 
the condition \eqref{GENPOINTSET1req}, since
the matrices $M_1,\ldots,M_N$ are unchanged.
It is also immediate that the %
condition \eqref{THINDISJcond2} remains true for the new presentation.
Hence %
it remains to prove that we can choose the positive integers $q_\psi$
in such a way that the analogue of \eqref{strongadmissible} holds for the new presentation,
i.e.\ so that
$\bigl(c_{j,1}',\ldots,c_{j,r_j'}'\bigr)$ is admissible, for every $j\in\{1,\ldots,N\}$.

The proof of this fact is essentially completed by
lemmas \ref{ADMproplem1} and \ref{ADMproplem2new} below.
\begin{lem}\label{IDCLcharLEM3a}
Let $1\leq r\leq r'$ and let $T\in \M_{r'\times r}(\Q)$.
Then for any non-empty subset $S\subset\R^r$
we have 
$\fL(TS)=T\,\fL(S)$.
\end{lem}
(Here for any subset $A\subset\R^r$ we write $TA:=\{T\vecv\col\vecv\in A\}$.
Recall also that we view vectors in $\R^r$ as column matrices;
hence $T\vecv\in\R^{r'}$ for every $\vecv\in\R^r$.)

\begin{proof}
Both $\fL(TS)$ and $T\,\fL(S)$ are rational subspaces of $\R^{r'}$;
hence it suffices to prove that 
the equivalence [$\veca\perp\fL(TS)\Leftrightarrow\veca\perp T\,\fL(S)$]
holds for all $\veca\in\Q^{r'}$.
However, $\veca\perp T\,\fL(S)$ holds if and only if
$\veca\cdot T\vecv=0$ for all $\vecv\in\fL(S)$,
or equivalently $(T\trans \veca)\cdot \vecv=0$
for all $\vecv\in\fL(S)$.
By Lemma \ref{IDCLcharLEM2},
this holds if and only if 
there is some $n\in\Z^+$ such that $T\trans\veca\cdot\vecv\in n^{-1}\Z$
for all $\vecv\in S$.
On the other hand, 
Lemma \ref{IDCLcharLEM2} also gives that
$\veca\perp \fL(TS)$
holds if and only if there is some $n\in\Z^+$ such that
$\veca\cdot T\vecv\in n^{-1}\Z$
for all $\vecv\in S$,
or equivalently $T\trans\veca\cdot\vecv\in n^{-1}\Z$
for all $\vecv\in S$.
Hence the equivalence is established.
\end{proof}

\begin{lem}\label{ADMproplem1}
If the vectors $\vecu=(u_1,\ldots,u_r)\in\R_{>0}^r$
and $\vecu'=(u_1',\ldots,u'_{r'})\in\R_{>0}^{r'}$
have the same set of coordinates,
i.e.\
$\{u_1,\ldots,u_r\}=\{u_1',\ldots,u'_{r'}\}$,
then $\vecu$ is admissible if and only if $\vecu'$ is admissible.
\end{lem}
\begin{proof}
Assume that
$\{u_1,\ldots,u_r\}=\{u_1',\ldots,u'_{r'}\}$.
This means that there exist uniquely determined
matrices $T\in\M_{r,r'}(\Z)$ and $T'\in\M_{r',r}(\Z)$
with %
exactly one entry of $1$ in each row and $0$s elsewhere,
such that $\vecu'=T\vecu$ and $\vecu=T'\vecu'$;
thus also $\widetilde{\vecu'}=T\tvecu$ and $\tvecu=T'\widetilde{\vecu'}$.
Assume that $\vecu$ is admissible.
Then for every $i\in\{1,\ldots,r\}$
we have $u_i\tvecu\in\fL(u_i\tvecu)+\Z^r$.
Multiplying this relation by $T$ from the left,
and using Lemma \ref{IDCLcharLEM3a}
and $T\Z^r\subset\Z^{r'}$,
we obtain
$u_i\widetilde{\vecu'}\in\fL(u_i\widetilde{\vecu'})+\Z^{r'}$.
This holds for all $i\in\{1,\ldots,r\}$,
and for each $i'\in\{1,\ldots,r'\}$ there exists some
$i\in\{1,\ldots,r\}$ such that $u_{i'}'=u_i$.
Hence $\vecu'$ is admissible.
The opposite implication is proved analogously, using $T'$.
\end{proof}

\begin{lem}\label{ADMproplem2new}
For any $\vecu=(u_1,\ldots,u_r)\in\R_{>0}^r$,
there exist positive integers $q_1,\ldots,q_r$ such that 
the vector $(q_1u_1,\ldots,q_ru_r)$ is admissible.
\end{lem}
\begin{proof}
Given $\vecq=(q_1,\ldots,q_r)\in\Z_{>0}^r$ we write $D_\vecq=\diag(q_1,\ldots,q_r)$;
then the task is to prove that we there exists some %
$\vecq\in\Z_{>0}^r$ such that
$q_iu_iD_\vecq^{-1}\tvecu\in\fL(q_iu_iD_\vecq^{-1}\tvecu)+\Z^r$ for all $i\in\{1,\ldots,r\}$.
By Lemma \ref{IDCLcharLEM3a}
applied for the matrix $q_iD_\vecq^{-1}\in\M_r(\Q)$, this is equivalent to:
\begin{align}\label{ADMproplem2newpf2}
q_iu_i\tvecu\in\fL(u_i\tvecu)+D_\vecq\,\Z^r\qquad\text{for all }\: i\in\{1,\ldots,r\}.
\end{align}

For each $i$, 
$\R\tvecu+\fL(u_i\tvecu)$ is a rational subspace of $\R^r$,
and $u_i\tvecu\in\Q^r+\fL(u_i\tvecu)$
by Lemma~\ref{IDCLcharLEM},
and $\R\tvecu=\R u_i\tvecu$.
But $\fL(\R\tvecu)$ is the unique smallest rational subspace of $\R^r$ containing $\R\tvecu$;
hence
\begin{align}\label{ADMproplem2newpf10}
\fL(\R\tvecu)=\R\tvecu+\fL(u_i\tvecu).
\end{align}

We will now describe a choice of $q_1,\ldots,q_r$ which makes \eqref{ADMproplem2newpf2} hold.
Let $p_i:\R^r\to\R$ denote projection onto the $i$th coordinate.
Note that $p_i(\Z^r\cap \fL(\R\tvecu))$ is a subgroup of $\Z$.
By considering the expansion of $\tvecu$ 
with respect to a basis of $\fL(\R\tvecu)$ consisting of vectors in $\Z^r$,
it follows that $p_i(\Z^r\cap \fL(\R\tvecu))\neq\{0\}$ for each $i$,
and hence %
there exist
unique positive integers $q_i$ such that
\begin{align}\label{ADMproplem2newpf3}
p_i(\Z^r\cap \fL(\R\tvecu))=q_i\Z
\qquad\text{for all }\: i\in\{1,\ldots,r\}.
\end{align}
Choose vectors
$\vech^{(i)}\in\Z^r\cap \fL(\R\tvecu)$ with $p_i(\vech^{(i)})=q_i$.
We now claim that
\begin{align}\label{ADMproplem2newpf4}
q_iu_i\tvecu\in\fL(u_i\tvecu)+\vech^{(i)}
\qquad\text{for all }\: i\in\{1,\ldots,r\}.
\end{align}
This implies that \eqref{ADMproplem2newpf2} holds,
since $\vech^{(i)}\in\Z^r\cap \fL(\R\tvecu)\subset D_\vecq\Z^r$,
where the last inclusion follows from \eqref{ADMproplem2newpf3}.

In order to prove \eqref{ADMproplem2newpf4},
let $i$ be given, and set
$\vecw:=q_iu_i\tvecu-\vech^{(i)}$.
We have $\vecw\in\fL(\R\tvecu)$,
since $\tvecu$ and $\vech^{(i)}$ lie in $\fL(\R\tvecu)$;
hence by \eqref{ADMproplem2newpf10} there exists some $t\in\R$ such that
$\vecw-t\tvecu\in\fL(u_i\tvecu)$.
However, both $\vecw$ and $\fL(u_i\tvecu)$
lie in the orthogonal complement of 
$\vece_i$, the $i$th standard unit vector in $\R^r$;\label{vecekcolDEF}
for $\vecw$ this holds since 
$\vece_i\cdot\vecw=\vece_i\cdot(q_iu_i\tvecu- \vech^{(i)})=q_i-p_i(\vech^{(i)})=0$
and for $\fL(u_i\tvecu)$ it holds by Lemma \ref{IDCLcharLEM2}, since
$\vece_i\cdot u_i\tvecu=1$.
On the other hand, $\tvecu$ is not orthogonal to $\vece_i$.
Hence we must have $t=0$, i.e.\
$\vecw\in\fL(u_i\tvecu)$,
and so \eqref{ADMproplem2newpf4} holds.
\end{proof}

Now to finally prove Proposition \ref{ADMpropMAIN} 
we proceed as follows,
for each fixed $j\in\{1,\ldots,N\}$:
Choose 
$0<u_1<\cdots<u_r$ so that
$\{c_{j,1},\ldots,c_{j,r_j}\}=\{u_1,\ldots,u_r\}$,
and let $\tau:\{1,\ldots,r_j\}\to\{1,\ldots,r\}$ be the map so that $c_{j,i}=u_{\tau(i)}$ for all $i\in\{1,\ldots,r_j\}$;
then by Lemma \ref{ADMproplem2new} 
there exist positive integers $q_1,\ldots,q_r$ such that
the vector 
$(q_1u_1,\ldots,q_ru_r)$
is admissible.
Setting now $q_{j,i}=q_{\tau(i)}$,
Lemma \ref{ADMproplem1} implies
that the vector 
$\bigl(q_{j,1}c_{j,1},\ldots,q_{j,r_j}c_{j,r_j}\bigr)$
is admissible;
and since 
$c'_{j,i}=q_{\psi(j,i)}c_{\psi(j,i)}$ for all $i=1,\ldots,r_j'$,
another application of Lemma \ref{ADMproplem1} gives that also
the vector
$\bigl(c_{j,1}',\ldots,c_{j,r_j'}'\bigr)$ is admissible,
which was exactly the desired condition.
This completes the proof of Proposition \ref{ADMpropMAIN}.
\hfill$\square$

\vspace{5pt}

Clearly Lemma \ref{ADMproplem2new} provided %
the key technical step for the proof of Proposition \ref{ADMpropMAIN}.
We end this section by some basic remarks related to the statement of that lemma.

First of all, it should be noted that if $\vecq=(q_1,\ldots,q_r)\in\Z_{>0}^r$
makes $(q_1u_1,\ldots,q_ru_r)$ admissible,
then so does $c\,\vecq=(cq_1,\ldots,cq_r)$ for any $c>0$ such that $c\,\vecq\in\Z^r$.
However, in general there are also other vectors 
satisfying the same condition; that is,
the vector $\vecq$ in the statement Lemma \ref{ADMproplem2new}
is in general \textit{not} unique up to scalar multiplication
-- even though the proof of Lemma \ref{ADMproplem2new} produces a certain
$\vecq\in\Z_{>0}^r$ which is uniquely determined for any given $\vecu$.
For example, if the numbers $u_1^{-1},\ldots,u_r^{-1}\in\R_{>0}$ are linearly independent over $\Q$
then $(q_1u_1,\ldots,q_ru_r)$ is admissible for \textit{every} $\vecq\in\Z_{>0}^r$.

On the other hand it is a nice fact that if the given vector 
$\vecu\in\R_{>0}^r$ satisfies $u_j/u_k\in\Q$ for some pair of indices $j<k$,
then \textit{any} $\vecq\in\Z_{>0}^r$ which makes 
$(q_1u_1,\ldots,q_ru_r)$ admissible must satisfy $q_ju_j=q_ku_k$.
As a special case we have that if $u_j/u_k\in\Q$ for \textit{all} $j,k$,
then $q_1u_1=\cdots=q_ru_r$ holds for any $\vecq\in\Z_{>0}^r$
such that $(q_1u_1,\ldots,q_ru_r)$ is admissible,
and in particular in this case the vector $\vecq$ 
in the statement of Lemma \ref{ADMproplem2new} \textit{is} uniquely determined up to scalar multiplication.

(To prove the claim in the previous paragraph, 
note that if $u_j/u_k\in\Q$,
then every vector $\veca\in\Q^r$ satisfying $a_{\ell}=0$ for all $\ell\notin\{i,j\}$
is orthogonal to both $\fL(u_j\tvecu)$ and $\fL(u_k\tvecu)$,
by Lemma~\ref{IDCLcharLEM2}.
Hence both $\fL(u_j\tvecu)$ and $\fL(u_k\tvecu)$ %
are contained in the orthogonal complement of $\{\vece_j,\vece_k\}$,
and so the condition \eqref{ADMproplem2newpf2}, applied with $i=j$ and $i=k$,
implies that $q_ju_j/u_k\in q_k\Z$
and $q_ku_k/u_j\in q_j\Z$.
Using also $u_j,u_k>0$ and $q_j,q_k\in\Z_{>0}$,
it follows that $q_ju_j=q_ku_k$, as claimed.)

\section{Precise statement of the main result}
\label{KINTHEORYsec}

In this section 
the statement of the main technical result of the paper, %
Theorem \ref{MAINTHM},
will be made more explicit, by
presenting the precise choice of the
data $\vs,\Sigma,\mm$ and $\vs\mapsto\mu_\vs$
for which we will later prove that %
the six conditions [P1]--[P3] and [Q1]--[Q3] 
in Section \ref{KINTHEORYrecapsec} hold.

Let $\scrP$ be a finite union of grids in $\R^d$,
and fix an \textit{admissible} presentation of $\scrP$ as in
\eqref{GENPOINTSET1}, \eqref{LpsiDEF}.
This is possible by Proposition~\ref{ADMpropMAIN}.

\subsection{The subspace $\XX^\psi\subset\XX$}

We first need to make a preparation of a technical nature.
For any $\psi\in\Psi$ we 
let $\XX^\psi$ be the closed set consisting of those $\Gamma g\in\XX$ for which $\Z^d\,p_\psi(g)$ is a \textit{lattice,}
i.e.
\begin{align}\label{Xpsi0def}
\XX^\psi:=\{\Gamma g\in\XX\col g\in G,\: \bn\in\Z^d\,p_\psi(g)\}.
\end{align}
(Recall that the grid $\Z^d\,p_\psi(g)$ is independent of the choice of representative $g$,
i.e.\ if $\Gamma g=\Gamma g'$ then $\Z^d\,p_\psi(g)=\Z^d\,p_\psi(g')$.)
Thus in terms of the parametrization
\eqref{Gammagpointset},
$\XX^\psi$ is the subset of those $\Gamma g$ in $\XX$ 
for which, in the corresponding union of grids in $\R^d$,
the grid indexed by $\psi$ contains the origin. %
Clearly, for $\psi=(j,i)$,
\begin{align}\label{XjiDEF}
\XX^\psi=p_j^{-1}(\XX_j^{(i)}),
\qquad
\text{with }\:\XX_{j}^{(i)}=\{\Gamma_jg\col g\in G_j,\: \bn\in\Z^d\a_i(g)\}.
\end{align}

\begin{lem}\label{g0qinXpsiLEM}
For any $\psi\in\Psi$ and $\vecq\in\scrL_\psi$,
$\Gamma g_0^{(\vecq)}\in\XX^\psi$.
\end{lem}
\begin{proof}
This is immediate from \eqref{KEYUjqdefmotivation}.
\end{proof}
We will write $\delta_{\bn}$ for the Dirac measure on $(\R/\Z)^d$ at the point $\bn\in(\R/\Z)^d$.
Recall that $\tr_i:\TT_j^d\to(\R/\Z)^d$ is the projection induced by
the $i$th row map
$\r_i:\M_{r_j\times d}(\R)\to\R^d$.
\begin{lem}\label{oomegapinullLEM}
Assume that $\psi=(j,i)\in\Psi$,
$\omega\in P(\TT_j^d)'$
and
$\tr_{i\,*}\,\omega=\delta_{\bn}$.
Then $\oomega(\XX_j^{(i)})=1$.
\end{lem}
\begin{proof}
In view of the definition of $\oomega$
in \eqref{toomegaDEF},
it suffices to verify that $\Gamma_j \I_UA\in\XX_j^{(i)}$
for all $U\in\r_i^{-1}(\Z^d)\subset\M_{r_j\times d}(\R)$
and all $A\in\SL_d(\R)$.
The verification of this fact is immediate.
\end{proof}

\begin{lem}\label{omegajVtriLEM}
Let $\psi=(j,i)\in\Psi$ and $V\in\TT_j^d$.
If $\tr_i(V)=\bn$ in $(\R/\Z)^d$ then
$\tr_{i\,*}\,\omega_j^{(V)}=\delta_{\bn}$,
while if $\tr_i(V)\neq\bn$ then
$\bigl(\tr_{i\,*}\,\omega_j^{(V)}\bigr)(\{\bn\})=0$.   %
\end{lem}
\begin{proof}
Write $V=(V_1,\ldots,V_d)$.
If $\tr_i(V)=\bn$, then the $i$th coordinate of each $V_\ell$
vanishes;
hence all points in $\SS_j^{(V)}$ have vanishing $i$th coordinate,
and so $\r_i\bigl(\bigl(\SS_j^{(V)}\bigr)^d\bigr)=\{\bn\}$ in $(\R/\Z)^d$,
which forces $\omega_j^{(V)}(\tr_i^{-1}(\{\bn\}))=1$,
i.e.,
$\tr_{i\,*}\,\omega_j^{(V)}=\delta_{\bn}$.

Next assume $\tr_i(V)\neq\bn$.
If $L_j^{(V)}\not\perp\vece_i$,
then $\tr_i\bigl(\bigl(\SS_j^{(V)}\bigr)^{\circ\,d}\bigr)=(\R/\Z)^d$
and so
by \eqref{OjVDEF} and the definition of $\omega_j^{(V)}$, we have
$\omega_j^{(V)}(\tr_i^{-1}(\{\bn\}))=0$.
In the remaining case, when 
$L_j^{(V)}\perp\vece_i$
and therefore 
$\tr_i\bigl(\bigl(\SS_j^{(V)}\bigr)^{\circ\,d}\bigr)=\{\bn\}$,
we again have $\omega_j^{(V)}(\tr_i^{-1}(\{\bn\}))=0$,
since $\tr_i(V\gamma)=\tr_i(V)\gamma\neq\bn$ for all $\gamma\in\SL_d(\Z)$.
\end{proof}

\begin{lem}\label{omegajqpiinv0LEM}
Let $\psi=(j,i)\in\Psi$ and $\vecq\in\R^d$.
If $\vecq\in\scrL_\psi$ then 
$\tr_{i\,*}\,\omega_j^{(\vecq)}=\delta_{\bn}$,
while if $\vecq\notin\scrL_\psi$ then
$\bigl(\tr_{i\,*}\,\omega_j^{(\vecq)}\bigr)(\{\bn\})=0$.   %
\end{lem}
\begin{proof}
Note that $\vecq\in\scrL_\psi$ holds if and only if
$\r_i(U_j^{(\vecq)})\in\Z^d$
(cf.\ \eqref{LpsiDEF} and \eqref{Ujqdef}),
that is, if and only if $\tr_i(\pi(U_j^{(\vecq)}))=\bn$ in $(\R/\Z)^d$.
Hence the lemma follows from
Lemma \ref{omegajVtriLEM}.
\end{proof}

\begin{lem}\label{LjpsispecLEM}
Let $j\in\{1,\ldots,N\}$
and $i\neq i'\in\{1,\ldots,r_j\}$.
Then $L_j^{(j,i')}\perp\vece_{i}$
holds if and only if $L_j^{(j,i)}\perp\vece_{i'}$,
and in this case also $L_j^{(j,i')}=L_j^{(j,i)}$ and $c_{j,i}=c_{j,i'}$
and $\vecw_{j,i}-\vecw_{j,i'}\notin\Z^d$.
\end{lem}
\begin{proof}
Assume $L_{j}^{(j,i')}\perp\vece_{i}$.
Then by Lemma \ref{LjvsLjpsiLEM2},
$L_j^{(j,i')}\subset L_j\cap\vece_i^\perp=L_j^{(j,i)}$,
and both $L_j^{(j,i')}$ and $L_j^{(j,i)}$ are subspaces of $L_j$
of codimension one; hence $L_j^{(j,i')}=L_j^{(j,i)}$,
and since $L_j^{(j,i')}\perp\vece_{i'}$
by Lemma \ref{LjvsLjpsiLEM2}, we conclude that
$L_j^{(j,i)}\perp\vece_{i'}$.
Of course, the converse implication,
$L_j^{(j,i)}\perp\vece_{i'}\Rightarrow
L_{j}^{(j,i')}\perp\vece_{i}$,
holds by symmetry.
When both these orthogonality relations hold,
admissibility
(see Definition \ref{admissibleDEF})
implies that 
both the $i\,$th coordinate of $\vecc_j^{(j,i')}$
and the $i'$th coordinate of $\vecc_j^{(j,i)}$ are \textit{integers},
i.e.\ both $c_{j,i'}/c_{j,i}$ and $c_{j,i}/c_{j,i'}$ are integers;
thus $c_{j,i'}=c_{j,i}$.
Finally, by \eqref{THINDISJcond2},
$\vecw_{j,i}-\vecw_{j,i'}\notin\Z^d$.
\end{proof}

\begin{lem}\label{omegajpsiLEM}
For every $\psi\in\Psi$ we have
$\tr_{i\,*}\,\omega_j^{\psi}=\delta_{\bn}$ when $(j,i)=\psi$,
while $\bigl(\tr_{i\,*}\,\omega_{j}^{\psi}\bigr)(\{\bn\})=0$ for every $(j,i)\in\Psi\setminus\{\psi\}$.
\end{lem}
\begin{proof}
The first statement follows from Lemma \ref{omegajqpiinv0LEM},
since $\omega_j^\psi$ can be obtained as a limit of
measures $\omega_j^{(\vecq)}$ for a sequence of points $\vecq$ in $\scrL_\psi$
(cf.\ Remark \ref{CONVremark}),
and since the map $\tr_{i\,*}:P(\TT_j^d)\to P((\R/\Z)^d)$ is continuous.
Alternatively it is easy to give a direct verification: %
In \eqref{cjdef} and \eqref{Wjdef} one notes that
the $i$th coordinate of $\vecc_j^\psi$ equals $1$;
thus $\r_i(W_j^\psi)=\bn$,
and so by \eqref{Ljpsi0DEF}, $L_j^\psi\perp\vece_i$,
meaning that every point in $\SS_j^\psi$ has a vanishing $i$th coordinate.
We know from the proof of Lemma \ref{omegajqpiinv0LEM}
that the same fact holds for every point in $\SS_j^{(\vecq)}$, for any $\vecq\in\scrL_\psi$;
hence by \eqref{Ojpsi0DEF} it also holds for every point in $\tSS_j^\psi$,
and so $\tr_i((\tSS_j^\psi)^d)=\{\bn\}$
and $\tr_{i\,*}\,\omega_j^{\psi}=\delta_{\bn}$.

We turn to the proof of the second statement
(this is similar to the proof of the second half of Lemma \ref{omegajVtriLEM}).
Thus let $(j,i)\in\Psi$, $(j,i)\neq\psi$.
If $L_{j}^\psi\not\perp\vece_{i}$ then
$\tr_{i}\bigl((\SS_{j}^\psi)^d\bigr)=(\R/\Z)^d$
and hence $\omega_{j}^{\psi}\bigl(\tr_{i}^{-1}(\{\bn\})\bigr)=0$
(cf.\ \eqref{Ojpsi0DEF}).
Now assume $L_{j}^\psi\perp\vece_{i}$.
Then $j=j_{\psi}$, by \eqref{Ljpsi0DEF}.
Pick an arbitrary point $\vecq\in\scrL_\psi$,
so that \eqref{Ojpsi0DEF} holds.
Now by Lemma \ref{LjpsispecLEM},
$c_\psi=c_{j,i}$ and $\vecw_\psi-\vecw_{j,i}\notin\Z^d$,
and so via \eqref{Wjdef} and \eqref{Ujqformula},
$\r_i(U_j^{(\vecq)})\notin\Z^d$,
and hence $\tr_i(\pi(U_j^{(\vecq)})\gamma)\neq\bn$ for all $\gamma\in\SL_d(\Z)$.
Also $\tr_i\bigl(\SS_j^\psi\bigr)=\{\bn\}$, since 
$L_{j}^\psi\perp\vece_{i}$.
Hence again $\omega_{j}^{\psi}\bigl(\tr_{i}^{-1}(\{\bn\})\bigr)=0$
(see \eqref{Ojpsi0DEF}).
\end{proof}

\subsection{The space of marks $\Sigma$ and the associated maps}
\label{spaceofmarksandmapsSEC}

We are now finally in position to introduce our precise choice of space of marks $\Sigma$
and the associated maps and measure.
To prepare for this, %
let $\Omega$ be the following Cartesian product:
\begin{align}\label{OMEGAdef}
\Omega=\prod_{j=1}^N P(\TT_{\!j}^d)'.
\end{align}
The point of the set $\Omega$ is that it
parametrizes a family of homogeneous probability measures on $\XX$ 
from which the limit measures appearing in the spherical equidistribution conditon [P2] 
in Section \ref{KINTHEORYrecapsec} can in turn be obtained.
Indeed, if $\omega=(\omega_1,\ldots,\omega_N)\in\Omega$
then $\overline{\omega_j}\in P(\XX_j)$ by the discussion near Lemma \ref{oomegapartDEFlem},
and so we have a 
natural map $\omega\mapsto\oomega$ from $\Omega$ to $P(\XX)$,
defined by
\begin{align}\label{oomegaDEF}
\oomega:=\overline{\omega_1}\otimes\cdots\otimes\overline{\omega_N}\in P(\XX) \qquad (\omega\in\Omega).
\end{align}
In particular, defining
\begin{align}\label{omegaqDEF}
\omega^{(\vecq)}:=\bigl(\omega_1^{(\vecq)},\ldots,\omega_N^{(\vecq)}\bigr)\in\Omega
\end{align}
for any $\vecq\in\R^d$,
we now have $\overline{\omega^{(\vecq)}}=\mu^{(\vecq)}$
(see \eqref{muqDEF});
and as described below Theorem \ref{HOMDYNintrononunifTHM},
it will turn out to be possible to obtain 
the limit measure in %
[P2]
as a pushforward of this mesure $\mu^{(\vecq)}$.

Now we define our space of marks $\Sigma$ through:
\begin{align}\label{Sigmadef}
\Sigma:=&\Bigl\{((j,i),\omega)\in\Psi\times\Omega\col %
\tr_{i\,*}\,\omega_{j}=\delta_{\bn}\Bigr\},
\end{align}
where $\omega_{j}\in P(\TT_j^d)'$ is the $j$th entry of $\omega$,
and where $\delta_{\bn}$ is the Dirac measure at the point $\bn\in(\R/\Z)^d$.
It is to be understood that in \eqref{Sigmadef}, 
$\Psi$ is equipped with the discrete topology,
and $\Psi\times\Omega$ with the product topology;
this makes $\Sigma$ a closed and hence compact subset of $\Psi\times\Omega$.

\begin{remark}
The reason why we cannot simply choose $\Sigma$ to be $\Psi\times\Omega$ is that
then the map $\vs\mapsto\mu_{\vs}$ which we define below would in general not be continuous.
\end{remark}

Next we define our marking $\vs$
through
\begin{align}\label{vsmapDEF}
\vs:\scrP\to\Sigma,\qquad
\vs(\vecq)=(\psi(\vecq),\omega^{(\vecq)}),
\end{align}
where $\psi:\scrP\to\Psi$ is as in \eqref{psimarkdef},
i.e.\ a fixed function such that $\vecq\in\scrL_{\psi(\vecq)}$ for all $\vecq\in\scrP$.
It follows from Lemma \ref{omegajqpiinv0LEM} that $\vs$ is indeed a map into $\Sigma$.

For each $\psi\in\Psi$ we set
\begin{align}\label{OMEGAPSIdef} 
\omega^\psi:=\bigl(\omega_1^{\psi},\ldots,\omega_N^{\psi}\bigr)\in\Omega
\end{align}
and
\begin{align}\label{sigmapsiDEF}
\sigma^{\psi}:=(\psi,\omega^\psi)\in\Sigma.
\end{align}
It follows from Lemma~\ref{omegajpsiLEM} that $\sigma^\psi$ indeed lies in $\Sigma$.
Next we define the Borel probability measure $\mm$ on $\Sigma$
to be the atomic
measure on $\Sigma$ supported on the points
$\sigma^\psi$,
with
\begin{align}\label{mmdef}
\mm\bigl(\sigma^\psi\bigr):=\frac{\nbar_{\psi}}{\nbar_{\scrP}}
\qquad(\psi\in\Psi).
\end{align}
Note that $\mm$ is supported on the finite 
subset
\begin{align}\label{tPsiDEF}
\tPsi:=\{\sigma^\psi\col\psi\in\Psi\}
\end{align}
of $\Sigma$.

As in Section \ref{KINTHEORYrecapsec} we set
\begin{align*}
\scrX=\R^d\times\Sigma
\qquad\text{and}\qquad
\mu_\scrX=\vol\times\mm.
\end{align*}

\vspace{5pt}

Finally we will make our choice of the map $\vs\mapsto\mu_{\vs}$ from $\Sigma$ to $P(N(\scrX))$.
To prepare for this, we first 
introduce the following map:   %
\begin{align}\label{JmapFULL}
&J:\XX\to N_s(\scrX),
\qquad
J(\Gamma g)=\bigcup_{\psi\in\Psi} c_\psi\bigl(\Z^d\,p_\psi(g)\bigr)\times \{\sigma^{\psi}\}\qquad(g\in G).
\end{align}
This map extends the map $J_0$ defined in \eqref{Gammagpointset}
in an obvious sense.
\begin{lem}\label{JcontLEM}
$J$ is continuous.
\end{lem}
\begin{proof}
Recall that $N_s(\scrX)$ is equipped with the vague topology;
hence the task is to prove that if 
$x_1,x_2,\ldots$ is a sequence in $\XX$ converging to
$x\in \XX$,
then for any $f\in\C_c(\scrX)$
we have $\sum_{p\in J(x_k)}f(p)\to\sum_{p\in J(x)}f(p)$.
This is immediate using the formula %
\begin{align*}
\sum_{p\in J(\Gamma g)}f(p)=
\sum_{\psi\in\Psi}\sum_{\vecm\in\Z^d}
f\Bigl(c_{\psi}\bigl(\vecm\, p_{\psi}(g)\bigr),\sigma^{\psi}\Bigr),
\qquad \forall g\in G.
\end{align*}
\end{proof}
Next, for every $\psi\in\Psi$ we introduce the following modification of the map $J$:
\begin{align}\label{Jpsidef}
J_{\psi}:\XX^\psi\to N_s(\scrX);
\qquad
J_{\psi}(\Gamma g):=J(\Gamma g)\setminus\{(\bn,\sigma^\psi)\}.
\end{align}
This map $J_\psi$ is also continuous;
this is proved in the same way as Lemma \ref{JcontLEM},
using also the fact that $(\bn,\sigma^\psi)\in J(x)$ for all $x\in\XX^\psi$
(see \eqref{Xpsi0def} and \eqref{JmapFULL}).
At last, we now define the map $\vs\mapsto\mu_{\vs}$
from $\Sigma$ to $P(N_s(\scrX))$ by setting
\begin{align}\label{muvsDEF}
\mu_{\vs}=J_{\psi\,*}\hspace{2pt}\oomega
\in P(N_s(\scrX))
\qquad \text{for }\:\vs=(\psi,\omega)\in\Sigma.
\end{align}
To see that $\mu_{\vs}$ is indeed a probability measure,
note that
\begin{align}\label{oomegaXpsieq1}
\oomega(\XX^\psi)=1,\qquad\text{for all }\: \vs=(\psi,\omega)\in\Sigma;
\end{align}
cf.\ \eqref{XjiDEF},
Lemma \ref{oomegapinullLEM},
\eqref{Sigmadef}
and \eqref{oomegaDEF}.

\begin{lem}\label{muvscontLEM}
The map $\vs\mapsto\mu_{\vs}$ is continuous.
\end{lem}
\begin{proof}
Let $\vs_1,\vs_2,\ldots$ be an arbitrary sequence in $\Sigma$ converging to a point $\vs\in\Sigma$.
Write $\vs_k=(\psi^{(k)},\omega^{(k)})$ 
with $\omega^{(k)}=(\omega^{(k)}_1,\ldots,\omega^{(k)}_N)\in\Omega$;
also write $\vs=(\psi,\omega)$.
Throwing away finitely many initial points from the sequence we may assume that
$\psi^{(k)}=\psi$ for all $k$.
By \eqref{oomegaXpsieq1} we then have $\oomega(\XX^\psi)=1$ and $\overline{\omega^{(k)}}(\XX^\psi)=1$
for all $k$.
Hence we may just as well regard $\oomega$ and all $\overline{\omega^{(k)}}$ as elements in $P(\XX^\psi)$.
For each fixed $j\in\{1,\ldots,N\}$
we have $\omega^{(k)}_j\to\omega_j$ in $P(\TT_j^d)'$ as $k\to\infty$;
hence by \eqref{oomegaDEF},
Lemma~\ref{mapPTtoPXjcontLEM}
and \cite[Thm.\ 2.8(ii)]{pB99},
$\overline{\omega^{(k)}}\to\oomega$ in $P(\XX)$.
Hence also $\overline{\omega^{(k)}}\to\oomega$ in $P(\XX^\psi)$
\cite[Lemma 4.26]{oK2002},
and by the continuous mapping theorem,
$J_{\psi*\,}\overline{\omega^{(k)}}
\to J_{\psi*\,}\oomega$, in $P(N_s(\scrX))$,
viz.,
$\mu_{\vs_k}\to\mu_{\vs}$.
\end{proof}

\subsection{Some comments on the construction of $\Sigma$} %
\label{SigmaremarksexSEC}
The facts mentioned in the present section will not be used later in the text,
and we will omit some proofs.

\vspace{5pt}

The construction of the space of marks $\Sigma$ in the present paper
is %
more complicated and non-intuitive than 
for any of the previous types of scatterer configurations for which %
the framework from 
\cite{jMaS2019}
is known to apply; see 
\cite[Ch.\ 5]{jMaS2019}.
In the special cases when $\scrP$ is periodic
or when the lattices in $\scrP$ are pairwise incommensurable,
we could have chosen the space of marks to be
simply equal to the indexing set $\Psi$;
however this simple choice of $\Sigma$ is not possible in the general case;
see Remarks \ref{sigmaplimitpointsREM},
\ref{MAMSSec5p2comparisonREM} and \ref{sigmaplimitpointsREM2} below.

For a general locally finite subset $\scrP\subset\R^d$,
when trying to verify the hypotheses from
\cite[Sec.\ 2.3]{jMaS2019}
(which we recalled in Section \ref{KINTHEORYrecapsec})
without an apriori guess of what the space of marks $\Sigma$ should be,
a natural first task is 
to ignore the marking, thus 
thinking of each measure $\mu_{\vecq,\rho}^{(\lambda)}$ %
as lying in $N_s(\R^d)$,\footnote{In the notation of Lemma \ref{MAMScondsimpleconsLEM} this means:
consider $p_*(\mu_{\vecq,\rho}^{(\lambda)})$
in place of $\mu_{\vecq,\rho}^{(\lambda)}$.}
and then verify that for any fixed $\lambda\in\Pac(\US)$
and $\vecq\in\scrP$ (possibly excluding a subset $\scrE\subset\scrP$ of density zero),
the weak limit of $\mu_{\vecq,\rho}^{(\lambda)}$ as $\rho\to0$
exists %
and is independent of $\lambda$; %
call this limit measure $\mu_{\vecq,0}\in P(N(\R^d))$.
Then by Lemma \ref{MAMScondsimpleconsLEM},
the space of marks $\Sigma$ must admit a continuous map
$\vs\mapsto\hmu_{\vs}$ from $\Sigma$ to $P(N(\R^d))$
such that $\hmu_{\vs(\vecq)}=\mu_{\vecq,0}$ for all $\scrP\setminus\scrE$.

This to a large extent motivates our choice of $\Sigma$
in Section \ref{spaceofmarksandmapsSEC}:
We have constructed the space
$\Omega=\prod_{j=1}^N P(\TT_{\!j}^d)'$
to encode all the spherical equidistribution limits $\mu_{\vecq,0}$ in $P(N(\R^d))$.
Indeed, %
it follows from Theorem \ref{HOMDYNintrononunifTHM} that\footnote{The %
proof of $\mu_{\vecq,0}=(J_{0,\psi(\vecq)})_*\bigl(\overline{\omega^{(\vecq)}}\bigr)$ %
from Theorem \ref{HOMDYNintrononunifTHM} is %
somewhat technical, and we omit it. %
(The details are similar to some of the arguments in %
Sec.\ \ref{P2initialdiscsec}.)}
$\mu_{\vecq,0}=(J_{0,\psi(\vecq)})_*\bigl(\overline{\omega^{(\vecq)}}\bigr)$
for any $\vecq\in\scrP\setminus\cup_{\psi\neq\psi'\in\Psi}(\scrL_\psi\cap\scrL_{\psi'})$,
where for any $\psi\in\Psi$,
$J_{0,\psi}$ is the following modified version of %
the map defined in \eqref{Gammagpointset}:
\begin{align}\label{J0psiDEF}
J_{0,\psi}:\XX\to N(\R^d);
\qquad J_{0,\psi}(\Gamma g):=\sum_{\psi'\in\Psi}\hspace{5pt}
\sum_{\substack{\vecx\in c_{\psi'}(\Z^d\,p_{\psi'}(g))\\
(\text{if }\psi'=\psi:\:\vecx\neq\bn)}}\delta_{\vecx}
\qquad(g\in G).
\end{align}
(That is, $J_{0,\psi}$ is as in \eqref{Gammagpointset}
except that we (1) count multiplicity, and (2) if possible remove the origin from 
$c_\psi\,(\Z^dp_\psi(g))$.)
This suggests that we could make the choice ``$\Sigma=\Omega$''.
It is technically convenient to 
also include the factor $\Psi$ in the definition of $\Sigma$;
and we then restrict by the condition 
$\tr_{i\,*}\,\omega_{j}=\delta_{\bn}$
(see \eqref{Sigmadef}) %
to ensure that the map
$\vs\mapsto\mu_{\vs}$ be continuous. 

However, even though this space $\Sigma$ by construction classifies
all the relevant spherical equidistribution limits %
\textit{in $P(N(\R^d))$},
it might still %
be too crude to classify the corresponding limits
in the refined space $P(N(\scrX))$! %
This is exactly the reason why we require that %
the presentation of $\scrP$ %
should be \textit{admissible}
(Definition \ref{admissibleDEF});
only under this assumption have we managed to find a proof of
the condition [P2].
The crucial property that we need
(see the proof of Lemma \ref{P2condthmLEM1} below)
is that for each fixed $\psi\in\Psi$,
$\omega^{(\vecq)}$ tends to a \textit{unique}
limit $\omega^\psi$
as $\vecq\to\infty$ within $\scrL_\psi\setminus\scrE$; %
see Section \ref{SSjpsi0Ojpsi0sec}
where the admissibility is needed already to define $\omega_j^\psi$.

\vspace{5pt}

We remark that it would in fact be possible
to take the space $\Sigma$ to be a significantly \textit{smaller}
set than we have done.\footnote{This is an aposteriori fact;
as far as we can see, such a choice %
would not lead to a simplification of the presentation of the proof of our main
Theorem \ref{MAINTHM}.}
Indeed, assuming that the conditions [P1]--[P3] and [Q1]--[Q3]
hold for the data
$\bigl[\Sigma,\vs,\vs\mapsto\mu_\vs$, $\mm\bigr]$,
and choosing $\scrE\subset\scrP$ as in condition [P2],
set %
\begin{align}\label{genuineSigmaREMfact2}
\Sigma':=\overline{\{\vs(\vecq)\col\vecq\in\scrP\setminus\scrE\}}.
\end{align}
This is a compact subset of $\Sigma$,
and as explained in %
\cite[Remark 2.6]{jMaS2019},
after modifying the marking of the points in $\scrE$ %
in an arbitrary fashion 
to ensure that $\vs(\vecq)\in\Sigma'$ for all $\vecq\in\scrP$,
the conditions [P1]--[P3] and [Q1]--[Q3] 
\textit{remain valid with $\Sigma'$ in place of $\Sigma$.}
It is also important to note that if [P2] holds for a certain set $\scrE\subset\scrP$,
then [P2] also holds for any \textit{larger} subset $\scrE\subset\scrP$ of 
density zero.
In our setting, %
with $\Sigma$ chosen as in Section \ref{spaceofmarksandmapsSEC},
it follows from Remark~\ref{CONVremark}
that $\tPsi\subset\Sigma'$, 
but \textit{also} that, by taking $\scrE$ appropriately large,
we can ensure that the set $\{\vs(\vecq)\col\vecq\in\scrP\}$
has no limit point outside $\tPsi$.
This means that
\begin{align}\label{genuineSigmaREMfact1}
\Sigma'=\{\vs(\vecq)\col\vecq\in\scrP\setminus\scrE\}
\cup \tPsi, %
\end{align}
and that %
$\Sigma'$ has quite a simple structure as a
(metrizable, compact)
topological space: %
$\Sigma'$ is a finite or countable set containing the finite set $\tPsi$,
and its set of limit points is contained in $\tPsi$.

\begin{remark}\label{sigmaplimitpointsREM}
The description of $\Sigma'$ %
in the last sentence 
is essentially as simple as it can be: %
For the scatterer configurations %
studied in the present paper, %
it is in general impossible to satisfy
the conditions in Section \ref{KINTHEORYrecapsec} %
using a \textit{finite} space of marks.
For example, consider the %
configuration
\begin{align}\label{limitpointsEX3}
\scrP=\Z^2
\:\cup\:
\bigl(\Z^2+\bigl(0,\sqrt2\bigr)\bigr)
\:\cup\: \Z^2M_2\hspace{10pt}\text{in }\R^2,
\qquad\text{with }\: M_2=\matr 1{\sqrt2}1{\sqrt2+1}\in\SL_2(\R).
\end{align}
In this case, one obtains infinitely many distinct
spherical equidistribution limits $\mu_{\vecq,0}$ in $P(N(\R^2))$
as $\vecq$ varies through the lattice $\Z^2M_2\subset\scrP$.
Indeed, for two arbitrary points
$\vecq,\vecq'\in\Z^2M_2$,
writing 
$\vecq=\vecm M_2$ 
with $\vecm=(m_1,m_2)\in\Z^2$
and similarly $\vecq'=\vecm'M_2$,
it turns out that 
\begin{align}\label{limitpointsEX3disc1}
\mu_{\vecq,0}=\mu_{\vecq',0}\quad\Leftrightarrow\quad
\bigl[m_1+m_2=m_1'+m_2'\:\text{ or }\:m_1+m_2=1-(m_1'+m_2')\bigr].
\end{align}
We give an outline of the proof below.
It follows from \eqref{limitpointsEX3disc1} that for any set $\scrE\subset\scrP$ of density zero,
the set of spherical equidistribution limits $\{\mu_{\vecq,0}\col\vecq\in\scrP\setminus\scrE\}$
is \textit{infinite.} (Indeed, if it were finite, then by \eqref{limitpointsEX3disc1}
the point set $\Z^2M_2\setminus\scrE$
would be contained in a finite number of lines %
and thus be of density zero,
whereas in fact it has density one.) %
Hence as in the discussion in the beginning of this subsection, it follows from 
Lemma~\ref{MAMScondsimpleconsLEM} that
for any marking of $\scrP$
used to satisfy the conditions 
[P1]--[P3] and [Q1]--[Q3] in Section~\ref{KINTHEORYrecapsec},
the set of marks $\{\vs(\vecq)\col\vecq\in\scrL_{(2,1)}\setminus\scrE\}$ must be infinite,
and in particular \textit{the space of marks must be infinite.}
Of course, this also implies that the space of marks must have some limit point, since the 
space of marks is required to be compact.
\end{remark}

\textit{Outline of proof of \eqref{limitpointsEX3disc1}.}
We express $\scrP$ in \eqref{limitpointsEX3} as in \eqref{PSIdef}--\eqref{GENPOINTSET1}
with $N=2$, $r_1=2$, $r_2=1$,
$M_1=\smatr1001$, and $M_2$ as above;
also $c_{1,1}=c_{1,2}=c_{2,1}=1$
and $\vecw_{1,1}=\vecw_{2,1}=(0,0)$ and $\vecw_{1,2}=(0,\sqrt2)$.
For any point $\vecq=\vecm M_2$ in $\scrL_{(2,1)}=\Z^2M_2$,
$\omega_2^{(\vecq)}$ %
is the Dirac measure at the origin
of the 2-dimensional torus $\TT_2^2$.
Furthermore, by \eqref{Ujqdef} and \eqref{LjDEF},
\begin{align*}
L_1^{(\vecq)}=L_1^{(\pi(U_1^{(\vecq)}))}=\fL\left(
\begin{pmatrix}0 \\[3pt]  0\end{pmatrix},
\begin{pmatrix}-(m_1+m_2)\sqrt2 \\[3pt] (1-m_1-m_2)\sqrt2\end{pmatrix}\right)
=\R\begin{pmatrix}m_1+m_2\\ m_1+m_2-1\end{pmatrix}\hspace{10pt}\subset\R^2,
\end{align*}
and so by \eqref{OjVDEF},
$\scrO_1^{(\vecq)}=\scrO_1^{(\pi(U_1^{(\vecq)}))}=\pi\bigl(L_1^{(\vecq)}\bigr)^2$, %
which is a $2$-dimensional subtorus of $\TT_1^2$,
and $\omega_1^{(\vecq)}$ is the normalized Haar measure of this subtorus.
Now %
$\mu_{0,\vecq}=(J_{0,(2,1)})_*\bigl(\overline{\omega^{(\vecq)}}\bigr)
=(J_{0,(2,1)})_*\bigl(\overline{\omega_1^{(\vecq)}}\otimes\overline{\omega_2^{(\vecq)}}\bigr)$.
The left implication in \eqref{limitpointsEX3disc1} follows quite easily.
Indeed, it is clear from the above that $\mu_{0,\vecq}$ only depends on 
$m_1+m_2$, and 
to prove that $\mu_{\vecq,0}=\mu_{\vecq',0}$
also holds in the case $m_1+m_2=1-(m_1'+m_2')$,
one makes use of a symmetry originating from the diffeomorphism
$\scmatr{\vecv_1}{\vecv_2}\mapsto\scmatr{\vecv_2}{\vecv_1}$
from $\pi(L_1^{(\vecq)})^2$ onto $\pi(L_1^{(\vecq')})^2$. %
The proof of the converse %
is more involved;
one approach goes via a study of the \textit{support} of $\mu_{\vecq,0}$ in $N(\R^2)$.
For example, by a careful explicit analysis one verifies that for any $\vecq=\vecm M_2$ in $\scrL_{(2,1)}$, the set
$\Bigl\{k\in\Z_{\geq3}\col
\sum_{\vecx\in\Z^2\setminus\{\bn\}}\delta_{\vecx}
+\sum_{\vecx\in\Z^2}\bigl(\delta_{\vecx}+\delta_{\vecx+k^{-1}\vece_1}\bigr)\in\supp(\mu_{\vecq,0})\Bigr\}$
contains exactly those %
$k\in\Z_{\geq3}$ which divide $m_1+m_2$ or $m_1+m_2-1$;
this gives the right implication in \eqref{limitpointsEX3disc1}.
\hfill$\square$

\begin{remark}\label{MAMSSec5p2comparisonREM}
On the other hand, if $\scrP$ is \textit{periodic},
then $\{\vs(\vecq)\col\vecq\in\scrP\}=\tPsi$,
i.e.\ the space of marks $\Sigma'$ in \eqref{genuineSigmaREMfact2}
is simply equal to $\tPsi$, even if we take $\scrE$ empty!
Indeed, any periodic $\scrP$ may be represented as in
\eqref{PSIdef}--\eqref{GENPOINTSET1} with
$N=1$ and all $c_{1,i}$ ($i=1,\ldots,r_1$) equal to a common value $c>0$.
(This representation is admissible; note also that all the grids 
$\scrL_{(1,1)},\ldots,\scrL_{(1,r_1)}$ are pairwise disjoint, by 
\eqref{THINDISJcond2}.)
Then for any $i\in\{1,\ldots,r_1\}$,
all points $\vecq\in\scrL_{(1,i)}$ give the \textit{same}
$\omega^{(\vecq)}$.
Indeed, for any two $\vecq,\vecq'\in\scrL_{(1,i)}$ 
we have $U_1^{(\vecq)}-U_1^{(\vecq')}\in\M_{1\times d}(\Z)$
by \eqref{Ujqdef}; %
thus $\pi(U_1^{(\vecq)})=\pi(U_1^{(\vecq')})$ in $\TT_1^d$,
and so $\omega_1^{(\vecq)}=\omega_1^{(\vecq')}$.
It then follows from Proposition \ref{muconvLEM},
or by direct inspection of the definitions,
that in fact $\omega_1^{(\vecq)}=\omega_1^{(1,i)}$ for all $\vecq\in\scrL_{(1,i)}$.
Hence by %
\eqref{vsmapDEF} and \eqref{sigmapsiDEF},
$\vs(\vecq)=\sigma^{\psi(\vecq)}\in\tPsi$ for all $\vecq\in\scrP$.

As we have mentioned, in the special case of periodic $\scrP$,
Theorem \ref{MAINTHM} was proved in 
\cite[Prop.\ 5.6]{jMaS2019},
and it is easy to give the translation between the two formulations: %
Our marking data $[\Sigma',\vs,\vs\mapsto\mu_\vs,\mm]$ with $\Sigma'=\tPsi$
agrees exactly with the marking data in 
\cite[Sec.\ 5.2, Prop.\ 5.6]{jMaS2019},
with the only difference that 
the mark $\sigma^{(1,\ell)}$ is called simply ``$\ell$'' in 
\cite[Sec.\ 5.2]{jMaS2019},
for $\ell=1,\ldots,r_1$
(also our $r_1$ is called ``$m$'', and our $c$ is ``$\delta^{1/d}$'').
In particular, thus, for each $\ell\in\{1,\ldots,r_1\}$
the measure $\mu_{\sigma^{(1,\ell)}}\in P(N_s(\R^d\times\Sigma))$
defined by \eqref{muvsDEF} above,
equals\footnote{After noticing that
$\mu_{\sigma^{(1,\ell)}}$ is in fact supported on $N_s(\R^d\times\tPsi)$,
and then identifying $\tPsi$ with $\{1,\ldots,r_1\}$ in the way we have just explained.}
the measure ``$\mu_{\ell}$''
in \cite[(5.35)]{jMaS2019},
although it requires some work to verify this fact from the definitions.

As noted in \cite[Prop.\ 5.6]{jMaS2019},
for $\scrP$ periodic,
the condition [P2] holds with $\scrE=\emptyset$.
Indeed, because of the periodicity,
we have the even much stronger statement that the convergence
$\mu_{\vecq,\rho}^{(\lambda)}\xrightarrow[]{\textup{ w }}\mu_{{\vs}(\vecq)}$
required in \eqref{repASS:KEY}
holds \textit{uniformly over all $\vecq\in\scrP$.}
\end{remark}

\begin{remark}\label{sigmaplimitpointsREM2}
Also in the case when $\scrP$ is a finite union of pairwise incommensurable grids,
one can use $\Sigma'=\tPsi$ as the space of marks.
Indeed, in this case we have $r_1=\cdots=r_N=1$,
and for any $j\neq j'$ in $\{1,\ldots,N\}$ we have
$L_j^{(j',1)}=L_j=\R^1$
by Lemma \ref{LjvsLjpsiLEM} and \eqref{trueLjDEF};
this also implies that $\omega_j^{(j',1)}$ equals Haar measure on $\TT_j^d$.
Now by Lemma~\ref{MuconvauxLEM1},
applied with $\veca=1\in\Z^1$,
for any $j\neq j'$
the set $\{\vecq\in\scrL_{(j',1)}\col L_j^{(\vecq)}=\{0\}\}$
has density zero.
Hence we may assume that $\scrE$ contains the union of these sets for all pairs of $j\neq j'$.
Then, for any $\vecq\in\scrP\setminus\scrE$,
taking $j'$ so that $\vecq\in\scrL_{(j',1)}$ it follows that
for each $j\neq j'$ we have $L_j^{(\vecq)}\neq\{0\}$,
viz., $L_j^{(\vecq)}=\R^1$;
and therefore $\scrO_j^{(\vecq)}=\TT_j^d$ and 
$\omega_j^{(\vecq)}=\omega_j^{(j',1)}$
(this also implies that $\vecq\notin\scrL_{(j,1)}$,
and so $\psi(\vecq)$ must equal $(j',1)$).
On the other hand, immediately from the definitions we have that
both $\omega_{j'}^{(\vecq)}$ and $\omega_{j'}^{(j',1)}$ equals the Dirac measure at the origin
of $\TT_{j'}^d$.
Hence $\omega^{(\vecq)}=\omega^{(j',1)}=\omega^{\psi(\vecq)}$
and $\vs(\vecq)=\sigma^{\psi(\vecq)}\in\tPsi$.
Since this holds for all $\vecq\in\scrP\setminus\scrE$,
we get $\Sigma'=\tPsi$
in \eqref{genuineSigmaREMfact2}.
\end{remark}

\section{Verification of [Q1],[Q2],[Q3],[P1],[P3], and initial discussion regarding [P2]}
\label{verQ1Q2etcSEC}

In order to prove the main result of the paper, 
Theorem \ref{MAINTHM},
we now wish to prove that all the conditions [P1]--[P3] and [Q1]--[Q3]
are satisfied for the 
maps $\vs:\scrP\to\Sigma$ and $\vs\mapsto\mu_\vs$ and measure $\mm$
which we have
defined in the previous section.
In the present section we will prove all of these conditions except [P2];
we will also reduce the verification of [P2] to a certain statement about
uniform equidistribution in the homogeneous space $\XX$,
Theorem \ref{HOMDYNMAINTHM} below.

\subsection{Verification of  [Q1], [Q2], [Q3]}

We start with the conditions [Q1]--[Q3].

The following lemma shows that [Q1] holds, in a much stronger form.
\begin{lem}\label{muvsSLdinvLEM}
For every $\vs\in\Sigma$, $\mu_\vs$ is $\SL_d(\R)$-invariant.
\end{lem}
\begin{proof}
For any $\omega_j\in P(\TT_j^d)'$,
the measure $\oomega_j$ on $\XX_j$ %
is right $\SL_d(\R)$ invariant by Lemma \ref{rhosoomegaeqnuLEM};
hence for any $\omega\in\Omega$,
the measure $\oomega$ on $\XX$ is %
right $\SL_d(\R)^N$-invariant,
and in particular it is
right $\varphi(\SL_d(\R))$-invariant,
where $\varphi:\SL_d(\R)\to G$ is the diagonal embedding.
Next we note that, for any $\psi\in\Psi$,
\begin{align*}
J_\psi(\Gamma g\varphi(h))=J_\psi(\Gamma g)h\qquad\text{for all }\: \Gamma g\in\XX^\psi,\: h\in\SL_d(\R).
\end{align*}
It follows that the measure $J_{\psi*\,}\oomega$ on $N_s(\scrX)$ is $\SL_d(\R)$-invariant, for any $\omega\in\Omega$.
This implies the lemma, via the definition \eqref{muvsDEF}.
\end{proof}

Next, because of the following general fact,
also [Q2] %
is an immediate consequence of Lemma~\ref{muvsSLdinvLEM}:
\begin{lem}\label{Q2frominvarianceLEM}
If $\mu\in P(N_s(\scrX))$ is invariant under the action of $\SO(d)$
then 
\begin{align}
\mu(\{\nu\in N(\scrX)\col \exists x_1\in\R\text{ s.t.\ }\nu(\{x_1\}\times\R^{d-1}\times\Sigma)>1\})=0. 
\end{align}
\end{lem}
\begin{proof}
For any $\vecv\in\US$, set
\begin{align}
A_{\vecv}=\{\nu\in N(\scrX)\col \exists x_1\in\R\text{ s.t. }\nu((x_1\vecv+\vecv^\perp)\times\Sigma)>1\}.
\end{align}
Then our task is to prove $\mu(A_{\vece_1})=0$.
The fact that $\mu$ is $\SO(d)$-invariant implies that
$\mu(A_\vecv)=\mu(A_{\vece_1})$ for all $\vecv\in\US$.
Hence we have, with $\lambda_1$ being the uniform probability measure on $\US$:
\begin{align*}
\mu(A_{\vece_1})=\int_{\US}\mu(A_\vecv)\,d\lambda_1(\vecv)
=\int_{N_s(\scrX)}\int_{\US}I(\nu\in A_\vecv)\,d\lambda_1(\vecv)\,d\mu(\nu)=0.
\end{align*}
Here the second equality holds by Fubini's Theorem
and since 
$\mu(N(\scrX)\setminus N_s(\scrX))=0$
(because of $\mu\in P(N_s(\scrX))$),
and the last equality holds since for any $\nu\in N_s(\scrX)$ 
we have $\int_{\US}I(\nu\in A_\vecv)\,d\lambda_1(\vecv)=0$,
since the set $\{\vecv\in\US\col \nu\in A_\vecv\}$
is a countable union of 
subspheres of $\US$ of codimension one.
\end{proof}

Next we turn to the condition [Q3].
\begin{lem}\label{Q3lem}
[Q3] holds.
\end{lem}
\begin{proof}
(Cf.\ the proof of
\cite[Lemma 5.3.13]{jMaS2019}.)
Set 
$\Lambda=\SL_d(\Z)$ %
and $\YY=\Lambda\bs\SL_d(\R)$.
Let us also fix a choice of $\psi=(j,i)\in\Psi$.
For $R>0$ we set
\begin{align*}
\YY(R)=\{\Lambda h\in\YY\col c_\psi(\Z^dh)+\scrB_{R/2}^d=\R^d\}.
\end{align*}
Note that the set $\YY(R)$ is increasing with respect to $R$.
Also, for every $\Lambda h\in\YY$, the set $c_\psi(\Z^dh)$ is a lattice in $\R^d$
and hence there exists some $R=R(h)>0$ such that 
$c_\psi(\Z^dh)+\scrB_{R/2}^d=\R^d$.
Hence $\cup_{R>0}\YY(R)=\YY$,
and it follows that for any given $\ve>0$ we can choose $R>0$ so that
$\eta(\YY(R))>1-\ve$,
where $\eta$ is the $\SL_d(\R)$-invariant probability measure on $\YY$.

With this choice of $R$, 
we now claim that for any $\vs=(\psi',\omega)\in\Sigma$ and $\vecx\in\R^d$,
\eqref{Q3cond} holds.
By \eqref{muvsDEF},
this is equivalent to the following:
\begin{align}\label{Q3lempf1}
\oomega\bigl(\bigl\{\Gamma g\in\XX^{\psi'}\col J_{\psi'}(\Gamma g)\cap(\scrB_R^d(\vecx)\times\Sigma)=\emptyset\bigr\}\bigr)<\ve.
\end{align}
Choose $\vecy\in\R^d$ so that
$\scrB^d_{R/2}(\vecy)\subset\scrB^d_R(\vecx)$ and $\bn\notin\scrB^d_{R/2}(\vecy)$.
For any $g\in G$,
letting $h=\iota(p_j(g))\in\SL_d(\R)$
we have that
the grid $c_\psi(\Z^d p_\psi(g))$ is a translate of the lattice $c_\psi(\Z^d h)$;
and in particular if $\Lambda h\in\YY(R)$
then $c_\psi(\Z^d p_\psi(g))$ must contain a point in $\scrB^d_{R/2}(\vecy)$.
This implies that for every $\Gamma g\in\XX$
satisfying $\tiota(\tp_j(\Gamma g))\in\YY(R)$,
the point set $J(\Gamma g)$
(cf.\ \eqref{JmapFULL})
must contain a point in $\scrB^d_{R/2}(\vecy)\times\Sigma$.
Using also $\bn\notin\scrB^d_{R/2}(\vecy)$ and $\scrB^d_{R/2}(\vecy)\subset\scrB^d_R(\vecx)$,
it follows that the measure in the left hand side of \eqref{Q3lempf1}
is bounded above by
\begin{align}\label{Q3lempf2}
\oomega\bigl(\bigl\{\Gamma g\in\XX^{\psi'}\col \tiota(\tp_j(\Gamma g))\notin\YY(R)\bigr\}\bigr).
\end{align}
However,
writing $\omega=(\omega_1,\ldots,\omega_j)$ we have
$\tp_{j*\,}(\oomega)=\oomega_j$, which is an $\SL_d(\R)$-invariant probability measure on $\XX_j$
(this is immediate from the definition \eqref{mapPTtoPXjrep}, as we have noted previously).
Hence the pushforward of $\oomega$ by $\tiota\circ\tp_j$ equals $\eta$,
and so the measure in \eqref{Q3lempf2}
equals $\eta(\YY\setminus\YY(R))$, which by our choice of $R$ is less than $\ve$.
Hence \eqref{Q3lempf1}, and thereby the lemma, is proved.
\end{proof}

We end this section by giving a closely related fact, %
namely a formula for the intensity measure of 
a point process with distribution $\mu_\vs$, for any $\vs\in\Sigma$.
This formula will not be used until Section \ref{transkerSEC}.
\begin{lem}\label{INTENSITYmunuLEM}
Let $\vs=(\psi',\omega)\in\Sigma$.
Then for any Borel set $B\subset\scrX$ we have
\begin{align}\label{INTENSITYmunuLEMres}
\int_{N_s(\scrX)}\#(B\cap Y)\,d\mu_\vs(Y)=\nbar_{\scrP}\mu_{\scrX}(B)
+\sum_{\psi\in\Psi\setminus\{\psi'\}}
\omega_{j_{\psi}}\bigl(\tr_{i_{\psi}}^{-1}(\{\bn\})\bigr)\cdot I\bigl((\bn,\sigma^{\psi})\in B\bigr).
\end{align}
In particular, for every $\psi'\in\Psi$, a point process with distribution 
$\mu_{\sigma^{\psi'}}$ has intensity measure $\nbar_{\scrP}\mu_{\scrX}$.
\end{lem}
\begin{proof}
By \eqref{muvsDEF}, \eqref{Jpsidef} and \eqref{JmapFULL},
the left hand side of \eqref{INTENSITYmunuLEMres} equals
\begin{align*}
-I\bigl((\bn,\sigma^{\psi'})\in B\bigr)
+\int_{\XX^{\psi'}}\sum_{\psi\in\Psi}\sum_{\vecm\in\Z^d}I\Bigl(\bigl(c_{\psi}\cdot\bigl( \vecm p_{\psi}(g)\bigr),\sigma^{\psi}\bigr)\in B\Bigr)
\,d\oomega(\Gamma g).
\end{align*}
By \eqref{oomegaXpsieq1}, the integration may just as well be taken over all $\XX$.
Moving out the sum over $\psi$ and then using \eqref{oomegaDEF} and \eqref{ppsiDEF}, we get
\begin{align*}
-I\bigl((\bn,\sigma^{\psi'})\in B\bigr)
+\sum_{\psi\in\Psi}\int_{\XX_{j_\psi}}\sum_{\vecm\in\Z^d}I\Bigl(\bigl(c_{\psi}\cdot\bigl( \vecm \a_{i_\psi}(g)\bigr),\sigma^{(\psi)}\bigr)\in B\Bigr)
\,d\oomega_{j_\psi}(\Gamma_{j_\psi} g),
\end{align*}
and by Proposition \ref{SIEGELFORMULALEM1}, this equals
\begin{align*}
-I\bigl((\bn,\sigma^{\psi'})\in B\bigr)
+\sum_{\psi\in\Psi}\nbar_\psi\int_{\R^d}I\bigl((\vecx,\sigma^{\psi})\in B\bigr)\,d\vecx
+\sum_{\psi\in\Psi}\omega_{j_{\psi}}\bigl(\tr_{i_{\psi}}^{-1}(\{\bn\})\bigr)\cdot I\bigl((\bn,\sigma^{\psi})\in B\bigr).
\end{align*}
Using 
$\omega_{j_{\psi'}}\bigl(\tr_{i_{\psi'}}^{-1}(\{\bn\})\bigr)=1$,
which holds since $(\psi',\omega)\in\Sigma$,
together with $\mu_{\scrX}=\vol\times\mm$ and \eqref{mmdef},
we obtain the right hand side of \eqref{INTENSITYmunuLEMres}.

To obtain the last statement of the lemma, we apply 
\eqref{INTENSITYmunuLEMres} for $\vs=\sigma^{\psi'}=(\psi',\omega^{\psi'})$;
in this case the sum over $\Psi\setminus\{\psi'\}$ in \eqref{INTENSITYmunuLEMres}
vanishes, by Lemma \ref{omegajpsiLEM};
hence the right hand side of \eqref{INTENSITYmunuLEMres}
equals $\nbar_{\scrP}\,\mu_{\scrX}(B)$.
\end{proof}

\subsection{Verification of [P1] (uniform density)}

\begin{prop}\label{P1prop}
[P1] holds,
i.e.\ for any bounded subset $B \subset\scrX$ with $\mu_\scrX(\partial B)=0$, we have
\begin{align}\label{P1lemRES}
\lim_{T\to\infty}\frac{\#(\tP\cap TB)}{T^d}=\nbar_{\scrP}\mu_\scrX(B).
\end{align}
\end{prop}
\begin{proof}
Note that $\scrX$ decomposes as the disjoint union $\sqcup_{\psi\in\Psi}\scrX_\psi$,
where $\scrX_\psi=\R^d\times\Sigma_\psi$ with
$\Sigma_\psi:=(\{\psi\}\times\Omega)\cap\Sigma$;
note also that each set $\scrX_\psi$ is both open and closed is $\scrX$.
It follows that it suffices to prove \eqref{P1lemRES} under the extra assumption that
$B\subset\scrX_\psi$ for some fixed $\psi$.
Using also the fact that 
the set $\{\vecq\in\scrL_\psi\col\psi(\vecq)\neq\psi\}$
has  density zero
(cf.\ Remark \ref{psiaedef}),
it follows that our task is to prove the following,
for any bounded set $B\subset\scrX_\psi$ with
$\mu_\scrX(\partial B)=0$:
\begin{align}\label{P1lempf1}
\lim_{T\to\infty}\frac{\#\{\vecq\in\scrL_{\psi}\col(\vecq,(\psi,\omega^{(\vecq)}))\in TB\}}{T^d}=\nbar_{\scrP}\mu_\scrX(\scrX_\psi).
\end{align}

Let us first verify that \eqref{P1lempf1} holds for any set $B$ of the form
\begin{align}\label{P1lempf2}
B=\Bigl(\prod_{i=1}^d[\alpha_i,\beta_i)\Bigr)\times U,
\end{align}
for any real numbers $\alpha_i<\beta_i$ ($i=1,\ldots,d$)
and any open neighbourhood
$U$ of $\sigma^\psi$ in %
$\Sigma_\psi$.
Indeed, given such a $U$, there exist open neighbourhoods $U_j$ of $\omega_j^\psi$ in $P(\TT_j^d)$ for $j=1,\ldots,N$
such that $\Sigma_\psi\cap(\{\psi\}\times\prod_{j=1}^NU_j)\subset U$.
Applying now Proposition \ref{muconvLEM} to the set $U_j$,
for each $j=1,\ldots,N$,
it follows that there exists a subset $\scrZ'\subset\scrL_{\psi}$
of  density zero such that
$(\psi,\omega^{(\vecq)})\in U$
for all $\vecq\in\scrL_{\psi}\setminus\scrZ'$.
Furthermore, for $B$ as in \eqref{P1lempf2},
we have $\nbar_{\scrP}\mu_\scrX(B)
=\nbar_\psi \prod_{i=1}^d(\beta_i-\alpha_i)$,
and so \eqref{P1lempf1}
follows from the fact that 
the grid $\scrL_\psi$ has asymptotic density $\nbar_\psi$ in $\R^d$.

Next, if $B$ is any bounded \textit{open} subset of $\scrX_\psi$,
then there exists a sequence $B_1\subset B_2\subset\cdots$
of subsets of $B$ such that
each $B_k$ is a finite disjoint union of sets of the form \eqref{P1lempf2},
and $B\cap(\R^d\times\{\sigma^\psi\})\subset\cup_{k=1}^\infty B_k$.
Using $\mu_{\scrX}(B)=\mu_{\scrX}(B\cap(\R^d\times\{\sigma^\psi\}))$,
it follows that $\mu_\scrX(B_k)\to\mu_\scrX(B)$ as $k\to\infty$,
and hence since \eqref{P1lempf1} holds for each of our sets $B_k$,
it follows that
\begin{align}\label{P1lempf3}
\liminf_{T\to\infty}\frac{\#\{\vecq\in\scrL_{\psi}\col(\vecq,(\psi,\omega^{(\vecq)}))\in TB\}}{T^d}\geq \nbar_{\scrP}\mu_\scrX(B).
\end{align}
Next if $\tB$ is any bounded \textit{closed} subset of $\scrX_\psi$,
then by taking $R>0$ so that $\tB\subset B':=\scrB_R^d\times\Sigma_\psi$,
and noticing that \eqref{P1lempf1}
holds for $B'$ 
(since the grid $\scrL_\psi$ has asymptotic density $\nbar_\psi$ in $\R^d$),
and also \eqref{P1lempf3} holds for the bounded open set $B'\setminus\tB$,
it follows that
\begin{align}\label{P1lempf4}
\limsup_{T\to\infty}\frac{\#\{\vecq\in\scrL_{\psi}\col(\vecq,(\psi,\omega^{(\vecq)}))\in T\tB\}}{T^d}
\leq \nbar_{\scrP}\mu_\scrX(\tB).
\end{align}

Finally consider an arbitrary bounded subset $B \subset\scrX_\psi$ with $\mu_\scrX(\partial B)=0$.
Let $B^\circ$ and $\oB$ be the interior and the closure of $B$, respectively.
Then \eqref{P1lempf3} holds for $B^\circ$
and \eqref{P1lempf4} holds for $\oB$,
and furthermore $\mu_\scrX(B^\circ)=\mu_\scrX(\oB)$,
since $\mu_\scrX(\partial B)=0$.
Hence \eqref{P1lempf1} holds.
\end{proof}

\subsection{Initial discussion regarding [P2] (uniform spherical equidistribution)}
\label{P2initialdiscsec}

We have the following result:
\begin{thm}\label{P2condthm}
[P2] holds, i.e.\
there exists a subset $\scrE\subset\scrP$ of  density zero
such that
for any fixed $T\geq1$ and $\lambda\in \Pac(\US)$, we have
$\mu_{\vecq,\rho}^{(\lambda)}\xrightarrow[]{\textup{ w }}\mu_{{\vs}(\vecq)}$
as $\rho\to0$, uniformly for $\vecq\in\scrP\cap\scrB^d_{T\rho^{1-d}}\setminus\scrE$.
\end{thm}

In this section we will 
prove that Theorem \ref{P2condthm}
follows from the following 
theorem on uniform equidistribution in the homogeneous space $\XX$,
the proof of which is the main goal of the later sections in this paper.

Recall that $\varphi:\SL_d(\R)\to G$ is the diagonal embedding.

\begin{thm}\label{HOMDYNMAINTHM}
Given any $\psi\in\Psi$ and any decreasing function $\scrT:(0,1)\to\R^+$,
there exists a 
subset $\scrE\subset\scrL_\psi$ of  density zero
such that for any fixed $f\in\C_b(\XX^\psi)$ and $\lambda\in\Pac(\US)$,
we have
\begin{align}\label{HOMDYNMAINTHMres}
\int_{\US}f\bigl(\Gamma g_0^{(\vecq)}\varphi(R(\vecv)D_\rho)\bigr)\,d\lambda(\vecv)
-\int_{\XX^\psi} f \,d\overline{\omega^{(\vecq)}}\to0
\end{align}
as $\rho\to0$,
uniformly over all $\vecq\in\scrL_\psi\cap\scrB_{\scrT(\rho)}^d\setminus\scrE$.
\end{thm}

To see that the statement of Theorem \ref{HOMDYNMAINTHM}
makes sense, note that for every $\vecq\in\scrL_\psi$ we have
$\Gamma g_0^{(\vecq)}\in\XX^\psi$
by Lemma \ref{g0qinXpsiLEM},
and so 
$\Gamma g_0^{(\vecq)}\varphi(R(\vecv)D_\rho)\in\XX^\psi$ for all $\vecv\in\US$ and $\rho>0$;
furthermore, $\overline{\omega^{(\vecq)}}(\XX^{\psi})=1$
by Lemmas \ref{omegajqpiinv0LEM} and \ref{oomegapinullLEM},
meaning that $\overline{\omega^{(\vecq)}}\in P(\XX)$ 
gives a probability measure on the subset $\XX^\psi$.
It is also worth noticing that 
without the \textit{uniformity over} $\vecq$,
the statement of Theorem \ref{HOMDYNMAINTHM}
would be an immediate consequence of Theorem \ref{HOMDYNintrononunifTHM}.

\vspace{5pt}

We will now give the proof of Theorem \ref{P2condthm},
assuming Theorem \ref{HOMDYNMAINTHM}.
As the very first step, let us apply Theorem \ref{HOMDYNMAINTHM}
with $\scrT(\rho)=\rho^{-d}$
and for each $\psi\in\Psi$;
this gives the existence of subsets $\scrE_\psi\subset\scrL_\psi$ of  density zero
such that
for any fixed $f\in\C_b(\XX^\psi)$ and $\lambda\in\Pac(\US)$,
\begin{align}\label{HOMDYNMAINTHMresREPpre}
\int_{\US}f(\Gamma g_0^{(\vecq)}\varphi(R(\vecv)D_\rho))\,d\lambda(\vecv)
-\int_{\XX^\psi} f \,d\overline{\omega^{(\vecq)}}\to0
\end{align}
as $\rho\to0$,
uniformly over all $\vecq\in\scrL_\psi\cap\scrB_{\rho^{-d}}^d\setminus\scrE_\psi$.
Let us now set
\begin{align}\label{scrEass}
\scrE:=\bigcup_{\psi\in\Psi}\scrE_\psi \: \cup \:
\bigcup_{\psi\neq\psi'\in\Psi}(\scrL_\psi\cap\scrL_{\psi'}).
\end{align}
By Lemma \ref{THINDISJcondLEM}, %
this set $\scrE$ is a subset of $\scrP$ of density zero.
We will keep this set $\scrE$ \textit{fixed} in the rest of the section,
and we will prove that the statement of 
Theorem \ref{P2condthm} holds with \textit{this} set $\scrE$.
For any $T\geq1$ and $\rho>0$ we write,
as in \eqref{repASS:KEY},
\begin{align*}
\scrP_T(\rho):=\scrP\cap\scrB_{T\rho^{1-d}}^d\setminus\scrE.
\end{align*}
Note that for any given $T$ we have $T\rho^{1-d}\leq\rho^{-d}$ for all sufficiently small $\rho$.
Recall also that $\scrP=\cup_{\psi\in\Psi}\scrL_\psi$.
Hence
the statement around \eqref{HOMDYNMAINTHMresREPpre}
now implies that the following holds:
\begin{align}\label{HOMDYNMAINTHMresREP}
\begin{cases}
\text{for any fixed $T\geq1$, $f\in\C_b(\XX)$ and $\lambda\in\Pac(\US)$,}
\\[9pt]
{\displaystyle
\rule{0pt}{0pt}
\hspace{80pt}\int_{\US}f(\Gamma g_0^{(\vecq)}\varphi(R(\vecv)D_\rho))\,d\lambda(\vecv)
-\int_{\XX} f \,d\overline{\omega^{(\vecq)}}\to0}
\hspace{50pt}
\\[14pt]
\text{as $\rho\to0$,
uniformly over all $\vecq\in\scrP_T(\rho)$.}
\end{cases}
\end{align}

\begin{lem}\label{P2condthmLEM6}
Let $k\in\Z_{>0}$ and let $B$ be a bounded subset of $\R^d$ with $\vol(\partial B)=0$.
Then for any 
$V>\vol(B)$,
$T\geq1$ and $\lambda\in\Pac(\US)$,
there exists $\rho_0\in(0,1)$ such that
\begin{align}\label{P2condthmLEM6res}
\lambda(\{\vecv\in\US\col\#((\scrP-\vecq)R(\vecv)D_\rho\cap B\setminus\{\bn\})\geq k\})
<\nbar_{\scrP}V/k
\end{align}
for all $\rho\in(0,\rho_0)$ and $\vecq\in\scrP_T(\rho)$.
\end{lem}
\begin{proof}
The assumptions imply that $B$ is Jordan measurable,
and hence there is a function $f\in\C_c(\R^d)$ such that $f=1$ on $B$,
$0\leq f\leq 1$ everywhere, and $V_f:=\int_{\R^d}f\,d\vol<V$.
For each $\psi=(j,i)\in\Psi$,
let $\hf_\psi\in\C(\XX_j)$ be the ``$\psi$th Siegel transform'' of $f$,
as defined in \eqref{SIEGELFORMULALEM1res1}.
The function $\hf_\psi$ is typically unbounded;
therefore we set $\tf_\psi=\min(k+1,\hf_\psi)$; this is a nonnegative function in $\C_b(\XX_j)$,
and hence $\tf_\psi\circ p_j\in\C_b(\XX)$.
Hence by \eqref{HOMDYNMAINTHMresREP},
and since $p_j(g_0^{(\vecq)})=\I_{U_j^{(\vecq)}}M_j$,
\begin{align*}
\int_{\US}\tf_\psi\bigl(\Gamma_j \I_{U_j^{(\vecq)}}M_j R(\vecv)D_\rho\bigr)\,d\lambda(\vecv)
-\int_{\XX} \tf_\psi\circ p_j \,d\overline{\omega^{(\vecq)}}\to0
\end{align*}
as $\rho\to0$,
uniformly over all
$\vecq\in\scrP_T(\rho)$.
Here $p_{j*}\bigl(\overline{\omega^{(\vecq)}}\bigr)=\overline{\omega_j^{(\vecq)}}$
(cf.\ \eqref{oomegaDEF});
hence by Proposition~\ref{SIEGELFORMULALEM1}
and Lemma \ref{omegajqpiinv0LEM},
\begin{align*}
\int_{\XX} \tf_\psi\circ p_j \,d\overline{\omega^{(\vecq)}}
\leq\int_{\XX} \hf_\psi\circ p_j \,d\overline{\omega^{(\vecq)}}
=\nbar_\psi V_f+\delta_{\vecq\in\scrL_\psi}\cdot f(\bn).
\end{align*}
Adding the above over all $\psi$
and using $\nbar_{\scrP}:=\sum_{\psi\in\Psi}\nbar_{\psi}$, it follows that 
there exists some $\rho_0\in(0,1)$ such that
\begin{align}\label{P2condthmLEM6pf1}
\int_{\US}\sum_{\psi\in\Psi}\Bigl(\tf_\psi(\Gamma_j \I_{U_j^{(\vecq)}}M_j R(\vecv)D_\rho)
-\delta_{\vecq\in\scrL_\psi}\cdot f(\bn)\Bigr)
\,d\lambda(\vecv)<
\nbar_{\scrP}V
\end{align}
for all $\rho\in(0,\rho_0)$ and $\vecq\in\scrP_T(\rho)$.
Here, by \eqref{SIEGELFORMULALEM1res1} and 
\eqref{KEYUjqdefmotivation},
we have for every $\vecv$ and $\psi$:
\begin{align*}
\tf_\psi\bigl(\Gamma_j \I_{U_j^{(\vecq)}}M_j R(\vecv)D_\rho\bigr)-\delta_{\vecq\in\scrL_\psi}\cdot f(\bn)
=\min\biggl(k+1,\sum_{\vecp\in(\scrL_\psi-\vecq)R(\vecv)D_\rho}f(\vecp)\biggr)
-\delta_{\vecq\in\scrL_\psi}\cdot f(\bn)
\\
\geq\min\biggl(k,\sum_{\vecp\in(\scrL_\psi-\vecq)R(\vecv)D_\rho\setminus\{\bn\}}f(\vecp)\biggr).
\end{align*}
Recalling also that $\cup_{\psi\in\Psi}\scrL_\psi=\scrP$,
it follows that
\begin{align*}
\int_{\US}\min\biggl(k,\sum_{\vecp\in(\scrP-\vecq)R(\vecv)D_\rho\setminus\{\bn\}}f(\vecp)\biggr)\,d\lambda(\vecv)
<\nbar_{\scrP}V
\end{align*}
for all $\rho\in(0,\rho_0)$ and $\vecq\in\scrP_T(\rho)$.
Here, for every 
$\vecv$ such that
$\#((\scrP-\vecq)R(\vecv)D_\rho\cap B\setminus\{\bn\})\geq k$,
the integrand equals $k$.
Hence we obtain the statement of the lemma.
\end{proof}

\begin{lem}\label{P2condthmLEM5}
Let $\scrZ$ be any subset of $\scrP$ of  density zero,
and let $T\geq1$, $\lambda\in\Pac(\US)$ and $S>0$.
Then 
\begin{align}\label{P2condthmLEM5res}
\lambda(\{\vecv\in\US\col\scrZ\cap(\vecq+\scrB_S^d D_\rho^{-1}R(\vecv)^{-1})\setminus\{\vecq\}\neq\emptyset\})\to0
\end{align}
as $\rho\to0$, uniformly over all $\vecq\in\scrP_T(\rho)$.
\end{lem}
\begin{proof}
(This is similar to the proof of 
\cite[Lemma 2.17]{jMaS2019}.)
Let $\lambda_1$ be the normalized Lebesgue measure on $\US$.
By a standard approximation argument,
using the fact that $C_c(\US)$ is dense in $L^1(\US)$,
it suffices to prove \eqref{P2condthmLEM5res} for those $\lambda$
which have a continuous density with respect to $\lambda_1$;
and thus in fact it suffices to prove \eqref{P2condthmLEM5res} for the single case
$\lambda=\lambda_1$.

Let $T\geq1$, $S>0$ and $\ve>0$ be given.
Take $0<r<S$ so small that $\nbar_{\scrP}\vol(\scrB_r^d)<\ve$.
Set $k=2S/r>2$ and $T'=k^{d-1}T$.
By Lemma \ref{P2condthmLEM6},
there exists $\rho_0\in(0,1)$
such that
\begin{align}\label{P2condthmLEM5pf1}
\lambda_1(\{\vecv\in\US\col(\scrP-\vecq)R(\vecv)D_\rho\cap\scrB_r^d\setminus\{\bn\}\neq\emptyset\})
<\ve
\end{align}
for all $\rho\in(0,\rho_0)$ and $\vecq\in\scrP_{T'}(\rho)$.
Set $\tB:=\scrB_r^dD_k^{-1}$.
Replacing $\rho$ by $k\rho$ in \eqref{P2condthmLEM5pf1},
it follows that
for all $\rho\in(0,\rho_0/k)$ and $\vecq\in\scrP_{T'}(k\rho)=\scrP_T(\rho)$ 
we have
\begin{align}\label{P2condthmLEM5pf5}
\lambda_1(\{\vecv\in\US\col(\scrP-\vecq)R(\vecv)D_\rho\cap\tB\setminus\{\bn\}\neq\emptyset\})
<\ve.
\end{align}
One verifies that $|x_1|\geq k_1:=(r/2)^dS^{1-d}$
for all $\vecx\in\scrB_S^d\setminus\tB$,
and hence 
\begin{align}\label{P2condthmLEM5pf5a}
(\scrB_S^d\setminus\tB){D}_\rho^{-1}\subset A(\rho):=\scrB_{S\rho^{1-d}}^d\setminus\scrB_{k_1\rho^{1-d}}^d,
\qquad\forall \rho>0.
\end{align}
Now for any $\rho\in(0,\rho_0/k)$ and $\vecq\in\scrP_{T}(\rho)$ we have,
using \eqref{P2condthmLEM5pf5} and \eqref{P2condthmLEM5pf5a}:
\begin{align}\label{P2condthmLEM5pf3}
\lambda_1(\{\vecv\in\US\col (\scrZ-\vecq)R(\vecv)D_\rho\cap\scrB_S^d\setminus\{\bn\}\neq\emptyset\})
\hspace{150pt}
\\\notag
<\ve+\sum_{\vecp\in \scrZ\cap(\vecq+A(\rho))}\lambda_1\bigl(\bigl\{
\vecv\in\US\col (\vecp-\vecq)R(\vecv)D_\rho\in \scrB_S^d\bigr\}\bigr).
\end{align}
But if
$(\vecp-\vecq)R(\vecv)D_\rho\in \scrB_S^d$,
or equivalently
$\vecp\in\vecq+\scrB_S^d{D}_\rho^{-1} R(\vecv)^{-1}$,
then $\vecp$ has a distance less than $S\rho$ to  the line $\vecq+\R\vecv$;
and if also $\vecp\in\vecq+A(\rho)$ then
the angle $\varphi(\vecv,\vecp-\vecq)$ between the vectors $\vecv$ and $\vecp-\vecq$
satisfies $\sin\varphi(\vecv,\vecp-\vecq)<(S/k_1)\rho^d$.
The measure of the set of such points $\vecv\in\US$ with respect to $\lambda_1$
is
bounded above by $C_1\rho^{d(d-1)}$,
where $C_1$ depends on $d,S,r$ but not on $\rho$ or $\vecp$.
Hence \eqref{P2condthmLEM5pf3} is
\begin{align*}
\leq\ve+\#(\scrZ\cap(\vecq+\scrB_{S\rho^{1-d}}^d))\cdot C_1\rho^{d(d-1)}
\leq\ve+\#(\scrZ\cap\scrB_{(T+S)\rho^{1-d}}^d)\cdot C_1\rho^{d(d-1)},
\end{align*}
and since $\scrZ$ has  density zero,
the last term is less than $\ve$ for $\rho$ sufficiently small.
\end{proof}

Recall that 
$\mu_{\vecq,\rho}^{(\lambda)}$ is the distribution of $\scrQ_\rho(\vecq,\vecv)$
for $\vecv$ random in $(\US,\lambda)$.
We now introduce a certain approximation
$\scrQ'_\rho(\vecq,\vecv)$ to $\scrQ_\rho(\vecq,\vecv)$,
which will be easier to handle.
We set
\begin{align}\label{tPpdef}
\tP'=\bigcup_{\psi\in\Psi}\{(\vecp,\sigma^\psi)\col\vecp\in\scrL_\psi\}.
\end{align}
Note that, unlike the projection $\tP\to\scrP$, the projection $\tP'\to\scrP$ 
is not necessarily injective!
(However, by Remark \ref{psiaedef}, it becomes injective after removing a set of density zero from $\tP'$.)
For any $\vecq\in\scrP$, we set
\begin{align}\label{tPpqDEF}
\tP'_{\vecq}=
\begin{cases}
\tP'\setminus\{(\vecq,\sigma^{\psi(\vecq)})\} & (\vecq\in\scrP)
\\
\tP' & (\vecq\notin\scrP) 
\end{cases}
\end{align}
and %
\begin{align}\label{QprhoqvDEF}
\scrQ'_\rho(\vecq,\vecv)=(\tP_{\vecq}'-\vecq)R(\vecv)D_\rho.
\end{align}

\begin{lem}\label{QrhopeqJpsiqofLEM}
For every $\vecq\in\scrP$ 
we have 
$\scrQ_\rho'(\vecq,\vecv)=J_{\psi(\vecq)}(\Gamma g_0^{(\vecq)}\varphi(R(\vecv)D_\rho))$.
\end{lem}
\begin{proof}
By parsing the definitions
\eqref{tPpqDEF},
\eqref{QprhoqvDEF}
and \eqref{Jpsidef},
we see that it suffices to prove 
\begin{align}\label{QrhopeqJpsiqofLEMpf1}
\tP'-\vecq=J(\Gamma g_0^{(\vecq)}).
\end{align}
However, comparing 
\eqref{JmapFULL}
and
\eqref{tPpdef},
we see that \eqref{QrhopeqJpsiqofLEMpf1}
is an immediate consequence of
\eqref{KEYUjqdefmotivation}
(just as is \eqref{J0UqtMeqPmq}).
\end{proof}

The following lemma shows that 
$\scrQ'_\rho(\vecq,\vecv)$ approximates
$\scrQ_\rho(\vecq,\vecv)$ in a sense that is appropriate for us.
\begin{lem}\label{P2condthmLEM1}
Let $f\in\C_c(\scrX)$,
$T\geq1$, $\lambda\in\Pac(\US)$
and $\ve>0$.
Then
\begin{align}\label{P2condthmLEM1res}
\lambda\biggl(\biggl\{\vecv\in\US\col 
\biggl|\sum_{x\in \scrQ_\rho(\vecq,\vecv)}f(x)-\sum_{x\in \scrQ'_\rho(\vecq,\vecv)}f(x)\biggr|>\ve\biggr\}\biggr)\to0
\end{align}
as $\rho\to0$,
uniformly over all $\vecq\in\scrP_T(\rho)$.
\end{lem}
\begin{proof}
Choose $S>0$ so that $\supp(f)\subset\scrB_S^d\times\Sigma$.
Note that for every $\vecq\in\scrP\setminus\scrE$ and every $\vecv\in\US$ we have,
using the fact that $\vecq\notin\scrL_{\psi'}$ for all $\psi'\neq\psi(\vecq)$
(which follows from $\vecq\notin\scrE$ and \eqref{scrEass}):
\begin{align}\notag
\sum_{x\in \scrQ_\rho(\vecq,\vecv)}f(x)- &\sum_{x\in \scrQ'_\rho(\vecq,\vecv)}f(x)
\\\label{P2condthmLEM1pf1}
&=\sum_{\vecp\in\scrP\setminus\{\vecq\}} \biggl(f\bigl((\vecp-\vecq)R(\vecv)D_\rho,\vs(\vecp)\bigr)
-\sum_{\substack{\psi\in\Psi\\(\vecp\in\scrL_\psi)}}f\bigl((\vecp-\vecq)R(\vecv)D_\rho,\sigma^{\psi}\bigr)\biggr).
\end{align}
Set 
\begin{align*}
A(\vecq,\vecv,\rho):=\scrP\cap(\vecq+\scrB_S^dD_\rho^{-1}R(\vecv)^{-1})\setminus\{\vecq\}.
\end{align*}
Note that for every $\vecp\in\scrP\setminus A(\vecq,\vecv,\rho)$
we have $(\vecp-\vecq)R(\vecv)D_\rho\notin\scrB_S^d$,
so that the corresponding term in \eqref{P2condthmLEM1pf1} vanishes.
Also for every $\vecp\notin\scrE$ we have
$\vecp\notin\scrL_{\psi'}$ for all $\psi'\neq\psi(\vecp)$
(cf.\ \eqref{scrEass}),
which implies that the corresponding term in \eqref{P2condthmLEM1pf1}
is bounded in absolute value by $d(\vs(\vecp))$,
where the function $d:\Sigma\to\R_{\geq0}$ is defined 
by
\begin{align}\label{P2condthmLEM1pf5}
d(\psi,\omega)=\sup\{|f(\vecx,(\psi,\omega))-f(\vecx,\sigma^{\psi})|\col\vecx\in\R^d\}
\qquad(\psi,\omega)\in\Sigma.
\end{align}
Hence %
for every $\vecq\in\scrP\setminus\scrE$ and $\vecv\in\US$
such that $\scrE\cap A(\vecq,\vecv,\rho)=\emptyset$, we have
\begin{align}\label{P2condthmLEM1pf4}
\biggl|\sum_{x\in \scrQ_\rho(\vecq,\vecv)}f(x)- &\sum_{x\in \scrQ'_\rho(\vecq,\vecv)}f(x)\biggr|
\leq\sum_{\vecp\in A(\vecq,\vecv,\rho)}d(\vs(\vecp)).
\end{align}

Now let $\ve'>0$ be given.
Take $K\in\Z^+$ and $\rho_0\in(0,1)$ such that
\begin{align}\label{P2condthmLEM1pf2}
\lambda(\{\vecv\in\US\col \#((\scrP-\vecq)R(\vecv)D_\rho\cap \scrB_S^d)>K\})<\ve'
\end{align}
for all $\rho\in(0,\rho_0)$ and all $\vecq\in\scrP_T(\rho)$.
This is possible by Lemma \ref{P2condthmLEM6}.
Next set
\begin{align*}
\scrZ:=\{\vecp\in\scrP\col d(\vs(\vecp))\geq\ve/K\}.
\end{align*}
We claim that the set $\scrZ$ has  density zero.
To prove this,
set
\begin{align*}
U_\psi:=\{\omega\in\Omega\col (\psi,\omega)\in\Sigma\text{ and }d(\psi,\omega)<\ve/K\},
\end{align*}
so that $\scrZ\subset\cup_{\psi\in\Psi}\{\vecq\in\scrL_\psi\col \omega^{(\vecq)}\notin U_\psi\}$.
Note that $\{\psi\}\times U_\psi$ is an open neighbourhood of $\sigma^\psi$ in $\Sigma$,
since the function $d$ is continuous.
Hence as in the proof of Proposition \ref{P1prop}
(making crucial use of Proposition \ref{muconvLEM}),
the set $\{\vecq\in\scrL_\psi\col\omega^{(\vecq)}\notin U_\psi\}$ has density zero.
Hence also $\scrZ$ has density zero, as claimed.

It follows that also $\scrE\cup\scrZ$ has  density zero,
and so by Lemma \ref{P2condthmLEM5},
after possibly shrinking $\rho_0$, we have
\begin{align}\label{P2condthmLEM1pf3}
\lambda(\{\vecv\in\US\col(\scrE\cup\scrZ)\cap   A(\vecq,\vecv,\rho) %
\neq\emptyset\})<\ve'
\end{align}
for all $\rho\in(0,\rho_0)$ and all $\vecq\in\scrP_T(\rho)$.
Now note that for any $\vecq\in\scrP\setminus\scrE$ and for any $\vecv\in\US$ which belongs to neither of the two sets in
\eqref{P2condthmLEM1pf2} and \eqref{P2condthmLEM1pf3},
the set $A(\vecq,\vecv,\rho)$
has cardinality at most $K$ and is disjoint from $\scrE\cup\scrZ$;
therefore the inequality in \eqref{P2condthmLEM1pf4} holds,
and the right hand side in that inequality is $\leq\ve$.
It follows that for all
$\rho\in(0,\rho_0)$ and $\vecq\in\scrP_T(\rho)$,
the measure in \eqref{P2condthmLEM1res} is less than $2\ve'$.
\end{proof}

\begin{proof}[Proof of Theorem \ref{P2condthm}]
Let $T\geq1$ and $\lambda\in\Pac(\US)$ be given.
Let $\rho_1,\rho_2\ldots$ be an arbitrary sequence in $(0,1)$ with $\rho_n\to0$,
and let $\vecq_n\in\scrP_T(\rho_n)$
for $n=1,2,\ldots$ 
be such that 
the limit $\vs=(\psi,\omega):=\lim_{n\to\infty}\vs(\vecq_n)\in\Sigma$ exists.
By \cite[Lemma 2.1.2]{jMaS2019},
it suffices to prove that 
in this situation
we have
\begin{align}\label{P2condthmNEWpf1}
\mu_{\vecq_n,\rho_n}^{(\lambda)\hspace{10pt}}\xrightarrow[]{\textup{ w }}\mu_{\vs}\qquad\text{as }\: n\to\infty.
\end{align}
Since $\vs(\vecq_n)\to(\psi,\omega)$
implies that  
$\psi(\vecq_n)=\psi$
for all large $n$,
we may without loss of generality assume that
$\psi(\vecq_n)=\psi$ for \textit{all} $n$.
This means that $\vecq_n\in\scrL_{\psi}$ for all $n$.

For any $\vecq\in\scrP$,  
$\rho>0$ and $\lambda\in P(\US)$,
let
$\nu_{\vecq,\rho}^{(\lambda)}\in P(\XX)$
be the distribution of $\Gamma g_0^{(\vecq)}\varphi(R(\vecv)D_\rho)$
for $\vecv$ random in $(\US,\lambda)$.
As a first step, let us note that 
\eqref{HOMDYNMAINTHMresREP}
implies that
\begin{align}\label{P2condthmNEWpf3}
\nu_{\vecq_n,\rho_n}^{(\lambda)}
\xrightarrow[]{\textup{ w }}\overline{\omega}\qquad\text{as }\: n\to\infty.
\end{align}
Indeed, let $f\in\C_b(\XX)$ be given.
Then by \eqref{HOMDYNMAINTHMresREP}
we have
$\nu_{\vecq_n,\rho_n}^{(\lambda)}(f)-\overline{\omega^{(\vecq_n)}}(f)\to0$ as $n\to\infty$.
Also $\vs(\vecq_n)\to(\psi,\omega)$ 
implies that $\omega_j^{(\vecq_n)}\xrightarrow[]{\textup{ w }}\omega_{j}$ in $P(\TT_j^d)'$
for each $j\in\{1,\ldots,N\}$;
hence by Lemma~\ref{mapPTtoPXjcontLEM}
and \cite[Thm.\ 2.8(ii)]{pB99},
we have 
$\overline{\omega^{(\vecq_n)}}\xrightarrow[]{\textup{ w }}\overline{\omega}$
in $P(\XX)$,
and thus 
$\overline{\omega^{(\vecq_n)}}(f)\to\overline{\omega}(f)$.
Hence
$\nu_{\vecq_n,\rho_n}^{(\lambda)}(f)\to\overline{\omega}(f)$,
and \eqref{P2condthmNEWpf3} is proved.

Next, for each $n$ we have
$\Gamma g_0^{(\vecq_n)}\in\XX^\psi$
by Lemma \ref{g0qinXpsiLEM},
and hence
$\nu_{\vecq_n,\rho_n}^{(\lambda)}(\XX^{\psi})=1$;
also 
$\overline{\omega}(\XX^{\psi})=1$
since $(\psi,\omega)\in\Sigma$;
cf.\ \eqref{oomegaXpsieq1}.
Hence all $\nu_{\vecq_n,\rho_n}^{(\lambda)}$ as well as $\oomega$
may be regarded as elements in $P(\XX^{\psi})$,
and \eqref{P2condthmNEWpf3}
implies that
$\nu_{\vecq_n,\rho_n}^{(\lambda)}
\xrightarrow[]{\textup{ w }}\overline{\omega}$
also in $P(\XX^{\psi})$
\cite[Lemma 4.26]{oK2002}.
Hence by the continuous mapping theorem,
\begin{align}\label{P2condthmNEWpf2}
J_{\psi\, *\,}\nu_{\vecq_n,\rho_n}^{(\lambda)}
\xrightarrow[]{\textup{ w }}J_{\psi\,*\,}\overline{\omega}\qquad\text{as }\: n\to\infty.
\end{align}
Here
$J_{\psi\,*\,}\overline{\omega}=\mu_{\vs}$,
by \eqref{muvsDEF}.
Now let $f\in\C_c(\scrX)$ be given,
and let $\pi_f$ be the continuous map from $N_s(\scrX)$ to $\R$ given by
$\pi_f(Q)=\sum_{x\in Q}f(x)$.
Then \eqref{P2condthmNEWpf2} implies that
\begin{align}\label{P2condthmLEM2pf4}
\pi_{f*\,}J_{\psi\, *\,}\nu_{\vecq_n,\rho_n}^{(\lambda)}
\xrightarrow[]{\textup{ w }}
\pi_{f*\,}\mu_{\vs}\qquad\text{as }\: n\to\infty.
\end{align}
But note that for each $\vecq\in\scrP$,
by Lemma \ref{QrhopeqJpsiqofLEM},
$J_{\psi(\vecq)\hspace{1pt}*\,}\nu_{\vecq,\rho}^{(\lambda)}$
is the distribution of $\scrQ_\rho'(\vecq,\vecv)$ 
in $N_s(\scrX)$
for $\vecv$ random in $(\US,\lambda)$.
Hence
$\pi_{f*\,}J_{\psi\, *\,}\nu_{\vecq_n,\rho_n}^{(\lambda)}$
is the distribution of the real-valued random variable
$F'(\vecv)=\sum_{x\in\scrQ'_{\rho_n}(\vecq_n,\vecv)}f(x)$,
for $\vecv$ random in $(\US,\lambda)$.
Similarly,
$\pi_{f*\,}\mu_{\vecq_n,\rho_n}^{(\lambda)}$
is the distribution of the real-valued random variable
$F(\vecv)=\sum_{x\in\scrQ_{\rho_n}(\vecq_n,\vecv)}f(x)$.
By Lemma~\ref{P2condthmLEM1},
$|F(\vecv)-F'(\vecv)|$ converges in probability to $0$.\label{P2condthmendofpf}
Hence by \cite[Thm.\ 3.1]{pB99},
\eqref{P2condthmLEM2pf4} implies that
\begin{align}\label{P2condthmLEM2pf10}
\pi_{f*\,}\mu_{\vecq_n,\rho_n}^{(\lambda)}
\xrightarrow[]{\textup{ w }}
\pi_{f*\,}\mu_{\vs}\qquad\text{as }\: n\to\infty.
\end{align}
We have proved that this holds for any given
$f\in\C_c(\scrX)$.
By \cite[Thm.\ 16.16(ii)$\Rightarrow$(i)]{oK2002},
this implies that \eqref{P2condthmNEWpf1} holds.
\end{proof}

\subsection{Verification of [P3], and the macroscopic limit}
\label{P3proofSEC}

\begin{prop}\label{propP3}
[P3] holds, i.e.\ for every bounded Borel set $B\subset\RR^d$ we have
\begin{align}\label{propP3res}
\lim_{\xi\to\infty}\limsup_{\rho\to0}\hspace{7pt}
[\vol\times\sigma]\bigl(\bigl\{(\vecq,\vecv)\in B\times\US\col
\scrQ_{\rho}(\rho^{1-d}\vecq,\vecv)\cap(\fZ_\xi\times\Sigma)=\emptyset\bigr\}\bigr)=0.
\end{align}
\end{prop}
\begin{proof}
Using the fact that $\scrP$ contains at least one grid,
the proposition follows from %
the existence of a
limit distribution 
on $\R_{>0}$ (with zero mass at $+\infty$)
for the macroscopic free path length
in the Boltzmann-Grad limit of
the Lorentz gas on a lattice scatterer configuration,
\cite[Theorem 1.2]{jMaS2010a}.
Indeed, fix an arbitrary $\psi\in\Psi$.
For any $\rho>0$ and $\vecq\in\R^d$, $\vecv\in\US$,
set
\begin{align*}
Q_\rho'(\vecq,\vecv)=(\scrL_\psi-\vecq)R(\vecv)D_\rho \qquad(\subset\R^d).
\end{align*}
Comparing with the definition of
$\scrQ_\rho(\vecq,\vecv)$ in
\eqref{repXIRHOqv},
\eqref{repPqdef},
and using $\scrL_\psi\subset\scrP$ and $\bn\notin\fZ_\xi$,
one verifies that for any $\xi>0$,
$\scrQ_{\rho}(\vecq,\vecv)\cap(\fZ_\xi\times\Sigma)=\emptyset$
forces $Q_\rho'(\vecq,\vecv)\cap\fZ_\xi=\emptyset$.
Hence to prove the proposition it suffices to prove that
\begin{align}\label{propP3pf1}
\lim_{\xi\to\infty}\limsup_{\rho\to0}\hspace{7pt}
[\vol\times\sigma]\bigl(\bigl\{(\vecq,\vecv)\in B\times\US\col
Q'_{\rho}(\rho^{1-d}\vecq,\vecv)\cap\fZ_\xi=\emptyset\bigr\}\bigr)=0.
\end{align}
By a simple translation and rescaling argument we may reduce to the case when $\scrL_\psi$ has covolume 
one and is a lattice,
and then \eqref{propP3pf1} is a simple consequence of 
\cite[Theorem 1.2]{jMaS2010a}.
\end{proof}

In \cite[Sec.\ 2.5]{jMaS2019},
the condition [P3] is used to prove,
for an arbitrary fixed point set $\scrP\subset\R^d$
satisfying the hypotheses in Section \ref{KINTHEORYrecapsec},
the existence of a 
canonical measure $\mu^{\g}\in P(N_s(\scrX))$ \footnote{In \cite{jMaS2019}\label{mugfootnote}
this measure is called ``$\mu$''.}
giving the limit distribution of $\scrQ_\rho(\vecq,\vecv)$ in the case of a 
\textit{macroscopic} initial condition.
That is, $\mu^{\g}$
equals the limit distribution
of $\scrQ_\rho(\rho^{1-d}\vecq,\vecv)$ for $(\vecq,\vecv)$ random in 
$(\T^1(\R^d),\Lambda)$,
for any fixed probability measure $\Lambda\in P(\T^1(\R^d))$
absolutely continuous with respect to the Liouville measure $\vol\times\sigma$
\cite[Theorem 2.19]{jMaS2019}.
This measure $\mu^{\g}$ also appears in the definition of
the transition kernel for generic initial data, $k^{\g}$;
see Section \ref{transkerdefSEC} below. %

In our case of 
$\scrP$ being a finite union of grids as in
\eqref{GENPOINTSET1},
the macroscopic limit measure $\mu^{\g}$
can be explicitly defined as follows: Set
\begin{align}\label{omegagdef}
\omega^{\g}:=\bigl(\omega_1^{\g},\ldots,\omega_N^{\g}\bigr)\in\Omega,
\end{align}
and then let
\begin{align}\label{mugdef}
\mu^{\g}=J_*(\overline{\omega^{\g}}),
\end{align}
with $J:\XX\to N_s(\scrX)$ being the map in \eqref{JmapFULL}.

We next state without proof a limit result
which significantly strengthens the above mentioned
\cite[Theorem 2.19]{jMaS2019}
for our special class of $\scrP$.
For any $\Lambda\in P(\T^1(\R^d))$, $s>0$ and $\rho\in(0,1)$,
let $\mu_\rho^{(\Lambda,s)}$ be the distribution of
$\scrQ_\rho(s\vecq,\vecv)$ for $(\vecq,\vecv)$ random in $(\T^1(\R^d),\Lambda)$.
\begin{thm}\label{MACROSCOPICNsTHM}
For any $\Lambda\in P(\T^1(\R^d))$
which is absolutely continuous with respect to Liouville measure,
and any $s_0>0$,
we have $\mu_\rho^{(\Lambda,s)}\xrightarrow[]{\textup{ w }}\mu^{\g}$
as $\rho\to0$, uniformly over all $s\geq s_0$.
\end{thm}

Note that \cite[Theorem 2.19]{jMaS2019}
corresponds to the particular choice $s=\rho^{1-d}$ in Theorem~\ref{MACROSCOPICNsTHM}.
The formulation of Theorem \ref{MACROSCOPICNsTHM} is inspired by
\cite[Theorem 1.1]{jMaS2013}.

As mentioned, we will not give the %
proof of Theorem \ref{MACROSCOPICNsTHM} in the present paper.
However we remark that Theorem \ref{MACROSCOPICNsTHM}
can be deduced,
by similar arguments as in Section \ref{P2initialdiscsec},
from the following equidistribution result in the homogeneous space $\XX$,
which is a kind of 
macroscopic analogue of Theorem \ref{HOMDYNintrononunifTHM}.
\begin{thm}\label{HOMDYNmacroscopicTHM}
For any $\Lambda\in P(\T^1(\R^d))$
which is absolutely continuous with respect to Liouville measure,
and any $f\in\C_b(\XX)$ and $s_0>0$,
we have
\begin{align}\label{HOMDYNmacroscopicTHMres}
\int_{\T^1(\R^d)}f\bigl(\Gamma g_0^{(s\vecq)}\varphi(R(\vecv)D_\rho)\bigr)\,d\Lambda(\vecq,\vecv)
\to\int_{\XX}f\,d\overline{\omega^{\g}}
\end{align}
as $\rho\to0$, uniformly over all $s\geq s_0$.
\end{thm}
We will not give the proof of Theorem \ref{HOMDYNmacroscopicTHM} either;
however we note that it
is to a large extent similar to the proof of
Theorem \ref{HOMDYNMAINTHM}
which we give in Section \ref{NEWISOsec} below.

\vspace{5pt}

We end this section by pointing out a couple of invariance properties of $\mu^{\g}$.
By the same argument as in the proof of Lemma \ref{muvsSLdinvLEM},
the measure $\overline{\omega^{\g}}$ on $\XX$ is $\varphi(\SL_d(\R))$-invariant,
and thus $\mu^{\g}$ is invariant under the action of $\SL_d(\R)$ on $N_s(\scrX)$.
We also have:
\begin{lem}\label{mugtranslinvLEM}
The measure $\mu^{\g}$ is invariant under the action of $\R^d$ on $N_s(\scrX)$ by translations.
\end{lem}
\begin{proof}
This is part of 
\cite[Prop.\ 2.24]{jMaS2019};
however let us note that it also follows directly from 
the explicit definition in \eqref{mugdef}.
Indeed, it follows from \eqref{trueLjDEF} that $\R\tvecc_j\subset L_j$;
hence $\tvecc_j\vecv\in L_j^d$ for all $\vecv\in\R^d$
(where $\tvecc_j\vecv$ is the matrix product of $\tvecc_j\in\M_{r_j\times 1}(\R)$
and $\vecv\in\M_{1\times d}(\R)$),
and so $\omega_j^{\g}$ is invariant under the translation
$X\mapsto X+\tvecc_j\vecv$ on $\TT_j^d$, for every $\vecv\in\R^d$.
This implies that the measure
$\widetilde{\overline{\omega_j^{\g}}}$ on $G_j$
(see \eqref{toomegaDEF}) is invariant under
$g\mapsto g\I_{\tvecc_j\vecv}$ for every $\vecv\in\R^d$,
and hence $\overline{\omega_j^{\g}}$ is invariant under
the translation $x\mapsto x\I_{\tvecc_j\vecv}$ on $\XX_j$.
Now the lemma follows from the definition of
$\mu^{\g}$ in \eqref{mugdef}
by using the formula
\begin{align*}
J\bigl(x_1\I_{\tvecc_1\vecv},\ldots,x_N\I_{\tvecc_N\vecv}\bigr)=J(x_1,\ldots,x_N)+\vecv
\qquad(\forall (x_1,\ldots,x_N)\in\XX,\: \vecv\in\R^d),
\end{align*}
which is immediate from
\eqref{cjdef} and \eqref{JmapFULL}.
\end{proof}

\section{Application of the classification of invariant measures of unipotent flows} %
\label{UNIPOTAPPLsec}

In this section we will state and prove a result,
Theorem \ref{MAINcleanUNIFCONVthm},
on the equidistribution of certain expanding unipotent orbits
in a slightly
generalized version of the homogeneous space
$\XX$ introduced in Section \ref{LIEGPHOMSPsec}.
This theorem is tailor-made to serve as 
the main ingredient in the proof of Theorem \ref{HOMDYNMAINTHM}
which we give in Section \ref{NEWISOsec} below;
in particular, it will be crucial for us to have %
a certain uniformity with respect to the position of the initial point in the torus fiber
(that is, uniformity with respect to the variable ``$V$'' in Theorem \ref{MAINcleanUNIFCONVthm} below).
The proof of Theorem \ref{MAINcleanUNIFCONVthm}
builds on Ratner's classification of ergodic measures invariant under unipotent flows
\cite{mR91a}
and further characterization results by Mozes and Shah
\cite{sM95}.
We also remark that if it were not for the uniformity requirement,
Theorem \ref{MAINcleanUNIFCONVthm}
could be deduced as a consequence of
Shah \cite[Theorem 1.4]{nS96}.

We start by introducing some notation.
We stress that in this %
Section \ref{UNIPOTAPPLsec},
some of our notation
(for example, ``$\XX$'', ``$G$'', ``$\Gamma\,$'', ``$\Gamma_j$'' and ``$\tM$'')
will be used in a slightly different and more general way
than in all the other sections of the paper.
The reason is that the results of the present section
will be applied, in Section~\ref{NEWISOsec},
to certain %
homogeneous \textit{submanifolds}
of our original space ``$\XX$'' %
(see the proofs of Theorems \ref{nonunifTHM1} and \ref{MAINUNIFPROP4new}).
To start with, similarly as before,
we set
\begin{align*}
G=G_1\times\cdots\times G_N=\S_{r_1}(\R)\times\cdots\times\S_{r_N}(\R);
\end{align*}
however now we allow $r_1,\ldots,r_N$ to be arbitrary (fixed)
\textit{non-negative} integers.
That is, unlike all the other sections,
we allow one or several of the $r_j$s to be \textit{zero,}
with the natural convention that $\S_0(\R):=\SL_d(\R)$.
Next, we fix $\Gamma_1',\ldots,\Gamma_N'$ to be arbitrary, fixed, finite index subgroups %
of $\SL_d(\Z)$,
and set
\begin{align}\label{GammajdefinUNIPOTAPPLsec}
\Gamma_j=\Gamma_j'\ltimes\M_{r_j\times d}(\Z)
=\bigl\{(M,U)\in\S_{r_j}(\Z)\col M\in\Gamma_j'\bigr\}
\qquad (j=1,\ldots,N)
\end{align}
(if $r_j=0$, this should be understood as $\Gamma_j=\Gamma_j'$),
and
\begin{align*}
\Gamma=\Gamma_1\times\cdots\times\Gamma_N.
\end{align*}
Then, as before, we set $\XX_j:=\Gamma_j\bs G_j$ and 
\begin{align*}
\XX:=\GaG=\XX_1\times\cdots\times\XX_N,
\end{align*}
and write $p_j:G\to\S_{r_j}(\R)$ and 
$\tp_j:\XX\to\XX_j$ ($j=1,\ldots,N$) for the projection maps.

Recall that we consider $\SL_d(\R)$ to be an embedded subgroup of each group $G_j$,
through $M\mapsto(M,0)$.
Now we also set $G':=\SL_d(\R)^N$; this is an embedded subgroup of $G$.\label{Gpdef}
We also set
\begin{align*}
\Gamma':=\Gamma_1'\times\cdots\times\Gamma_N'\subset G'
\end{align*}
and
\begin{align*}
\XX'=\Gamma'\bs G'.
\end{align*}
Recall that we have a projection morphism $\iota: G_j\to\SL_d(\R)$ for each $j$,
and $\iota$ induces a projection map
\begin{align}\label{tiotaDEFgen}
\tiota:\XX_j\to\Gamma_j'\bs\SL_d(\R),
\qquad\tiota(\Gamma_jg)=\Gamma_j'\,\iota(g)\quad (g\in G_j),
\end{align}
generalizing \eqref{tiotaDEF}.
In the present section we will also write $\iota$ for the product morphism from $G$ to $G'$,
and write $\tiota$ for the induced projection map from $\XX$ to $\XX'$.

As before we set
\begin{align*}
\TT_j:=\R^{r_j}/\Z^{r_j}
\qquad\text{and}\qquad
\TT_j^d:=\underbrace{\TT_j\times\cdots\times\TT_j}_{d\: \text{ copies}};
\end{align*}
if $r_j=0$ this should be understood as $\TT_j=\TT_j^d=\{\bn\}$, 
the trivial group.
The definition in \eqref{xmapDEF} of the embedding $x:\TT_j^d\to\XX_j$ 
carries over unchanged to our present setting,
although our ``$\XX_j$'' is now more general.
(If $r_j=0$ then we set $x(\{\bn\}):=\Gamma_j\in\XX_j$.)
We now also set
\begin{align}\label{tTTdef2}
\tTT:=\TT_1^d\times\TT_2^d\times\cdots\times\TT_N^d,
\end{align}
and let $\tp_j:\tTT\to\TT_j^d$ ($j=1,\ldots,N$) be the projection maps;\label{pjdef2}
and we will write ``$x$'' also for the map $\tTT\to\XX$ 
which is the product of the maps $x:\TT_j^d\to\XX_j$. \label{xmapDEF2}
The fact that both ``$x$'' and ``$\tp_j$'' now denote more than one map should not cause any confusion;
in particular note that
with this abuse of notation we have $x\circ\tp_j=\tp_j\circ x:\tTT\to\XX_j$ for each $j$.

We now come to the statement of the main result of the present section,
Theorem \ref{MAINcleanUNIFCONVthm}.
It concerns the equidistribution of pieces of expanding unipotent orbits in $\XX$ of the form
\begin{align}\label{MAINcleanUNIFCONVthmorbits}
\bigl\{x(V)\tM\,\varphi(n_-(\vecu)D_\rho)\col \vecu\in\R^{d-1}\bigr\},
\end{align}
where $V\in\tTT$;
$\tM$ is an arbitrary element in $G'$ not belonging to the subset
\begin{align}\label{DSdef}
\fD_{\scrS}:=
\bigcup_{i<j}\bigl\{(M_1,\ldots,M_N)\in G'\col M_iM_j^{-1}\in\scrS\bigr\},
\end{align}
with $\scrS$ as in \eqref{COMMENSURATOR};
$\varphi$ is the diagonal embedding of $\SL_d(\R)$ in $G$;
and finally
\begin{align}\label{nmuDEF}
n_-(\vecu):=\matr 1{\vecu}0{I_{d-1}}\in\SL_d(\R)
\end{align}
(block diagonal notation).
The equidistribution is with respect to the 
$G$-invariant probability measure on $\XX$,
which we call $\mu$.

In order for orbits of the form
\eqref{MAINcleanUNIFCONVthmorbits}
to equidistribute in $(\XX,\mu)$ as $\rho\to0$,
we have to assume that $V$ avoids a certain
'singular' subset $\Delta_k^{(\eta)}$ of $\tTT$,
which we now introduce.
For each $j\in\{1,\ldots,N\}$ with $r_j\neq0$, 
let us write $\pi_j$ for the projection from $(\R^{r_j})^d$ to $\TT_j^d$
(it was called ``$\pi$'' in \eqref{pitorusprojDEF}).
Then for any %
$q\in\Z^+$ and $\vecm\in\Z^{r_j}\setminus\{\bn\}$,
we set
\begin{align}\label{DELTAjqLclean}
\Delta_{j,q,\vecm}:=
\pi_j((q^{-1}\Z^{r_j}+\vecm^\perp)^d)
\subset\TT_j^d,
\end{align}
where $\vecm^\perp$ is the orthogonal complement of $\vecm$ in $\R^{r_j}$.
Also, for any $k\in\Z^+$, we set
\begin{align}\label{DeltajkDEF}
\Delta_{j,k}:=\bigcup_{q=1}^k\bigcup_{\substack{\vecm\in\Z^{r_j}\\ 0<\|\vecm\|\leq k}}\Delta_{j,q,\vecm}.
\end{align}
Note that $\Delta_{j,q,\vecm}$ and $\Delta_{j,k}$ are only defined when $r_j\neq0$,
in which case they are both closed regular submanifolds of $\TT_j^d$ of codimension $d$.
Next, for any $\eta>0$ we define $\Delta_{j,k}^{(\eta)}$ to be the open
$\eta$-neighbourhood of $\Delta_{j,k}$ in $\TT_j^d$,
with respect to the metric induced by the standard Euclidean metric on $\M_{r_j\times d}(\R)=(\R^{r_j})^d$.
Finally, we set
\begin{align}\label{scrDkveDEF1}
\Delta_k^{(\eta)}=\bigcup_{\substack{j=1\\(r_j\neq0)}}^N\tp_j^{\hspace{3pt}-1}\bigl(\Delta_{j,k}^{(\eta)}\bigr)\subset\tTT.
\end{align}

Let $\Pac(\R^{d-1})$ be the set of Borel probability measures on $\R^{d-1}$
which are absolutely continuous with respect to Lebesgue measure.
\begin{thm}\label{MAINcleanUNIFCONVthm}
Let $f\in\C_b(\XX)$
and $\ve>0$ be given.
Then there exists some $k\in\Z^+$ such that for every 
$\lambda\in\Pac(\R^{d-1})$, $\eta>0$ and 
$\tM\in G'\setminus\fD_{\scrS}$, there exists some $\rho_0\in(0,1)$ such that
\begin{align}\label{MAINcleanUNIFCONVthmRES}
\biggl|\int_{\R^{d-1}}f\Bigl(x(V)\tM\varphi(n_-(\vecu)D_\rho)\Bigr)\,d\lambda(\vecu)-
\int_{\XX} f \,d\mu\biggr|<\ve
\end{align}
for all $\rho\in(0,\rho_0)$ and all $V\in\tTT\setminus \Delta_k^{(\eta)}$.
\end{thm}

\vspace{5pt}

The rest of this section is devoted to the proof of Theorem \ref{MAINcleanUNIFCONVthm}.
\begin{definition}\label{Pkdef}
For each $k\in\Z^+$,
we let $P_k$ be the set of all measures $\nu\in P(\XX)$ which can be obtained
as a weak limit of a sequence of probability measures
$\nu_1,\nu_2,\ldots$ given by
\begin{align}\label{numDEFclean}
\nu_m:\quad f\mapsto \int_{\R^{d-1}}f\bigl(x(V_m)\tM\varphi(n_-(\vecu)D_{\rho_m})\bigr)\,d\lambda(\vecu)
\qquad (f\in\C_b(\XX)),
\end{align}
for some 
$\lambda\in\Pac(\R^{d-1})$,
$\tM\in G'\setminus\fD_{\scrS}$,
real numbers $\rho_1>\rho_2>\cdots\to0$,
and points $V_1,V_2,\ldots$ in $\tTT$
such that [$\exists\eta>0$: $\forall m\in\Z^+$: $V_m\notin\Delta_k^{(\eta)}$].
\end{definition}

Note that $P_1\supset P_2\supset\cdots$,
since $\Delta_k^{(\eta)}\subset\Delta_{k'}^{(\eta)}$ 
whenever $k<k'$. %

\vspace{5pt}

Throughout the rest of this section, we will let $W$ denote the following subgroup of $G$:
\begin{align*}
W:=\bigl\{\varphi(n_-(\vecw))\col\vecw\in\R^{d-1}\bigr\}.
\end{align*}
\begin{lem}\label{Winvlem}
Every $\nu\in P_k$ is $W$-invariant.
\end{lem}
\begin{proof}
This is a (very) standard consequence of the fact that 
for $\rho<1$,
the action (from the right) of $\varphi(D_\rho)$ on $\XX$ expands
any $W$-orbit.
The details are as follows.
Let $\nu\in P_k$ be given.
Then the task is to prove that for any given $\vecw\in\R^{d-1}$
and $f\in\C_b( \XX)$ we have $\nu(f\circ R_{\vecw})=\nu(f)$,
where $R_{\vecw}:\XX\to \XX$ denotes right multiplication by $\varphi(n_-(\vecw))$.
Choose $\lambda,\tM,(\rho_m),(V_m)$ as in
Definition \ref{Pkdef},
so that $\nu$ is the weak limit of the measures $\nu_m$ given by \eqref{numDEFclean}.
Using the relation $D_\rho \, n_-(\vecw)=n_-(\rho^d\vecw) \, D_\rho$,
and writing $\lambda'\in\L^1(\R^{d-1})$ for the density of $\lambda$ with respect to Lebesgue measure, %
we now have:  %
\begin{align*}
\nu_m(f\circ R_{\vecw}) %
=\int_{\R^{d-1}}f\Bigl(x(V_m)\tM \varphi(n_-(\vecu)D_{\rho_m})\Bigr)\,
\lambda'(\vecu-\rho_m^d\vecw)\,d\vecu.
\end{align*}
Hence $|\nu_m(f\circ R_{\vecw})-\nu_m(f)|\leq
\|f\|_{\L^\infty}\cdot \|\tau_{\rho_m^d\vecw}\lambda'-\lambda'\|_{\L^1(\R^{d-1})}$,
and so by \cite[Prop.\ 8.5]{gF99} we have
$\lim_{m\to\infty}\nu_m(f\circ R_{\vecw})=\lim_{m\to\infty}\nu_m(f)$,
that is, $\nu(f\circ R_{\vecw})=\nu(f)$.
\end{proof}
Recall that we write $\tiota$
for the natural projection map from $\XX$ to $\XX'$;
in particular $\tmu:=\tiota_*\,\mu$
is the unique $G'$-invariant probability measure on $\XX'$.
\begin{lem}\label{PkmeasrhoprojLEM}
Every $\nu\in P_k$ satisfies $\tiota_*\,\nu=\tmu$.
\end{lem}
\begin{proof}
Let $\nu\in P_k$ be given, and let $(\nu_m)$ be a sequence as in Definition \ref{Pkdef},
tending weakly to $\nu$.
For any $f\in\C_b(\XX')$ we have
\begin{align*}
\nu_m(f\circ\tiota)
=\int_{\R^{d-1}}f\bigl(\tiota\bigl(x(V_m)\tM \varphi(n_-(\vecu)D_{\rho_m})\bigr)\bigr)\,d\lambda(\vecu)
=\int_{\R^{d-1}}f\bigl(\Gamma'\tM \varphi(n_-(\vecu)D_{\rho_m})\bigr)\,d\lambda(\vecu).
\end{align*}
This integral tends to $\tmu(f)$ as $m\to\infty$,
by \cite[Thm.\ 5]{jMaS2013a}
applied to the function
$g\mapsto f(\Gamma'\tM g)$
(which is left $\prod_{j=1}^N (M_j^{-1}\Gamma_j'M_j)$--invariant).
On the other hand, by the definition of $\nu$ we have
$\nu_m(f\circ\tiota)\to\nu(f\circ\tiota)=(\tiota_*\,\nu)(f)$.  %
Hence $(\tiota_*\,\nu)(f)=\tmu(f)$.
\end{proof}

Recall that a subgroup $U$ of $G$ is said to be \textit{unipotent},
if the linear automorphism $\Ad(u)$ of the Lie algebra of $G$ is unipotent for all $u\in U$.
For any $h\in G$ let us write $R_h:\XX\to\XX$ for the map 
$\Gamma g\mapsto\Gamma gh$.
For any $\alpha\in P(\XX)$, let us define $H_\alpha$ 
to be the identity component of the subgroup of $G$ consisting of all $g$
which preserve $\alpha$;
\begin{align}\label{Halphadef}
H_\alpha:=\{g\in G\col R_{g*}\alpha=\alpha\}^\circ.
\end{align}
This is a closed connected Lie subgroup of $G$.
We let $\scrQ(\XX)$ be the set of all $\alpha\in P(\XX)$
such that the group generated by all unipotent one-parameter subgroups of $G$
contained in $H_\alpha$ acts ergodically on $\XX$ with respect to $\alpha$.
(Note that this definition of $\scrQ(\XX)$ is equivalent to the one in 
\cite[p.\ 150]{sM95},
although our $H_\alpha$ equals ``$\Lambda(\alpha)^\circ\,$'' in the notation of
\cite{sM95}.)

A key ingredient in our proof of Theorem \ref{MAINcleanUNIFCONVthm}
will be Ratner's classification of invariant measures of unipotent flows,
\cite[Thm.\ 1]{mR91a}.
Applied in our setting, this result says that for every $\alpha\in\scrQ(\XX)$,
there is some $g_\alpha\in G$ such that
$\alpha(\Gamma\bs\Gamma g_\alpha H_\alpha)=1$.
Note that in this situation, 
$\Gamma\cap g_\alpha H_\alpha g_\alpha^{-1}$
is a lattice in $g_\alpha H_\alpha g_\alpha^{-1}$,
and the support of $\alpha$ equals $\Gamma\bs\Gamma g_\alpha H_\alpha$,
which is a smooth embedded submanifold of
$\XX$.

\begin{lem}\label{iotaHeqGpLEM}
For any $\alpha\in\scrQ(\XX)$ such that $\tiota_*\,\alpha=\tmu$,
we have $\iota(H_\alpha)=G'$.
\end{lem}
\begin{proof}
(This generalizes \cite[Lemma 6]{cDjMaS2016},
and the proof is the same.)
Using the fact that the map $\iota:\XX\to\XX'$ has compact fibers,
we have $\tiota(\supp\alpha)=\supp\tiota_*\alpha=\supp\tmu=\XX'$.
But $\supp\alpha=\Gamma\bs\Gamma g_\alpha H_\alpha$.
Hence $\Gamma'\iota(g_\alpha)\iota(H_\alpha)=G'$,
and thus $\iota(H_\alpha)=G'$.
\end{proof}

Next, using basic Lie group and Lie algebra theory,
we will derive a completely explicit description of any 
Lie subgroup $H_\alpha$ as in Lemma \ref{iotaHeqGpLEM};
cf.\ Lemma \ref{charHsbgpsLEM} %
below.

For any $r\in\Z_{\geq0}$, let $\is_r(\R)$ be the Lie algebra of $\S_r(\R)$,
which we represent as the set of pairs $(A,X)\in\lsl_d(\R)\times \M_{r\times d}(\R)$,
with the Lie bracket given by 
\begin{align}\label{LIEBRACKETrep}
[(A_1,X_1),(A_2,X_2)]=([A_1,A_2],X_1A_2-X_2A_1).
\end{align}
(For $r=0$ we have $\is_0(\R)=\lsl_d(\R)$, and in \eqref{LIEBRACKETrep} 
we should view $\M_{0\times d}(\R)$ as a singleton set containing only the ``empty matrix''.)
Just as for the Lie groups,
we always consider $\lsl_d(\R)$ to be embedded in $\is_r(\R)$
through $A\mapsto(A,0)$.
We also set $\ig=\is_{r_1}(\R)\oplus\cdots\oplus\is_{r_N}(\R)$
and $\ig'=\lsl_d(\R)^N$;
these are the Lie algebras of $G$ and of $G'$, respectively.
Next, as in \eqref{SVdefrepnew},
given any linear subspace $L$ of $\R^r$, we let $\S_L(\R)$ be the closed connected subgroup of $\S_r(\R)$ given by
\begin{align}\label{SVdefrep}
\S_L(\R):=\SL_d(\R)\ltimes L^d=\bigl\{(M,U)\in\S_r(\R)\col %
U\in L^d\bigr\}
\end{align}
Recall here that 
via our identification
$\M_{r\times d}(\R)=(\R^r)^d$,
$L^d$ is the set of matrices in 
$\M_{r\times d}(\R)$ all of whose column vectors lie in $L$.
For any matrix $X\in\M_{r\times d}(\R)$,
we also write:
\begin{align}\label{SVXdefrep}
&\S_L^X(\R):=\I_X\S_L(\R)\,\I_X^{-1}.
\end{align}
(For $r=0$ we have $\R^r=\{\bn\}$, the only linear subspace $L\subset\R^0$ is $L=\R^0$,
and the only matrix in $\M_{0\times d}(\R)$ is $X=$ the empty matrix,
and for these $L,X$ we have $\S_L(\R)=\S_L^X(\R)=\SL_d(\R)$.)
We write $\is_L(\R)$ and $\is_L^X(\R)$ for the Lie subalgebras of $\is_r(\R)$
corresponding to $\S_L(\R)$ and $\S_L^X(\R)$, respectively.
Thus in particular,
\begin{align}\label{isLformula}
\is_L(\R)=\{(A,Y)\in\is_r(\R)\col Y\in L^d\}.
\end{align}
Recall that we write $\iota$ for the natural projection $\S_r(\R)\to\SL_d(\R)$;
hence $d\iota$ is the natural projection $\is_r(\R)\to\lsl_d(\R)$.
\begin{lem}\label{DMSlem7}
If $\ih$ is a Lie subalgebra of $\is_r(\R)$ satisfying $d\iota(\ih)=\lsl_d(\R)$,
then there exist a linear subspace $L\subset\R^r$ and a matrix $X\in \M_{r\times d}(\R)$
such that $\ih=\is_L^X(\R)$.
\end{lem}
\begin{proof}
This is a Lie algebra version of 
\cite[Lemma 7]{cDjMaS2016},
and essentially the same proof works.
Therefore we here give a rather terse presentation, referring to the proof in 
\cite{cDjMaS2016}
for further details.
Of course if $r=0$
then $d\iota(\ih)=\lsl_d(\R)$
implies $\ih=\lsl_d(\R)$ and the lemma is trivial;
hence in the following we may assume $r>0$.

Set $L'=\{X\in \M_{r\times d}(\R)\col (0,X)\in\ih\}$;
this is a linear subspace of $\M_{r\times d}(\R)$.
Using $d\iota(\ih)=\lsl_d(\R)$
it follows that $XA\in L'$ for all $X\in L'$ and all $A\in\lsl_d(\R)$;
and this in turn implies that $L'$ must be of the form
$L'=L^d$ for some linear subspace $L\subset\R^r$.
Let $L^\perp$ be the orthogonal complement of $L$ in $\R^r$;
then $\M_{r\times d}(\R)=L^d\oplus(L^\perp)^d$,
and for each $A\in\lsl_d(\R)$ there exists a unique $Y\in(L^\perp)^d$ such that $(A,Y)\in\ih$.
This implies that the Lie subalgebra $\is_{L^\perp}(\R)\cap\ih$ is a Levi subalgebra of $\is_{L^\perp}(\R)$,
and so by Malcev's Theorem
\cite[Ch.\ III.9]{nJ62},
there exists some $X\in(L^\perp)^d$ such that
$\is_{L^\perp}(\R)\cap\ih=\bigl(\Ad\I_X\bigr)(\lsl_d(\R))$.
But $\ih$ is the vector space direct sum of 
$\is_{L^\perp}(\R)\cap\ih$
and $\{(0,U)\col U\in L^d\}$;
hence in fact $\ih=\is_L^X(\R)$.
\end{proof}

\begin{lem}\label{charHsbgpsLEM}
Assume that $H$ is a connected Lie subgroup of $G$ satisfying $\iota(H)=G'$.
Then there exist linear subspaces $L_j\subset\R^{r_j}$ 
and matrices $X_j\in \M_{r_j\times d}(\R)$ %
such that
\begin{align}\label{charHsbgpsLEMres}
H=\S_{L_1}^{X_1}(\R)\times\cdots\times\S_{L_N}^{X_N}(\R).
\end{align}
\end{lem}
\begin{proof}
Let $\ih$ be the Lie subalgebra of
$\ig=\is_{r_1}(\R)\times\cdots\times\is_{r_N}(\R)$
corresponding to $H$.
It follows from $\iota(H)=G'$ that 
$d\iota(\ih)=\ig'$. %
Recall that $p_j:G\to G_j=\S_{r_j}(\R)$ denotes the projection onto the $j$th factor.
It follows from $d\iota(\ih)=\ig'$ that, for each $j$,
the Lie subalgebra $dp_j(\ih)$ of $\is_{r_j}(\R)$
satisfies $d\iota(dp_j(\ih))=\lsl_d(\R)$,
and so by Lemma \ref{DMSlem7} there exist
a linear subspace $L_j\subset\R^{r_j}$ and a matrix $X_j\in \M_{r_j\times d}(\R)$
such that 
\begin{align}\label{charHsbgpsLEMpf2}
dp_j(\ih)=\is_{L_j}^{X_j}(\R)
\qquad (\forall j\in\{1,\ldots,N\}).
\end{align}
This implies:
\begin{align}\label{charHsbgpsLEMpf3}
\ih\subset\is_{L_1}^{X_1}(\R)\times\cdots\times\is_{L_N}^{X_N}(\R).
\end{align}

We claim that the two sides of \eqref{charHsbgpsLEMpf3} are in fact equal.
In order to prove this equality,
it suffices to prove that $\varphi_j(\is_{L_j}^{X_j}(\R))\subset\ih$ for each $j$,
where
\begin{align*}
\varphi_j:\is_{r_j}(\R)\to\ig
\end{align*}
is the Lie group homomorphism mapping $X$ to $(0,\cdots,X,\cdots,0)$
($0$s in all positions except the $j$th).
Let $j$ be fixed, and set
\begin{align*}
\il=\{Z\in\is_{L_j}^{X_j}(\R)\col\varphi_j(Z)\in\ih\}.
\end{align*}
Using \eqref{charHsbgpsLEMpf2} it follows that
$\il$ is an ideal of $\is_{L_j}^{X_j}(\R)$.
Hence also $d\iota(\il)$ is an ideal of $\lsl_d(\R)$.
But given any two elements $Y,Y'\in\lsl_d(\R)$,
it follows from $d\iota(\ih)=\ig'$
that there exist $Z,Z'\in\ih$ such that
$d\iota(dp_j(Z))=Y$, $d\iota(dp_j(Z'))=Y'$,
and $d\iota(dp_i(Z))=d\iota(dp_i(Z'))=0$ for all $i\neq j$.
Then also $[Z,Z']\in\ih$,
and one computes that $[Z,Z']=\varphi_j(([Y,Y'],C))$ for some
$C\in\M_{r_j\times d}(\R)$
(and then in fact $([Y,Y'],C)\in \is_{L_j}^{X_j}$, because of \eqref{charHsbgpsLEMpf3}).
Hence we conclude
that 
$[Y,Y']\in d\iota(\il)$, for all $Y,Y'\in\lsl_d(\R)$.
Since $\lsl_d(\R)$ is a simple Lie algebra,
it follows that $d\iota(\il)=\lsl_d(\R)$.

Next fix some $Y\in\lsl_d(\R)$ which is invertible as a $d\times d$ matrix.
Because of $d\iota(\il)=\lsl_d(\R)$ there is some $C\in L_j^d$
such that $\varphi_j((Y,C))\in\ih$.
Using also \eqref{charHsbgpsLEMpf2} it follows that for any
$C'\in L_j^d$ there exists some $Z\in\ih$
satisfying $p_j(Z)=(Y,C')$.
Then $\ih$ also contains the Lie product
$[\varphi_j((Y,C)),Z]=\varphi_j([(Y,C),(Y,C')])=\varphi_j((0,(C-C')Y))$.
Hence $(0,(C-C')Y)\in\il$.
Since $C'$ is an arbitrary element in $L_j^d$ 
and $Y$ is invertible,
it follows that $(0,C)\in\il$ for \textit{all} $C\in L_j^d$.
Combining this fact with $d\iota(\il)=\lsl_d(\R)$,
we finally conclude that
$\il=\is_{L_j}^{X_j}(\R)$.
Hence 
$\varphi(\is_{L_j}^{X_j}(\R))\subset\ih$.
We have proved that this holds for all $j$;
hence 
\begin{align}\label{charHsbgpsLEMpf4}
\ih=\is_{L_1}^{X_1}(\R)\times\cdots\times\is_{L_N}^{X_N}(\R),
\end{align}
and so \eqref{charHsbgpsLEMres} holds.
\end{proof}

\begin{lem}\label{tscrHcharacterizationLEM}
Let $H$ be as in \eqref{charHsbgpsLEMres},
and assume that $\Gamma\cap H$ is a lattice in $H$.
Then for each $j$, $L_j$ is a \textrm{rational} subspace of $\R^{r_j}$,
and there exist $Y_j\in \M_{r_j\times d}(\Q)$ such that
\begin{align}\label{tscrHcharacterizationLEMres}
H=\S_{L_1}^{Y_1}(\R)\times\cdots\times\S_{L_N}^{Y_N}(\R).
\end{align}
\end{lem}
\begin{proof} %
The assumption implies that $\Gamma_j\cap\S_{L_j}^{X_j}(\R)$ is a lattice in 
$\S_{L_j}^{X_j}(\R)$, for each $j$.
As in
\cite[(5.58), (5.59)]{jMaS2019},
this implies that $L_j\cap\Z^{r_j}$ is a lattice in $L_j$,
i.e.\ $L_j$ is a rational subspace of $\R^{r_j}$,
and furthermore $X_j\in\M_{r_j\times d}(\Q)+L_j^d$.
Choosing now any $Y_j\in\M_{r_j\times d}(\Q)$ such that $X_j\in Y_j+L_j^d$,
we have $\S_{L_j}^{X_j}(\R)=\S_{L_j}^{Y_j}(\R)$.
Carrying this out for each $j$, we obtain \eqref{tscrHcharacterizationLEMres}.
\end{proof}
The following very basic observation will also be useful for us:
\begin{lem}\label{SLXcontainmentcritLEM}
Let $L$ and $L'$ be linear subspaces of $\R^r$, and $X,X'\in\M_{r\times d}(\R)$.
Then $\S_{L'}^{X'}(\R)\subset\S_L^X(\R)$ holds if and only if
$L'\subset L$ and $X'-X\in L^d$.
\end{lem}
\begin{proof}
Let $Y:=X'-X$.
Then $\S_{L'}^{X'}(\R)\subset\S_L^X(\R)$ holds 
if and only if 
$\I_Y(M,U)\I_Y^{-1}\in\S_{L}(\R)$ for all 
$(M,U)\in\S_{L'}(\R)$,
that is,
$U+Y(M-I)\in L^d$ for all
$M\in\SL_d(\R)$ and all $U\in{L'}^d$.
Assume that this holds.
Then, considering first only $M=I$ %
it follows that $L'\subset L$, thus ${L'}^d\subset L^d$,
and using this fact it follows that we must have $Y(M-I)\in L^d$ for all $M\in\SL_d(\R)$.
Considering only the first column of $M-I$ we conclude that
$Y\veca\in L$ for all column vectors $\veca\in\R^d\setminus\{-\vece_1\}$,
and this in turn implies $Y\in L^d$, i.e.\ $X'-X\in L^d$.
The converse direction %
is immediate.
\end{proof}
In the next lemma we derive an important consequence of the condition
``[$\exists\eta>0$: $\forall m\in\Z^+$: $V_m\notin\Delta_k^{(\eta)}$]''
in Definition \ref{Pkdef}.
Given any $j\in\{1,\ldots,N\}$ with $r_j\neq0$
and $\vecm\in\Z^{r_j}\setminus\{\bn\}$,
we set: %
\begin{align}\label{KjLXclean}
K_{j,\vecm}:=\left\{\I_B A \I_Y
\col B\in \|\vecm\|^{-2}\M_{r_{j}\times d}(\Z),\:
A\in\SL_d(\R),\:
Y\in \vecm^\perp\times (\R^{r_{j}})^{d-1}\right\}.
\end{align}
Note that $K_{j,\vecm}$ is left $\Gamma_j$-invariant $F_\sigma$ set in $\S_{r_j}(\R)$;
hence $p_j^{-1}(K_{j,\vecm})$ is a left $\Gamma$-invariant $F_\sigma$ set in $G$,
and $\pi(p_j^{-1}(K_{j,\vecm}))$ is an $F_\sigma$ set in $\XX$.
\begin{lem}\label{ELKIESMCMULLENtypeLEMcleaned}
For any $k\in\Z^+$, $\nu\in P_k$,
$j\in\{1,\ldots,N\}$ with $r_j\neq0$ and $\vecm\in\Z^{r_j}\setminus\{\bn\}$, if $\|\vecm\|^2\leq k$
then $\nu(\pi(p_j^{-1}(K_{j,\vecm})))=0$.
\end{lem}
\begin{proof}
(This is similar to
\cite[pp.\ 114--115]{nEcM2001}
and \cite[Lemma 9]{cDjMaS2016}.)
Let $k,\nu,j,\vecm$ be given as in the statement of the lemma.
Set $q:=\|\vecm\|^2$ (thus $0<q\leq k$).
Let us fix a vector $\vecb\in\Z^{r_j}$ satisfying $\vecb\cdot\vecm=\gcd(m_1,\ldots,m_{r_j})$.
Then we have
\begin{align}\label{ELKIESMCMULLENtypeLEMcleanedpf9}
\Z^{r_{j}}=(\vecm^\perp\cap\Z^{r_j})\oplus\Z\vecb.
\end{align}
Let $p_{\vecb}:\R^{r_j}\to\R$ be the linear map
such that $\vecv-p_{\vecb}(\vecv)\vecb\in\vecm^\perp$ for all $\vecv\in\R^{r_j}$.
For any matrix $Z\in\M_{r_j\times e}(\R)$,
we write $Z_1,\ldots,Z_e\in\R^{r_j}$ for its column vectors (in order),
and we define 
$p_{\vecb}(Z)=(p_\vecb(Z_1),\ldots,p_\vecb(Z_e))$,
i.e.\ $p_\vecb(Z)$ is the vector in $\R^e$
obtained by applying $p_\vecb$ individually to each column of $Z$.
Furthermore, for any matrix $Z\in\M_{r_j\times d}(\R)$ we will write
$Z'=(Z_2,\ldots,Z_d)$ for the matrix in $\M_{r_j\times(d-1)}(\R)$ formed by 
removing the first column vector from $Z$.
For any $T>0$ %
and $\delta>0$ %
we now introduce the following subsets of $\M_{r_{j}\times d}(\R)$:
\begin{align*}
&\Omega_{T}:=\{Z\in\M_{r_j\times d}(\R)\col p_{\vecb}(Z_1)=0,\: \|p_\vecb(Z')\|<T\};
\\
&\Omega_{T,\delta}:=
\{Z\in\M_{r_j\times d}(\R)\col |p_\vecb(Z_1)|<\delta,\: \|p_\vecb(Z')\|<T\}.
\end{align*}
We also introduce the following subsets of $\S_{r_j}(\R)$:
\begin{align}\notag
K_{T}=\left\{\I_{B} A \I_Y\col B\in q^{-1}\M_{r_{j}\times d}(\Z),
A\in\SL_d(\R), 
Y\in \Omega_{T}\right\}
\end{align}
and
\begin{align}\notag
K_{T,\delta}=\left\{\I_{B} A \I_Y\col B\in q^{-1}\M_{r_{j}\times d}(\Z),
A\in\SL_d(\R), 
Y\in \Omega_{T,\delta}\right\}.
\end{align}
Note that both $K_T$ and $K_{T,\delta}$ are left $\Gamma_j$-invariant;
also $K_{T,\delta}$ is open; %
hence $p_j^{-1}(K_{T,\delta})$ is an open subset of $G$,
and
$\pi(p_j^{-1}(K_{T,\delta}))$ is an open subset of $\XX$.

Now our goal is to prove that %
for any given 
$\eta>0$,
$\lambda\in\Pac(\R^{d-1})$,
$\tM=(M_1,\ldots,M_N)\in G'\setminus\fD_{\scrS}$, real numbers $\rho_1>\rho_2>\cdots\to0$,
and points $V_1,V_2,\ldots$ in $\tTT\setminus\Delta_k^{(\eta)}$,
if $\nu_m\in P(\XX)$ is defined as in \eqref{numDEFclean},
then we have:
\begin{align}\label{ELKIESMCMULLENtypeLEMcleanedpf3}
\forall T,\ve>0:\hspace{5pt}
\exists\delta,m_0>0:\hspace{5pt}
\forall m\geq m_0:\hspace{20pt}
\nu_m(\pi(p_j^{-1}(K_{T,\delta})))<\ve.
\end{align}
This will complete the proof of Lemma \ref{ELKIESMCMULLENtypeLEMcleaned}.
Indeed, if $\nu$ is any weak limit of $\nu_1,\nu_2,\ldots$,
then \eqref{ELKIESMCMULLENtypeLEMcleanedpf3} together with 
the Portmanteau Theorem implies that
for any $T,\ve>0$ there exists $\delta>0$ such that
$\nu(\pi(p_j^{-1}(K_{T,\delta})))\leq\ve$.
This forces
$\nu(\pi(p_j^{-1}(K_{T})))=0$ for all $T>0$,
and so 
$\nu(\pi(p_j^{-1}(K_{j,\vecm})))=0$, 
i.e.\ the lemma is proved.

Let us note that it suffices to prove \eqref{ELKIESMCMULLENtypeLEMcleanedpf3}
for special choices of $\lambda$.
Indeed, since $\C_c(\R^{d-1})$ is dense in $L^1(\R^{d-1})$,
it suffices to prove \eqref{ELKIESMCMULLENtypeLEMcleanedpf3} for 
measures $\lambda\in\Pac(\R^{d-1})$ of the form
$d\lambda(\vecx)=\lambda'(\vecx)\,d\vecx$ with $\lambda'\in\C_c(\R^{d-1})$.
Next, using the fact that any such function $\lambda'$ is bounded, it follows that 
it actually suffices to prove \eqref{ELKIESMCMULLENtypeLEMcleanedpf3}
when $\lambda=\vol\big|_{\scrB_R^{d-1}}$,
i.e.\ Lebesgue measure restricted to a ball $\scrB_R^{d-1}$,\footnote{This measure 
should really be normalized
by a factor $\vol(\scrB_R^{d-1})^{-1}$, to make $\lambda$ and $\nu_m$ probability measures;
however such a normalizing factor clearly does not affect the validity of \eqref{ELKIESMCMULLENtypeLEMcleanedpf3},
and so we will ignore it.} %
with $R>1$ arbitrary and fixed.
Hence from now on we assume that $\lambda$ is of this form.

Let us write $V_m=(V_{m,1},\ldots,V_{m,N})$ with $V_{m,j}\in\TT_j^d$.
Then for any $T,\delta,m$, %
\begin{align}\label{ELKIESMCMULLENtypeLEMcleanedpf7}
\nu_m(\pi(p_j^{-1}(K_{T,\delta})))
=
\vol\Bigl(\Bigl\{\vecu\in\scrB_R^{d-1}\col
x(V_{m,j})M_{j} n_-(\vecu)D_{\rho_m}\in \pi(K_{T,\delta})\Bigr\}\Bigr).
\end{align}
Take $U_{m,j}\in\M_{r_j\times d}(\R)$ 
with $\pi(U_{m,j})=V_{m,j}$.
Since $K_{T,\delta}$ is %
$\Gamma_j$-invariant,
the condition
$x(V_{m,j})M_{j} n_-(\vecu)D_{\rho_m}\in \pi(K_{T,\delta})$
is equivalent with
$\I_{U_{m,j}}M_{j} n_-(\vecu)D_{\rho_m}\in K_{T,\delta}$,
which in turn is equivalent with
\begin{align}\label{ELKIESMCMULLENtypeLEMcleanedpf4}
\bigl(U_{m,j}-q^{-1}\M_{r_j\times d}(\Z)\bigr)M_{j} n_-(\vecu)D_{\rho_m}
\:\bigcap\:
\Omega_{T,\delta}\neq\emptyset.
\end{align}
But for any $Z\in\M_{r_j\times d}(\R)$ the condition
$Z n_-(\vecu)D_{\rho_m}\in\Omega_{T,\delta}$ holds if and only if
the vector $\vecz:=p_\vecb(Z)$ satisfies
$|z_1|<\delta\rho_m^{1-d}$
and $\|z_1\vecu+(z_2,\ldots,z_d)\|<T\rho_m$.
We also have $p_\vecb(AM_j)=p_\vecb(A)M_j$ for all $A\in\M_{r_j\times d}(\R)$;
furthermore $p_\vecb(\Z^{r_j})=\Z$, by \eqref{ELKIESMCMULLENtypeLEMcleanedpf9},
and thus $p_{\vecb}(\M_{r_j\times d}(\Z))=\Z^d$;
therefore
\begin{align}\label{ELKIESMCMULLENtypeLEMcleanedpf20}
\bigl\{p_\vecb(Z)\col Z\in(U_{m,j}-q^{-1}\M_{r_j\times d}(\Z))M_{j}\bigr\}
=\bigl(p_\vecb(U_{m,j})+q^{-1}\Z^d\bigr)M_j=:L_m.
\end{align}
Note that this set $L_m$ is a grid in $\R^d$.
It now follows that the measure in \eqref{ELKIESMCMULLENtypeLEMcleanedpf7} equals
\begin{align}\notag
\vol\Bigl(\Bigl\{\vecu\in\scrB_R^{d-1}\col 
\bigl[\exists\vecz\in L_m:\: 
|z_1|<\delta\rho_m^{1-d}
\text{ and }
\|z_1\vecu+(z_2,\ldots,z_d)\|<T\rho_m
\bigr]\Bigr\}\Bigr)
\\\label{ELKIESMCMULLENtypeLEMcleanedpf21}
\leq\sum_{\substack{\vecz\in L_m\\(|z_1|<\delta\rho_m^{1-d})}}
\vol\Bigl(\Bigl\{\vecu\in\scrB_R^{d-1}\col \|z_1\vecu+(z_2,\ldots,z_d)\|<T\rho_m
\Bigr\}\Bigr).
\end{align}

Recall that we are assuming
$V_m\notin \Delta_k^{(\eta)}$ for all $m$;
this implies that 
for all $q'\in\{1,\ldots,k\}$
and all $\vecn\in\Z^{r_j}$ with $0<\|\vecn\|\leq k$,
the point $V_{m,j}$ in $\TT_j^d$ has distance $\geq\eta$ from
the set $\Delta_{j,q',\vecn}$,
and therefore $U_{m,j}$ has distance $\geq\eta$ from
the set 
$({q'}^{-1}\Z^{r_j}+\vecn^\perp)^d$
in 
in $\M_{r_{j}\times d}(\R)=(\R^{r_j})^d$.
In particular,
$U_{m,j}$ has distance $\geq\eta$ from
$(q^{-1}\Z^{r_j}+\vecm^\perp)^d$,
and this is seen to be equivalent to the statement that
$\|p_\vecb(U_{m,j})-q^{-1}\veca\|\geq\frac{\|\vecm\|}{\gcd(\vecm)}\eta$
for all $\veca\in\Z^{r_j}$.
Combining this with the definition of $L_m$ in \eqref{ELKIESMCMULLENtypeLEMcleanedpf20},
it follows that there exists a constant $\eta'>0$,
independent of $m$, such that
\begin{align}\label{ELKIESMCMULLENtypeLEMcleanedpf22}
\forall m\in\Z^+:\hspace{5pt}
\forall\vecz\in L_m:\hspace{5pt}
\|\vecz\|\geq\eta'.
\end{align}

Now let $T>0$ be given, and keep $m\in\Z^+$ so large that
$T\rho_m<\eta'/6$.
Consider any vector $\vecz\in L_m$ which gives a non-zero contribution in the sum in 
\eqref{ELKIESMCMULLENtypeLEMcleanedpf21}.
This means that there exists some 
$\vecu\in\scrB_R^{d-1}$
such that 
$\|z_1\vecu+(z_2,\ldots,z_d)\|<T\rho_m<\eta'/6$,
and thus
$\|(z_2,\ldots,z_d)\|<\eta'/6+R|z_1|$.
If $R|z_1|\leq\eta'/3$ then it would follow that
$\|\vecz\|<\eta'$,
which is impossible by \eqref{ELKIESMCMULLENtypeLEMcleanedpf22}.
Hence we have proved that
every $\vecz\in L_m$ which gives a non-zero contribution in the sum in 
\eqref{ELKIESMCMULLENtypeLEMcleanedpf21} satisfies
\begin{align}\label{ELKIESMCMULLENtypeLEMcleanedpf23}
|z_1|>\frac{\eta'}{3R}
\quad\text{and}\quad
\|(z_2,\ldots,z_d)\|<2R|z_1|.
\end{align}
Let $L_{m,0}$ be the set of all $\vecz\in L_m$ satisfying 
\eqref{ELKIESMCMULLENtypeLEMcleanedpf23}
and $|z_1|\leq2$,
and for each $\ell\in\Z^+$ let
$L_{m,\ell}$ be the set of all $\vecz\in L_m$ satisfying 
\eqref{ELKIESMCMULLENtypeLEMcleanedpf23}
and $2^\ell<|z_1|\leq2^{\ell+1}$.
Then the sum in \eqref{ELKIESMCMULLENtypeLEMcleanedpf21} is
\begin{align*}
\leq\sum_{0\leq\ell<\log_2(\delta\rho_m^{1-d})}\sum_{\vecz\in L_{m,\ell}}\vol(\scrB_1^{d-1})\cdot\Bigl(\frac{T\rho_m}{|z_1|}\Bigr)^{d-1}.
\end{align*}
But for every $\ell\geq0$ and every $\vecz\in L_m$ we have
\begin{align*}
\|\vecz\|\leq|z_1|+\|(z_2,\ldots,z_d)\|
< (1+2R)|z_1| %
<2^{\ell+3}R.
\end{align*}
Recall also that each grid $L_m$ is a translate of the fixed lattice $q^{-1}\Z^dM_j$.
It follows that there exists a constant $C>0$ which is independent of $m$ and $\ell$
(but which depends on $R,q$ and $M_j$)
such that
$\#L_{m,\ell}< C2^{d\ell}$ for all $\ell\geq0$.
It follows that our sum is
\begin{align*}
\leq 
C\cdot \vol(\scrB_1^{d-1})\cdot (T\rho_m)^{d-1}\cdot\biggl(
\Bigl(\frac{\eta'}{3R}\Bigr)^{1-d}
+ \sum_{1\leq\ell<\log_2(\delta\rho_m^{1-d})}2^{\ell}\biggr)
<C'\bigl(\rho_m^{d-1}+\delta\bigr),
\end{align*}
where $C'$ is a constant which is independent of $m$ or $\delta$
(but which depends on $R,T,\eta'$).

To sum up, we have proved that for any $T,\delta>0$,
and all $m\in\Z^+$ so large that $T\rho_m<\eta'/6$,
we have 
$\nu_m(\pi(p_j^{-1}(K_{T,\delta})))
\leq C'(\rho_m^{d-1}+\delta)$,
with a constant $C'$ which may depend on $T$, but is independent of $m$ and $\delta$.
This bound implies that \eqref{ELKIESMCMULLENtypeLEMcleanedpf3} holds,
and hence Lemma \ref{ELKIESMCMULLENtypeLEMcleaned} is proved.
\end{proof}

\begin{lem}\label{heightLEM}
Let $j\in\{1,\ldots,N\}$,
let $L$ be a rational subspace of $\R^{r_j}$,
$L\neq\R^{r_j}$ (thus $r_j\neq0$),
and let $X\in\M_{r_j\times d}(\Q)$.
Then for any 
$\vecm\in\Z^{r_j}\cap L^\perp\setminus\{\bn\}$
satisfying $X\trans\,\vecm\in\Z^d$,
and any $Y\in\M_{r_j\times d}(\R)$
satisfying $n_-(\R^{d-1})\subset \S_L^{X-Y}(\R)$,
we have %
$\Gamma_j \S_L^X(\R)\I_Y\subset K_{j,\vecm}$.
\end{lem}
\begin{proof}
Set 
$\tX=\|\vecm\|^{-2}\vecm\,(X\trans\,\vecm)\trans\in \|\vecm\|^{-2}\M_{r_j\times d}(\Z)$,
and note that $\tX\trans\,\vecm=X\trans\vecm$.
Let $Y\in\M_{r_j\times d}(\R)$ 
be such that $n_-(\R^{d-1})\subset \S_L^{X-Y}(\R)$.
This implies that 
$(Y-X)(n_-(\vecw)-I)\subset L^d$ for all $\vecw\in\R^{d-1}$,
which forces the first column of $Y-X$ lies in $L$.

Now consider an arbitrary element in 
$\Gamma_j \S_L^X(\R)\I_Y$.
This element can be expressed as follows,
for some $(\gamma,B)\in\Gamma_j$ and 
$(A,U)\in\S_L(\R)$:
\begin{align}\label{heightLEMpf1}
(\gamma,B)\I_X(A,U)\I_{\,-X}\I_Y
=\bigl(\gamma A,(B+X)A+U+Y-X\bigr)
=\I_{(B+\tX)\gamma^{-1}}\, \gamma A\,
\I_W,
\end{align}
where $W:=(X-\tX)A+U+Y-X$.
Here $(B+\tX)\gamma^{-1}\in \|\vecm\|^{-2}\M_{r_j\times d}(\Z)$.
Furthermore, $\tX\trans\,\vecm=X\trans\vecm$
implies $X-\tX\in(\vecm^\perp)^d$ and so $(X-\tX)A\in(\vecm^\perp)^d$.
Also $U\in L^d\subset(\vecm^\perp)^d$,
and finally the first column of $Y-X$ lies in $L$, hence in $\vecm^\perp$.
It follows that the first column of $W$ lies in $\vecm^\perp$.
Hence the element in \eqref{heightLEMpf1} lies in $K_{j,\vecm}$,
and we have proved that 
$\Gamma_j \S_L^X(\R)\I_Y\subset K_{j,\vecm}$.
\end{proof}

Let us write $\scrE$ for the set of ergodic $W$-invariant probability measures on $\XX$.
\begin{lem}\label{PkweakconvSUBLEM}
Given any $f\in\C_b(\XX)$ and $\ve>0$,
there exists $k\in\Z^+$ such that
every $\alpha\in\scrE$ 
which satisfies
$\tiota_*\alpha=\tmu$ and
\begin{align}\label{PkweakconvSUBLEMass}
\forall j\in\{1,\ldots,N\}:\:\: r_j\neq0\Rightarrow
\:\:\:
\forall \vecm\in\Z^{r_j}\setminus\{\bn\}:\:\:
\|\vecm\|^2\leq k
\Rightarrow
\alpha(\pi(p_j^{-1}(K_{j,\vecm})))=0,
\end{align}
also satisfies $\bigl|\alpha(f)-\mu(f)\bigr|<\ve$.
\end{lem}
\begin{proof} %
Assume the opposite;
this means that there exist some
$f\in\C_b(\XX)$, $\ve>0$
and measures $\alpha_1,\alpha_2,\ldots$ in $\scrE$ %
such that for every $k$ we have
$\tiota_*\alpha_k=\tmu$ and 
$\bigl|\alpha_k(f)-\mu(f)\bigr|\geq\ve$,
and $\alpha=\alpha_k$ satisfies \eqref{PkweakconvSUBLEMass}.

Clearly $\scrE\subset\scrQ(\XX)$,
and hence for each $k$, we can apply Ratner's theorem,
\cite[Thm.\ 1]{mR91a},
to $\alpha_k$.
As discussed below \eqref{Halphadef},
this implies that there exists some element 
$g_k\in G$ such that,
writing $H_k:=H_{\alpha_k}$,
$\Gamma\cap g_kH_kg_k^{-1}$ is a lattice in $g_kH_kg_k^{-1}$
and
the support of $\alpha_k$ equals $\Gamma\bs\Gamma g_kH_k$.
The validity of the previous statements remain intact when replacing
$g_k$ by any other element from the double coset $\Gamma g_kH_k$.
Hence, since
$\iota(H_{k})=G'$ by Lemma \ref{iotaHeqGpLEM},
after right multiplying $g_k$ by an appropriate element in $H_k$
we may assume that
$g_k$ is of the form
$g_k=(\I_{Y_{k,1}},\cdots,\I_{Y_{k,N}})$ for some 
matrices $Y_{k,j}\in\M_{r_j\times d}(\R)$.
Next, by left multiplying $g_k$ by an appropriate element in
$\Gamma$ (in fact in $\Gamma\cap\iota^{-1}(\{I\})$),
we may furthermore assume 
that every entry of every matrix $Y_{k,j}$ lies in the interval $[0,1)$.

Note that the sequence $\alpha_1,\alpha_2,\ldots$ in $P(\XX)$
is tight, since $\tiota_*\alpha_k=\tmu$ for all $k$
and the map $\tiota:\XX\to\XX'$ is proper.
Hence by Prohorov's Theorem,
there exists a subsequence, say $\alpha_{k_1},\alpha_{k_2},\ldots$
where $1\leq k_1<k_2<\cdots$,
which converges to some limit measure $\nu\in P(\XX)$.
In view of our assumption on the entries of the matrices $Y_{k,j}$,
we may \textit{also}
assume that $g_{k_\ell}$ converges to some element
$\tg$ in $G$ as $\ell\to\infty$.
We have $\tiota_*\nu=\tmu$, by the continuous mapping theorem.
Furthermore, by \cite[Cor.\ 1.1]{sM95}
we have $\nu\in\scrQ(\XX)$,
and hence by Ratner's 
\cite[Thm.\ 1]{mR91a},
there exists some $g_\nu\in G$ such that
$\Gamma\cap g_\nu H_\nu g_\nu^{-1}$
is a lattice in $g_\nu H_\nu g_\nu^{-1}$
and $\supp(\nu)=\Gamma\bs\Gamma g_\nu H_\nu$.
Note also that $\iota(H_\nu)=G'$, by Lemma \ref{iotaHeqGpLEM}.
Next we will apply
\cite[Thm.\ 1.1(2)]{sM95}.
As a preparation,
note that by \cite[Lemma 2.3]{sM95},
for each $\ell$ we can find a one-parameter subgroup $\{u_\ell(t)\}_{t\in\R}$ of
$W$ which %
acts ergodically with respect to $\alpha_{k_\ell}$,
and we can then find an element 
$s_\ell\in G$ such that the trajectory $\{\Gamma g_\nu s_\ell u_\ell(t)\col t>0\}$
is uniformly distributed with respect to $\alpha_{k_\ell}$.
This last property remains valid if $s_\ell$ is replaced by
$s_\ell u_\ell(t)$ for any $t>0$,
and in this way we may modify the elements $s_1,s_2,\ldots$
so that $s_\ell\to e$ in $G$ as $\ell\to\infty$
(this is possible since $\alpha_{k_\ell}\to\nu$ in $P(\XX)$
and since the point $\Gamma g_\nu$ lies in the support of $\nu$).
Hence by \cite[Thm.\ 1.1(2)]{sM95},
for all sufficiently large $\ell$
we have $\supp(\alpha_{k_\ell})\subset\supp(\nu)\cdot s_\ell$,
or equivalently,
$\Gamma\bs\Gamma g_{k_\ell} H_{k_\ell}\subset\Gamma\bs\Gamma g_\nu H_\nu s_\ell$,
viz.,
$\Gamma (g_{k_\ell} H_{k_\ell}g_{k_\ell}^{-1}) g_{k_\ell}s_\ell^{-1}g_\nu^{-1} 
\subset\Gamma (g_\nu H_\nu g_\nu^{-1})$.
But we know that $\Gamma (g_\nu H_\nu g_\nu^{-1})$
is a closed regular submanifold of $G$,
and for each $\gamma\in\Gamma$,
$\gamma (g_\nu H_\nu g_\nu^{-1})$ is a connected component of this submanifold
(cf.\ \cite[Theorem 1.13]{mR72}).
Hence for each large $\ell$
there exists some $\gamma_\ell\in\Gamma$ such that
\begin{align}\label{PkweakconvLEMcleanpf11}
(g_{k_\ell} H_{k_\ell}g_{k_\ell}^{-1}) g_{k_\ell}s_\ell^{-1}g_\nu^{-1} 
\subset \gamma_\ell (g_\nu H_\nu g_\nu^{-1}).
\end{align}
Recall that we also have $g_{k_\ell}\to\tg$ as $\ell\to\infty$;
hence 
$g_{k_\ell}s_\ell^{-1}g_\nu^{-1}\to \tg g_\nu^{-1}$,
and since 
$g_{k_\ell}s_\ell^{-1}g_\nu^{-1}\in\gamma_\ell(g_\nu H_\nu g_\nu^{-1})$
for all large $\ell$
it follows that
there is some $\tgamma\in\Gamma$ such that
$\gamma_\ell(g_\nu H_\nu g_\nu^{-1})
=\tgamma(g_\nu H_\nu g_\nu^{-1})$
for all sufficiently large $\ell$.
For these $\ell$,
\eqref{PkweakconvLEMcleanpf11} implies
\begin{align}\label{PkweakconvLEMcleanpf12}
g_{k_\ell}s_\ell^{-1}g_\nu^{-1} \tgamma^{-1}
\in \tgamma g_\nu H_\nu g_\nu^{-1}\tgamma^{-1}
\qquad\text{and}\qquad
g_{k_\ell} H_{k_\ell}g_{k_\ell}^{-1}
\subset \tgamma g_\nu H_\nu g_\nu^{-1}\tgamma^{-1}.
\end{align}
We have proved that \eqref{PkweakconvLEMcleanpf12} holds for all sufficiently large $\ell$;
however by removing the initial elements from the sequence $k_1<k_2<\cdots$,
we may from now on assume that \eqref{PkweakconvLEMcleanpf12} holds for \textit{all} $\ell\in\Z^+$.

Next let us apply Lemmas \ref{charHsbgpsLEM} and \ref{tscrHcharacterizationLEM}
to the group
$\tgamma g_\nu H_\nu g_\nu^{-1}\tgamma^{-1}$
and
the groups $g_{k_\ell} H_{k_\ell}g_{k_\ell}^{-1}$ for all $\ell$.
This gives that 
there exist rational subspaces $L_{j}$ and $L_{\ell,j}$ of $\R^{r_j}$ 
and matrices $X_{j}$ and $X_{\ell,j}$ in $\M_{r_j\times d}(\Q)$
(for all $j\in\{1,\ldots,N\}$ and $\ell\in\Z^+$),
such that
\begin{align}\label{PkweakconvLEMcleanpf13a}
\tgamma g_\nu H_\nu g_\nu^{-1}\tgamma^{-1}=\S_{L_{1}}^{X_{1}}(\R)\times\cdots\times\S_{L_{N}}^{X_{N}}(\R).
\end{align}
and
\begin{align}\label{PkweakconvLEMcleanpf13}
g_{k_\ell} H_{k_\ell}g_{k_\ell}^{-1}
=\S_{L_{\ell,1}}^{X_{\ell,1}}(\R)\times\cdots\times\S_{L_{\ell,N}}^{X_{\ell,N}}(\R)
\qquad
(\forall \ell\in\Z^+).
\end{align}
It now follows from %
\eqref{PkweakconvLEMcleanpf12} 
that $\S_{L_{\ell,j}}^{X_{\ell,j}}(\R)\subset \S_{L_{j}}^{X_{j}}(\R)$,
and hence by Lemma \ref{SLXcontainmentcritLEM},
\begin{align}\label{PkweakconvLEMcleanpf9}
L_{\ell,j}\subset L_j\quad\text{and}\quad
X_{\ell,j}-X_j\in L_j^d,\qquad
\forall \ell\in\Z^+,\: j\in\{1,\ldots,N\}.
\end{align}
Let us also note that \eqref{PkweakconvLEMcleanpf13}
and 
$g_k=(\I_{Y_{k,1}},\cdots,\I_{Y_{k,N}})$ imply that
\begin{align}\label{PkweakconvLEMcleanpf13b}
H_{k_\ell}=\S_{L_{\ell,1}}^{X_{\ell,1}-Y_{k_{\ell},1}}(\R)\times\cdots\times\S_{L_{\ell,N}}^{X_{\ell,N}-Y_{k_{\ell},N}}(\R).
\end{align}

Now we obtain a contradiction as follows:
We have 
$\nu\neq\mu$,
since
$\nu(f)=\lim_{\ell\to\infty}\alpha_{k_\ell}(f)$
and $|\alpha_k(f)-\mu(f)|\ge\ve$ for all $k$.
Hence $H_\nu\neq G$,
and so by \eqref{PkweakconvLEMcleanpf13a} there is some $j\in\{1,\ldots,N\}$
such that $L_j\neq\R^{r_j}$ (this implies in particular $r_j>0$).
Hence we can choose some
$\vecm\in\Z^{r_j}\cap L_{j}^\perp\setminus\{\bn\}$
with $X_{j}\trans\,\vecm\in\Z^d$.
Now by \eqref{PkweakconvLEMcleanpf9}
we also have $\vecm\perp L_{\ell,j}$
and $X_{\ell,j}\trans\,\vecm=X_j\trans\,\vecm\in\Z^d$
for all $\ell$.
Let us apply this for some fixed choice of $\ell$ so large that
$k_\ell\geq\|\vecm\|^2$.
It is immediate from the definition of 
$H_{k_\ell}=H_{\alpha_{k_\ell}}$
(cf.\ \eqref{Halphadef})
that $W\subset H_{k_\ell}$;
thus 
$n_-(\R^{d-1})\subset\S_{L_{\ell,j}}^{X_{\ell,j}-Y_{k_{\ell},j}}(\R)$
(cf.\ \eqref{PkweakconvLEMcleanpf13b}),
and so by Lemma~\ref{heightLEM},
the set
$\Gamma_j\I_{Y_{k_{\ell},j}}\S_{L_{\ell,j}}^{X_{\ell,j}-Y_{k_{\ell},j}}(\R)$
is contained in $K_{j,\vecm}$.
By \eqref{PkweakconvLEMcleanpf13b},
this implies that $\Gamma\bs\Gamma g_{k_\ell}H_{k_\ell}$,
i.e.\ the support of $\alpha_{k_\ell}$,
is contained in 
$\pi(p_j^{-1}(K_{j,\vecm}))$,
and so $\alpha_{k_\ell}(\pi(p_j^{-1}(K_{j,\vecm})))=1$.
Since $\|\vecm\|^2\leq k_\ell$,
this gives a contradiction against
our assumption that 
\eqref{PkweakconvSUBLEMass} holds for $k=k_\ell$ and 
$\alpha=\alpha_{k_\ell}$.

Hence the lemma is proved.
\end{proof}

\begin{remark}
Our proof of Lemma \ref{PkweakconvSUBLEM} made use of
tools developed in \cite{sM95}.
In view of the explicit format of the groups $H_k$
known from Lemma \ref{charHsbgpsLEM},
it seems that it might be possible to 
alternatively give a more direct proof,
circumventing the use of 
\cite{sM95}, and that this may also lead to 
more explicit information on the admissible $k=k(f,\ve)$ in the statement of the lemma
(it will be seen below that this $k=k(f,\ve)$ is also admissible in the statement of Theorem \ref{MAINcleanUNIFCONVthm}).
By contrast, the number $\rho_0=\rho_0(f,\ve,\lambda,\eta,\tM)$ in
the statement of Theorem \ref{MAINcleanUNIFCONVthm}
is genuinly non-explicit, because of the application
of  %
Ratner's measure classification
in our proof.
\end{remark}

\begin{lem}\label{PkweakconvLEMclean}
Given any $f\in\C_b(\XX)$ and $\ve>0$,
there exists $k\in\Z^+$ such that
$\bigl|\nu(f)-\mu(f)\bigr|<\ve$ holds for all
$\nu\in P_k$.
\end{lem}

\begin{proof}
For the given $f$ and $\ve>0$,
we choose $k$ as in Lemma \ref{PkweakconvSUBLEM};
then consider an arbitrary measure $\nu\in P_k$.
Since $\nu$ is $W$-invariant by Lemma \ref{Winvlem},
we can apply ergodic decomposition;
this gives that there exists 
a unique Borel probability measure $\omega$ 
on $\scrE$ such that
\begin{align}\label{nuERGDEC}
\nu=\int_{\scrE}\alpha\,d\omega(\alpha).
\end{align}
Cf., e.g., \cite[Theorem 4.4]{vV63}.
Note that \eqref{nuERGDEC} together with Lemma \ref{PkmeasrhoprojLEM}
implies $\tmu=\tiota_*\nu=\int_{\scrE}\tiota_*\alpha\,d\omega(\alpha)$,
and for each $\alpha\in\scrE$,
$\tiota_*\alpha$ is an ergodic $W$-invariant measure on $\XX'$.
Hence in fact $\tiota_*\alpha=\tmu$ for $\omega$-almost all $\alpha\in\scrE$, by uniqueness of the 
ergodic decomposition of $\tmu$. 
Furthermore, 
for every 
$j\in\{1,\ldots,N\}$ with $r_j\neq0$
and every $\vecm\in\Z^{r_j}$
with $0<\|\vecm\|\leq k$,
we have
$\int_{\scrE}\alpha(\pi(p_j^{-1}(K_{j,\vecm})))\,d\omega(\alpha)=\nu(\pi(p_j^{-1}(K_{j,\vecm})))=0$,
by Lemma \ref{ELKIESMCMULLENtypeLEMcleaned},
and hence $\alpha(\pi(p_j^{-1}(K_{j,\vecm})))=0$ for $\omega$-almost all $\alpha$. %
Therefore, by Lemma \ref{PkweakconvSUBLEM},
$\bigl|\alpha(f)-\mu(f)\bigr|<\ve$ for $\omega$-almost all $\alpha$.
Hence
$\bigl|\nu(f)-\mu(f)\bigr|
=\bigl|\int_{\scrE}\alpha(f)\,d\omega(\alpha)-\mu(f)\bigr|
\leq \int_{\scrE}\bigl|\alpha(f)-\mu(f)\bigr|\,d\omega(\alpha)<\ve$.
\end{proof}

\vspace{5pt}

\begin{proof}[Proof of Theorem \ref{MAINcleanUNIFCONVthm}]
Given $f\in\C_b(\XX)$ and $\ve>0$, we choose 
$k\in\Z^+$ as in 
Lemma \ref{PkweakconvLEMclean}.
Now also let arbitrary $\lambda\in\Pac(\R^{d-1})$, $\eta>0$, $\tM\in G'\setminus\fD_{\scrS}$
be given.
Assume that there does \textit{not} exist any $\rho_0\in(0,1)$ 
such that \eqref{MAINcleanUNIFCONVthmRES} holds for all
$\rho\in(0,\rho_0)$ and $V\in\tTT\setminus \Delta_k^{(\eta)}$.
This means that there exist sequences 
$\rho_1>\rho_2>\cdots\to0$ 
and $V_1,V_2,\ldots$
in $\tTT\setminus\Delta_k^{(\eta)}$ satisfying
\begin{align}\label{MAINUNIFcleanPROP4step3pf1}
\biggl|\int_{\R^{d-1}}f\bigl(x(V_m)\tM\varphi(n_-(\vecu)D_{\rho_m})\bigr)\,d\lambda(\vecu)-
\int_{\XX} f \,d\mu\biggr|\geq\ve,
\qquad\forall m\in\Z^+.
\end{align}
Define $\nu_m\in P(\XX)$
through $\nu_m(g)=\int_{\R^{d-1}}g\bigl(x(V_m)\tM\varphi(n_-(\vecu)D_{\rho_m})\bigr)\,d\lambda(\vecu)$
for all $g\in \C_b(\XX)$
(just as in \eqref{numDEFclean}).
By \cite[Thm.\ 5]{jMaS2013a}
(applied in the same way as in the proof of Lemma \ref{PkmeasrhoprojLEM})
we have $\tiota_*\,\nu_m\to\tmu$ in $P(\XX')$.
Hence the sequence $\nu_1,\nu_2,\ldots$ in $P(\XX)$
is tight, and so by Prohorov's Theorem,
after passing to a subsequence we may assume that $\nu_1,\nu_2,\ldots$ converges to some
$\nu\in P(\XX)$. 
Then $\nu\in P_k$ (cf.\ Def.\ \ref{Pkdef}),
and thus by our choice of $k$ we have
$\bigl|\nu(f)-\mu(f)\bigr|<\ve$.
But $\nu_m$ converges weakly to $\nu$;
in particular $\nu_m(f)\to\nu(f)$ as $m\to\infty$,
and hence %
we conclude that $\bigl|\nu_m(f)-\mu(f)\bigr|<\ve$ for all sufficiently large $m$.
This is a contradiction against \eqref{MAINUNIFcleanPROP4step3pf1},
and thus Theorem \ref{MAINcleanUNIFCONVthm} is proved.
\end{proof}

Next we establish a variant of Theorem \ref{MAINcleanUNIFCONVthm},
where %
the unipotent element $n_-(\vecu)$ is replaced by a rotation:
\begin{thm}\label{MAINcleanUNIFCONVROTthm}
Let $f\in\C_b(\XX)$
and $\ve>0$ be given.
Then there exists some $k\in\Z^+$ such that for every 
$\lambda\in\Pac(\US)$, $\eta>0$ and 
$\tM\in G'\setminus\fD_{\scrS}$, there exists some $\rho_0\in(0,1)$ such that
\begin{align}\label{MAINcleanUNIFCONVROTthmRES}
\biggl|\int_{\US}f\Bigl(x(V)\tM\varphi(R(\vecv)D_\rho)\Bigr)\,d\lambda(\vecv)-
\int_{\XX} f \,d\mu\biggr|<\ve
\end{align}
for all $\rho\in(0,\rho_0)$ and all $V\in\tTT\setminus \Delta_k^{(\eta)}$.
\end{thm}
\begin{proof}
This follows from Theorem \ref{MAINcleanUNIFCONVthm}
via a fairly standard approximation argument;
cf., e.g.,
\cite[Thm.\ 5.3 and Cor.\ 5.4]{jMaS2010a}.
Some care is needed to ensure that we can obtain a uniform statement as in the theorem.

To start with,
we restrict to functions $f$ of compact support.
Thus let $f\in\C_c(\XX)$ and $\ve>0$ be given.
Fix a corresponding positive integer $k$ as in Theorem \ref{MAINcleanUNIFCONVthm}.
Now also let $\lambda\in\Pac(\US)$, $\eta>0$ and 
$\tM\in G'\setminus\fD_{\scrS}$ be given.

For $\delta>0$ we set
$\scrU_\delta=\{T\in\SL_d(\R)\col \|T-I\|<\delta\}$,
where $\|\cdot\|$ denotes the entrywise maximum norm on $d\times d$ matrices.
Thus $\scrU_\delta$ is an open neighborhood of the identity in $\SL_d(\R)$.
Since $f$ has compact support, we can fix %
$0<\delta<1$ so small that
\begin{align}\label{MAINUNIFPROP4variation2pf1}
|f(x\varphi(T))-f(x)|<\ve,   \qquad\forall x\in\XX,\: T\in\scrU_\delta.
\end{align}
Recall that $R$ is continuous when restricted to $\US$ minus one point;
it follows that there exists a compact subset
$S\subset\US$ such that the restriction of $R$ to $S$ is continuous,
and
\begin{align}\label{MAINcleanUNIFCONVROTthmpf1}
\|f\|_\infty\cdot\lambda(\US\setminus S)\leq\ve.
\end{align}
For each $\vecv_0\in S$ we set
$\Omega_{\vecv_0}=\{\vecv\in S\col R(\vecv_0)^{-1}R(\vecv)\in\scrU_{\delta/2}\}$;
this is a relatively open neighborhood of $\vecv_0$ in $S$.
Since $S$ is compact, we
can fix a finite subset $\fQ_0\subset\US$
such that the sets $\Omega_{\vecv_0}$ for $\vecv_0\in\fQ_0$ cover $S$.
Let us 
fix an arbitrary total order $\prec$ on $\fQ_0$,
and set $\Omega_{\vecv_0}':=\Omega_{\vecv_0}\setminus
\bigl(\cup_{\substack{\vecv_0'\in\fQ_0\\\vecv_0'\prec\vecv_0}}\Omega_{\vecv_0'}\bigr)$.
Then the sets %
$\Omega_{\vecv_0}'$ for $\vecv_0\in\fQ_0$ form a partition of $S$.
Set $\fQ_0'=\{\vecv_0\in\fQ_0\col\lambda(\Omega_{\vecv_0}')>0\}$,
and for each $\vecv_0\in\fQ_0'$ let
$\lambda_{\vecv_0}:=\lambda(\Omega_{\vecv_0}')^{-1}\,\lambda\big|_{\Omega_{\vecv_0}'}\in\Pac(\US)$.
Note that we now have
\begin{align}\label{MAINcleanUNIFCONVROTthmpf1a}
\lambda|_S=\sum_{\vecv_0\in\fQ_0'}\lambda(\Omega_{\vecv_0}')\lambda_{\vecv_0}.
\end{align}

Let us fix $\vecv_0\in\fQ_0'$ temporarily, and consider the functions
$E:\US\to\M_{d}(\R)$,
$a:\US\to\R$,
$\vecb,\vecc:\US\to\R^{d-1}$,
$D:\US\to\M_{d-1}(\R)$ defined by
\begin{align*}
E(\vecv)=\matr{a(\vecv)}{\vecb(\vecv)}{\vecc(\vecv)\trans}{D(\vecv)}:=R(\vecv_0)^{-1}R(\vecv)
\qquad(\vecv\in\US).
\end{align*}
Note that 
$(a(\vecv),\vecc(\vecv))=\vece_1 E(\vecv)\trans=\vecv R(\vecv_0)$
for all $\vecv\in\US$;
in particular $a(\vecv)=\vecv\cdot\vecv_0$,
and, since $\delta<1$, it follows that 
$\Omega_{\vecv_0}$ is contained in the open disc
$\scrH_{\vecv_0}=\{\vecv\in\US\col\vecv\cdot\vecv_0>\frac12\}$.
We introduce the function
$\vecx:\scrH_{\vecv_0}\to\R^{d-1}$,
$\vecx(\vecv)=-a(\vecv)^{-1}\vecc(\vecv)$;
this is a diffeomorphism of $\scrH_{\vecv_0}$ onto the open ball $\scrB_{\sqrt3}^{d-1}$.
We set $\tlambda_{\vecv_0}=\vecx_*(\lambda_{\vecv_0})\in\Pac(\R^{d-1})$.

Note that $\tM\notin\fD_{\scrS}$ implies that
$\tM\varphi(R(\vecv_0))\notin\fD_{\scrS}$ for every $\vecv_0\in\fQ_0'$.
Hence by Theorem~\ref{MAINcleanUNIFCONVthm} and our choice of $k$,
there exists $\rho_0\in(0,1)$
such that %
for all $\vecv_0\in\fQ_0'$, $\rho\in(0,\rho_0)$
and $V\in\tTT\setminus\Delta_k^{(\eta)}$, we have:
\begin{align}\label{MAINUNIFPROP4variation2pf2}
\biggl|\int_{\R^{d-1}}f\Bigl(x(V)\tM\varphi(R(\vecv_0))\,\varphi(n_-(\vecx)D_\rho)\Bigr)\,d\tlambda_{\vecv_0}(\vecx)
-\int_{\XX} f \,d\mu
\biggr|<\ve.
\end{align}
Here in the integral over $\R^{d-1}$,
we substitute $\vecx=\vecx(\vecv)$
and then use 
\eqref{MAINUNIFPROP4variation2pf1}
combined with
the fact that 
for every $\vecv\in\Omega_{\vecv_0}'$
and $\rho\in(0,\rho_0)$
we have $\smatr{a(\vecv)^{-1}}{\bn}{\rho^d\vecc(\vecv)\trans}{D(\vecv)}\in\scrU_\delta$
(this is immediate from the fact that
$E(\vecv)\in\scrU_{\delta/2}$,
once we note that $|a-1|<\delta/2$ implies $|a^{-1}-1|<2|1-a|<\delta$).
This gives:
\begin{align}\notag
\Biggl|\int_{\R^{d-1}}f(x(V)\tM\varphi(R(\vecv_0)n_-(\vecx)D_\rho))\,d\tlambda_{\vecv_0}(\vecx)
\hspace{170pt}
\\\label{MAINUNIFPROP4variation2pf3}
-\int_{\Omega_{\vecv_0}'}f\biggl(x(V)\tM\varphi\biggl(R(\vecv_0)n_-(\vecx(\vecv))D_\rho
\matr{a(\vecv)^{-1}}{\bn}{\rho^d\vecc(\vecv)\trans}{D(\vecv)}\biggr)\biggr)\,d\lambda_{\vecv_0}(\vecv)
\Biggr|<\ve.
\end{align}
Here the last integral can be simplified using
$R(\vecv_0)n_-(\vecx(\vecv))D_\rho\smatr{a(\vecv)^{-1}}{\bn}{\rho^d\vecc(\vecv)\trans}{D(\vecv)}
=R(\vecv)D_\rho$.
Hence, by \eqref{MAINUNIFPROP4variation2pf2}
and \eqref{MAINUNIFPROP4variation2pf3},
we have 
for any $\vecv_0\in\fQ_0'$, $\rho\in(0,\rho_0)$
and $V\in\tTT\setminus\Delta_k^{(\eta)}$:
\begin{align}\label{MAINUNIFPROP4variation2pf5}
\biggl|\int_{\Omega_{\vecv_0}'}f(x(V)\tM\varphi(R(\vecv)D_\rho))\,d\lambda_{\vecv_0}(\vecv)
-\int_{\XX} f \,d\mu
\biggr|<2\ve.
\end{align}
Multiplying this inequality by $\lambda(\Omega_{\vecv_0}')$
and then adding over all $\vecv_0\in\fQ_0'$
and using \eqref{MAINcleanUNIFCONVROTthmpf1a}
and \eqref{MAINcleanUNIFCONVROTthmpf1}, 
we conclude that \eqref{MAINcleanUNIFCONVROTthmRES}
holds with $4\ve$ in place of $\ve$,
for all 
$\rho\in(0,\rho_0)$
and $V\in\tTT\setminus\Delta_k^{(\eta)}$.

Thus the theorem is proved under the extra assumption that $f\in\C_c(\XX)$.
Finally, the extension to the case of arbitrary functions $f\in\C_b(\XX)$
is achieved by a completely standard approximation argument.
\end{proof}

The following is an immediate corollary of
Theorem \ref{MAINcleanUNIFCONVROTthm}:
\begin{cor}\label{MAINcleanCONVcor2}
Let $V$ be an arbitrary, fixed point in $\tTT\setminus\bigcup\, \tp_j^{\hspace{3pt}-1}(\Delta_{j,q,\vecm})$,
where the union is taken over all triples $\langle j,q,\vecm\rangle$ with
$j\in\{1,\ldots,N\}$, $r_j\neq0$,
$q\in\Z^+$ and $\vecm\in\Z^{r_j}\setminus\{\bn\}$.
Then for any $\tM\in G'$, 
$f\in\C_b(\XX)$
and
$\lambda\in\Pac(\US)$,
we have 
\begin{align}\label{MAINcleanCONVcor2RES}
\int_{\US}f\Bigl(x(V)\tM\varphi(R(\vecv)D_\rho)\Bigr)\,d\lambda(\vecv)
\to
\int_{\XX} f \,d\mu
\qquad\text{as }\:\rho\to0.
\end{align}
\end{cor}
\begin{proof}
The assumption on $V$ implies that for any $k\in\Z^+$
there is some $\eta>0$ such that $V\notin\Delta_k^{(\eta)}$.
Using this fact, the corollary is an immediate consequence of 
Theorem \ref{MAINcleanUNIFCONVROTthm}.
\end{proof}

\section{Proof of [P2] (uniform spherical equidistribution)}
\label{NEWISOsec}

We now return to using the same notation as in 
Sections \ref{SETUPsec}--\ref{verQ1Q2etcSEC};
in particular $\scrP$ is a finite union of grids in $\R^d$,
and we assume that an admissible presentation of $\scrP$ has been fixed
(cf.\ \eqref{GENPOINTSET1}, \eqref{LpsiDEF}),
and corresponding to this presentation we let the homogeneous space $\XX=\GaG$
be as defined in Section \ref{LIEGPHOMSPsec}.
In particular we have $r_1,\ldots,r_N>0$ and
\begin{align*}
\Gamma=\S_{r_1}(\Z)\times\cdots\times\S_{r_N}(\Z),
\end{align*}
which is a stricter requirement than what was imposed in 
Section \ref{UNIPOTAPPLsec}.

Our goal in this section is to complete the proof of [P2]. %
Recall that by our initial discussion %
in Section \ref{P2initialdiscsec},
the task which remains is to prove Theorem \ref{HOMDYNMAINTHM}.
A key tool in our proof will be Theorem \ref{MAINcleanUNIFCONVROTthm};
note that %
this theorem will in general be applied to a certain 
homogeneous \textit{submanifold} of our present homogeneous space $\XX$.

\subsection{Equidistribution without uniformity}
\label{nonunifequidistrSEC}

We will start by proving the following non-uniform result,
which as we will see fairly easily implies Theorem \ref{HOMDYNintrononunifTHM},
and which will also play an important role in our proof of 
Theorem \ref{HOMDYNMAINTHM}.

As in Section \ref{UNIPOTAPPLsec}, we let the subset $\fD_{\scrS}\subset G'$ be given by \eqref{DSdef}.
Recall that $\Omega=\prod_{j=1}^N P(\TT_j^d)'$; cf.\ \eqref{OMEGAdef}.
As in \eqref{tTTdef2} we set $\tTT=\TT_1^d\times\TT_2^d\times\cdots\times\TT_N^d$,
and we let $x:\tTT\to\XX$ be the natural embedding.
Finally, for any $V=\bigl( V_1,\ldots,V_N\bigr)\in\tTT$ we define:
\begin{align}\label{omegaVDEF}
\omega^{(V)}:=\bigl(\omega_1^{(V_1)},\ldots,\omega_N^{(V_N)}\bigr)\in\Omega.
\end{align}
\begin{thm}\label{nonunifTHM1}
For any $V\in\tTT$, $\tM\in G'\setminus\fD_{\scrS}$, $f\in\C_b(\XX)$ and 
$\lambda\in\Pac(\US)$
we have
\begin{align}\label{nonunifTHM1res}
\int_{\US}f(x(V) \tM\varphi(R(\vecv)D_\rho))\,d\lambda(\vecv)
\to \int_{\XX} f \,d\overline{\omega^{(V)}}
\end{align}
as $\rho\to0$.
\end{thm}
\begin{proof} %
Let $V=\bigl( V_1,\ldots,V_N\bigr)\in\tTT$ be given.
For each $j\in\{1,\ldots,N\}$
we write $L_j:=L_j^{(V_j)}$,
fix some $X_j\in\M_{r_j\times d}(\Q)$ 
such that 
$V_j-\pi(X_j)\in\bigl(\SS_j^{(V_j)}\bigr)^{\circ\,d}$,
and fix some 
$\tV_j\in X_j+L_j^d\subset\M_{r_j\times d}(\R)$ with $\pi(\tV_j)=V_j$
(cf.\ the discussion above \eqref{xVSLR}).
Using the notation in \eqref{SVXdefrep},
we then set:
\begin{align*}
H:=\S_{L_1}^{X_1}(\R)\times\cdots\times\S_{L_N}^{X_N}(\R)
=\S_{L_1}^{\tV_1}(\R)\times\cdots\times\S_{L_N}^{\tV_N}(\R).
\end{align*}
As in \eqref{xVSLR} and the proof of Lemma \ref{mujqjustifyLEM},
we have for each $j$ that
$\Gamma_j$ intersects $\S_{L_j}^{\tV_j}(\R)$ in a lattice,
and the orbit $x(V_j)\cdot\S_{L_j}(\R)=\Gamma_j\bs\Gamma_j\S_{L_j}^{\tV_j}(\R)\I_{\tV_j}$ %
is a closed embedded submanifold of $\XX_j$ which carries a unique
$\S_{L_j}(\R)$-invariant probability measure;
and by Proposition \ref{SLRinvprobmeasLEM}
this measure equals $\overline{\omega_j^{(V_j)}}$.
Taking the product over all $j$,
and writing $\tV:=\bigl(\tV_1,\ldots,\tV_N\bigr)\in\prod_{j=1}^N\M_{r_j\times d}(\R)$,
it follows that $\Gamma\bs\Gamma H\I_{\tV}$ is a closed embedded submanifold of $\XX$
and, using also \eqref{omegaVDEF}
and \eqref{oomegaDEF},
that $\overline{\omega^{(V)}}$ is the unique 
$\prod_{j=1}^N\S_{L_j}(\R)$-invariant probability measure on 
$\Gamma\bs\Gamma H\I_{\tV}$.

Let $\Gamma_{\! H}:=\Gamma\cap H$,
and let $\mu$ be the unique $H$-invariant probability measure on 
the homogeneous submanifold $\Gamma\bs\Gamma H=\Gamma_{\! H}\bs H$ of $\XX$.
In order to prove the theorem, %
we will prove that for any
$\tM\in G'\setminus\fD_{\scrS}$, $F\in\C_b(\Gamma_{\! H}\bs H)$ and $\lambda\in\Pac(\US)$,
we have
\begin{align}\label{nonunifTHM1pf1}
\int_{\US}F
\Bigl(\Gamma_{\! H} \I_{\tV}\tM\varphi(R(\vecv)D_\rho)\I_{\tV}^{-1}\Bigr)\,d\lambda(\vecv)
\to
\int_{\Gamma_{\! H}\bs H} F \,d\mu
\end{align}
as $\rho\to0$.
(To see that the integral to the left in \eqref{nonunifTHM1pf1} is well-defined,
note that $\I_{\tV} g\I_{\tV}^{-1}\in H$ for all $g\in G'$.)

To see that the convergence in \eqref{nonunifTHM1pf1}
implies the statement of the theorem,
let  $\tau:\XX\to\XX$ be right multiplication by $\I_{\tV}$;
then %
$\overline{\omega^{(V)}}=\tau_*(\mu)$,
and the left side in \eqref{nonunifTHM1res} can be expressed as
\begin{align*}
\int_{\US}(f\circ\tau)\Bigl(\Gamma\I_{\tV} \tM\varphi(R(\vecv)D_\rho)\I_{\tV}^{-1}\Bigr)\,d\lambda(\vecv).
\end{align*}
Hence \eqref{nonunifTHM1res} follows from
\eqref{nonunifTHM1pf1} if we let $F$ be the restriction of $f\circ\tau$
to $\Gamma\bs\Gamma H$.

\vspace{4pt}

We now turn to the proof of \eqref{nonunifTHM1pf1}.
We will start by fixing an isomorphism from $H$ onto the Lie group
\begin{align*}
\tG:=\S_{s_1}(\R)\times\cdots\times\S_{s_N}(\R),
\end{align*}
where $s_j:=\dim L_j$.
Given any linear bijection $\varphi:L\stackrel{\sim}{\rightarrow}\R^s$,
where $L$ is a linear subspace of $\R^r$ (for some $r\in\Z^+$) of dimension $s\geq0$,
we write $\S_\varphi$ for the following Lie group isomorphism:
\begin{align}\label{SphiDEFnew}
\S_\varphi:\S_L(\R)\stackrel{\sim}{\rightarrow}\S_s(\R),
\qquad \S_\varphi\bigl((M,U)\bigr)=(M,\varphi^d(U)),
\end{align}
where $\varphi^d$ is the linear bijection
from $L^d$ onto $\M_{s\times d}(\R)$
given by applying $\varphi$ to each column of the matrix.
(If $s=0$ so that $L=\{\bn\}$ and $\S_s(\R)=\SL_d(\R)$,
the definition in \eqref{SphiDEFnew} should of course be interpreted to say $\S_\varphi\bigl((M,0)\bigr)=M$.)
One verifies immediately that $\S_\varphi$ is indeed a Lie group isomorphism.
Also for any $X\in\M_{r\times d}(\R)$, we introduce the following
Lie group isomorphism: %
\begin{align}\label{SphiXDEFnew}
\S_\varphi^X:\S_L^X(\R)\stackrel{\sim}{\rightarrow}\S_s(\R),
\qquad \S_\varphi^X(g)=\S_\varphi\bigl(\I_X^{-1}\,g\,\I_X\bigr).
\end{align}
Next, for each $j$, we fix, once and for all,
a linear bijection $\varphi_j:L_j\stackrel{\sim}{\rightarrow}\R^{s_j}$
with the property that
$\varphi_j(L_j\cap\Z^{r_j})=\Z^{s_j}$;
this is possible since $L_j$ is a rational subspace of $\R^{r_j}$.
Finally we let $\Phi$ be the Lie group isomorphism
\begin{align*}
\Phi:=\S_{\varphi_1}^{X_1}\times\cdots\times\S_{\varphi_N}^{X_N}:
\quad H\stackrel{\sim}{\rightarrow}\tG.
\end{align*}

For each $j$, fix a positive integer $q_j$ such that $X_j\in q_j^{-1}\M_{r_j\times d}(\Z)$,
and let $\Gamma_j'$ be the principal congruence subgroup of $\SL_d(\Z)$ of level $q_j$:
\begin{align}\label{princcongrsubgpdef}
\Gamma_j':=\{M\in\SL_d(\Z)\col M\equiv I\mod q_j\}.
\end{align}
Then set
\begin{align}\label{tGammachoice}
\tGamma_j:=\Gamma_j'\ltimes\M_{s_j\times d}(\Z)
\hspace{12pt} (j=1,\ldots,N)
\qquad\text{and}\quad
\tGamma:=\tGamma_1\times\cdots\times\tGamma_N.
\end{align}
We now claim that
\begin{align}\label{nonunifTHM1pf2}
\tGamma\subset\Phi(\Gamma_H).
\end{align}
To verify this, it suffices to verify that
$\tGamma_j\subset\S_{\varphi_j}^{X_j}\bigl(\S_{r_j}(\Z)\cap\S_{L_j}^{X_j}(\R)\bigr)$
for each $j$.
To do so, note that given any 
$(M,U)\in \tGamma_j$,
we have 
$\bigl(\S_{\varphi_j}^{X_j}\bigr)^{-1}(M,U)\in\S_{L_j}^{X_j}(\R)$
and
\begin{align*}
\bigl(\S_{\varphi_j}^{X_j}\bigr)^{-1}(M,U)
=\bigl(M,(\varphi_j^d)^{-1}(U)+X_j(M-I)\bigr).
\end{align*}
Here
$(\varphi_j^d)^{-1}(U)\in\M_{r_j\times d}(\Z)$
since $U\in\M_{s_j\times d}(\Z)$ and $\varphi_j^{-1}(\Z^{s_j})=L_j\cap\Z^{r_j}$;
also $X_j(M-I)\in\M_{r_j\times d}(\Z)$ since 
$X_j\in q_j^{-1}\M_{r_j\times d}(\Z)$ and 
$M\in\Gamma_j'$;
hence 
$(\S_{\varphi_j}^{X_j})^{-1}(M,U)\in\S_{r_j}(\Z)$.
This completes the proof of \eqref{nonunifTHM1pf2}.
Note that it follows from \eqref{nonunifTHM1pf2}
that we have a well-defined covering map
\begin{align}\label{PHIGammaHdisc3new}
J:\tGamma\bs \tG\to \Gamma_{\!H}\bs H,
\qquad J(\tGamma g)=\Gamma_{\!H}\,\Phi^{-1}(g).
\end{align}

Next, the result of Corollary \ref{MAINcleanCONVcor2},
applied to the homogeneous space $\tGamma\bs \tG$,
can be stated as follows:
Let $\tmu$ be the invariant probability measure on
$\tGamma\bs \tG$.
Let $W=(W_1,\ldots,W_N)$ be an arbitrary element in $\prod_{j=1}^N\M_{s_j\times d}(\R)$
such that for every 
$j\in\{1,\ldots,N\}$ 
and every rational subspace $L'\subsetneq\R^{s_j}$,
we have
\begin{align}\label{MAINcleanCONVcor2rep2cond}
W_j\notin \M_{s_j\times d}(\Q)+(L')^d.
\end{align}
Then for any $\tM\in G'\setminus\fD_{\scrS}$,
$F_1\in\C_b(\tGamma\bs \tG)$ and $\lambda\in\Pac(\US)$,
we have
\begin{align}\label{MAINcleanCONVcor2rep2}
\int_{\US}F_1\Bigl(\tGamma \I_W\tM\varphi(R(\vecv)D_\rho)\Bigr)\,d\lambda(\vecv)
\to
\int_{\tGamma\bs \tG} F_1 \,d\tmu
\qquad\text{as }\:\rho\to0.
\end{align}
(Note that $\I_W\in\tG$ since $W\in\prod_{j=1}^N\M_{s_j\times d}(\R)$;
cf.\ \eqref{IVdef}.)
Starting from \eqref{MAINcleanCONVcor2rep2} and
applying the 
continuous mapping theorem
with the covering map $J$
(cf.\ \eqref{PHIGammaHdisc3new}),
we conclude:
For any element $W=(W_1,\ldots,W_N)$ in
$\prod_{j=1}^N L_j^d$
such that for every $j\in\{1,\ldots,N\}$
and every rational subspace $L'\subsetneq\R^{s_j}$,
\begin{align}\label{MAINcleanCONVcor2rep2condtransl}
\varphi_j^d(W_j)\notin \M_{s_j\times d}(\Q)+(L')^d,
\end{align}
and for any $\tM\in G'\setminus\fD_{\scrS}$,
$F_2\in\C_b(\Gamma_{\!H}\bs H)$ and $\lambda\in\Pac(\US)$,
\begin{align}\label{MAINcleanCONVcor2reptransl2}
\int_{\US}F_2\Bigl(\Gamma_{\!H}\I_W \I_X
\tM\varphi(R(\vecv)D_\rho)
\I_X^{-1}\Bigr)\,d\lambda(\vecv)
\to
\int_{\Gamma_{\!H}\bs H} F_2\,d\mu
\qquad\text{as }\:\rho\to0,
\end{align}
where $X:=(X_1,\ldots,X_N)\in\prod_{j=1}^N\M_{r_j\times d}(\Q)$.
In the above deduction we used the fact that
$J_*(\tmu)=\mu$,
the unique $H$-invariant probability measure on
$\Gamma_{\!H}\bs H$.

We wish to apply the last convergence relation 
with $W:=\tV-X$, i.e.\ $W_j:=\tV_j-X_j$ for each $j$.
This $W$ lies in $\prod_{j=1}^N L_j^d$,
and we proceed to verify that also the condition \eqref{MAINcleanCONVcor2rep2condtransl}
holds for every $j\in\{1,\ldots,N\}$
and every rational subspace $L'\subsetneq\R^{s_j}$.
Assume the opposite,
i.e.\ assume that there exists $j\in\{1,\ldots,N\}$ and a 
rational subspace $L'\subsetneq\R^{s_j}$
such that $\varphi_j^d(W_j)\in \M_{s_j\times d}(\Q)+(L')^d$.
Using the fact that
$\varphi_j^{-1}(\Q^{s_j})\subset\Q^{r_j}$
(since $\varphi_j^{-1}(\Z^{s_j})=L_j\cap\Z^{r_j}$),
we conclude $W_j\in\M_{r_j\times d}(\Q)+(L'')^d$,
where $L'':=\varphi_j^{-1}(L')$.
Since $W_j:=\tV_j-X_j$ and $X_j\in\M_{r_j\times d}(\Q)$,
it follows that
$\tV_j\in\M_{r_j\times d}(\Q)+(L'')^d$,
or equivalently
\begin{align}\label{MAINcleanCONVcor2rep2condtranslpf}
\tV_{j,1},\ldots,\tV_{j,d}\in\Q^{r_j}+L'',
\end{align}
where
$\tV_{j,1},\ldots,\tV_{j,d}\in\R^{r_j}$ are the column vectors of $\tV_j$.
But we have $L''\subsetneq L_j$ and $L''$ is a rational subspace of $\R^{r_j}$;
also $L_j=L_j^{(V_j)}=\fJ(\{\tV_{j,1},\ldots,\tV_{j,d}\})$
(cf.\ \eqref{LjDEF}),
and thus by Lemma \ref{IDCLcharLEM},
$L_j$ is the smallest rational subspace of $\R^{r_j}$ with the property that 
$\tV_{j,1},\ldots,\tV_{j,d}\in\Q^{r_j}+L_j$.
This is a contradiction against \eqref{MAINcleanCONVcor2rep2condtranslpf}.
This completes the proof that the condition
\eqref{MAINcleanCONVcor2rep2condtransl} is fulfilled for our choice of $W$.

Note also that $\I_W\in H$ since $W\in\prod_{j=1}^N L_j^d$,
and $\I_W\I_X=\I_{\tV}$.
Hence for any given $F\in\C_b(\Gamma_{\! H}\bs H)$,
also the function $F_2$ defined by
$F_2(\Gamma_{\! H}h)=F(\Gamma_{\!H} h\I_W^{-1})$
lies in $\C_b(\Gamma_{\! H}\bs H)$;
and applying \eqref{MAINcleanCONVcor2reptransl2} to this function $F_2$
we conclude that \eqref{nonunifTHM1pf1} holds for the given function $F$.
\end{proof}
\begin{remark}\label{XpsireplXrem}
If the point $V=\bigl( V_1,\ldots,V_N\bigr)\in\tTT$ satisfies
$\tr_{i_\psi}(V_{j_\psi})=\bn$ for some $\psi\in\Psi$,
then the statement of Theorem \ref{nonunifTHM1} also 
holds with ``$\XX^\psi$'' in place of ``$\XX$'',
i.e.\ for any $\tM\in G'\setminus\fD_{\scrS}$, $f\in\C_b(\XX^\psi)$ and $\lambda\in\Pac(\US)$
we have
\begin{align}\label{XpsireplXremres}
\int_{\US}f(x(V) \tM\varphi(R(\vecv)D_\rho))\,d\lambda(\vecv)
\to \int_{\XX^\psi} f \,d\overline{\omega^{(V)}}
\end{align}
as $\rho\to0$.
\end{remark}
\begin{proof}
By \cite[Lemma 4.26]{oK2002},
this follows from Theorem \ref{nonunifTHM1}
if we can only verify that 
$\overline{\omega^{(V)}}(\XX^\psi)=1$
and %
$x(V) \tM\varphi(R(\vecv)D_\rho)\in\XX^\psi$ for all $\rho$ and $\vecv$.
The first of these statements %
is immediate from
\eqref{XjiDEF}, Lemma \ref{oomegapinullLEM} and Lemma \ref{omegajVtriLEM}.
For the second statement,
note that $\XX^\psi$ is preserved by right multiplication of any $G'$-element;
hence it suffices to verify that $x(V)\in\XX^\psi$.
But writing $\psi=(j,i)$,
and taking $\tV_j\in\M_{r_j\times d}(\R)$ so that $V_j=\pi(\tV_j)$,
we have $\r_i(\tV_j)\in\Z^d$ since 
$\tr_{i}(V_j)=\bn$.
Now $x(V_j)=\Gamma_j\I_{\tV_j}$ and
$\Z^d\,\a_i(\I_{\tV_j})=\Z^d\,(\I,\r_i(\tV_j))=\Z^d$,
which is a lattice containing $\bn$.
Hence, by \eqref{XjiDEF},
$x(V)\in\XX^\psi$,
and the proof is complete.
\end{proof}

Let us note that 
Theorem \ref{HOMDYNintrononunifTHM}
is an immediate consequence of 
Theorem \ref{nonunifTHM1}:
\begin{proof}[Proof of Theorem \ref{HOMDYNintrononunifTHM}]
By \eqref{g0qdef},
$g_0^{(\vecq)}:=\I_{U^{(\vecq)}}\,\tM$,
where $\tM=(M_1,\ldots,M_N)$ with the $M_j$s coming from the fixed presentation of
$\scrP$ in \eqref{LpsiDEF}, \eqref{GENPOINTSET1}.
This $\tM$ lies outside $\fD_{\scrS}$,
by \eqref{GENPOINTSET1req}.
Hence the left hand side of \eqref{nonunifTHM1res}
equals the left hand side of \eqref{HOMDYNintrononunifTHMRES}
if we choose $V:=\pi(U^{(\vecq)})$,
i.e.\ $V=(V_1,\ldots,V_N)$ with $V_j=\pi(U_j^{(\vecq)})$.
With this choice,
$\omega_j^{(\vecq)}=\omega_j^{(V_j)}$ %
holds by definition,
and hence by Proposition \ref{SLRinvprobmeasLEM}
and
\eqref{muqDEF}, \eqref{oomegaDEF}, \eqref{omegaVDEF},
we have 
$\overline{\omega^{(V)}}=\mu^{(\vecq)}$,
meaning that also the right hand sides of \eqref{nonunifTHM1res}
and \eqref{HOMDYNintrononunifTHMRES} agree.
\end{proof}

\subsection{A first uniform result}
\label{firstunifresultSEC}
We will now prove a %
uniform equidistribution result,
Theorem~\ref{MAINUNIFPROP4new} below,
which, in combination with the non-uniform result of Theorem \ref{nonunifTHM1},
will play a key role in our proof of Theorem \ref{HOMDYNMAINTHM}.

For any $\psi\in\Psi$ and $j\in\{1,\ldots,N\}$ we pick an arbitrary point $\vecq\in\scrL_\psi$,
and define
\begin{align}\label{Yjpsidef}
\YY^\psi_j:=\pi(U_j^{(\vecq)})+(\SS_j^\psi)^d\subset\TT_j^d.
\end{align}
This is a connected component %
of the group 
$(\tSS_j^\psi)^d$;
cf.\ Lemma \ref{Ojpsi0welldefLEM}.
Note that 
$\YY^\psi_j$ is independent of the choice of $\vecq$,
since $\pi\bigl(U_j^{(\vecq)}\bigr)-\pi\bigl(U_j^{(\vecq')}\bigr)\in(\SS_j^{\psi})^d$ for any two $\vecq,\vecq'\in\scrL_\psi$,
as was noted in the proof of Lemma \ref{Ojpsi0welldefLEM}.
Let us also fix %
a matrix $X_j^\psi\in\M_{r_j\times d}(\Q)$ with the property that
\begin{align}\label{Xjpsireq}
\YY_j^\psi=\pi\bigl(X_j^\psi+(L_j^\psi)^d\bigr).
\end{align}
(Proof of existence: Choose any $\vecq\in\scrL_\psi$;
then by Lemma \ref{IDCLcharLEM} 
we can choose $X_j^\psi\in\M_{r_j\times d}(\Q)$ 
so that $U_j^{(\vecq)}-X_j^\psi\in \bigl(L_j^{(\vecq)}\bigr)^d$;
using \eqref{Yjpsidef} and Lemma \ref{LjqsubsetLjpsi0LEM}
it then follows that
\eqref{Xjpsireq} holds.)
Furthermore, we fix
a linear bijection $\varphi_j^\psi:L_j^\psi\stackrel{\sim}{\rightarrow}\R^{s}$
(with $s=s(\psi,j)=\dim L_j^\psi$)
with the property that
$\varphi_j^\psi(L_j^\psi\cap\Z^{r_j})=\Z^{s}$.
These matrices $X_j^\psi$ and bijections $\varphi_j^\psi$
will be kept fixed throughout the present section.
We also introduce the following map:
\begin{align}\label{tvarphijpsiDEF}
\widetilde{\varphi_j^\psi}:\YY_j^\psi\to\M_{s\times d}(\R/\Z);\qquad
\widetilde{\varphi_j^\psi}\bigl(\pi(X_j^\psi+W)\bigr)=\pi'\bigl((\varphi_j^\psi)^d(W)\bigr)
\quad \bigl(W\in(L_j^\psi)^d\bigr),
\end{align}
where $\pi'$ is the projection map $\M_{s\times d}(\R)\to\M_{s\times d}(\R/\Z)$.
It follows from the defining property of $\varphi_j^\psi$ that the map
$\widetilde{\varphi_j^\psi}$ is well-defined,
and that $\widetilde{\varphi_j^\psi}$ is a diffeomorphism from $\YY_j^\psi$ onto the torus $\M_{s\times d}(\R/\Z)$.

Next, for any $\psi\in\Psi$, 
$j\in\{1,\ldots,N\}$, %
$k\in\Z^+$ and $\eta>0$, we set
\begin{align}\label{DeltapsijketaDEF}
\Delta_{\psi,j,k}^{(\eta)}:=
\begin{cases}
\bigl(\widetilde{\varphi_j^\psi}\bigr)^{-1}\bigl(\Delta_{j,k}^{(\eta)}\bigr)
&\text{if }\:L_j^\psi\neq\{\bn\},
\\
\emptyset
&\text{if }\:L_j^\psi=\{\bn\},
\end{cases}
\end{align}
where if $L_j^\psi\neq\{\bn\}$, %
the set $\Delta_{j,k}^{(\eta)}\subset\M_{s\times d}(\R/\Z)$
is defined as on p.\ \pageref{DeltajkDEF},
but using $s=\dim L_j^\psi$ ($>0$)
in the place of $r_j$, so that ``$\TT_j^d$'' on p.\ \pageref{DeltajkDEF}
becomes %
$\M_{s\times d}(\R/\Z)$.

Next, for each $\psi\in\Psi$, we define 
\begin{align*}
\YY^\psi=\YY^\psi_1\times\cdots\times\YY^\psi_N\subset\tTT
\end{align*}
and for any $k\in\Z^+$ and $\eta>0$:
\begin{align}\label{DeltapsiketaDEF}
\Delta_{\psi,k}^{(\eta)}
:=\bigl\{V=\bigl( V_1,\ldots,V_N\bigr)\in\YY^\psi\col V_j\in\Delta_{\psi,j,k}^{(\eta)}
\text{ for some }j\bigr\}.
\end{align}

\begin{thm}\label{MAINUNIFPROP4new}
Let $\psi\in\Psi$,
$f\in\C_b(\XX^{\psi})$
and $\ve>0$ be given.
Then there exists some $k\in\Z^+$ such that for every 
$\lambda\in\Pac(\US)$, $\eta>0$ and 
$\tM\in G'\setminus\fD_{\scrS}$, 
there exists some $\rho_0\in(0,1)$ such that
\begin{align}\label{MAINUNIFPROP4newres2}
\biggl|\int_{\US}f\bigl(x(V)\tM\varphi(R(\vecv)D_\rho)\bigr)\,d\lambda(\vecv)-
\int_{\XX^{\psi}} f \,d\overline{\omega^{\psi}}\biggr|<\ve
\end{align}
for all $\rho\in(0,\rho_0)$ and all $V\in\YY^{\psi}\setminus \Delta_{\psi,k}^{(\eta)}$.
\end{thm}
To see that the statement of Theorem \ref{MAINUNIFPROP4new} makes sense,
note that for every $V\in\YY^\psi$ we have $x(V)\in\XX^\psi$
by the following Lemma \ref{xofYpsilem};
thus also $x(V)g\in\XX^\psi$ for all $g\in G'$;
and we also have $\overline{\omega^\psi}(\XX^\psi)=1$, by 
Lemma \ref{oomegapinullLEM} and Lemma \ref{omegajpsiLEM}.

\begin{lem}\label{xofYpsilem}
For any $\psi\in\Psi$ and $V\in\YY^\psi$, 
we have $\tr_{i_\psi}(V_{j_\psi})=\bn$ 
and $x(V)\in\XX^\psi$.
\end{lem}
\begin{proof}
Assume $V=(V_1,\ldots,V_N)\in\YY^\psi$.
Write $\psi=(j,i)$, and choose a point $\vecq\in\scrL_\psi$.
Take $W\in\M_{r_j\times d}(\R)$ so that $V_j=\pi(W)$.
It follows from $V_j\in\YY^\psi_j$ that
$W\in U_j^{(\vecq)}+(L_j^\psi)^d+\M_{r_j\times d}(\Z)$.
It follows from \eqref{LpsiDEF} and \eqref{Ujqdef}
that $\r_i\bigl(U_j^{(\vecq)}\bigr)\in\Z^d$,
and we noted in the proof of Lemma \ref{omegajpsiLEM}
that $L_j^\psi\perp\vece_i$.
Hence $\r_i(W)\in\Z^d$.
This shows that $\tr_{i}(V_{j})=\bn$,
and it also implies that the grid $\Z^d\,\a_i(\I_W)=\Z^d+W$ contains $\bn$, %
i.e.\ the point $x(V_j)=\Gamma_j\I_W$ lies in $\XX_j^{(i)}$.
Hence $x(V)\in\XX^\psi$.
\end{proof}

\begin{proof}[Proof of Theorem \ref{MAINUNIFPROP4new}]
Let $\psi\in\Psi$ be given.
Let us set $X:=(X_1^\psi,\ldots,X_N^\psi)$ and
\begin{align}\label{Hpsidef}
H:=\S_{L_1^\psi}^{X_1^\psi}(\R)\times\cdots\times\S_{L_N^\psi}^{X_N^\psi}(\R)
=\I_X\Bigl(\S_{L_1^\psi}(\R)\times\cdots\times\S_{L_N^\psi}(\R)\Bigr)\I_X^{-1}.
\end{align}
Recall that 
$\overline{\omega^\psi}=\overline{\omega^\psi_1}\otimes\cdots\otimes\overline{\omega^\psi_N}$.
We claim that $\overline{\omega^\psi}$ equals the unique 
$\I_X^{-1}H \I_X$-invariant probability measure
on $\Gamma\bs\Gamma H\I_X$.
To prove this, it suffices to prove that for each fixed $j$,
$\overline{\omega_j^\psi}$ equals the unique 
$\S_{L_j^\psi}(\R)$-invariant
probability measure on 
$\Gamma_j\bs\Gamma_j\I_{X_j^\psi}\S_{L_j^\psi}(\R)$.
To this end, fix an arbitrary matrix $W\in X_j^\psi+(L_j^\psi)^d$
which is generic in the sense that it lies 
outside $\M_{r_j\times d}(\Q)+L^d$ for every
rational subspace $L\subsetneq L_j^\psi$,
and set $V=\pi(W)\in\TT_j^d$.
Then $L_j^{(V)}=L_j^\psi$ by Lemma~\ref{IDCLcharLEM},
and hence by Proposition \ref{SLRinvprobmeasLEM},
$\overline{\omega_j^{(V)}}$ is the unique
$\S_{L_j^\psi}(\R)$-invariant
probability measure on 
$\Gamma_j\bs\Gamma_j\I_{W}\S_{L_j^\psi}(\R)=
\Gamma_j\bs\Gamma_j\I_{X_j^\psi}\S_{L_j^\psi}(\R)$.
Furthermore, we have
$\bigl(\SS_j^{(V)}\bigr)^\circ=\pi\bigl(L_j^{(V)}\bigr)=\pi(L_j^\psi)$ %
and thus
$V+\bigl(\SS_j^{(V)}\bigr)^{\circ\, d}=\pi(X_j^\psi+(L_j^\psi)^d)
=\YY_j^\psi=\pi\bigl(U_j^{(\vecq)}\bigr)+(\SS_j^\psi)^d$
for any $\vecq\in\scrL_\psi$.
This implies $\scrO_j^{(V)}=\scrO_j^{\psi}$
(cf.\ \eqref{Ojpsi0DEF} and \eqref{OjVDEF}),
and hence 
$\omega_j^{(V)}=\omega_j^\psi$.
This completes the proof of the claim.

Set $\Gamma_H:=\Gamma\cap H$,
and let $\mu$ be the $H$-invariant probability measure on
$\Gamma_H\bs H$.

Next, for each $j$ let us write $s_j:=\dim L_j^\psi$,
and recall that we have fixed a 
linear bijection $\varphi_j^\psi:L_j^\psi\stackrel{\sim}{\rightarrow}\R^{s_j}$
with the property that
$\varphi_j^\psi(L_j^\psi\cap\Z^{r_j})=\Z^{s_j}$.
Set
\begin{align*}
\tG:=\S_{s_1}(\R)\times\cdots\times\S_{s_N}(\R),
\end{align*}
and let $\Phi$ be the Lie group isomorphism
\begin{align*}
\Phi:=\S_{\varphi_1^\psi}^{X_1^\psi}\times\cdots\times\S_{\varphi_N^\psi}^{X_N^\psi}:
\quad H\stackrel{\sim}{\rightarrow}\tG
\end{align*}
(using the notation from \eqref{SphiXDEFnew}).
For each $j$, choose $q_j\in\Z^+$ so that $X_j^\psi\in q_j^{-1}\Z$,
and let $\Gamma_j'$ be the principal congruence subgroup of $\SL_d(\Z)$ of order $q_j$ (cf.\ \eqref{princcongrsubgpdef});
then define $\tGamma_j$ and $\tGamma$ as in \eqref{tGammachoice}.
By an argument entirely similar to the discussion in the proof of
Theorem \ref{nonunifTHM1} (leading up to \eqref{MAINcleanCONVcor2reptransl2}),
one verifies that Theorem \ref{MAINcleanUNIFCONVROTthm}
applied to the homogeneous space $\tGamma\bs\tG$,
yields the following result:
For any $f\in\C_b(\Gamma_H\bs H)$
and $\ve>0$,
there exists some $k\in\Z^+$ such that for every 
$\lambda\in\Pac(\US)$, $\eta>0$ and $\tM\in G'\setminus\fD_{\scrS}$,
there exists some $\rho_0\in(0,1)$ such that
\begin{align}\label{MAINcleanUNIFCONVpropRESreptransl}
\biggl|\int_{\R^{d-1}}f\Bigl(\Gamma_H\I_W \I_X
\tM\varphi(R(\vecv)D_\rho)
\I_X^{-1}\Bigr)\,d\lambda(\vecv)
-
\int_{\Gamma_H\bs H} f\,d\mu\biggr|
<\ve
\end{align}
for all $\rho\in(0,\rho_0)$ and all 
$W=\bigl( W_1,\ldots,W_N\bigr)\in\prod_{j=1}^N(L_j^\psi)^d$
satisfying
$\pi'_j((\varphi_j^\psi)^d(W_j))\notin\Delta_{j,k}^{(\eta)}$ for every $j\in\{1,\ldots,N\}$ with $s_j>0$.
In the last condition, 
the set $\Delta_{j,k}^{(\eta)}$ is defined exactly as on p.\ \pageref{DeltajkDEF}
but using the dimension $s_j$ in place of $r_j$
(so that ``$\TT_j^d$'' on p.\ \pageref{DeltajkDEF}
becomes %
$\M_{s_j\times d}(\R/\Z)$),
and $\pi'_j$ is the projection $\M_{s_j\times d}(\R)\to\M_{s_j\times d}(\R/\Z)$. 

Finally, let us write $\tV:=W+X$ %
in the previous result;
thus $\tV_j:=W_j+X_j^\psi\in X_j^\psi+(L_j^\psi)^d$ for each $j$,
and also $\I_W\I_X=\I_{\tV}$ in \eqref{MAINcleanUNIFCONVpropRESreptransl}.
Then in view of the definitions
\eqref{DeltapsijketaDEF} and \eqref{DeltapsiketaDEF},
the condition on $W$ 
is equivalent to 
$\pi(\tV)\notin\Delta_{\psi,k}^{(\eta)}$.
Hence, 
by an argument completely
similar to the proof that \eqref{nonunifTHM1pf1} suffices to give Theorem \ref{nonunifTHM1},
the result stated around \eqref{MAINcleanUNIFCONVpropRESreptransl}
implies the statement of Theorem \ref{MAINUNIFPROP4new}.
\end{proof}

Next let us note that by combining
Theorem \ref{MAINUNIFPROP4new} with Theorem \ref{nonunifTHM1},
we immediately obtain a variant of %
Theorem \ref{MAINUNIFPROP4new},
where the limit measure 
$\overline{\omega^\psi}$
in \eqref{MAINUNIFPROP4newres2}
is replaced by $\overline{\omega^{(V)}}$:

\begin{thm}\label{MAINUNIFPROP4COR}
For any $\psi\in\Psi$,
$f\in\C_b(\XX^{\psi})$
and $\ve>0$,
there exists some $k\in\Z^+$ such that for every 
$\lambda\in\Pac(\US)$,
$\eta>0$
and every $\tM\in G'\setminus\fD_{\scrS}$, 
there exists some $\rho_0\in(0,1)$ such that
\begin{align}\label{MAINUNIFPROP4CORres2}
\biggl|\int_{\US}f\bigl(x(V)\tM\varphi(R(\vecv)D_\rho)\bigr)\,d\lambda(\vecv)-
\int_{\XX^{\psi}} f \,d\overline{\omega^{(V)}}\biggr|<\ve
\end{align}
for all $\rho\in(0,\rho_0)$ and all $V\in\YY^\psi\setminus\Delta_{\psi,k}^{(\eta)}$.
\end{thm}

\begin{proof}
Given $\psi,f,\ve$,
take $k$ as in Theorem \ref{MAINUNIFPROP4new}.
Now also let $\lambda\in\Pac(\US)$, $\eta>0$ and $\tM\in G'\setminus\fD_{\scrS}$
be given.
Take $\rho_0\in(0,1)$ as in Theorem \ref{MAINUNIFPROP4new},
i.e.\ so that \eqref{MAINUNIFPROP4newres2} holds
for all $\rho\in(0,\rho_0)$ and all $V\in\YY^{\psi}\setminus \Delta_{\psi,k}^{(\eta)}$.
Now for any fixed $V\in\YY^\psi$
we have $\tr_{i_\psi}(V_{j_\psi})=\bn$ 
by Lemma \ref{xofYpsilem},
and so by Theorem \ref{nonunifTHM1} and Remark \ref{XpsireplXrem},
the convergence in \eqref{XpsireplXremres} holds as $\rho\to0$.
Combining this fact with \eqref{MAINUNIFPROP4newres2} gives
\begin{align}\label{MAINUNIFPROP4CORpf1}
\biggl|\int_{\XX^{\psi}} f \,d\overline{\omega^{(V)}}-
\int_{\XX^{\psi}} f \,d\overline{\omega^{\psi}}\biggr|\leq\ve,
\qquad\forall V\in\YY^{\psi}\setminus \Delta_{\psi,k}^{(\eta)}.
\end{align}
Combining \eqref{MAINUNIFPROP4newres2} and \eqref{MAINUNIFPROP4CORpf1},
we conclude that \eqref{MAINUNIFPROP4CORres2}, with $2\ve$ in place of $\ve$, 
holds for all
$\rho\in(0,\rho_0)$ and all $V\in\YY^\psi\setminus\Delta_{\psi,k}^{(\eta)}$.
\end{proof}

\subsection{Proof of the uniformity in [P2]; later steps}
\label{P2prooflaterSEC}
Note that for any $\psi\in\Psi$ and $\vecq\in\scrL_\psi$
we have 
$\pi\bigl(U_j^{(\vecq)}\bigr)\in\YY_j^\psi$ for all $j\in\{1,\ldots,N\}$,
by \eqref{Yjpsidef},
and hence $\pi\bigl(U^{(\vecq)}\bigr)\in\YY^\psi$.
The following proposition shows that 
by taking $\eta$ %
small,
we can ensure that the density of points $\vecq\in\scrL_\psi$
for which $\pi\bigl(U^{(\vecq)}\bigr)$ falls inside the
``singular'' set $\Delta_{\psi,k}^{(\eta)}$,
is %
small.

\begin{prop}\label{UksetsaregoodPROP}
For every $\psi\in\Psi$ and $k\in\Z^+$ we have
\begin{align}\label{UksetsaregoodPROPres}
\lim_{\eta\to 0}
\limsup_{T\to\infty}
\frac{\#\bigl\{\vecq\in\scrL_{\psi}\cap\scrB_T^d\col
\pi\bigl(U^{(\vecq)}\bigr)\in \Delta_{\psi,k}^{(\eta)}\bigr\}}{T^d}=0.
\end{align}
\end{prop}
\begin{proof}
It follows from the definitions 
in \eqref{DeltapsijketaDEF}, \eqref{DeltapsiketaDEF}
and on p.\ \pageref{DeltajkDEF}
that $\Delta_{\psi,k}^{(\eta)}$ is a union of sets of the form
\begin{align}\label{UksetsaregoodPROPpf4}
\Delta^{(\eta)}_{j,A,L}:=\bigl\{V=\bigl( V_1,\ldots,V_N\bigr)\in\YY^\psi\col V_j-\pi(X_j^\psi)\text{ is $\eta$-near }\pi(A+L^d)\bigr\},
\end{align}
the union being taken over a finite set of triples
$\langle j,A,L\rangle$
with $j\in\{1,\ldots,N\}$,
$L_j^\psi\neq\bn$,
$A\in (L_j^\psi)^d\cap M_{r_j\times d}(\Q)$
and $L$ being a rational subspace of $L_j^\psi$,
$L\neq L_j^\psi$.
Note that in \eqref{UksetsaregoodPROPpf4}, ``$\eta$-near''
refers to the Riemannian metric on $\YY_j^\psi$
induced from the standard Euclidean metric on $\M_{s\times d}(\R/\Z)$ (with $s=\dim L_j^\psi$)
via the diffeomorphism in \eqref{tvarphijpsiDEF}.

It follows that it suffices to prove that for any fixed such triple
$\langle j,A,L\rangle$,
\begin{align}\label{UksetsaregoodPROPpf1}
\lim_{\eta\to0}\limsup_{T\to\infty}
\frac{\#\bigl\{\vecq\in\scrL_{\psi}\cap\scrB_T^d\col
\pi\bigl(U_j^{(\vecq)}\bigr)\in \Delta_{j,A,L}^{(\eta)}\bigr\}}{T^d}=0.
\end{align}
But it follows from
the formula for $U_j^{(\vecq)}$
in \eqref{Ujqformula}
and 
Weyl equidistribution
that if we let $Z$ be the closed subgroup of $\TT_j^d$
which is the closure of the set
$\bigl\{\pi(U_j^{(\vecq)}-W_j^\psi)\col\vecq\in\scrL_\psi\bigr\}$,
and if $\nu$ is the Haar measure on $Z$ normalized so that $\nu(Z)=1$,
then for any fixed closed subset $C\subset Z$,
\begin{align}\label{UksetsaregoodPROPpf5}
\limsup_{T\to\infty}\frac{\#\{\vecq\in\scrL_\psi\cap\scrB_T^d\col\pi(U_j^{(\vecq)}-W_j^\psi)\in C\}}{\nbar_{\psi}\vol(\scrB_T^d)}\leq\nu(C).
\end{align}
Note that it follows from \eqref{Yjpsidef} and \eqref{Xjpsireq}
that $\pi(U_j^{(\vecq)}-W_j^\psi)\in\pi(X_j^\psi-W_j^\psi)+(\SS_j^\psi)^d$ for all $\vecq\in\scrL_\psi$;
hence also $Z\subset\pi(X_j^\psi-W_j^\psi)+(\SS_j^\psi)^d$,
and since $0\in Z$ it follows that
$\pi(X_j^\psi-W_j^\psi)\in(\SS_j^\psi)^d$
and $Z\subset(\SS_j^\psi)^d$.
Using \eqref{UksetsaregoodPROPpf5} and \eqref{UksetsaregoodPROPpf4},
it follows that in order to prove
\eqref{UksetsaregoodPROPpf1},
it suffices to prove that $\nu(C_\eta)\to0$ as $\eta\to0$,
where $C_\eta$ is the closed $\eta$-neighborhood of $\pi(X_j^\psi-W_j^\psi+A+L^d)$.
But we have $\cap_{k=1}^\infty C_{1/k}=\pi(X_j^\psi-W_j^\psi+A+L^d)$;
hence in fact it suffices to prove that 
\begin{align}\label{UksetsaregoodPROPpf6}
\nu(\pi(X_j^\psi-W_j^\psi+A+L^d))=0.
\end{align}

Fix $q\in\Z^+$ so that $X_j^\psi+A\in q^{-1}\M_{r_j\times d}(\Z)$.
Let $Z^\circ$ be the identity component of $Z$;
then $Z$ is a union of a finite number of $Z^\circ$-cosets.
Now if $Z'$ is any of these cosets, we may argue as follows.
Set 
\begin{align*}
\scrL':=\{\vecq\in\scrL_\psi\col \pi(U_j^{(\vecq)}-W_j^\psi)\in Z'\};
\end{align*}
this is a subgrid of $\scrL_\psi$.
Hence by Lemma \ref{Ljpsi0keypropLEM2},
$L_j^\psi=\fL(\{U_{j,\ell}^{(\vecq)}\col\vecq\in\scrL',\:\ell\in\{1,\ldots,d\}\})$,
and so by Lemma \ref{IDCLcharLEM}, since $L$ is a rational subspace of $L_j^\psi$
and $L\neq L_j^\psi$,
there exist some $\vecq\in\scrL'$ and $\ell\in\{1,\ldots,d\}$
for which $U_{j,\ell}^{(\vecq)}\notin q^{-1}\Z^{r_j}+L$.
This implies that $U_j^{(\vecq)}\notin q^{-1}\M_{r_j\times d}(\Z)+L^d$,
and in particular $\pi(U_j^{(\vecq)})\notin\pi(X_j^\psi+A+L^d)$.
But
$\pi(U_j^{(\vecq)}-W_j^\psi)\in Z'$.
Hence we conclude that
\begin{align}\label{UksetsaregoodPROPpf7}
Z'\not\subset\pi(X_j^\psi-W_j^\psi+A+L^d).
\end{align}
Now both $Z'$ and $\pi(X_j^\psi-W_j^\psi+A+L^d)$ are translates of closed connected subtori of $(\SS_j^\psi)^d$;
hence \eqref{UksetsaregoodPROPpf7}
implies that $Z'\cap\pi(X_j^\psi-W_j^\psi+A+L^d)$
is either empty or a submanifold of codimension $\geq1$ of $Z'$.
Therefore %
\begin{align}\label{UksetsaregoodPROPpf8}
\nu(Z'\cap\pi(X_j^\psi-W_j^\psi+A+L^d))=0.
\end{align}

We have proved that \eqref{UksetsaregoodPROPpf8} holds for every $Z^\circ$-coset $Z'$ in $Z$;
hence \eqref{UksetsaregoodPROPpf6} holds, and the lemma is proved.
\end{proof}

Next we prove an auxiliary lemma 
concerning the type of uniform convergence which we require.
We define the \textit{upper density} of a subset $\scrZ\subset\R^d$ 
to be the number\footnote{To conform with the definition of asymptotic density in \eqref{density000},
it would be more natural to divide with $\vol(\scrB_T^d)$ 
instead of $T^d$
in \eqref{upperdensityDEF};
however using \eqref{upperdensityDEF} makes some 
computations in the following slightly cleaner.}
\begin{align}\label{upperdensityDEF}
\limsup_{T\to\infty}T^{-d}\#(\scrZ\cap\scrB_T^d).
\end{align}
\begin{lem}\label{densityzeroconstrLEM}
Let $\scrQ$ be a locally finite subset of $\R^d$, let $J$ be a countable set,
and let a function $F:J\times\scrQ\times(0,1)\to\R$
be given.
Assume that
\\[3pt]\rule{0pt}{0pt}\hspace{15pt}
\parbox{420pt}{(i) For any fixed $j\in J$ and $\vecq\in\scrQ$,
$F(j,\vecq,\rho)\to 0$ as $\rho\to0$,}
\\[3pt]
and
\\[3pt]\rule{0pt}{0pt}\hspace{15pt}
(ii) \parbox[t]{400pt}{for any $j\in J$ and $\ve,\ve'>0$, there exist $\rho_0\in(0,1)$
and a subset $\scrZ\subset\scrQ$ 
of upper density $\leq\ve'$,
such that 
$|F(j,\vecq,\rho)|<\ve$
\text{for all $\rho\in(0,\rho_0)$ and all $\vecq\in\scrQ\setminus\scrZ$.}}
\\[6pt]
Then for any decreasing function $\scrT:(0,1)\to\R^+$,
there exists a subset $\scrE\subset\scrQ$ of 
density zero,
such that for each fixed $j\in J$,
we have
$F(j,\vecq,\rho)\to0$ as $\rho\to0$,
uniformly over all $\vecq\in\scrQ\cap\scrB_{\scrT(\rho)}^d\setminus\scrE$.
\end{lem}
\begin{proof}
We assume, without loss of generality, that $J=\Z^+$.
For any $j,k,m\in\Z^+$, by (ii) there exist
$\rho_0'\in(0,1)$ and a subset $\scrZ'\subset\scrQ$
of upper density $\leq 2^{-m-1}$ such that
$|F(j,\vecq,\rho)|<2^{-k}$
for all $\rho\in(0,\rho_0')$ and all $\vecq\in\scrQ\setminus\scrZ'$.
We may then choose $T_0>0$ so that
$T^{-d}\#(\scrZ'\cap\scrB_T^d)<2^{-m}$ for all $T\geq T_0$,
and set $\scrZ=\scrZ(j,k,m):=\scrZ'\setminus\scrB_{T_0}^d$.
Now $\scrZ'\setminus\scrZ$ is finite, 
and for each $\vecq\in\scrZ'\setminus\scrZ$ there exists some $\rho_0^{(\vecq)}\in(0,1)$
such that $|F(j,\vecq,\rho)|<2^{-k}$ for all $\rho\in(0,\rho_0^{(\vecq)})$, by (i).
Set 
\begin{align*}
\rho_0=\rho_0(j,k,m):=\min(\{\rho_0'\}\cup\{\rho_0^{(\vecq)}\col\vecq\in\scrZ'\setminus\scrZ\}).
\end{align*}
Now for any $j,k,m\in\Z^+$, we have constructed a
number $\rho_0(j,k,m)\in(0,1)$ and a subset $\scrZ(j,k,m)\subset\scrQ$,
and it is clear from our construction that
\begin{align}
\#(\scrZ(j,k,m)\cap\scrB_T^d)<2^{-m}T^d,\qquad\forall T>0
\end{align}
and that
\begin{align}\label{densityzeroconstrLEMpf8}
|F(j,\vecq,\rho)|<2^{-k},\qquad\forall \rho\in(0,\rho_0(j,k,m)),\: 
\vecq\in\scrQ\setminus\scrZ(j,k,m).
\end{align}
Let us now also set, for any $m\in\Z^+$,
\begin{align*}
\tscrZ(m):=\bigcup_{j\in\Z^+}\bigcup_{k\in\Z^+}\scrZ(j,k,j+k+m).
\end{align*}
Then %
\begin{align}\label{densityzeroconstrLEMpf2}
\#(\tscrZ(m)\cap\scrB_T^d)\leq\sum_{j,k\in\Z^+}2^{-j-k-m}T^d=2^{-m}T^d,
\qquad\forall T>0.
\end{align}

Now let a decreasing function $\scrT:(0,1)\to\R^+$ be given,
as in the statement of the lemma.
For any $m\in\Z^+$ we set
\begin{align}\label{densityzeroconstrLEMpf6}
\trho_0(m):=\min\{\rho_0(j,k,j+k+m)\col j,k\in\{1,\ldots,m\}\}.
\end{align}
Next choose numbers $1\leq R_1<R_2<\cdots$ so that
$R_m\geq\scrT(\trho_0(m+1))$ 
and $R_{m+1}\geq 2R_m$ for all $m$.
Finally set 
\begin{align*}
\scrE:=\bigcup_{m=1}^\infty\bigl(\tscrZ(m)\cap\scrB_{R_m}^d\bigr).
\end{align*}
Let us prove that this set $\scrE$ has density zero.
Given any $T\geq R_1$, choosing $n$ so that $R_n\leq T<R_{n+1}$
we have
\begin{align*}
\#(\scrE\cap\scrB_T^d)
&\leq\sum_{m=1}^{n}\#(\tscrZ(m)\cap\scrB_{R_m}^d)
+\sum_{m=n+1}^{\infty}\#(\tscrZ(m)\cap\scrB_{T}^d)
\leq\sum_{m=1}^{n} 2^{-m}R_m^d
+\sum_{m=n+1}^{\infty}2^{-m}T^d,
\end{align*}
by \eqref{densityzeroconstrLEMpf2}.
Here 
for all $m\leq n$ we have
$R_m\leq 2^{m-n}R_n$
and hence $R_m^d\leq 2^{2(m-n)}T^d$, since $d\geq2$.
Plugging in these bounds we get
$\#(\scrE\cap\scrB_T^d)<3\cdot 2^{-n}T^d$.
Hence since $n\to\infty$ as $T\to\infty$,
we conclude that $\scrE$ has density zero.

It remains to %
prove the uniform convergence stated in the lemma.
Thus let $j\in\Z^+$ and $\ve>0$ be given.
By (i), we can take $\rho_1\in(0,1)$ so small that
$|F(j,\vecq,\rho)|<\ve$ for all 
$\vecq\in\scrQ\cap\scrB_{R_j}^d$ and all
$\rho\in(0,\rho_1)$.
Choose $k\in\Z^+$ so that $2^{-k}<\ve$,
and set
\begin{align}\label{densityzeroconstrLEMpf4}
\rho_0:=\min\bigl(\{\rho_1\}\cup\{\rho_0(j,k,j+k+m)\col m\in\{1,\ldots,k\}\}\bigr).
\end{align}
Now for any $\rho\in(0,\rho_0)$
and $\vecq\in \scrQ\cap\scrB_{\scrT(\rho)}^d\setminus\scrE$
we can argue as follows:
If $\|\vecq\|<R_j$ then $|F(j,\vecq,\rho)|<\ve$
since $\rho<\rho_0\leq\rho_1$.
Next assume instead that $\|\vecq\|\geq R_j$,
and let $m>j$ be the minimal positive integer such that $\|\vecq\|<R_m$.
Then $\scrT(\rho)>\|\vecq\| \geq R_{m-1}\geq\scrT(\trho_0(m))$,
and so $\rho<\trho_0(m)$.
If $m\geq k$ then $\trho_0(m)\leq\rho_0(j,k,j+k+m)$
by \eqref{densityzeroconstrLEMpf6};
on the other hand if $m<k$ then
$\rho_0\leq\rho_0(j,k,j+k+m)$ by \eqref{densityzeroconstrLEMpf4}.
Hence we always have $\rho<\rho_0(j,k,j+k+m)$.
Furthermore,
$\vecq\notin\scrE$ and $\|\vecq\|<R_m$ implies
$\vecq\notin\tscrZ(m)$,
and in particular $\vecq\notin\scrZ(j,k,j+k+m)$.
Hence by \eqref{densityzeroconstrLEMpf8},
$|F(j,\vecq,\rho)|<2^{-k}<\ve$.
Summing up, we have proved:
\begin{align}\label{densityzeroconstrLEMpf7}
|F(j,\vecq,\rho)|<\ve,\qquad
\forall \rho\in(0,\rho_0),
\:\vecq\in \scrQ\cap\scrB_{\scrT(\rho)}^d\setminus\scrE.
\end{align}
This completes the proof of the uniform convergence
stated in the lemma.
\end{proof}

We are now ready to prove Theorem \ref{HOMDYNMAINTHM}.

\begin{proof}[Proof of Theorem \ref{HOMDYNMAINTHM}]
Let $\psi$ and $\scrT$ be given as in the statement of 
Theorem \ref{HOMDYNMAINTHM}.
\label{pfHOMDYNMAINTHMstart}

Let $J_1$ be a fixed countable dense subset of $\C_c(\XX^\psi)$ 
with respect to the uniform norm.
We equip $\Pac(\US)$ 
the metric $d$ defined by
$d(\lambda_1,\lambda_2):=\int_{\US}|\lambda_1'-\lambda_2'|\,d\vecv$,
where $\lambda_j'$ is the 
density of $\lambda_j$ with respect to $\sigma$,
that is,
$\lambda_j'\in\L^1(\US)$ 
and $\lambda_j=\lambda_j'\, d\sigma$ ($j=1,2$).
Now let $J_2$ be a fixed countable dense subset of $\Pac(\US)$
with respect to the metric $d$;
such a set exists since %
$\C(\US)$ is dense in $\L^1(\US)$
\cite[Thm.\ 3.14]{wR87}.

We will apply Lemma \ref{densityzeroconstrLEM} with
$\scrQ=\scrL_\psi$, 
$J:=J_1\times J_2$,
and with the function 
$F:J\times\scrL_\psi\times(0,1)\to\R$
defined by
\begin{align}\label{MAINUNIFPROP2pf10}
F\bigl(\langle f,\lambda\rangle,\vecq,\rho\bigr):=
\int_{\US}f\bigl(\Gamma g_0^{(\vecq)}\varphi(R(\vecv)D_\rho)\bigr)\,d\lambda(\vecv)
-\int_{\XX^\psi} f \,d\overline{\omega^{(\vecq)}}.
\end{align}
Recall that $\overline{\omega^{(\vecq)}}(\XX^{\psi})=1$ 
for every $\vecq\in\scrL_\psi$,
by Lemmas \ref{omegajqpiinv0LEM} and \ref{oomegapinullLEM};
hence the 
integral over $\XX^\psi$ in \eqref{MAINUNIFPROP2pf10}
may just as well be taken over all $\XX$.
Recall also that for any $\vecq\in\R$ %
we have $\overline{\omega^{(\vecq)}}=\mu^{(\vecq)}$,
by Proposition \ref{SLRinvprobmeasLEM}
and \eqref{muqDEF}, \eqref{omegaqDEF}, \eqref{oomegaDEF}.
Hence by Theorem \ref{HOMDYNintrononunifTHM}
we have
$F\bigl(\langle f,\lambda\rangle,\vecq,\rho\bigr)\to0$
for any fixed 
$\langle f,\lambda\rangle\in J$
and $\vecq\in\scrL_\psi$,
i.e.\ the assumption (i) in Lemma \ref{densityzeroconstrLEM} holds.

We next verify that also 
the assumption (ii) in Lemma \ref{densityzeroconstrLEM} holds.
Thus let $\langle f,\lambda\rangle\in J$ and $\ve,\ve'>0$ be given.
Choose $k\in\Z^+$ as in Theorem \ref{MAINUNIFPROP4COR}  %
(for our given $\psi,f,\ve$).
By Proposition~\ref{UksetsaregoodPROP},
we can now fix some $\eta>0$ such that the set
$\scrZ:=\bigl\{\vecq\in\scrL_{\psi}\col\pi\bigl(U^{(\vecq)}\bigr)\in \Delta_{\psi,k}^{(\eta)}\bigr\}$
has upper density $<\ve'$.
We keep $\tM\in G'\setminus\fD_{\scrS}$
fixed as in \eqref{tMdef},
for our given, fixed union of grids $\scrP$.
Because of our choice of $k$,
we may now fix $\rho_0\in(0,1)$
in such a way that
\eqref{MAINUNIFPROP4CORres2}
holds for all $\rho\in(0,\rho_0)$ and
all $V\in\YY^{\psi}\setminus \Delta_{\psi,k}^{(\eta)}$,
for our chosen $\tM,\psi,f,\lambda,\ve,\eta$.
Now for any $\vecq\in\scrL_\psi\setminus\scrZ$,
we may apply \eqref{MAINUNIFPROP4CORres2}
with $V:=\pi\bigl(U^{(\vecq)}\bigr)\in\tTT$;
indeed, for this $V$ we have
$V\in\YY^{\psi}\setminus \Delta_{\psi,k}^{(\eta)}$
since $\vecq\notin\scrZ$;
also $\Gamma g_0^{(\vecq)}=x(V)\tM$
(by \eqref{g0qdef})
and $\omega^{(V)}=\omega^{(\vecq)}$.
Hence we conclude that for all 
$\vecq\in\scrL_\psi\setminus\scrZ$
and all $\rho\in(0,\rho_0)$ we have
$\bigl|F\bigl(\langle f,\lambda\rangle,\vecq,\rho\bigr)\bigr|<\ve$.
Hence assumption (ii) in Lemma~\ref{densityzeroconstrLEM}
is indeed fulfilled.

It now follows from 
Lemma~\ref{densityzeroconstrLEM}
that there exists a 
subset $\scrE\subset\scrL_\psi$ of  density zero
such that 
for any $f\in J_1$ and $\lambda\in J_2$
we have 
$F\bigl(\langle f,\lambda\rangle,\vecq,\rho\bigr)\to0$
as $\rho\to0$, uniformly over all $\vecq\in\scrL_\psi\cap\scrB_{\scrT(\rho)}^d\setminus\scrE$,
viz.,
the uniform convergence in
\eqref{HOMDYNMAINTHMres} in Theorem \ref{HOMDYNMAINTHM}
holds. %
To complete the proof of Theorem \ref{HOMDYNMAINTHM},
we will give a (standard) approximation argument to show that
the uniform convergence in
\eqref{HOMDYNMAINTHMres} in fact
holds for arbitrary $f\in\C_b(\XX^\psi)$ and $\lambda\in\Pac(\US)$.

Note that equation \eqref{MAINUNIFPROP2pf10} defines
$F\bigl(\langle f,\lambda\rangle,\vecq,\rho\bigr)$
for arbitrary
$f\in\C_b(\XX^\psi)$ and $\lambda\in\Pac(\US)$;
and for any 
$f_1,f_2\in\C_b(\XX^\psi)$,
$\lambda_1,\lambda_2\in\Pac(\US)$,
$\vecq\in\scrL_\psi$ and $\rho\in(0,1)$
we have,
with $\|\cdot\|_u$ denoting the uniform norm on 
$\C_b(\XX^\psi)$:
\begin{align*}
\Bigl|F\bigl(\langle f_1,\lambda_1\rangle,\vecq,\rho\bigr)
&-F\bigl(\langle f_2,\lambda_2\rangle,\vecq,\rho\bigr)\Bigr|
\\
&\leq 2\|f_1-f_2\|_u+
\int_{\US}\bigl|f_1\bigl(\Gamma g_0^{(\vecq)}\varphi(R(\vecv)D_\rho)\bigr)\bigr|\cdot|\lambda_1'(\vecv)-\lambda_2'(\vecv)|\,d\sigma(\vecv)
\\
&\leq 2\|f_1-f_2\|_u+\|f_1\|_u\cdot d(\lambda_1,\lambda_2).
\end{align*}
Using this bound,
and the fact that $J_1$ and $J_2$ are dense in $\C_c(\XX^\psi)$
and $\Pac(\US)$, respectively,
it is immediate to extend
the uniform convergence in
\eqref{HOMDYNMAINTHMres}
(with the subset $\scrE\subset\scrL_\psi$ fixed once and for all)
from $f\in J_1$ and $\lambda\in J_2$
to arbitrary $f\in\C_c(\XX^\psi)$
and $\lambda\in\Pac(\US)$.

It remains to extend to arbitrary functions $f\in\C_b(\XX^\psi)$.
Thus let $f\in \C_b(\XX^\psi)$, $\lambda\in\Pac(\US)$
and $\ve>0$ be given.
Set $B:=\|f\|_u$; we may assume $B>0$ since otherwise $f$ is identically zero.
Let $\nu$ be the $\SL_d(\R)$ invariant probability measure on 
$\SL_d(\Z)\bs\SL_d(\R)$;
fix a compact subset $K'$ of
$\SL_d(\Z)\bs\SL_d(\R)$ with $\nu(K')>\bigl(1-\ve/(4B)\bigr)^{1/N}$,
and then set $K:=\XX^\psi\cap\prod_{j=1}^N\tiota_j^{-1}(K')$,
where $\tiota_j$ is the projection from $\XX_j$ to $\SL_d(\Z)\bs\SL_d(\R)$ \footnote{This is the map which we called
``$\tiota$'' in \eqref{tiotaDEF}; we now call it $\tiota_j$ for clarity.}.
Then $K$ is a compact subset of $\XX^\psi$,
and for every $\vecq\in\scrL_\psi$ we have
\begin{align}
\overline{\omega^{(\vecq)}}(K)=
\overline{\omega^{(\vecq)}}\biggl(\prod_{j=1}^N\tiota_j^{-1}(K')\biggr)
=\prod_{j=1}^N\overline{\omega_j^{(\vecq)}}\bigl(\tiota_j^{-1}(K')\bigr)
=\nu(K')^N>1-\frac{\ve}{4B},
\end{align}
where we first used the fact that
$\overline{\omega^{(\vecq)}}(\XX^{\psi})=1$
(by Lemmas \ref{omegajqpiinv0LEM} and \ref{oomegapinullLEM}),
and then used Lemma~\ref{rhosoomegaeqnuLEM}.
Next fix a function
$h\in\C_c(\XX_\psi)$ with $0\leq h\leq 1$ and $h|_K\equiv1$.
By the uniform convergence which we have already proved, 
there exists $\rho_0\in(0,1)$
such that 
for all $\rho\in(0,\rho_0)$ and $\vecq\in\scrL_\psi\cap\scrB_{\scrT(\rho)}^d\setminus\scrE$
we have $\bigl|F\bigl(\langle h,\lambda\rangle,\vecq,\rho\bigr)\bigr|<\ve/(4B)$;
but also 
$\int_{\XX^\psi} h \,d\overline{\omega^{(\vecq)}}\geq 
\overline{\omega^{(\vecq)}}(K)>1-\ve/(4B)$,
and hence
$\int_{\US}h\bigl(\Gamma g_0^{(\vecq)}\varphi(R(\vecv)D_\rho)\bigr)\,d\lambda(\vecv)>1-\ve/(2B)$.
It follows that
\begin{align*}
\lambda\bigl(\bigl\{\vecv\in\US\col \Gamma g_0^{(\vecq)}\varphi(R(\vecv)D_\rho)\in\supp(h)\bigr\}\bigr)>1-\frac{\ve}{2B},
\end{align*}
for all $\rho\in(0,\rho_0)$ and $\vecq\in\scrL_\psi\cap\scrB_{\scrT(\rho)}^d\setminus\scrE$.

Next fix a function $h_1\in\C_c(\XX^\psi)$
with $0\leq h_1\leq 1$ and $h_1|_{\supp(h)}\equiv1$,
and set $f_1:=h_1f\in\C_c(\XX^\psi)$.
Then note that for all $x\in\XX^\psi$ we have
\begin{align*}
\bigl|f(x)-f_1(x)\bigr|
\leq\bigl|f(x)\bigr|\cdot\bigl|1-h_1(x)\bigr|
\leq B\bigl|1-h_1(x)\bigr|
\leq B\cdot I(x\notin\supp(h)).
\end{align*}
Hence for all $\rho\in(0,\rho_0)$ and $\vecq\in\scrL_\psi\cap\scrB_{\scrT(\rho)}^d\setminus\scrE$
we have
\begin{align*}
\biggl|\int_{\US}f\bigl(\Gamma g_0^{(\vecq)}\varphi(R(\vecv)D_\rho)\bigr)\,d\lambda(\vecv)
-\int_{\US}f_1\bigl(\Gamma g_0^{(\vecq)}\varphi(R(\vecv)D_\rho)\bigr)\,d\lambda(\vecv)\biggr|
\hspace{60pt}
\\
\leq B\cdot\lambda\bigl(\bigl\{\vecv\in\US\col
\Gamma g_0^{(\vecq)}\varphi(R(\vecv)D_\rho)\notin \supp(h)\bigr\}\bigr)
<\frac{\ve}2,
\end{align*}
and also
\begin{align*}
\biggl|\int_{\XX^\psi}f\,d\overline{\omega^{(\vecq)}}
-\int_{\XX^\psi}f_1\,d\overline{\omega^{(\vecq)}}\biggr|
\leq B\cdot\overline{\omega^{(\vecq)}}\bigl(\XX^\psi\setminus\supp(h)\bigr)
\leq B\cdot\overline{\omega^{(\vecq)}}\bigl(\XX^\psi\setminus K\bigr)
<\frac{\ve}{4}.
\end{align*}
Hence for these $\rho$ and $\vecq$ we have
$\bigl|F\bigl(\langle f,\lambda\rangle,\vecq,\rho\bigr)-F\bigl(\langle f_1,\lambda\rangle,\vecq,\rho\bigr)\bigr|<3\ve/4$. %
Furthermore, by again applying 
the uniform convergence result which we have already proved
it follows that after possibly shrinking $\rho_0$,
we have
$\bigl|F\bigl(\langle f_1,\lambda\rangle,\vecq,\rho\bigr)\bigr|<\ve/4$ %
for all $\rho\in(0,\rho_0)$
and all $\vecq\in\scrL_\psi\cap\scrB_{\scrT(\rho)}^d\setminus\scrE$.
Combining the last two inequalities, we conclude that
$\bigl|F\bigl(\langle f,\lambda\rangle,\vecq,\rho\bigr)\bigr|<\ve$ %
for all $\rho\in(0,\rho_0)$
and all $\vecq\in\scrL_\psi\cap\scrB_{\scrT(\rho)}^d\setminus\scrE$.\label{HOMDYNMAINTHMproofend}
This completes the proof of 
the uniform convergence in
\eqref{HOMDYNMAINTHMres}, %
viz., the proof of Theorem \ref{HOMDYNMAINTHM}.\label{pfHOMDYNMAINTHMend}
\end{proof}

Note that in view of %
the results in Section \ref{verQ1Q2etcSEC},
we have now also completed the proof of the main result of the paper,
Theorem \ref{MAINTHM}.

\section{The transition kernels}
\label{transkerSEC}

\subsection{Definitions; collision kernels and transition kernels}
\label{transkerdefSEC}

We will now give the explicit formulas for the collision kernels
$p^{(\psi)}$ and $p^{(\psi'\to\psi)}$ appearing in
Theorem \ref{limitflightprocessexplTHM}.
It should be noted that 
Theorem \ref{limitflightprocessexplTHM}
is a reformulation of 
\cite[Theorem 4.6]{jMaS2019},
in our special case
of $\scrP$ being a finite union of grids.
Crucially, our main result,
Theorem \ref{MAINTHM},
is required to ensure %
that the assumptions in 
\cite[Theorem 4.6]{jMaS2019}
are fulfilled.
We also use the fact that any integral over the space $\Sigma$ 
with respect to the measure $\mm$ can be reduced to an integral
over $\tPsi$, since $\mm(\Sigma\setminus\tPsi)=0$;
and in the statement of Theorem \ref{limitflightprocessexplTHM}
we view $\mm$ as a probability measure on $\Psi$,
via the bijection
$\psi\leftrightarrow\sigma^\psi$ between $\Psi$ and $\tPsi$.
The defining formulas
for the collision kernels
are \cite[(3.41) and (3.44)]{jMaS2019}\footnote{See also
\eqref{ktranslMAMS} below.}
\begin{align}\label{pgendef001}
p^{(\psi)}\bigl(\vecV';\xi,\vecV\bigr)=\frac{\sigma(\vecV',\vecV)}{v_{d-1}}
\,k^{\g}(\xi,\vecw,\psi) %
\end{align}
and
\begin{align}\label{pbndefG001}
p^{(\psi'\to\psi)}\bigl(\vecV'',\vecV';\xi,\vecV\bigr)=
\frac{\sigma(\vecV',\vecV)}{v_{d-1}}\,
k(\vecw',\psi',\xi,\vecw,\psi),  %
\end{align}
where $k$ and $k^{\g}$ are the \textit{transition kernels},
whose definition we will recall below;   %
$\sigma(\vecV',\vecV)$
is the differential cross section of the 
fixed scattering process;
and finally $\vecw\in\UB$ is the impact parameter 
which is a function of $\vecV'$ and $\vecV$,
and $\vecw'\in\UB$ is the exit parameter which is a function of
$\vecV''$ and $\vecV'$.
It should be noted that the formula
\eqref{pgendef001} %
only makes sense for $\vecV',\vecV$ which can form
entry and exit velocities in a scatterer collision; %
for any other $\vecV',\vecV\in\US$
we have $p^{(\psi)}\bigl(\vecV';\xi,\vecV\bigr)=0$ %
by definition; a similar statement holds for \eqref{pbndefG001}.

\vspace{5pt}

Next we recall in detail the %
definition of the transition kernels $k$ and $k^{\g}$,
from
\cite[Section 3.1]{jMaS2019},
with some parts of the notation mildly modified.

Recall that for any $\xi>0$ we denote
$\fZ_\xi=(0,\xi)\times\UB$,
an open cylinder in $\R^d$.
In the following we will also use this notation with $\xi=\infty$,
i.e., $\fZ_\infty=(0,\infty)\times\UB$.
Furthermore, we define the ``$\vece_1$-coordinate'' of any point $(\vecx,\vs)\in\scrX$
to be the number $\vecx\cdot\vece_1$.
Next we define a map
\begin{align}\label{mapzDEF}
\vecz:N_s(\scrX)\to\Delta:=(\fZ_\infty\times\Sigma)\sqcup\{\undef\},
\end{align}
in the following way:\footnote{This map $\vecz$ is the same map as in
\cite[(3.3)]{jMaS2019};
in particular, note that
in \cite{jMaS2019}, ``$\Omega$'' denotes $\UB\times\Sigma$,
and hence ``$\R_{>0}\times\Omega$'' can in an obvious way be identified with the set $\fZ_\infty\times\Sigma$.
Recall that in the present paper, ``$\Omega$'' has a completely different meaning;
see \eqref{OMEGAdef}.}
Given $Y\in N_s(\scrX)$, let $\vecz=\vecz(Y)$ be that point in 
$Y\cap(\fZ_\infty\times\Sigma)$ which has minimal $\vece_1$-coordinate;
if there does not exist a unique such point then let
$\vecz(Y)=\undef$.
Here ``$\undef$'' is a dummy element not in $\fZ_\infty\times\Sigma$,
and we provide $\Delta$ with the disjoint union topology.
We also introduce the reflection map
\begin{align}\label{reflmapDEF}
\iota:\scrX\to\scrX,\qquad\iota((w_1,\vecw),\vs)=((w_1,-\vecw),\vs)
\qquad (w_1\in\R,\:\vecw\in\R^{d-1},\:\vs\in\Sigma).
\end{align}
Next, given $\vecw\in\R^{d-1}$ 
and $\psi\in\Psi$,
we let
$\kappa_{\vecw,\psi}\in P(\Delta)$ 
be the pushforward of the measure $\mu_{\sigma^{\psi}}$
by the map $Y\mapsto \iota(\vecz(Y-(0,\vecw)))$ from
$N_s(\scrX)$ to $\Delta$.
Equivalently, 
\begin{align}\label{kappawppsipdeffact1}
\text{$\kappa_{\vecw,\psi}$ is the pushforward of $\overline{\omega^{\psi}}$ by the map
$x\mapsto\iota(\vecz(J(x)-(0,\vecw)))$ from $\XX$ to $\Delta$.}
\end{align}
Indeed, this holds since 
$\mu_{\sigma^{\psi}}=J_{\psi*}\,\overline{\omega^{\psi}}$
by definition, 
and furthermore
$\vecz(J_{\psi}(x))=\vecz(J(x))$ for all $x\in\XX^{\psi}$,
and $\overline{\omega^{\psi}}(\XX\setminus\XX^{\psi})=0$
by \eqref{oomegaXpsieq1}.
One always has %
$\kappa_{\vecw,\psi}(\{\undef\})=0$,
i.e.\ $\kappa_{\vecw,\psi}$ restricts to a probability measure on
$\fZ_{\infty}\times\Sigma$.
Furthermore, since $J(x)\subset\R^d\times\tPsi$ for all $x\in\XX^{\psi}$,
by %
\eqref{JmapFULL},
$\kappa_{\vecw,\psi}$ in fact restricts to a probability measure on
$\fZ_{\infty}\times\tPsi$.
By Lemma \ref{INTENSITYmunuLEM},
we have $\kappa_{\vecw,\psi}(B)\leq \nbar_{\scrP}\,\mu_{\scrX}(B)$
for every Borel set $B\subset\fZ_{\infty}\times\Sigma\subset\scrX$;
in particular $\kappa_{\vecw,\psi}$ is absolutely continuous with respect to $\mu_{\scrX}$.

Now the transition kernel $k$ is defined as follows:
For any given $\vecw'\in\R^{d-1}$ and $\psi'\in\Psi$,
we define $k(\vecw',\psi',\,\cdot\,)$ to be the probability density
of the measure $\kappa_{\vecw',\psi'}$
with respect to the measure $v_{d-1}^{-1}\,\mu_{\scrX}$,
but restricted to the set $\fZ_{\infty}\times\tPsi$,
parametrized by $(\xi,\vecw,\psi)\in\R_{>0}\times\UB\times\Psi$.
That is, we define the function
\begin{align}\label{transkDEF}
k:\R^{d-1}\times\Psi\times\R_{>0}\times\UB\times\Psi\to[0,\nbar_{\scrP}\,v_{d-1}]
\end{align}
so that for each $\vecw'\in\R^{d-1}$ and $\psi'\in\Psi$,
$k(\vecw',\psi',\cdot,\cdot,\cdot)$ is uniquely defined as an element in
$\L^1(\R_{>0}\times\UB\times\Psi,d\xi\,d\vecw\,d\mm)$,
and 
\begin{align}\label{transkDEF2}
\kappa_{\vecw',\psi'}(B)=v_{d-1}^{-1}\int_{(\xi,\vecw,\sigma^\psi)\in B} 
k(\vecw',\psi',\xi,\vecw,\psi)\,d\xi\,d\vecw\,d\mm(\psi)
\end{align}
for all Borel sets $B\subset\fZ_{\infty}\times\Sigma$.
Note that in \eqref{transkDEF2} we view $\mm$ as a 
probability measure on $\Psi$
through $\mm(\psi):=\mm(\sigma^\psi)$,
just as we %
did in the statement of Theorem \ref{limitflightprocessexplTHM}.

\vspace{5pt}

Similarly, we let $\kappa^{\g}$ be the pushforward of the measure $\mu^{\g}$
by the map $Y\mapsto \iota(\vecz(Y))$ from
$N_s(\scrX)$ to $\Delta$.
Then by \eqref{mugdef} and Lemma \ref{mugtranslinvLEM} we have, for any fixed $\vecv\in\R^d$:
\begin{align}\label{kappagwppsipdeffact1gen}
\text{$\kappa^{\g}$ is the pushforward of $\overline{\omega^{\g}}$ by the map
$x\mapsto\iota(\vecz(J(x)+\vecv))$ from $\XX$ to $\Delta$.}
\end{align}
By \cite[Lemma 3.5]{jMaS2019},
and since $J(x)\subset\R^d\times\tPsi$ for all $x\in\XX$
by \eqref{tPsiDEF} and \eqref{JmapFULL},
$\kappa^{\g}$ in fact restricts to a probability measure on
$\fZ_{\infty}\times\tPsi$.
By \cite[Prop.\ 2.27]{jMaS2019}
applied with $A=N_s(\scrX)$, we have
$\kappa^{\g}(B)\leq \nbar_{\scrP}\,\mu_{\scrX}(B)$
for every Borel set $B\subset\fZ_{\infty}\times\Sigma\subset\scrX$
(in our situation this may also be derived from
\eqref{mugdef} and Proposition \ref{SIEGELFORMULALEM1},
by an argument similar to that in the proof of Lemma \ref{INTENSITYmunuLEM}).
In particular $\kappa^{\g}$ is absolutely continuous with respect to $\mu_{\scrX}$.
We now define the transition kernel $k^{\g}$ to be the corresponding probability density
with respect to the measure $v_{d-1}^{-1}\,\mu_{\scrX}$,
but restricted to the set $\fZ_{\infty}\times\tPsi$,
which we parametrize by $(\xi,\vecw,\psi)\in\R_{>0}\times\UB\times\Psi$.
That is, we define the function
\begin{align}\label{transkDEF3}
k^{\g}:\R_{>0}\times\UB\times\Psi\to[0,\nbar_{\scrP}\,v_{d-1}]
\end{align}
so that 
\begin{align}\label{transkDEF4}
\kappa^{\g}(B)=v_{d-1}^{-1}\int_{(\xi,\vecw,\sigma^\psi)\in B}  k^{\g}(\xi,\vecw,\psi)\,d\xi\,d\vecw\,d\mm(\psi)
\end{align}
for all Borel sets $B\subset\fZ_{\infty}\times\Sigma$.
The function $k^{\g}$ is uniquely defined as an element in
$\L^1(\R_{>0}\times\UB\times\Psi,d\xi\,d\vecw\,d\mm)$.

\vspace{5pt}

To facilitate comparison, %
let us note that in the notation of \cite{jMaS2019} we have:
\begin{align}\label{ktranslMAMS}
k(\vecw',\psi',\xi,\vecw,\psi)=\text{``}k\bigl(\bigl(\vecw',\sigma^{\psi'}\bigr),\xi,\bigl(\vecw,\sigma^\psi\bigr)\bigr)\text{''};
\qquad
k^{\g}(\xi,\vecw,\psi)=
\text{``}k^{\g}\bigl(\xi,\bigl(\vecw,\sigma^\psi\bigr)\bigr)\text{''}.
\end{align}

\vspace{5pt}

Let us also recall from \cite{jMaS2019} the
following formula expressing $k^{\g}$ in terms of $k$.
\begin{lem}\label{kkgrelLEM}
There is a continuous version of $k^{\g}$,
which for all $(\xi,\vecw,\psi)\in\R_{>0}\times\UB\times\Psi$ satisfies
\begin{align}\label{MEMSProp3p22restat3}
k^{\g}(\xi,\vecw,\psi)
=\nbar_{\scrP}\int_{\xi}^\infty\int_{\UB}\int_{\Psi} k(-\vecw,\psi,\xi',\vecw',\psi')\,d\mm(\psi')\,d\vecw'\,d\xi'.
\end{align}
\end{lem}
\begin{proof}
This follows as in the proof of 
\cite[Cor.\ 3.24]{jMaS2019}.
(Indeed, for $d=2$ we are forced to take $R=-I_1\in O(1)$ in that proof, %
which directly gives the lemma;
for $d\geq 3$ we apply the proof %
with a fixed $R\in O(d-1)$ with $\det R=-1$,
and then also use 
\cite[Lemma 3.18]{jMaS2019}.)
\end{proof}

In the rest of Section \ref{transkerSEC} we will always assume that
the continuous version of $k^{\g}$ is used.

\vspace{5pt}

Let us now %
note that the formula for the 
free path length density $\Phi_{\scrP}$,
\eqref{PhiPdef}, can be rewritten as
follows using \eqref{pgendef001}:
\begin{align}\label{genoPhiPhirel4}
\Phi_{\scrP}(\xi)=
\frac1{v_{d-1}}\int_{\UB\times\Psi}k^{\g}(\xi,\vecw,\psi)\,d\vecw\,d\mm(\psi).
\end{align}
We also introduce 
the density function for the free path length between consecutive collisions,
\begin{align}\label{genoPhiPhirel5}
\oPhi_{\scrP}(\xi)
=\frac1{v_{d-1}^2}\int_{\UB\times\Psi}\int_{\UB\times\Psi}
k(\vecw',\psi',\xi,\vecw,\psi)\,d\vecw\,d\mm(\psi)\,d\vecw'\,d\mm(\psi').
\end{align}
It is immediate from the definitions of
$k^{\g}$ and $k$ 
that 
$\int_0^\infty\Phi_{\scrP}(\xi)\,d\xi
=\int_0^\infty\oPhi_{\scrP}(\xi)\,d\xi=1$
and that 
$0\leq\Phi_{\scrP}(\xi),\: \oPhi_{\scrP}(\xi) \leq\nbar_{\scrP}\,v_{d-1}$.
Furthermore $\Phi_{\scrP}(\xi)$ is continuous since $k^{\g}$ is continuous.
The formula \eqref{PhioPhirel} 
expressing $\Phi_{\scrP}(\xi)$ in terms of 
$\oPhi_{\scrP}(\xi)$ follows from Lemma \ref{kkgrelLEM},
and \eqref{PhioPhirel}  implies that
$\Phi_{\scrP}(\xi)$ is a decreasing function.
Also, \eqref{genoPhiPhirel6}
follows from $\int_0^\infty\Phi_{\scrP}(\xi)\,d\xi=1$ %
and \eqref{PhioPhirel}.
It should be noted that the formula \eqref{genoPhiPhirel6}
holds more generally for any scattering configuration
$\scrP$ which fits within the framework of \cite{jMaS2019}; see
\cite[Cor.\ 3.23]{jMaS2019}.

\begin{remark}\label{singlegridREM}
In the special case when $\scrP$ is \textit{a single grid}
(presented with $\#\Psi=1$),
the transition kernels $k$ and $k^{\g}$
have previously been studied in detail
in the papers
\cite{jMaS2008,jMaS2010a,jMaS2011a,jMaS2010b}.
In this case we have:
\begin{align*}
k(\vecz,\psi,\xi,\vecw,\psi)=\nbar_{\scrP}\,v_{d-1}\Phi_{\bn}(\nbar_{\scrP}\xi,\vecw,\vecz)
\qquad\text{and}\qquad
k^{\g}(\xi,\vecw,\psi)=\nbar_{\scrP}\,v_{d-1}\Phi(\nbar_{\scrP}\xi,\vecw),
\end{align*}
where ``$\Phi_{\bn}$'' and ``$\Phi$'' is the notation used in
\cite{jMaS2008,jMaS2010b}.
Also the free path length densities are given by
\begin{align*}
\oPhi_{\scrP}(\xi)=\nbar_{\scrP}\oPhi_{\bn}(\nbar_{\scrP}\xi)
\qquad\text{and}\qquad
\Phi_{\scrP}(\xi)=\nbar_{\scrP}\Phi(\nbar_{\scrP}\xi),
\end{align*}
where, again, ``$\Phi_{\bn}$'' and ``$\Phi$'' is the notation used in
\cite{jMaS2008,jMaS2010b}.
\end{remark}

\begin{remark}\label{incommensurableREM}
The more general special case %
when all the grids appearing in $\scrP$ are incommensurable
was previously studied in \cite{jMaS2013a}.
In that paper the normalizing condition $\nbar_{\scrP}=1$ was imposed.
Given such a point set $\scrP$, presented with $\Psi=\{(j,1)\col j=1,\ldots,N\}$, we have:
\begin{align}\label{incommensurableREMres}
k(\vecw,\psi',\xi,\vecz,\psi)
=\frac{v_{d-1}}   %
{\nbar_{j_\psi}}\cdot\Phi_{\bn}^{(j_\psi'\to j_\psi)}(\xi,\vecz,\vecw)
\qquad\text{and}\qquad
k^{\g}(\xi,\vecw,\psi)
=\frac{v_{d-1}} %
{\nbar_{j_\psi}}\cdot\Phi^{(j_\psi)}(\xi,\vecw),
\end{align}
where ``$\Phi_{\bn}^{(j\to k)}$''
and ``$\Phi^{(j)}$'' is the notation used in 
\cite{jMaS2013a}.
The first translation formula in \eqref{incommensurableREMres}
can be obtained e.g.\ by comparing the statements of
\cite[Theorem 4 and (5.4)]{jMaS2013a}
and \cite[Thm.\ 3.6]{jMaS2019},
and the second formula by comparing 
\cite[Theorem 4 and (5.11)]{jMaS2013a}
and \cite[Thm.\ 3.14]{jMaS2019}.
See also Remark \ref{oldprodformulacompareREM} below.
\end{remark}

\begin{remark}
For $\scrP$ a general finite union of grids as in  %
the present paper,
it seems that it should be possible to give explicit formulas 
for the transition kernels $k$ and $k^{\g}$
in terms of Haar measures on appropriate homogeneous submanifolds
of $\XX$,
similar in spirit to the formulas in
\cite[Sec.\ 7--8]{jMaS2010a}
for the case when $\scrP$ is a single grid.
Obtaining such formulas would be a first step towards the problem of 
generalizing the precise asymptotic estimates for the transition kernels
which were obtained in \cite{jMaS2010b} in the single grid case.
\end{remark}

\subsection{The product formula}
\label{prodformulapfsec}

In view of our explicit set-up from Section \ref{Preprsec}, %
any two $\psi,\psi'\in\Psi$ are equivalent
(``$\psi\sim\psi'$'') in the sense defined in
Section \ref{prodformulasSEC} 
if and only if $j_\psi=j_{\psi'}$.
Hence we have an obvious 
bijection between 
$C_\Psi$ and $\{1,\ldots,N\}$.
For any $j\in\{1,\ldots,N\}$,
let us denote the corresponding equivalence class by
$\Psi_j:=\{(j,i)\col i\in\{1,\ldots,r_j\}\}\in C_\Psi$. %
Let us also introduce the short-hand notation
$\scrP_j:=\scrP_{\Psi_j}$; that is
\begin{align}\label{scrPjdef}
\scrP_j=\bigcup_{\psi\in\Psi_j}\scrL_\psi. %
\end{align}
The asymptotic density of this point set is
$\nbar_j:=\nbar_{\scrP_j}=\sum_{\psi\in\Psi_j}\nbar_{\psi}$.
\begin{lem}\label{PjadmLEM}
For each $j\in\{1,\ldots,N\}$, the presentation of $\scrP_j$ in \eqref{scrPjdef} is admissible.
\end{lem}
\begin{proof}
This lemma is essentially clear %
by inspection;
we still write out some details in order to explain 
certain notation which we will use later.
Let $j\in\{1,\ldots,N\}$ be fixed.
The setup in Section \ref{Preprsec} applies to $\scrP_j$ if we set
\begin{align*}
\text{[$\Psi$ for $\scrP_j$]} \: :=\{(1,i)\col i\in\{1,\ldots,r_j\}\}
\end{align*}
(thus [$N$ for $\scrP_j$] $=1$ and [$r_1$ for $\scrP_j$] $=r_j$),
and then, for each $i\in\{1,\ldots,r_j\}$, set
\begin{align*}
\text{[$c_{1,i}$ for $\scrP_j$] $:=c_{j,i}$,}
\quad
\text{[$\vecw_{1,i}$ for $\scrP_j$] $:=\vecw_{j,i}$},
\quad
\text{and [$M_1$ for $\scrP_j$] $:=M_j$,}
\end{align*}
so that [$\scrL_{(1,i)}$ for $\scrP_j$] $=\scrL_{(j,i)}$
(see \eqref{LpsiDEF}).
Indeed, with these definitions the analogue of 
the formula in \eqref{GENPOINTSET1} holds for $\scrP_j$.
Of course, this formula %
is exactly the same as %
\eqref{scrPjdef},
once we identify [$\Psi$ for $\scrP_j$] with $\Psi_j$ through $(1,i)\leftrightarrow(j,i)$.

Now the condition \eqref{THINDISJcond2} trivially holds for $\scrP_j$
since it holds for $\scrP$,
and the condition \eqref{GENPOINTSET1req} for $\scrP_j$
is void, since [$N$ for $\scrP_j$]$=1$.
Next, we have 
[$\vecc_1^{(1,i)}$ for $\scrP_j$] $=\vecc_j^{(j,i)}$
by \eqref{cjdef},
and [$W_1^{(1,i)}$ for $\scrP_j$] $=W_j^{(j,i)}$
by \eqref{Wjdef};
therefore
[$L_j^{(1,i)}$ for $\scrP_j$] $=L_j^{(j,i)}$
by \eqref{Ljpsi0DEF}.
Hence by 
Definition \ref{admissibleDEF},
our presentation of $\scrP_j$ is admissible if and only if
$\vecc_j^{(j,i)}\in L_j^{(j,i)}+\Z^{r_j}$ for all $i\in\{1,\ldots,r_j\}$;
and these conditions are certainly fulfilled,
since we are assuming that the 
fixed presentation of 
$\scrP$ in
\eqref{GENPOINTSET1} is admissible.
\end{proof}

Recall from
Section \ref{prodformulasSEC} 
that for any $j\in\{1,\ldots,N\}$ and $\psi,\psi'\in\Psi_j$,
the collision kernels
$\cp^{(\psi)}$ and $\cp^{(\psi\to\psi')}$
are defined as 
[$p^{(\psi)}$ and $p^{(\psi\to\psi')}$ for $\scrP_{j_\psi}$].
Similarly let us define the transition kernels
$\ck(\vecw',\psi',\xi,\vecw,\psi)$
and $\ck^{\g}(\xi,\vecw,\psi)$
as [$k(\cdots)$ and $k^{\g}(\cdots)$ for $\scrP_{j_\psi}$].
Now the product formulas announced in
Section \ref{prodformulasSEC} are as follows.
\begin{thm}\label{Productformulasthm}
Given any
$\langle\vecw',\psi'\rangle\in\UB\times\Psi$
and $\psi\in\Psi$,
the following formula holds for Lebesgue almost all 
$\langle\xi,\vecw\rangle\in\R_{>0}\times\UB$:
\begin{align}\label{Productformulasthmres1och2}
k(\vecw',\psi',\xi,\vecw,\psi)
=\begin{cases}
{\displaystyle \frac{\nbar_{\scrP}}{\nbar_{j_\psi}}\ck(\vecw',\psi',\xi,\vecw,\psi)\,
\prod_{j\neq j_{\psi}}\int_\xi^\infty\Phi_{\scrP_j}(\xi')\,d\xi'}
\hspace{70pt}\text{if }\: j_{\psi'}=j_\psi
\\[20pt]
{\displaystyle \frac{\nbar_{\scrP}}{\nbar_{j_\psi}\nbar_{j_{\psi'}}\,v_{d-1}}
\ck^{\g}(\xi,\vecw,\psi)
\ck^{\g}(\xi,-\vecw',\psi')
\prod_{j\neq j_{\psi}, j_{\psi'}}\int_\xi^\infty\Phi_{\scrP_j}(\xi')\,d\xi'}
\\
\hspace{270pt}\text{if }\:j_{\psi'}\neq j_\psi.
\end{cases}
\end{align}
Furthermore, for all $\langle\xi,\vecw,\psi\rangle\in\R_{>0}\times\UB\times\Psi$:
\begin{align}\label{Productformulasthmres3}
k^{\g}(\xi,\vecw,\psi)=\frac{\nbar_{\scrP}}{\nbar_{j_\psi}}
\ck^{\g}(\xi,\vecw,\psi)
\,\prod_{j\neq j_\psi}\int_\xi^\infty\Phi_{\scrP_j}(\xi')\,d\xi'.
\end{align}
Each product over $j$ in 
\eqref{Productformulasthmres1och2} and \eqref{Productformulasthmres3}
should be understood to run over all $j\in\{1,\ldots,N\}$
except those numbers which are explicitly excluded.
\end{thm}
\begin{remark}\label{Productformulasthm2}
It is not difficult to translate
Theorem \ref{Productformulasthm} into a product formula involving the
collision kernels instead, using 
\eqref{pgendef001} and \eqref{pbndefG001}.
For example, when $j_{\psi'}=j_\psi$,
the formula in \eqref{Productformulasthmres1och2}
is equivalent to
\begin{align}\label{Productformulasthm2res1}
p^{(\psi'\to\psi)}(\vecV'',\vecV';\xi,\vecV)
=\frac{\nbar_{\scrP}}{\nbar_{j_\psi}}\cp^{(\psi'\to\psi)} %
(\vecV'',\vecV';\xi,\vecV)\,
\prod_{j\neq j_\psi}    %
\int_\xi^\infty\Phi_{\scrP_{j}}(\xi')\,d\xi'.
\end{align}
But when $j_{\psi'}\neq j_\psi$
(and assuming $\vecV'\in\scrV_{\vecV''}$),
\eqref{Productformulasthmres1och2}
translates into the following
somewhat more complicated formula:
\begin{align}\notag
p^{(\psi'\to\psi)}(\vecV'',\vecV';\xi,\vecV)
=\frac{\nbar_{\scrP}}{\nbar_{j_\psi}\nbar_{j_{\psi'}}\,\sigma(\vecV',\vecW)}
p^{(\psi)}(\vecV';\xi,\vecV)
p^{(\psi')}(\vecV';\xi,\vecW)
\prod_{j\neq j_{\psi}, j_{\psi'}}\int_\xi^\infty\Phi_{\scrP_j}(\xi')\,d\xi',
\end{align}
with $\vecW=\Psi_1(\vecV',-\vecbeta_{\vecV''}^+(\vecV'))$,
where the maps ``$\Psi_1$'' and ``$\vecbeta^+$'' are as in
\cite[Sec.\ 3.4]{jMaS2019}
(these are defined in terms of the fixed scattering process of the Lorentz gas). %
If the scattering process %
preserves angular momentum,
then we have in fact $\vecW=\vecV''$,
by \cite[(3.33)]{jMaS2019}.
\end{remark}

\vspace{5pt}

\noindent
\textit{Proof of Theorem \ref{Productformulasthm}.}
We first need to introduce some further notation.
For each $j\in\{1,\ldots,N\}$ let us set
$\Sigma_j:=\{(\psi,\omega)\in\Sigma\col j_\psi=j\}$
and $\scrX_j=\R^d\times\Sigma_j$.
Note that each $\Sigma_j$ is a clopen subset of $\Sigma$,
and each $\scrX_j$ is a clopen subset of $\scrX$,
and $\Sigma$ and $\scrX$ decompose as disjoint unions
$\Sigma=\sqcup_{j=1}^N\Sigma_j$ and $\scrX=\sqcup_{j=1}^N\scrX_j$.
Also for each $j\in\{1,\ldots,N\}$ we define the map
$J^{(j)}:\XX_j\to N_s(\scrX_j)$ 
through
\begin{align}\label{tnupsiabstractdefSTEP1}
J^{(j)}(\Gamma_jg)=\bigcup_{i=1}^{r_j}c_{j,i}\bigl(\Z^d\a_i(g)\bigr)\times\{\sigma^{(j,i)}\}
\qquad (g\in G_j).
\end{align}
Using these maps $J^{(j)}$,
the formula for $J:\XX\to N_s(\scrX)$ in \eqref{JmapFULL} can be
written
\begin{align}\label{tnupsiabstractdefSTEP2}
J(x)=\bigcup_{j=1}^N J^{(j)}(\tp_j(x))\qquad (x\in\XX).
\end{align}
Throughout the following, %
we will identify 
[$\Psi$ for $\scrP_j$] with $\Psi_j$ through $(1,i)\leftrightarrow(j,i)$
whenever convenient,
as explained %
in the proof of Lemma \ref{PjadmLEM} above.
It is now straightforward to
express the concepts defined in
Section \ref{spaceofmarksandmapsSEC} when 
\textit{applied to $\scrP_j$ with its presentation in \eqref{scrPjdef}.}
For example, %
[$\Omega$ for $\scrP_j$] now equals $P(\TT_j^d)'$,
[$\Sigma$ for $\scrP_j$] equals $\Sigma_j$,
and 
\begin{align}\label{mjformula}
\mm_j:=\text{[$\mm$ for $\scrP_j$]}=\frac{\nbar_{\scrP}}{\nbar_j}\mm\big|_{\Sigma_j}
\quad\in P(\Sigma_j).
\end{align}
As we did for $\mm$ in Section \ref{transkerdefSEC},
we also use $\mm_j$ to denote the probability measure on
$\Psi_j$ given by $\mm_j(\psi):=\mm_j(\sigma^\psi)$ ($\psi\in\Psi_j$);
this gives back the measure defined in
\eqref{mcDEF} (but now writing $\Psi_j$ in place of $c$).
Furthermore, we have [$\scrX$ for $\scrP_j$] $=\scrX_j$
and [$\mu_{\scrX}$ for $\scrP_j$] $=\vol\times\mm_j
={\displaystyle \frac{\nbar_{\scrP}}{\nbar_j}}\mu_{\scrX}\big|_{\scrX_j}$.
Also, the map [$J$ for $\scrP_j$] is given by $J^{(j)}$ defined in
\eqref{tnupsiabstractdefSTEP1}.

Next,  %
for any $\vecw'\in\R^{d-1}$ and 
$\psi'\in\Psi_j$, 
let us write 
\begin{align*}
\ckappa_{\vecw',\psi'}:=[\kappa_{\vecw',\psi'}\:\text{ for }\:\scrP_j],
\qquad\text{and also}\qquad
\ckappa_j^{\g}:=[\kappa^{\g}\:\text{ for }\:\scrP_j].
\end{align*}
Now it is easy to write out the definitions and facts from Section \ref{transkerdefSEC}
when applied for $\scrP_j$ in the place of $\scrP$.
For example, $\ck$,
which we defined above to be the transition kernel [$k$ for $\scrP_j$],
is a function 
\begin{align}\label{ckDEF}
\ck:\R^{d-1}\times\Psi_j\times\R_{>0}\times\UB\times\Psi_j\to[0,\nbar_j\,v_{d-1}].
\end{align}
such that for each $\vecw'\in\R^{d-1}$ and $\psi'\in\Psi_j$,
and for any Borel set $B\subset\fZ_{\infty}\times\Sigma_j$,
we have
\begin{align*}
\ckappa_{\vecw',\psi'}(B)=v_{d-1}^{-1}\int_{(\xi,\vecw,\sigma^\psi)\in B} \ck(\vecw',\psi',\xi,\vecw,\psi)\,d\xi\,d\vecw\,d\mm_j(\psi).
\end{align*}

\vspace{5pt}

We now start with the actual proof of Theorem \ref{Productformulasthm}.
Let $\vecw'\in\UB$ and $\psi'\in\Psi$; these will be kept fixed throughout the proof.
Let $j_0\in\{1,\ldots,N\}$,
and let $B$ be a Borel subset of 
$\fZ_\infty\times\Sigma_{j_0}$.
Then by \eqref{kappawppsipdeffact1},
\begin{align}\label{Productformulasthmpf10}
\kappa_{\vecw',\psi'}(B)
=\overline{\omega^{\psi'}}(A),
\qquad\text{with }\:
A:=\bigl\{x\in\XX\col\iota(\vecz(J(x)-(0,\vecw')))\in B\bigr\}.
\end{align}
Recall that $\XX=\XX_1\times\cdots\times\XX_N$
and $\overline{\omega^{\psi'}}
=\overline{\omega_1^{\psi'}}\otimes\cdots\otimes\overline{\omega_N^{\psi'}}$.
Let us write $\XX^{(j_0)}=\prod_{j\neq j_0}\XX_j$,
so that 
\begin{align*}
\XX=\XX_{j_0}\times\XX^{(j_0)}
=\{(x,x')\col x\in \XX_{j_0},\: x'\in\XX^{(j_0)}\},
\end{align*}
in an obvious identification.
Corresponding to this product decomposition we have
$\overline{\omega^{\psi'}}
=\overline{\omega_{j_0}^{\psi'}}\otimes\nu$
with $\nu:=\otimes_{j\neq j_0}\overline{\omega_j^{\psi'}}$,
and hence, by Fubini's Theorem,
\begin{align}\label{Productformulasthmpf2}
\kappa_{\vecw',\psi'}(B)=\overline{\omega^{\psi'}}(A)
=\int_{\XX_{j_0}}\nu(A_x)\,d\overline{\omega_{j_0}^{\psi'}}(x),
\end{align}
where for each $x\in\XX_{j_0}$,
\begin{align*}
A_x:=\{x'\in\XX^{(j_0)}\col \iota(\vecz(J(x,x')-(0,\vecw')))\in B\}.
\end{align*}

For each $j\in\{1,\ldots,N\}$ we set
$\tXX_j:=\{x\in\XX_j\col\vecz(J^{(j)}(x)-(0,\vecw'))\neq\undef\}$.
Then
\begin{align}\label{Productformulasthmpf1}
\overline{\omega_j^{\psi'}}\bigl(\XX_j\setminus \tXX_j\bigr)=0.
\end{align}
Indeed, if $j=j_{\psi'}$ then \eqref{Productformulasthmpf1}
follows from $\ckappa_{\vecw',\psi'}(\{\undef\})=0$
(which holds by \cite[Lemma 3.1]{jMaS2019}),
and if $j\neq j_{\psi'}$ then
$\omega_j^{\psi'}=\omega_j^{\g}$
by Lemma \ref{SjpsieqSjLEM},
and so \eqref{Productformulasthmpf1} follows from
the fact that 
$\ckappa_j^{\g}(\{\undef\})=0$
(which holds
by \cite[Lemma 3.5]{jMaS2019})
and the formula \eqref{kappagwppsipdeffact1gen}
with $\vecv=-(0,\vecw')$,
applied for $\scrP_j$ in the place of $\scrP$.
Hence, writing also $\tXX=\prod_{j=1}^N\tXX_j$
and $\tXX^{(j_0)}=\prod_{j\neq j_0}\tXX_j$,
we have
$\overline{\omega^{\psi'}}\bigl(\XX\setminus\tXX\bigr)=0$
and $\nu\bigl(\XX^{(j_0)}\setminus\tXX^{(j_0)}\bigr)=0$,
and using these facts together with \eqref{Productformulasthmpf1},
it follows that:
\begin{align}\label{Productformulasthmpf3}
\kappa_{\vecw',\psi'}(B)=\overline{\omega^{\psi'}}\bigl(A\cap\tXX\bigr)
=\int_{\tXX_{j_0}}\nu\bigl(A_x\cap\tXX^{(j_0)}\bigr)\,d\overline{\omega_{j_0}^{\psi'}}(x),
\end{align}

For any $x\in\tXX_{j_0}$, let us define
$\xi(x)\in\R_{>0}$, $\vecw(x)\in\UB$ and $\psi(x)\in\Psi_{j_0}$ through 
the relation
\begin{align}\label{Productformulasthmpf12}
\bigl((\xi(x),\vecw(x)),\psi(x)\bigr)=\iota(\vecz(J^{(j_0)}(x)-(0,\vecw')))
\qquad\text{in }\: \fZ_\infty\times\Sigma_{j_0}.
\end{align}
Let us also set
\begin{align*}
\tB=\{x\in\tXX_{j_0}\col ((\xi(x),\vecw(x)),\psi(x))\in B\}.
\end{align*}
We will write any element $x'$ in 
$\XX^{(j_0)}=\prod_{j\neq j_0}\XX_j$ as $x'=(x'_j)_{j\neq j_0}$
with $x'_j\in\XX_j$.
Also, for any $Y\in N_s(\scrX)$ such that $\vecz(Y)\neq\undef$,
let us denote by $z_1(Y)$ the $\vece_1$-coordinate of the point $\vecz(Y)$;
in particular we then have
$z_1(J^{(j_0)}(x)-(0,\vecw'))=\xi(x)$ for every $x\in\tXX_{j_0}$.
Using the definition of the map $\vecz$
and our assumption that $B\subset\fZ_\infty\times\Sigma_{j_0}$,
it follows that 
for every $(x,x')\in\tXX_{j_0}\times\tXX^{(j_0)}$,
the condition
$\iota(\vecz(J(x,x')-(0,\vecw')))\in B$ holds
if and only if
$x\in\tB$ and
$z_1(J^{(j)}(x'_j)-(0,\vecw'))>\xi(x)$ for all $j\neq j_0$.
Hence \eqref{Productformulasthmpf3} can be
rewritten as
\begin{align}\label{Productformulasthmpf4}
\kappa_{\vecw',\psi'}(B)
=\int_{\tB}\prod_{j\neq j_0}\overline{\omega_j^{\psi'}}
\bigl(\bigl\{x_j'\in\tXX_j\col z_1(J^{(j)}(x'_j)-(0,\vecw'))>\xi(x)\bigr\}\bigr)
\,d\overline{\omega_{j_0}^{\psi'}}(x),
\end{align}
Here for each $j\neq j_{\psi'}$ we have
\begin{align}\notag
\overline{\omega_j^{\psi'}}
\bigl(\bigl\{x_j'\in\tXX_j\col z_1(J^{(j)}(x'_j)-(0,\vecw'))>\xi(x)\bigr\}\bigr)
=\ckappa_j^{\g}\bigl((\xi(x),\infty)\times\UB\times\tPsi_j\bigr)
\hspace{50pt}
\\\notag
=v_{d-1}^{-1}\int_{\xi(x)}^\infty\int_{\UB}\int_{\Psi_j} \ck^{\g}(\xi',\vecw,\psi)\,d\mm_j(\psi)\,d\vecw\,d\xi
=\int_{\xi(x)}^\infty\Phi_{\scrP_j}(\xi')\,d\xi',
\end{align}
where the first equality
follows from $\omega_j^{\psi'}=\omega_j^{\g}$
and the formula
\eqref{kappagwppsipdeffact1gen}
with $\vecv=-(0,\vecw')$,
applied for $\scrP_j$ in place of $\scrP$,
and the last equality 
holds by 
\eqref{genoPhiPhirel4} applied for $\scrP_j$.
On the other hand,
for $j=j_{\psi'}$ we have
\begin{align}\notag
\overline{\omega_j^{\psi'}}
\bigl(\bigl\{x_j'\in\tXX_j\col z_1(J^{(j)}(x'_j)-(0,\vecw'))>\xi(x)\bigr\}\bigr)
=\ckappa_{\vecw',\psi'}\bigl((\xi(x),\infty)\times\UB\times\tPsi_j\bigr)
\hspace{30pt}
\\\notag
=v_{d-1}^{-1}\int_{\xi(x)}^\infty\int_{\UB}\int_{\Psi_j} \ck(\vecw',\psi',\xi',\vecw,\psi)\,d\mm_j(\psi)\,d\vecw\,d\xi
=\frac1{\nbar_j\,v_{d-1}}\ck^{\g}_j(\xi(x),-\vecw',\psi'),
\end{align}
where in the last equality we used
Lemma \ref{kkgrelLEM} for $\scrP_j$.

Let us first assume $j_0=j_{\psi'}$.
Then the above formulas imply that 
\begin{align}\label{Productformulasthmpf6}
\kappa_{\vecw',\psi'}(B)
=\int_{\tB}\biggl(\prod_{j\neq j_0}
\int_{\xi(x)}^\infty\Phi_{\scrP_j}(\xi')\,d\xi'\biggr)
\,d\overline{\omega_{j_0}^{\psi'}}(x),
\end{align}
We here change to new integration variables 
$(\xi,\vecw,\psi)$ through
\eqref{Productformulasthmpf12};
then by the definition of $\ck(\vecw',\psi',\cdots)$ we have
$d\overline{\omega_{j_0}^{\psi'}}(x)=v_{d-1}^{-1}\,\ck(\vecw',\psi',\xi,\vecw,\psi)\,d\xi\,d\vecw\,d\mm_{j_0}(\psi)$.
Hence, using also \eqref{mjformula}
to express $\mm_{j_0}$ in terms of $\mm$, we get
\begin{align}\label{Productformulasthmpf8}
\kappa_{\vecw',\psi'}(B)
=\frac{\nbar_{\scrP}}{\nbar_{j_0}\,v_{d-1}}\int_{B}\ck(\vecw',\psi',\xi,\vecw,\psi)\, \biggl(\prod_{j\neq j_0}
\int_{\xi}^\infty\Phi_{\scrP_j}(\xi')\,d\xi'\biggr)
\,d\xi\,d\vecw\,d\mm(\psi).
\end{align}
The fact that \eqref{Productformulasthmpf8} holds for all
Borel sets $B\subset\fZ_\infty\times\Sigma_{j_0}$
implies that the first formula in \eqref{Productformulasthmres1och2} holds 
for almost all
$\langle\xi,\vecw,\psi\rangle\in\R_{>0}\times\UB\times\Psi_{j_0}$.

Next assume $j_0\neq j_{\psi'}$.
Then in a similar way, we get
\begin{align}\notag
&\kappa_{\vecw',\psi'}(B)
=\frac1{\nbar_{j_{\psi'}}\,v_{d-1}}\int_{\tB}
\ck^{\g}_{j_{\psi'}}(\xi(x),-\vecw',\psi')\,
\biggl(\prod_{j\neq j_{\psi'},j_0}
\int_{\xi(x)}^\infty\Phi_{\scrP_j}(\xi')\,d\xi'\biggr)
\,d\overline{\omega_{j_0}^{\psi'}}(x)
\\\label{Productformulasthmpf9}
&=\frac{\nbar_{\scrP}}{\nbar_{j_0}\nbar_{j_{\psi'}}\,v_{d-1}^2}\int_{B}\ck_{j_0}^{\g}(\xi,\vecw,\psi)\,
\ck^{\g}_{j_{\psi'}}(\xi,-\vecw',\psi')\,
 \biggl(\prod_{j\neq j_{\psi'},j_0}
\int_{\xi}^\infty\Phi_{\scrP_j}(\xi')\,d\xi'\biggr)
\,d\xi\,d\vecw\,d\mm(\psi).
\end{align}
Here, to obtain the last equality, we 
again changed to new integration variables 
$(\xi,\vecw,\psi)$ through
\eqref{Productformulasthmpf12}
and used $\omega_{j_0}^{\psi'}=\omega_{j_0}^{\g}$
and the formula
\eqref{kappagwppsipdeffact1gen}
with $\vecv=-(0,\vecw')$
(applied for $\scrP_j$ in place of $\scrP$),
to see that 
$d\overline{\omega_{j_0}^{\psi'}}(x)
=v_{d-1}^{-1}\,\ck^{\g}_{j_{\psi'}}(\xi,-\vecw',\psi')\,d\xi\,d\vecw\,d\mm_{j_0}(\psi)$.
Finally, the fact that \eqref{Productformulasthmpf9} holds for all
Borel sets $B\subset\fZ_\infty\times\Sigma_{j_0}$
implies that the second formula in \eqref{Productformulasthmres1och2} holds 
for almost all
$\langle\xi,\vecw,\psi\rangle\in\R_{>0}\times\UB\times\Psi_{j_0}$.

Since $j_0$ was arbitrary, we have now proved 
all of \eqref{Productformulasthmres1och2}.
The formula \eqref{Productformulasthmres3} follows 
by a very similar computation,
starting from the relation
$\kappa^{\g}(B)=\overline{\omega^{\g}}(A)$
instead of \eqref{Productformulasthmpf10},
leading to
\begin{align*}
\kappa^{\g}(B)
=\frac{\nbar_{\scrP}}{\nbar_{j_0}\,v_{d-1}}\int_{B}\ck^{\g}(\xi,\vecw,\psi)\, \biggl(\prod_{j\neq j_0}
\int_{\xi}^\infty\Phi_{\scrP_j}(\xi')\,d\xi'\biggr)
\,d\xi\,d\vecw\,d\mm(\psi).
\end{align*}
The fact that \eqref{Productformulasthmres3} 
holds for \textit{all} (and not just almost all) $\xi,\vecw,\psi$
follows by continuity, see Lemma \ref{kkgrelLEM}.

This completes the proof of Theorem \ref{Productformulasthm},
and thus also the proof of Theorem \ref{Productformulasthm2}.
\hfill$\square$

\begin{remark}\label{oldprodformulacompareREM}
In the special case
when all the grids appearing in $\scrP$ are incommensurable
and $\nbar_{\scrP}=1$,
the product formulas in 
Theorem \ref{Productformulasthm}
agree with the formulas
\cite[(5.5), (5.6), (5.12)]{jMaS2013a}.
This is immediately verified using the 
formulas in Remark \ref{singlegridREM} 
and Remark \ref{incommensurableREM},
and the fact that 
$\ck^{\g}(\xi,-\vecw',\psi')=\ck^{\g}(\xi,\vecw',\psi')$,
since $\scrP_{j_{\psi'}}$ is a single grid.
(The last mentioned symmetry relation follows from
Remark \ref{singlegridREM} and 
\cite[Remark 4.5]{jMaS2010a}\footnote{Note that
the function ``$\Phi(\xi,\vecw)$'' in \cite{jMaS2008,jMaS2010b} is the same as
``$\Phi_{\vecalf}(\xi,\vecw,\vecz)$ with $\vecalf\notin\Q^d$''
in \cite{jMaS2010a}.}.
In fact, in dimension $d\geq3$ we have for our \textit{general} $\scrP$ that
$k^{\g}(\xi,\vecw,\psi)$ only depends 
on $\xi,\|\vecw\|,\psi$,
by \cite[Lemma 3.18]{jMaS2019}.)
\end{remark}

We next give the proof of Theorem \ref{indepmergeTHM}.
We will here stay close to the notation used in the statement of that
theorem, thus, e.g., we return to parametrizing the equivalence classes in $\Psi$
by $c\in C_\Psi$ instead of $j\in\{1,\ldots,N\}$.

\noindent
\textit{Proof of Theorem \ref{indepmergeTHM}.}
Since the set $A$ may be decomposed as the disjoint union of the subsets
$A\cap(\R_{>0}\times\{\psi\}\times\US)$ with $\psi$ running through $\Psi$,
it suffices to prove the statements of the theorem %
when $A$ is a subset of $\R_{>0}\times\{\psi_0\}\times\US$ for some fixed $\psi_0\in\Psi$.
In this case, for any given $\psi'\in\Psi$,
it follows from
the definition of $\tc_{\psi'}$ that
the right hand side of \eqref{indepmergeCORres2} equals
\begin{align}\label{indepmergeCORpf1}
\PP\Bigl(\txi_{[\psi_0],\psi'}<\txi_{c,\psi'}\:\forall c\neq[\psi_0]
\text{ and }\langle\txi_{[\psi_0],\psi'},\tpsi_{[\psi_0],\psi'},\tvecV_{[\psi_0],\psi'}\rangle\in A\Bigr).
\end{align}
Using the definition of the random triples
$\langle \txi_{c,\psi'},\tpsi_{c,\psi'},\tvecV_{c,\psi'}\rangle$, %
the above probability can be expressed as follows, if $\psi'\not\sim\psi_0$:
\begin{align*}
&\int_{A}\cp^{(\psi_0)}(\vecV';\xi,\vecV)
\biggl(\int_{(\xi,\infty)\times[\psi']\times\US}\cp^{(\psi'\to\psi_1)}(\vecV'',\vecV';\xi_1,\vecVone)\,d\xi_1\,d\mm_{[\psi']}(\psi_1)\,d\sigma(\vecVone)
\biggr)
\\
&\hspace{30pt}
\times\biggl(\prod_{c\neq[\psi_0],[\psi']}\int_{(\xi,\infty)\times c\times\US}\cp^{(\psi_1)}(\vecV';\xi_1,\vecVone)
\,d\xi_1\,d\mm_{c}(\psi_1)\,d\sigma(\vecVone)\biggr)
\,d\xi\,d\mm_{[\psi_0]}(\psi)\,d\sigma(\vecV)
\\
&=\int_A\cp^{(\psi_0)}(\vecV';\xi,\vecV)
\biggl(\frac1{v_{d-1}}\int_{(\xi,\infty)\times[\psi']\times\UB}\ck(\vecw',\psi',\xi_1,\vecw_1,\psi_1)\,d\xi_1\,d\mm_{[\psi']}(\psi_1)\,d\vecw_1
\biggr)
\\
&\hspace{70pt}\times\biggl(\prod_{c\neq[\psi_0],[\psi']}\int_\xi^\infty\Phi_{\scrP_c}(\xi')\,d\xi'\biggr)
\,d\xi\,d\mm_{[\psi_0]}(\psi)\,d\sigma(\vecV),
\end{align*}
where $\vecw'\in\UB$ is the exit parameter determined by the fixed vectors $\vecV'',\vecV'$.
To get the last equality,
in the integral over $(\xi,\infty)\times[\psi']\times\US$ we introduced the new variable
$\vecw_1:=$ the impact parameter determined by $\vecV',\vecVone$,
and used \eqref{pbndefG001} and 
the defining property of the differential cross section
(cf.\ \cite[Sec.\ 3.4 and Lemma 3.26]{jMaS2019});
also, inside the product over $c$ we applied 
\eqref{PhiPcformula}. %

Next we apply Lemma \ref{kkgrelLEM} to identify the integral over
$(\xi,\infty)\times[\psi']\times\UB$,
and in the integral over $A$ we 
substitute $\vecw_2:=$ the impact parameter determined by $\vecV',\vecV$,
and use \eqref{pgendef001}.
It follows that the above expression equals
\begin{align*}
\frac1{\nbar_{[\psi']}\,v_{d-1}^2}\int_{\oA_{\vecV'}}\ck(\xi,\vecw_2,\psi_0) \ck(\xi,-\vecw',\psi')
\biggl(\prod_{c\neq[\psi_0],[\psi']}\int_\xi^\infty\Phi_{\scrP_c}(\xi')\,d\xi'\biggr)\,
\,d\xi\,d\mm_{[\psi_0]}(\psi)\,d\vecw_2,
\end{align*}
where $\oA_{\vecV'}$ is the set of triples $(\xi,\psi,\vecw_2)$ which 
arise as images of the $(\xi,\psi,\vecV)\in A$.
Finally, using $\mm_{[\psi_0]}=(\nbar_{\scrP}/\nbar_{[\psi_0]})\,\mm\big|_{[\psi_0]}$
and the product formula \eqref{Productformulasthmres1och2} in Theorem \ref{Productformulasthm},
and our assumption that $A\subset\R_{>0}\times\{\psi_0\}\times\US$
($\Rightarrow \oA_{\vecV'}\subset\R_{>0}\times\{\psi_0\}\times\UB$),
the above expression becomes
\begin{align*}
\frac1{v_{d-1}}\int_{\oA_{\vecV'}} k(\vecw',\psi',\xi,\vecw_2,\psi)\,d\xi\,d\mm(\psi)\,d\vecw_2
=\int_A p^{(\psi'\to\psi)}(\vecV'',\vecV';\xi,\vecV)\,d\xi\,d\mm(\psi)\,d\sigma(\vecV).
\end{align*}
By \eqref{Markovprocexpl2} in Theorem \ref{limitflightprocessexplTHM},
this equals the left hand side of \eqref{indepmergeCORres2}!

In the case $\psi'\sim\psi_0$, the probability in \eqref{indepmergeCORpf1} instead equals
\begin{align*}
\int_A\cp^{(\psi'\to\psi_0)}(\vecV'',\vecV';\xi,\vecV)\,\biggl(\prod_{c\neq[\psi_0]}
\int_\xi^\infty\Phi_{\scrP_c}(\xi')\,d\xi'\biggr)
\,d\xi\,d\mm_{[\psi_0]}(\psi)\,d\sigma(\vecV),
\end{align*}
and by the same type of computation as above
(but easier, and using the 
first equality in \eqref{Productformulasthmres1och2} instead of the second),
this is again seen to be equal to the left hand side of \eqref{indepmergeCORres2}.
This completes the proof of \eqref{indepmergeCORres2}.

The proof of \eqref{indepmergeCORres1} is completely similar, but easier:
Again assuming $A\subset\R_{>0}\times\{\psi_0\}\times\US$, 
the right hand side of \eqref{indepmergeCORres1} equals
\begin{align*}
&\PP\Bigl(\txi_{[\psi_0]}<\txi_{c}\:\forall c\neq[\psi_0]
\text{ and }\langle\txi_{[\psi_0]},\tpsi_{[\psi_0]},\tvecV_{[\psi_0]}\rangle\in A\Bigr)
\\
&\hspace{40pt}
=\int_A\cp^{(\psi_0)}(\vecV';\xi,\vecV)\,\biggl(\prod_{c\neq[\psi_0]}
\int_\xi^\infty\Phi_{\scrP_c}(\xi')\,d\xi'\biggr)
\,d\xi\,d\mm_{[\psi_0]}(\psi)\,d\sigma(\vecV),
\end{align*}
and by the same type of computation as above,
using now \eqref{Productformulasthmres3},
this is seen to be equal to the left hand side of \eqref{indepmergeCORres1}.
\hfill$\square$

\vspace{10pt}

\noindent
\textit{Proof of Corollary \ref{ProductformulasthmCOR}.}
Given $\xi>0$, we apply \eqref{indepmergeCORres1} in Theorem \ref{indepmergeTHM}
to the set $A:=(\xi,\infty)\times\Psi\times\US$.
Then by \eqref{Markovprocexpl1} and \eqref{PhiPdef},
the left hand side of \eqref{indepmergeCORres1}
equals $\int_\xi^\infty\Phi_{\scrP}(\xi')\,d\xi'$.
On the other hand, by the definition of the 
(mutualy independent) random variables
$\langle \txi_{c},\tpsi_{c},\tvecV_{c}\rangle$, 
and the definition of $\tc$,
the right hand side of 
\eqref{indepmergeCORres1}
equals $\prod_{c\in C_\Psi}\PP(\txi_c>\xi)$,
which by 
\eqref{Markovprocexpl1} (for $\scrP_c$) and \eqref{PhiPcformula}
equals $\prod_{c\in C_\Psi}\int_\xi^\infty\Phi_{\scrP_{c}}(\xi')\,d\xi'$.
Hence:
\begin{align*}
\int_\xi^\infty\Phi_{\scrP}(\xi')\,d\xi'
=\prod_{c\in C_\Psi}\int_\xi^\infty\Phi_{\scrP_{c}}(\xi')\,d\xi',
\end{align*}
and the corollary follows by differentiation,
using the fact that the functions $\Phi_{\scrP}$ and $\Phi_{\scrP_c}$ are 
nonnegative and continuous.
\hfill$\square$

\subsection{Asymptotic estimates for the free path length distribution}
\label{asymptSEC}

We will now prove 
Theorem \ref{AsymptTHM}.
This theorem will be a quite easy consequence of the
following result.
\begin{prop}\label{AsymptPROP}
There exist constants $0<c_1<c_2$, which only depend on $\scrP$,
such that
\begin{align}\label{AsymptPROPres}
c_1\xi^{-N}<\int_\xi^{\infty}\Phi_{\scrP}(\xi')\,d\xi'< c_2\xi^{-N},
\qquad\forall \xi\geq1.
\end{align}
\end{prop}
\begin{proof}
By Corollary \ref{ProductformulasthmCOR},
$\int_\xi^{\infty}\Phi_{\scrP}(\xi')\,d\xi'
=\prod_{j=1}^N\int_\xi^{\infty}\Phi_{\scrP_j}(\xi')\,d\xi'$;
thus, if we can prove the proposition when $N=1$,
it follows for general $N$.
Hence from now on we assume $N=1$.

By \eqref{genoPhiPhirel4} and the definition of $k^{\g}$, we have
\begin{align}\notag
\int_\xi^\infty\Phi_{\scrP}(\xi')\,d\xi'
&=\frac1{v_{d-1}}\int_{[\xi,\infty)\times\UB\times\Psi} k^{\g}(\xi',\vecw,\psi)\,d\xi'\,d\vecw\,d\mm(\psi)
\\\label{AsymptPROPpf10}
&=\kappa^{\g}\bigl([\xi,\infty)\times\UB\times\Sigma\bigr)
=\overline{\omega^{\g}}\bigl(\bigl\{x\in\XX\col J(x)\cap(\fZ_\xi\times\Sigma)=\emptyset\bigr\}\bigr),
\end{align}
where in the last equality we used
\eqref{kappagwppsipdeffact1gen},
$\kappa^{\g}(\{\undef\})=0$, and the definition of the map $\vecz$.
Now since $N=1$ we have
$\XX=\XX_1$ and $\omega^{\g}=\omega_1^{\g}$,
and $\overline{\omega^{\g}}$ is $\SL_d(\R)$-invariant
by Lemma \ref{rhosoomegaeqnuLEM};
in particular $\overline{\omega^{\g}}$ is invariant under
(right) multiplication by
$D:=D_{\xi^{1/d}}=\diag\bigl(\xi^{1-\frac1d},\xi^{-\frac1d},\ldots,\xi^{-\frac1d}\bigr)$,
and so we may replace $J(x)$ by $J(xD)$ in the last expression in
\eqref{AsymptPROPpf10}.
Using also $J(xD)=J(x)D$,
we see that this is equivalent with replacing,
in the same expression,
$\fZ_\xi$ by
\begin{align*}
\tfZ_\xi:=\fZ_\xi D^{-1}=(0,\xi^{1/d})\times\scrB_{\xi^{1/d}}^{d-1}.
\end{align*}
Using also Lemma \ref{oomegapartDEFlem}
for $\omega^{\g}=\omega_1^{\g}$,
and the definition of $J$ in \eqref{JmapFULL},
we conclude:
\begin{align}\label{AsymptPROPpf1}
\int_\xi^\infty\Phi_{\scrP}(\xi')\,d\xi'
=\int_{F_d}
\omega^{\g}\bigl(\bigl\{U\in\TT_1^d\col
c_{1,i}\bigl(\Z^d+\tr_i(U))A\cap\tfZ_\xi=\emptyset\text{ for }i=1,\ldots,r_1\bigr\}\bigr)\,d\nu(A).
\end{align}
Here %
$\tr_i(U)\in(\R/\Z)^d$, and in the above expression we use 
the convention that 
for any $\vecu\in(\R/\Z)^d$, 
$\Z^d+\vecu$ denotes the inverse image of $\vecu$ in $\R^d$
(that is, $\Z^d+\vecu:=\Z^d+\vecv$, where $\vecv$ is any lift of $\vecu$ to $\R^d$).

To get an upper bound, we fix an arbitrary $i\in\{1,\ldots,r_1\}$,
and note that the previous integral is bounded above by
\begin{align*}
\int_{F_d}
[\tr_{i*}(\omega^{\g})]\bigl(\bigl\{\vecu\in(\R/\Z)^d\col  c_{1,i}(\Z^d+\vecu)A\cap\tfZ_\xi=\emptyset
\bigr\}\bigr)\,d\nu(A).
\end{align*}
Here note that the measure
$\tr_{i*}(\omega^{\g})$ is %
invariant under translations,
since, as we saw in the proof of Lemma \ref{mugtranslinvLEM}, 
$\omega^{\g}=\omega_1^{\g}$ is invariant under $X\mapsto X+\tvecc_1\vecv$
for all $\vecv\in\R^d$, and $\r_i(\tvecc_1\vecv)=c_{1,i}^{-1}\vecv$ with $c_{1,i}>0$.
Hence $\tr_{i*}(\omega^{\g})$ is in fact the Lebesgue probability measure
on $(\R/\Z)^d$, and it follows that 
in the last integral equals
\begin{align*}
\int_{\ASL_d(\Z)\bs\ASL_d(\R)}I\bigl(j(x)\cap c_{1,i}^{-1}\tfZ_\xi=\emptyset\bigr)\,d\tnu(x),
\end{align*}
where $\tnu$ is the $\ASL_d(\R)$ invariant probability measure on
$\ASL_d(\Z)\bs\ASL_d(\R)$ and 
$j$ is the standard identification between
$\ASL_d(\Z)\bs\ASL_d(\R)$ and the space of grids of density one in $\R^d$,
that is, $j(\ASL_d(\Z)g):=\Z^dg$ for any $g\in\ASL_d(\R)$.
This integral is known to be asymptotically equal to $c\,\xi^{-1}$ as $\xi\to\infty$
with an explicit constant $c$
\cite[Theorem 3.5 and (3.21), (3.22)]{jMaS2010b}.
Hence the upper bound in
\eqref{AsymptPROPres} (for $N=1$) holds for an appropriate constant $c_2$.

It remains to prove the lower bound.
We will need the following lemma.
\begin{lem}\label{AsymptPROPauxLEM1}
There exists a constant $c=c(d)>0$
such that for any $B\geq1$, there exists a subset 
$M_B\subset F_d$
with $\nu(M_B)\geq cB^{-d}$ such that for every 
$A\in M_B$ there exists a vector $\vecell\in\R^d$
satisfying $\|\vecell\|\geq B$
and $\Z^dA\subset \Z\vecell+\vecell^\perp$.
\end{lem}
(Here $\vecell^\perp$ denotes the hyperplane orthogonal to $\vecell$.)
\begin{proof}
Consider the Siegel set
\begin{align*}
\scrS_d:=\bigl\{u\,\diag(a_1,\ldots,a_d)\,k\col u\in F_U,\:0<a_{j+1}\leq\tfrac2{\sqrt3}a_j\: (j=1,\ldots,d-1),\: k\in\SO(d)\bigr\},
\end{align*}
where $F_U$ is the set of all unipotent upper triangular matrices in $\SL_d(\R)$
all of whose elements above the diagonal belong to $[-\frac12,\frac12]$.
It is known that $\scrS_d$ contains a fundamental region for $\SL_d(\Z)\bs\SL_d(\R)$,
and on the other hand,
the number of points in $\scrS_d$ which project to any given point in $\SL_d(\Z)\bs\SL_d(\R)$
is bounded above by a constant $K=K(d)$ %
(see \cite{aB69}).
Note that for any $A=u\,\diag(a_1,\ldots,a_d)\,k\in\scrS_d$ we have
$\vece_1 u\subset\vece_1+\vece_1^\perp$ and $\vece_j u\subset\vece_1^\perp$
for all $j\in\{2,\ldots,d\}$;
hence the vector $\vecell:=a_1\vece_1k$ satisfies
$\Z^dA\subset \Z\vecell+\vecell^\perp$.
Therefore, given $B\geq1$, if we set
\begin{align*}
M_B':=\{u\,\diag(a_1,\ldots,a_d)\,k\in\scrS_d\col a_1\geq B\},
\end{align*}
then for every $A\in M_B'$ %
there exists $\vecell\in\R^d$ satisfying 
$\|\vecell\|\geq B$
and $\Z^dA\subset \Z\vecell+\vecell^\perp$.
Note also that this property only depends on the lattice $\Z^dA$,
that is, it depends only on the coset $\SL_d(\Z)\,A$.
Hence if we set $M_B:=F_d\cap\bigl(\SL_d(\Z)\,M_B'\bigr)$,
then for every $A\in M_B$ there again exists 
$\vecell\in\R^d$ satisfying 
$\|\vecell\|\geq B$
and $\Z^dA\subset \Z\vecell+\vecell^\perp$.
Furthermore,
\begin{align*}
\nu(M_B)\geq K^{-1}\nu(M_B')
\gg\int_B^\infty a_1^{-d-1}\,da_1\gg B^{-d},
\end{align*}
where the lower bound on $\nu(M_B')$
follows from a standard integration
using the expression for $\nu$ in terms of the
Iwasawa decomposition of $\SL_d(\R)$
(see, e.g., 
\cite[(2.11) and (2.12)]{aS2011i},
and use also the fact that for $B$ sufficiently large,
the set of $A'\in\scrS_{d-1}$ with $a_1\leq\frac2{\sqrt3}B^{d/(d-1)}$
has Haar measure bounded away from zero independently of $B$).
\end{proof}
Now let $M_B\subset F_d$ be the set provided by
Lemma \ref{AsymptPROPauxLEM1}, applied with
\begin{align}\label{AsymptPROPpf4}
B:=4\Bigl(\sum_{i=1}^{r_1}c_{1,i}^{-1}\Bigr)C_\xi
\qquad\text{where }\:
C_\xi:=\sup\bigl\{\bigl|\vecv\cdot\vecw\bigr|\col\vecv\in\US,\:\vecw\in\tfZ_\xi\bigr\}.
\end{align}
We claim that for every $A\in M_B$,
\begin{align}\label{AsymptPROPpf2}
\omega^{\g}\bigl(\bigl\{U\in\TT_1^d\col
c_{1,i}\bigl(\Z^d+\tr_i(U))A\cap\tfZ_\xi=\emptyset\text{ for }i=1,\ldots,r_1\bigr\}\bigr)
\geq\tfrac12.
\end{align}
To prove this, we first use that the left hand side of
\eqref{AsymptPROPpf2} is
\begin{align}\label{AsymptPROPpf3}
\geq 1-\sum_{i=1}^{r_1}\omega^{\g}\bigl(\bigl\{U\in\TT_1^d\col c_{1,i}(\Z^d+\tr_i(U))A\cap \tfZ_\xi\neq\emptyset\bigr\}\bigr).
\end{align}
Now $A\in M_B$ implies that there is a vector 
$\vecell\in\R^d$ satisfying $\|\vecell\|\geq B$
and $\Z^dA\subset \Z\vecell+\vecell^\perp$.
It follows that for every $U\in\TT_1^d$ we have
$(\Z^d+\tr_i(U))A\subset(\Z+x)\vecell+\vecell^\perp$,
where $x:=\|\vecell\|^{-2}\vecell\cdot(\vecv A)$
for some vector $\vecv\in\Z^d+\tr_i(U)$.
Note that for every $\vecm\in\Z^d$ we have
$\vecm A\in\Z\vecell+\vecell^\perp$ and so $\|\vecell\|^{-2}\vecell\cdot (\vecm A)\in\Z$;
this implies that the
map $\vecv\mapsto\|\vecell\|^{-2}\vecell\cdot(\vecv A)$ from $\R^d$ to $\R$
induces a well-defined map 
$\delta:(\R/\Z)^d\to\R/\Z$. %
Using this map, the inclusion just mentioned can be written:
\begin{align*}
(\Z^d+\tr_i(U))A\subset(\Z+\delta(\tr_i(U)))\vecell+\vecell^\perp,
\qquad\forall U\in\TT_1^d.
\end{align*}
(Here for any $y\in\R/\Z$, we write $\Z+y$ for the inverse image of $y$ in $\R$.)
Hence we conclude that the expression in %
\eqref{AsymptPROPpf3} is
\begin{align}\label{AsymptPROPpf5}
\geq1-\sum_{i=1}^{r_1}
\bigl[(\delta\circ\tr_i)_*(\omega^{\g})\bigr]
\bigl(\bigl\{y\in\R/\Z\col (\Z+y)\vecell+\vecell^{\perp}\cap c_{1,i}^{-1}\tfZ_{\xi}\neq\emptyset\bigr\}\bigr).
\end{align}
But we have noted that $\tr_{i*}(\omega^{\g})$ is Lebesgue measure on $(\R/\Z)^d$;
hence $(\delta\circ\tr_i)_*(\omega^{\g})$ is Lebesgue measure on $\R/\Z$.
Note also that for any $x\in\R$, the set 
$c_{1,i}^{-1}\tfZ_{\xi}$ intersects $x\vecell+\vecell^\perp$
if and only there is some 
$\vecw\in c_{1,i}^{-1}\tfZ_{\xi}$
satisfying $\|\vecell\|^{-2}\vecell\cdot\vecw=x$.
Furthermore, because of our choice of $C_\xi$ in \eqref{AsymptPROPpf4}, we have
$\bigl|\|\vecell\|^{-2}\vecell\cdot\vecw\bigr|\leq C_\xi / \bigl(c_{1,i}\|\vecell\|\bigr)$
for every $\vecw\in c_{1,i}^{-1}\tfZ_{\xi}$.
Therefore, the expression in \eqref{AsymptPROPpf5} is
\begin{align*}
&\geq1-\sum_{i=1}^{r_1}
\Leb\,\biggl(\biggl\{y\in\R/\Z\col (\Z+y)\cap
\biggl[-\frac{C_\xi}{c_{1,i}\|\vecell\|},\frac{C_\xi}{c_{1,i}\|\vecell\|}\biggr]\neq\emptyset\biggr\}\biggr)
=1-\sum_{i=1}^{r_1}\frac{2C_\xi}{c_{1,i}\|\vecell\|}\geq\frac12,
\end{align*}
where the last bound holds since $\|\vecell\|\geq B$ and
because of our choice of  $B$ in \eqref{AsymptPROPpf4}
(and the preceding equality holds since 
$C_\xi / \bigl(c_{1,i}\|\vecell\|\bigr)\leq\frac14\leq\frac12$ for every $i$).
This completes the proof of \eqref{AsymptPROPpf2}.

Using \eqref{AsymptPROPpf2} and 
\eqref{AsymptPROPpf1}, we conclude
\begin{align*}
\int_\xi^\infty\Phi_{\scrP}(\xi')\,d\xi'\geq\tfrac12\nu(M_B)
\gg B^{-d}\gg\xi^{-1},
\end{align*}
where we first used the bound from 
Lemma \ref{AsymptPROPauxLEM1},
and then the fact that $B\ll C_\xi\ll\xi^{1/d}$.
Hence the lower bound in
\eqref{AsymptPROPres} (for $N=1$) holds for an appropriate constant $c_1$,
and Proposition \ref{AsymptPROP} is proved.
\end{proof}

\begin{proof}[Proof of Theorem \ref{AsymptTHM}]
This is immediate from Proposition \ref{AsymptPROP}
and the fact that $\Phi_{\scrP}(\xi)$ is a nonnegative,
decreasing function.
Indeed, the upper bound for $\xi\geq2$ follows using
$(\xi/2)\Phi_{\scrP}(\xi)\leq\int_{\xi/2}^{\xi}\Phi_{\scrP}(\xi')\,d\xi'$;
for $1\leq\xi\leq2$
one may e.g.\ use
$\Phi_{\scrP}(\xi)\leq\nbar_{\scrP}\,v_{d-1}$.
For the lower bound, let $c_1,c_2$ be as in 
Proposition \ref{AsymptPROP}, and fix $a>1$ so that 
$c_1-c_2a^{-N}>0$.
Then for any $\xi\geq1$,
\begin{align*}
(a-1)\xi\Phi_{\scrP}(\xi)
\geq\int_\xi^{a\xi}\Phi_{\scrP}(x)\,dx
=\int_\xi^{\infty}\Phi_{\scrP}(x)\,dx-\int_{a\xi}^\infty\Phi_{\scrP}(x)\,dx
>c_1\xi^{-N}-c_2(a\xi)^{-N},
\end{align*}
which implies that
the lower bound in \eqref{genoPhiPhirel7disc4}
holds with $c_2^{\new}=(a-1)^{-1}(c_1-c_2a^{-N})$.
\end{proof}

\begin{proof}[Proof of Corollary \ref{Asympt2momentCOR}]
By \eqref{PhioPhirel},
$-\bigl(\nbar_{\scrP}\,v_{d-1}\bigr)^{-1}\Phi_{\scrP}(\xi)$ is a primitive
function of $\oPhi_{\scrP}(\xi)$;
hence by
integration by parts,
\begin{align*}
\int_0^A\xi^2\,\oPhi_{\scrP}(\xi)
=\frac1{\nbar_{\scrP}\,v_{d-1}}\Bigl(-A^2\Phi_{\scrP}(A)+2\int_0^A\xi\Phi_{\scrP}(\xi)\,d\xi\Bigr),
\end{align*}
for any $A>0$.
Now the corollary follows by letting $A\to\infty$ and using Theorem \ref{AsymptTHM}.
\end{proof}

\newpage

\section*{Index of notation} %

For the convenience of the reader, we include an
index of some of the most important notation.
Note that %
in Section \ref{UNIPOTAPPLsec},
some of the notation listed below (for example, ``$\Gamma\,$'', ``$\Gamma_j$'' and ``$\tM$'')
is used in a slightly different and more general way; this is explained in the beginning of that section.

\begin{center}
\begin{footnotesize}
\begin{longtable}{llr}
$\scrB_r^d$ & open ball in $\R^d$ with center $\bn$ and radius $r$ & \pageref{scrBrhodDEF}
\\
$\scrB_r^d(\vecx)$ & open ball in $\R^d$ with center $\vecx$ and radius $r$ & \pageref{scrBRxdef}
\\
$\C_b(S)$ & the space of bounded continuous functions on $S$
& \pageref{HOMDYNintrononunifTHM}
\\
$C_\Psi$ & family of equivalence classes in $\Psi$
& \pageref{CPsidef}
\\
$\tvecc_j$
& the vector $\bigl(c_{j,1}^{-1}\: \cdots \: c_{j,r_j}^{-1}\bigr)\trans$ in $\R^{r_j}$
& \pageref{cjdef}
\\
$\vecc_j^\psi$ & the vector $c_\psi\tvecc_j$ in $\R^{r_j}$
& \pageref{cjdef}
\\
$c_{\psi}$ & 
fixed positive real numbers 
such that \eqref{LpsiDEF}--\eqref{GENPOINTSET1} hold
& \pageref{LpsiDEF}
\\
$D_\rho$ & diagonal matrix $\diag(\rho^{d-1},\rho^{-1},\cdots,\rho^{-1})$ & \pageref{Drhodef}
\\
$\fD_{\scrS}$ & the set $\bigcup_{i<j}\bigl\{(M_1,\ldots,M_N)\in G'\col M_iM_j^{-1}\in\scrS\bigr\}$
& \pageref{DSdef}
\\
$\scrE$ & subset of $\scrP$ of density zero (``exceptional points'') & \pageref{scrEDEF}, \pageref{scrEass}
\\
$\vece_k$ & the $k$th standard unit vector in $\R^d$ (or in $\R^r$ or $\R^{r_j}$) 
& \pageref{vece1DEF}, \pageref{vecekrowDEF}, \pageref{vecekcolDEF}
\\
$G$ & $\S_{r_1}(\R)\times\cdots\times\S_{r_N}(\R)$ & \pageref{GXDEF}
\\
$G_j$ & $\S_{r_j}(\R)$ & \pageref{GjDEF}
\\
$G'$ & $\SL_d(\R)^N$ (a subgroup of $G$) & \pageref{Gpdef}
\\
$g_0^{(\vecq)}$ & $\I_{U^{(\vecq)}}\,\tM$
& \pageref{g0qdef}
\\[2pt]
$\I_U$ & 
for $U\in\M_{r\times d}(\R)$, $\I_U:=(\I,U)\in\S_r(\R)$;
for $U\in\prod_{j=1}^N\M_{r_j\times d}(\R)$,
$\I_U:=(\I_{U_1},\ldots,\I_{U_N})$
& \pageref{IUdef}, \pageref{IVdef}
\\[2pt]
$I(\cdot)$ & indicator function: $I(P)=1$ if statement $P$ is true, otherwise $I(P)=0$
& \pageref{indicatorfunction}
\\[2pt]
$i_\psi$ & the second coordinate of $\psi$ (for any $\psi\in\Psi$)
&\pageref{ipsijpsidef}
\\[2pt]
$J_0$ & a map from $\XX$ to $N_s(\R^d)$
& \pageref{Gammagpointset}
\\[2pt]
$J$ & a map from $\XX$ to $N_s(\scrX)$
& \pageref{JmapFULL}
\\[2pt]
$J_\psi$ & a map from $\XX^\psi$ to $N_s(\scrX)$
& \pageref{Jpsidef}
\\[2pt]
$j_\psi$ & the first coordinate of $\psi$ (for any $\psi\in\Psi$)
&\pageref{ipsijpsidef}
\\[2pt]
$k$ & transition kernel, defined in \eqref{transkDEF}--\eqref{transkDEF2}
& \pageref{transkDEF}
\\[2pt]
$k^{\g}$ & transition kernel, defined in \eqref{transkDEF3}--\eqref{transkDEF4}
& \pageref{transkDEF3}
\\[2pt]
$\fL(S)$  & for $\emptyset\neq S\subset\TT^{r_j}$,
$\fL(S):=\overline{\big\langle \pi^{-1}(S)\big\rangle}^{\hspace{3pt}\circ}$
(a rational subspace of $\R^{r_j}$)
&\pageref{fLdef}
\\
$\scrL_\psi$ & 
a grid in $\R^d$. %
After \eqref{LpsiDEF} we have
$\scrL_\psi=c_{\psi}(\Z^d+\vecw_\psi)M_{j_\psi}$ %
& \pageref{Ppsidec1}, \pageref{LpsiDEF}
\\[2pt]
$L_j^{(\vecq)}$ & 
$L_j^{(\pi(U_j^{(\vecq)}))}$
& 
\pageref{Ljqfirstdef},
\pageref{Ljqdef}
\\[2pt]
$L_j^{(V)}$ & $\fL(V_1,\ldots,V_d)$, for $V=(V_1,\ldots,V_d)\in\TT_j^d$ 
& \pageref{LjDEF}
\\[2pt]
$L_j^{\psi}$ & the rational subspace of $\R^{r_j}$ given by \eqref{Ljpsi0DEF}
& \pageref{Ljpsi0DEF}
\\[2pt]
$L_j$ & the rational subspace of $\R^{r_j}$ defined in \eqref{trueLjDEF}
& \pageref{trueLjDEF}
\\
$\mm$ & probability measure on $\Psi$ or on $\Sigma$
& \pageref{mmpsiDEF}, \pageref{mmDEF2}, \pageref{mmdef}
\\
$\mm_c$ & probability measure on an equivalence class $c\subset\Psi$
& \pageref{mcDEF}
\\[2pt]
$M_j$ & $M_1,\ldots,M_N$ are fixed elements in $\SL_d(\R)$ %
such that \eqref{LpsiDEF}--\eqref{GENPOINTSET1} hold
& \pageref{PSIdef}
\\[2pt]
$\tM$ & $(M_1,\ldots,M_N)$ (an element in $G'$, thus in $G$) 
& \pageref{tMdef}
\\
$N(\scrX)$ & the set of locally finite counting measures on $\scrX$ & \pageref{NscrXDEF}
\\
$N_s(\scrX)$ & the set of simple measures in $N(\scrX)$ & \pageref{NscrXDEF}
\\
$\nbar_{\scrP}$ & the asymptotic density of the point set $\scrP$; $\nbar_{\scrP}=\sum_{\psi\in\Psi}\nbar_{\psi}$ 
& \pageref{nbarPDEF1}, \pageref{nbarPDEF0}
\\
$\nbar_{\psi}$ & the asymptotic density of $\scrL_\psi$.
(After \eqref{LpsiDEF}: $\nbar_\psi=c_{\psi}^{-d}$.) & \pageref{nbarDEF0}
\\
$\scrO_j^{(\vecq)}$ & $\scrO_j^{(\pi(U_j^{(\vecq)}))}$ & \pageref{OjqDEF}
\\[2pt]
$\scrO_j^{(V)}$ & the subset of  $\bigl(\SS_j^{(V)}\bigr)^d$  %
given by \eqref{OjVDEF}
&\pageref{OjVDEF}
\\
$\scrO_j^{\psi}$ & the subset of  $\TT_j^d$ given in \eqref{Ojpsi0DEF}
& \pageref{Ojpsi0DEF}
\\
$\scrO_j$ & the subset of $\TT_j^d$ given in \eqref{OjDEF}
& \pageref{OjDEF}
\\
$\scrP$ & the scatterer configuration; a fixed union of grids in $\R^d$ & \pageref{Pform0}
\\
$\tP$ & $\{(\vecp,{\vs}(\vecp))\col\vecp\in\scrP\}$ (a subset of $\scrX$)
& \pageref{tPDEF}
\\
$\tP_\vecq$ & $\tP\setminus\{(\vecq,\vs(\vecq))\}$ if $\vecq\in\scrP$, otherwise $\tP$
& \pageref{repPqdef}
\\
$\scrP_T(\rho)$ & the set $\scrP\cap\scrB^d_{T\rho^{1-d}}\setminus\scrE$
& \pageref{repASS:KEY}
\\
$P(S)$ & the set of Borel probability measures on $S$ (for any topological space $S$) 
& \pageref{PSdef} 
\\
$\Pac(\US)$ & the set of $\lambda\in P(\US)$ which are absolutely continuous with respect to $\sigma$
& \pageref{PacUSdef}
\\
$P(\TT_j^d)'$ & the subset of $\SL_d(\Z)$-invariant measures in $P(\TT_j^d)$ & \pageref{PTjdprimDEF}
\\
$p^{(\psi)}$ & collision kernel & \pageref{Markovprocexpl1}
\\
$p^{(\psi'\to\psi)}$ & collision kernel & \pageref{Markovprocexpl2}
\\
$\cp^{(\psi)}$ & collision kernel for the scatterer configuration $\scrP_{[\psi]}$
& \pageref{Markovprocexpl1}
\\
$\cp^{(\psi'\to\psi)}$ & collision kernel for the scatterer configuration $\scrP_{[\psi]}$ (here $[\psi']=[\psi]$)
& \pageref{Markovprocexpl2}
\\
$p_j$ & the projection map $G\to\S_{r_j}(\R)$
& \pageref{pjdef}
\\
$\tp_j$ & either of the projection maps $\XX\to\XX_j$ or $\tTT\to\TT_j^d$
& \pageref{pjdef}, \pageref{pjdef2}
\\
$p_\psi$ & 
for $\psi=(i,j)\in\Psi$,
$p_\psi$ is 
the map $\a_i\circ p_j$ from $G$ to $\ASL_d(\R)$ 
& \pageref{ppsiDEF}
\\
$\scrQ_\rho(\vecq,\vecv)$ & $(\tP_\vecq-\vecq)\,R(\vecv)\,D_\rho$
& \pageref{repXIRHOqv}
\\
$R$ & a fixed map $\US\to\SO(d)$ such that $\vecv R(\vecv)=\vece_1$, $\forall \vecv\in\US$ & \pageref{Rdef}
\\
$\r_i$ & the projection map $\M_{r\times d}(\R)\to\R^d$ which takes any matrix to its $i$th row & \pageref{riDEF}
\\
$\tr_i$ & for $1\leq i\leq r_j$, $\tr_i$ is the projection map $\TT_j^d\to(\R/\Z)^d$ induced by $\r_i$ & \pageref{triDEF}
\\
$\scrS$ & the commenturator of $\SL_d(\Z)$ in $\SL_d(\R)$ & \pageref{COMMENSURATOR}
\\
$\S_r(\R)$ & $\SL_d(\R)\ltimes\M_{r\times d}(\R)$ & \pageref{Srdefrep}
\\[3pt]
$\S_L(\R)$ & for $L$ a linear subspace of $\R^r$, $\S_L(\R):=\SL_d(\R)\ltimes L^d$, a subgroup of $\S_r(\R)$
& \pageref{SVdefrepnew}
\\[3pt]
$\SS_j^{(V)}$ & $\overline{\langle V_1,\ldots,V_d\rangle}$, a closed subgroup of $\TT_j$  & \pageref{SSjVdef}
\\[2pt]
$\SS_j^{(\vecq)}$ & $\SS_j^{(\pi(U_j^{(\vecq)}))}$ & \pageref{SSjdef} 
\\[2pt]
$\SS_j^{\psi}$ 
& $\pi(L_j^{\psi})$ 
& \pageref{Sjpsi0DEF}
\\[2pt]
$\SS_j$ 
& $\pi(L_j)$ 
& \pageref{SjDEF}
\\[2pt]
$\TT_j$ & $(\R/\Z)^{r_j}$
& \pageref{TTjdef}
\\
$\TT_j^d$ & $\TT_j\times\cdots\times\TT_j=\M_{r_j\times d}(\R/\Z)$ & \pageref{TTjdef}
\\
$\tTT$ & $\TT_1^d\times\TT_2^d\times\cdots\times\TT_N^d$ & \pageref{tTTdef2}
\\
$U_j^{(\vecq)}$ & the $r_j\times d$ matrix with row vectors
$\vecw_{j,i}-c_{j,i}^{-1}\vecq M_j^{-1}$ ($i=1,\ldots,r_j$)
&\pageref{Ujqdef}
\\
$U^{(\vecq)}$ & $(U_1^{(\vecq)},\cdots,U_N^{(\vecq)})$ & \pageref{tUqdef}
\\
$v_{d-1}$ & $\vol(\UB)$ & \pageref{vdm1DEF}
\\
$\vecw_{\psi}$ & 
fixed vectors in $\R^d$
such that \eqref{LpsiDEF}--\eqref{GENPOINTSET1} hold
& \pageref{PSIdef}
\\
$W_j$ & The $r_j\times d$ matrix with row vectors $\vecw_{j,i}$ ($i=1,\dots,r_j$)
& \pageref{Wjdef}
\\
$\scrX$ & $\R^d\times\Sigma$ & \pageref{scrXDEF}
\\
$\XX$ & $\GaG$ & \pageref{GXDEF}
\\
$\XX_j$ & $\Gamma_j\bs G_j$ & \pageref{GjDEF}
\\
$\XX^{\psi}$ & $\{\Gamma g\in\XX\col g\in G,\: \bn\in\Z^d\,p_\psi(g)\}$ & \pageref{Xpsi0def}
\\
$x$ & the embedding $\TT_j^d\to\XX_j$ defined in \eqref{xmapDEF},
or the embedding $\tTT\to\XX$
& \pageref{xmapDEF}, \pageref{xmapDEF2}
\\
$\fZ_\xi$ & The open cylinder $(0,\xi)\times\UB$ & \pageref{fZxidef}
\\
$\vecz$ & map from $N_s(\scrX)$ to $\Delta$ %
& \pageref{mapzDEF}
\\
$\Gamma$ & $\S_{r_1}(\Z)\times\cdots\times\S_{r_N}(\Z)$ & \pageref{GXDEF}
\\
$\Gamma_j$ & $\S_{r_j}(\Z)$ & \pageref{GjDEF}
\\
$\Delta$ & the set $(\fZ_\infty\times\Sigma)\sqcup\{\undef\}$
& \pageref{mapzDEF}
\\
$\iota$ & the projection $\S_r(\R)\to\SL_d(\R)$, or the reflection map $\scrX\to\scrX$ 
& \pageref{IOTAdef}, \pageref{reflmapDEF}
\\
$\tiota$ & the projection $\XX_j\to\SL_d(\Z)\bs\SL_d(\R)$
& \pageref{tIOTAdef}
\\
$\mu_{\vecq,\rho}^{(\lambda)}$ & distribution of $\scrQ_\rho(\vecq,\vecv)$ for $\vecv$
random in $(\US,\lambda)$
& \pageref{repmuqrholambdaDEF}
\\
$\mu_{\vs}$ & we have fixed a continuous map $\vs\mapsto\mu_{\vs}$ from $\Sigma$ to $P(N(\scrX))$
& \pageref{muvsmapDEF}, \pageref{muvsDEF}
\\
$\mu_{\scrX}$ & $\vol\times\mm$ & \pageref{muscrXDEF}
\\
$\mu_j^{(\vecq)}$ & measure in $P(\XX_j)$
& \pageref{muqDEF}
\\
$\mu^{(\vecq)}$ & the measure $\mu_1^{(\vecq)}\otimes\cdots\otimes\mu_N^{(\vecq)}$ in $P(\XX)$
& \pageref{muqDEF}
\\
$\mu^{\g}$ & the measure $J_*(\overline{\omega^{\g}})$ on $N_s(\scrX)$
& \pageref{mugdef}
\\
$\nu$ & Haar measure on $\SL_d(\R)$, normalized by $\nu(\SL_d(\Z)\bs\SL_d(\R))=1$
& \pageref{nuDEF}
\\
$\Xi$ & random flight process & \pageref{limitflightprocessTHM}
\\
$\pi$ & either of the projection maps 
$G_j\to\XX_j$, $G\to\XX$,
$\R^{r_j}\to\TT_j$ or $(\R^{r_j})^d\to\TT_j^d$
& \pageref{piXprojDEF} %
\\
$\Sigma$ & fixed compact metric space (the space of marks) 
& \pageref{scrXDEF}, \pageref{Sigmadef}
\\
$\sigma$ & $\vol_{\S_1^{d-1}}$, Lebesgue measure on $\US$ & \pageref{sigmadef}
\\
$\sigma^\psi$ &  $(\psi,\omega^\psi)$ (an element in $\Sigma$) & \pageref{OMEGAPSIdef}
\\
$\vs$ & function $\scrP\to\Sigma$ fixed in \eqref{vsmapDEF}
& \pageref{vsmapDEF}
\\
$\Phi_{\scrP}(\xi)$ & free path length density &
\pageref{PhiPdef}
\\
$\varphi$ & the diagonal embedding $\SL_d(\R)\to G$ & \pageref{varphiDEF}
\\
$\Psi$ & a set of indices. %
After \eqref{PSIdef}:
$\Psi=\{(j,i)\col j\in\{1,\ldots,N\},\: i\in\{1,\ldots,r_j\}\}$ & \pageref{Ppsidec1}, \pageref{PSIdef}
\\
$\tPsi$ & the set $\{\sigma^\psi\col\psi\in\Psi\}$ (subset of $\Sigma$)
& \pageref{tPsiDEF}
\\
$\psi$ & function $\scrP\to\Psi$ fixed in \eqref{psimarkdef}. But $\psi$ is also used to denote
a variable element in $\Psi$. & \pageref{psimarkdef}
\\
$\Omega$ & $\prod_{j=1}^N P(\TT_j^d)'$ & \pageref{OMEGAdef}
\\[4pt]
$\omega_j^{(\vecq)}$ & $\omega_j^{(\pi(U_j(\vecq)))}$ 
&\pageref{mujDEF}
\\[2pt]
$\omega_j^{(V)}$
& normalized restriction to $\scrO_j(V)$ of the Haar measure on $\SS_j(V)^d$
&\pageref{mujVDEF}
\\[2pt]
$\omega_j^{\psi}$ & 
normalized restriction to $\scrO_j^{\psi}$ of the Haar measure on $(\tSS_j^{\psi})^d$
& \pageref{mujpsi0def}
\\[2pt]
$\omega_j^{\g}$ & 
normalized restriction to $\scrO_j$ of the Haar measure on $(\tSS_j)^d$
& \pageref{omegagjdef}
\\[2pt]
$\omega^{(\vecq)}$ & $\bigl(\omega_1^{(\vecq)},\ldots,\omega_N^{(\vecq)}\bigr)$ (an element in $\Omega$)
& \pageref{omegaqDEF}
\\[2pt]
$\omega^{(V)}$ & $\bigl(\omega_1^{(V_1)},\ldots,\omega_N^{(V_N)}\bigr)$ (an element in $\Omega$)
&\pageref{omegaVDEF}
\\[2pt]
$\omega^\psi$
& $\bigl( \omega_1^{\psi},\ldots,\omega_N^{\psi}\bigr)$ (an element in $\Omega$)
& \pageref{OMEGAPSIdef}
\\[2pt]
$\omega^{\g}$
& $\bigl(\omega_1^{\g},\ldots,\omega_N^{\g}\bigr)$ (an element in $\Omega$)
& \pageref{omegagdef}
\\[2pt]
$\oomega$ &
for $\omega\in P(\TT_j^d)'$, $\oomega$ is
the probability measure on $\XX_j$ 
defined below \eqref{mapPTtoPXjrep};
& \pageref{oomegapartDEF}
\\
& for $\omega %
\in\Omega$, $\oomega$ is the probability measure on $\XX$ defined 
in \eqref{oomegaDEF}
& \pageref{oomegaDEF}
\end{longtable}
\end{footnotesize}
\end{center}

\section*{Funding and conflicts of interests}

The research leading to these results received funding from 
the Knut and Alice Wallenberg Foundation.
The authors have no competing interests to declare that are relevant to the content of this article.

\section*{Data availability statement}

This manuscript has no associated data.

\end{document}